\documentclass[reqno,english,11pt]{amsart}
\usepackage[]{amsmath,amssymb,amsfonts,latexsym,amsthm,enumerate}
\usepackage{amssymb,bbm,mathrsfs,tikz,stmaryrd}
\usetikzlibrary{calc,decorations.pathmorphing}
\usetikzlibrary{patterns}
\usepackage[normalem]{ulem}
\usepackage{enumerate,graphicx,paralist}
\usepackage[font=small]{caption}
\usepackage[noadjust]{cite}
\usepackage[perpage]{footmisc} %

\numberwithin{equation}{section}
\usepackage[psdextra=true,colorlinks=true, pdfstartview=FitV, linkcolor=blue,
  citecolor=blue, urlcolor=blue,pagebackref=false]{hyperref}
\usepackage[shortlabels]{enumitem}  
\usepackage{xspace}
\usepackage[letterpaper]{geometry}
\usepackage[colorinlistoftodos]{todonotes}
\presetkeys{todonotes}{inline, color=green}{}
\usepackage{comment}
\usepackage[capitalise]{cleveref}
\newtheorem{maintheorem}{Theorem}  
\newtheorem{maincoro}[maintheorem]{Corollary}

\newtheorem{theorem}{Theorem}[section]
\newtheorem*{theorem*}{Theorem}
\newtheorem{lemma}[theorem]{Lemma}
\newtheorem{claim}[theorem]{Claim}
\newtheorem{proposition}[theorem]{Proposition}
\newtheorem{observation}[theorem]{Observation}
\newtheorem{fact}[theorem]{Fact}
\newtheorem{corollary}[theorem]{Corollary}

\newtheorem{remark}[theorem]{Remark}
\newtheorem*{remark*}{Remark}

\theoremstyle{definition}{
\newtheorem{example}[theorem]{Example}
\newtheorem{definition}[theorem]{Definition}
\newtheorem*{definition*}{Definition}

\newtheorem*{question*}{Question}
\newtheorem*{example*}{Example}
\newtheorem*{examples*}{Examples}
}

\newcommand{\C}{\mathbb C}
\newcommand{\E}{\mathbb E}
\newcommand{\N}{\mathbb N}
\renewcommand{\P}{\mathbb P}

\newcommand{\R}{\mathbb R}
\newcommand{\T}{\mathbb T}
\newcommand{\Z}{\mathbb Z}

\newcommand{\cK}{\mathcal K}

\newcommand{\cC}{{\mathcal C}}
\newcommand{\cE}{{\mathcal E}}
\newcommand{\cF}{{\mathcal F}}
\newcommand{\cG}{{\mathcal G}}
\newcommand{\cH}{{\mathcal H}}
\newcommand{\cI}{{\mathcal I}}

\newcommand{\cM}{{\mathcal M}}

\newcommand{\cP}{{\mathcal P}}
\newcommand{\cQ}{{\mathcal Q}}
\newcommand{\cR}{{\mathcal R}}
\newcommand{\cS}{{\mathcal S}}

\newcommand{\cV}{{\mathcal V}}

\newcommand{\cZ}{{\mathcal Z}}
\newcommand{\fa}{\mathfrak a}
\newcommand{\fb}{\mathfrak b}
\newcommand{\fc}{\mathfrak c}
\newcommand{\fd}{\mathfrak d}
\newcommand{\ff}{\mathfrak f}
\newcommand{\fg}{\mathfrak g}
\newcommand{\fm}{\mathfrak m}
\newcommand{\fs}{\mathfrak s}

\newcommand{\fB}{\mathfrak B}

\newcommand{\fP}{\mathfrak P}

\newcommand{\fX}{\mathfrak X}
\newcommand{\fL}{\mathfrak L}
\renewcommand{\d}{\mathrm{d}}
\newcommand{\sw}{\mathsf w}
\newcommand{\sB}{\mathsf B}

\newcommand{\sL}{\mathsf L}
\newcommand{\sM}{\mathsf M}
\newcommand{\sQ}{\mathsf Q}
\newcommand{\sV}{\mathsf V}

\newcommand{\Dim}{\textsc{d} }
\newcommand{\reg}{\mathsf{reg}}

\newcommand{\Cov}{\operatorname{Cov}}
\newcommand{\Var}{\operatorname{Var}}

\newcommand{\diam}{\operatorname{diam}}
\newcommand{\dist}{\operatorname{dist}}

\newcommand{\one}{\mathbbm{1}}
\newcommand{\Zsos}{\cZ^\textsc{sos}}
\newcommand{\Zsosbar}{\bar{\cZ}^\textsc{sos}}
\newcommand{\Ztile}{\cZ^\textsc{t}}

\newcommand{\xor}{\vartriangle}
\newcommand{\gr}{\textup{\texttt{g}}}
\newcommand{\northeast}{\textup{\textsf{NE}}\xspace}
\newcommand{\east}{\textup{\textsf{E}}\xspace}
\newcommand{\southeast}{\textup{\textsf{SE}}\xspace}

\newcommand{\southwest}{\textup{\textsf{SW}}\xspace}
\newcommand{\west}{\textup{\textsf{W}}\xspace}
\newcommand{\northwest}{\textup{\textsf{NW}}\xspace}
\newcommand{\north}{\textup{\textsf{N}}\xspace}
\newcommand{\GFF}{\textsc{GFF}\xspace}
\newcommand{\SOS}{\textsc{SOS}\xspace}

\renewcommand{\restriction}{\mathord{\upharpoonright}}
\renewcommand{\epsilon}{\varepsilon}

\newcommand{\cupdot}{\mathbin{\mathaccent\cdot\cup}}
\DeclareMathOperator{\Int}{Int}

\makeatletter
\crefname{maintheorem}{Theorem}{Theorems}
\crefname{maincoro}{Corollary}{Corollaries}
\AddToHook{env/maincoro/begin}{\crefalias{maintheorem}{maincoro}}
\crefname{step}{Step}{Steps}
\crefname{part}{Part}{Parts}
\crefname{case}{Case}{Cases}
\crefname{claim}{Claim}{Claims}
\AddToHook{env/lemma/begin}{\crefalias{theorem}{lemma}}
\AddToHook{env/claim/begin}{\crefalias{theorem}{claim}}
\AddToHook{env/proposition/begin}{\crefalias{theorem}{proposition}}
\AddToHook{env/corollary/begin}{\crefalias{theorem}{corollary}}
\AddToHook{env/definition/begin}{\crefalias{theorem}{definition}}
\AddToHook{env/example/begin}{\crefalias{theorem}{example}}
\AddToHook{env/question/begin}{\crefalias{theorem}{question}}
\AddToHook{env/remark/begin}{\crefalias{theorem}{remark}}
\AddToHook{env/observation/begin}{\crefalias{theorem}{observation}}
\AddToHook{env/fact/begin}{\crefalias{theorem}{fact}}
\AddToHook{cmd/appendix/before}{%
  \crefalias{section}{appendix}%
  \crefalias{subsection}{appendix}%
}
\makeatother

\usepackage{crossreftools}
\title[Tilted Solid-On-Solid is liquid]{Tilted Solid-On-Solid is liquid: \\ scaling limit of SOS with a potential on a slope
\vspace*{-0.1in}}

\author{Beno\^{\i}t Laslier}
\address{B.\ Laslier\hfill\break
Universit\'e Paris Cit\'e and Sorbonne Universit\'e\\ CNRS\\ Laboratoire de Probabilit\'es, Statistique et Mod\'elisation\\
b\^{a}timent Sophie Germain,
 8 place Aur\'elie Nemour\\
  75205 Paris CEDEX 13, France\\
 \textit{and} \'Ecole normale sup\'erieure\\ Université PSL\\ CNRS\\ 75005 Paris\\ France.}
\email{laslier@lpsm.paris}
\author{Eyal Lubetzky}
\address{E.\ Lubetzky\hfill\break
Courant Institute\\ New York University\\
251 Mercer Street\\ New York, NY 10012, USA.}
\email{eyal@courant.nyu.edu}

\begin{document}

\begin{abstract}
\vspace{-0.1in}
The $(2+1)$\Dim Solid-On-Solid (SOS) model famously exhibits a roughening transition: on an $N\times N$ torus with the height at the origin rooted at $0$, the variance of $h(x)$, the height at~$x$, is $O(1)$ at large inverse-temperature $\beta$, vs.\ ~$\asymp \log |x|$ at small $\beta$ (as in the Gaussian free field (\GFF)).
The former---rigidity at large $\beta$---is known for a wide class of $|\nabla\phi|^p$ models ($p=1$ being~SOS) yet is believed to fail once the surface is on a slope (tilted boundary conditions). 
It is conjectured that the slope would destabilize the rigidity and induce the \GFF-type behavior of the surface at small $\beta$.
The only rigorous result on this is by Sheffield (2005): for these models of integer height functions, if the slope $\theta$ is irrational, then $\Var(h(x))\to\infty$ with $|x|$ (with no known quantitative bound).

We study a family of SOS surfaces at a large enough fixed $\beta$, on an $N\times N$ torus with a nonzero boundary condition slope $\theta$, perturbed by a potential $\sV$ on an $\epsilon_\beta$-fraction of sites (arbitrarily small). Our main result is (a) the measure on the height gradients $\nabla h$ has a weak limit $\mu_\infty$ as $N\to\infty$; and (b) the scaling limit of a sample from $\mu_\infty$ converges to a full plane \GFF. In particular, we recover the asymptotics $\Var(h(x))\sim c\log|x|$.
To our knowledge, this is the first example of a tilted $|\nabla\phi|^p$ model, or a perturbation thereof, where the limit is recovered at large $\beta$.
The proof looks at random monotone surfaces that approximate the SOS surface, and shows that (i) these form a weakly interacting dimer model, and (ii) the renormalization framework of Giuliani, Mastropietro and Toninelli (2017) leads to the \GFF limit. New ingredients are needed in both parts, including a nontrivial extension of [GMT17] from finite interactions to ones with exponential decay in the~radius.
\end{abstract}

{\mbox{}\vspace{-0.5cm}
\maketitle
}
\vspace{-1cm}

\section{Introduction}\label{sec:intro}

The Solid-On-Solid (SOS) model on a finite graph with vertices $\cV$ and edges $\cE$ is a distribution on height functions $h:\cV\to\Z$ rooted at a marked vertex $o$ (the origin) to $h(o)=0$. The probability assigned to each $h$ penalizes it for having large (in absolute value) gradients along the edges $e\in\cE$:
\[ \P_\beta(h) \propto \exp\Big[-\beta\sum_{e\in \cE}\left|\nabla h(e)\right| \Big] \,,\]
where $\nabla h(e):= h(y)-h(x)$ for $e=(x,y)\in\cE$  and the parameter $\beta>0$ is the inverse-temperature. 

The $(2+1)$\Dim SOS model takes the graph to be the $N\times N$ torus in $\Z^2$, denoted here $\Lambda_N$, where the heights are assigned to the $N^2$ unit-squares. Note that $\P_\beta$ is then translation-invariant (as the torus is vertex transitive, and rooting $h(o)=0$ has no effect when we only look at the gradients~$\nabla h$).
One associates to $h$ the surface in $\R^3$ consisting of a horizontal face of $\Z^3$ at height $h(x)$ for each $x\in\cV(\Lambda_N)$ and a minimum completion of vertical faces to make them simply connected (i.e., $|\nabla h(e)|$ faces for each $e\in\cE(\Lambda_N)$). Viewing $h$ as this set of faces, $|h|= N^2+\sum_{e\in\cE}|\nabla h(e)|$, thus
\begin{equation}\label{eq:SOS-model} \P_\beta(h) \propto \exp\big(-\beta |h|\big)\,.\end{equation}

The study of this model in statistical physics goes back to the early 1950's (\cite{BCF51,Temperley52}), pertaining to crystal formation at low temperature (large $\beta$). It further serves as an approximation of plus/minus interfaces in the 3\Dim Ising model, which resemble height functions as overhangs are microscopic in that regime. In line with predictions for 3\Dim Ising (given in the 1970's and yet unproved: cf.~\cite{Abraham86,FPS81}), the $(2+1)$\Dim SOS model undergoes a \emph{roughening transition} from being delocalized, or rough (formally defined below) at small~$\beta$ to localized, or rigid at large~$\beta$. It is widely believed that for 3\Dim Ising and its approximations, e.g., $(2+1)$\Dim SOS and more general $|\nabla\phi|^p$ models, when the interface is on a slope, the rigidity at large $\beta$ is destabilized, mirroring the small $\beta$ behavior (see \cref{fig:tilted-3D,fig:tilted-3D-ctr}).

In this work we give a first rigorous proof of the scaling limit of such a model at large $\beta$ on~a~slope: we show that the $(2+1)$\Dim SOS model perturbed by a potential on  $\epsilon_\beta N^2$ sites converges to a \GFF.
As we further explain below, it was unclear whether the scaling limit of tilted Ising-type interfaces
would in fact match the zero temperature \GFF picture, especially if the nonzero slope is \emph{rational}.

There is a vast body of works on low temperature 3\Dim Ising interfaces and $(2+1)$\Dim SOS surfaces, intrinsic to the study of crystals. In the physics literature there are numerous studies, experimental (e.g.,
\cite{ABKLLS80} comparing roughening in $^4\mathrm{He}$ crystals to 3\Dim Ising) and theoretical (e.g., \cite{WGL73} drawing evidence for the existence of a roughening transition in 3\Dim Ising); see  \cite{HasenbuschPinn1997,GilmerBennema72,Swendsen77,ChuiWeeks76,ElwenspoekvdEerden87,JST83,NHB84,MWLB88,MLS90} as well as \cite{Weeks1980} and the references therein, for a sample of such studies going back to the early 1950's.
The mathematical physics literature, aside from the celebrated work of Fr\"ohlich and Spencer \cite{FrohlichSpencer81b} that we expand on below, has been mostly confined to the rigid regime of these models---starting from the pioneering work of Dobrushin~\cite{Dobrushin72a} in the early 1970's and proceeding to delicate aspects such as entropic repulsion, layering and wetting; see, e.g., \cite{BEF86,KDM11,Lacoin18,Lacoin20,CLMST14,CLMST16,DinaburgMazel94,CesiMartinelli96a,CesiMartinelli96b} for but a small sample, as well as some recent works~\cite{CKL_SOSlevelline,GheissariLubetzky21,GheissariLubetzky22,GheissariLubetzky23a,GheissariLubetzky23b,GheissariLubetzky_SOSLayering} and the survey \cite{IoffeVelenik18} and references therein.

\begin{figure}
\vspace{-0.2in}
    \begin{tikzpicture}
   \node (fig1) at (10.5,0) {
    	\includegraphics[width=0.5\textwidth]{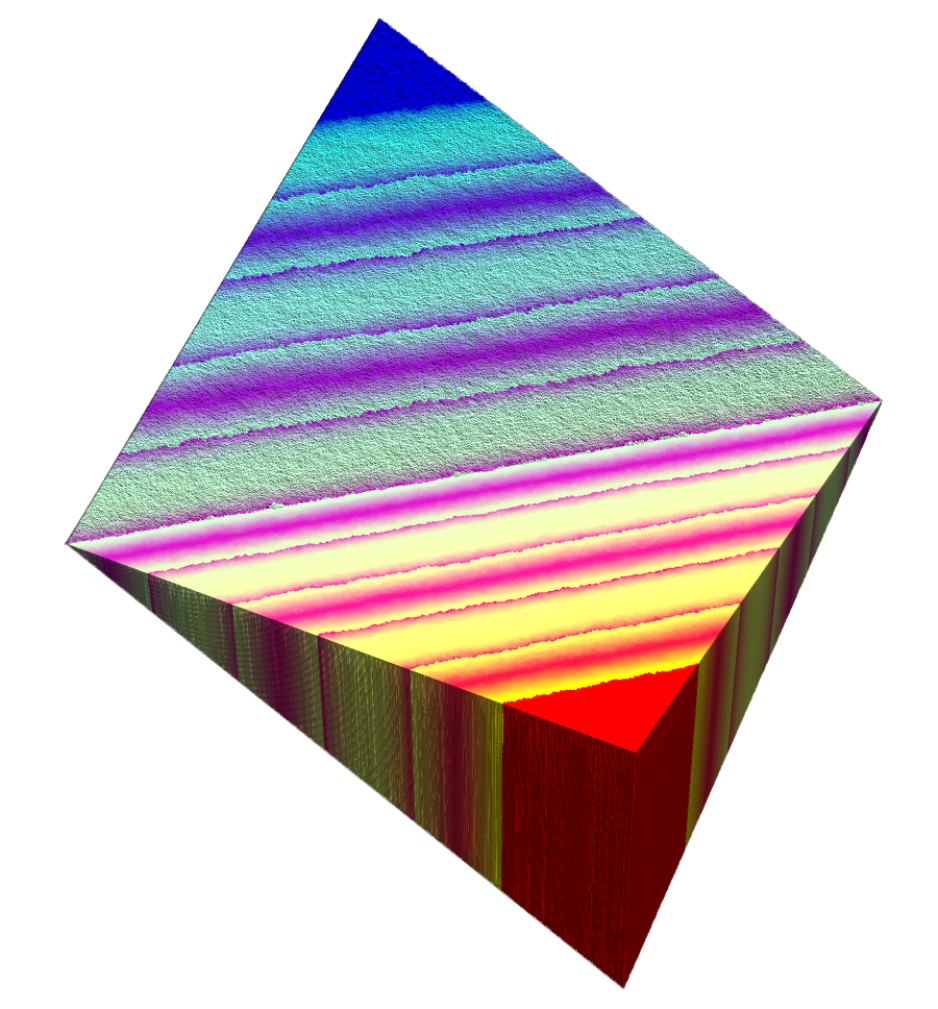}};
    \draw[color=black] (4.4,-0.25)--(6.77,-0.25);
    \draw[fill=gray!50,opacity=0.35] (7.6,-0.2) circle (0.82);
    \node (fig1z) at (2.5,-0.75) {
   	\includegraphics[width=0.35\textwidth]{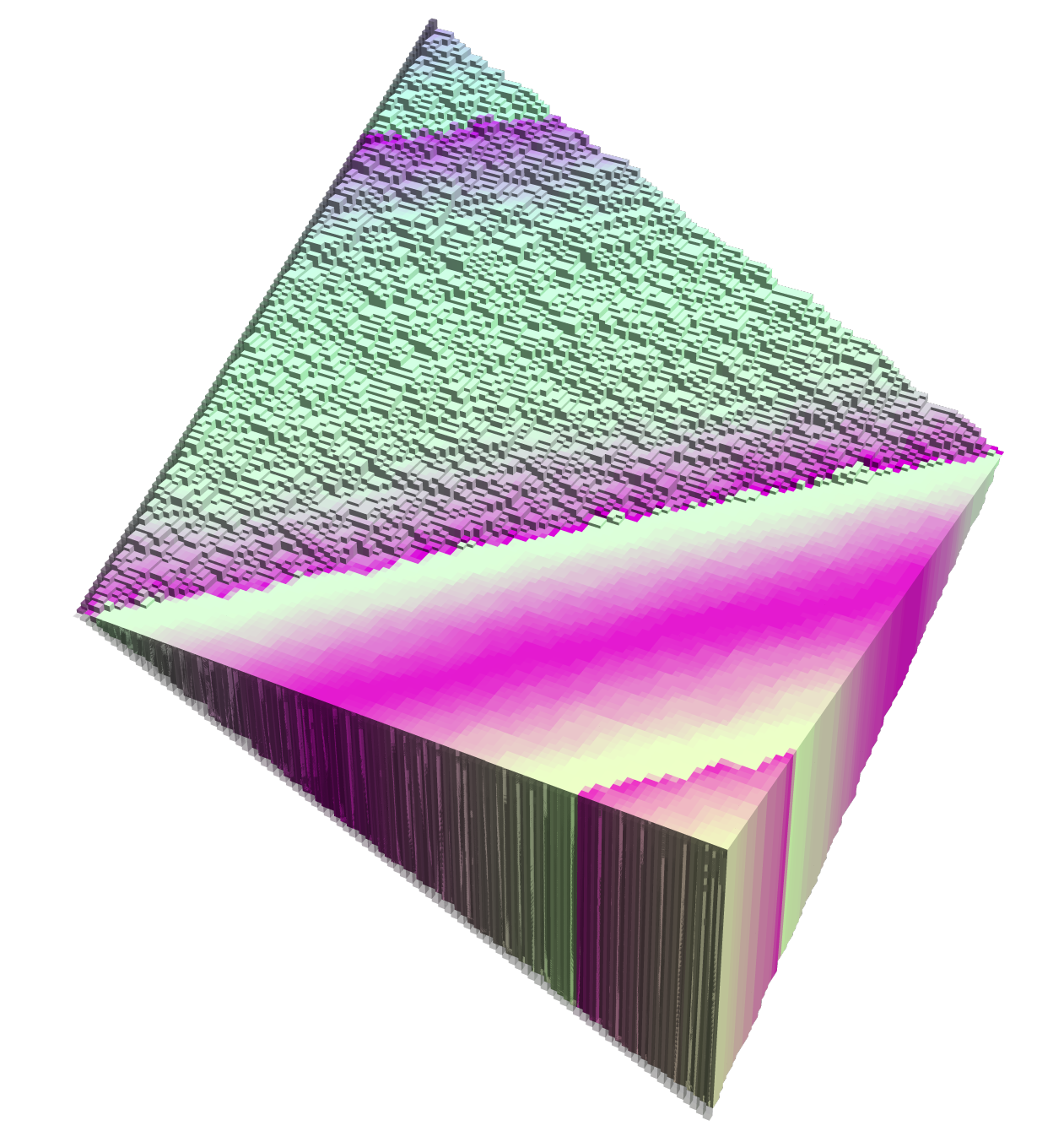}};
    \end{tikzpicture}
\vspace{-0.28in}
    \caption{Simulation of a $(2+1)$\Dim SOS model on a $500 \times 500$ box with slope $\theta=(1,1)$ at inverse-temperature $\beta=3$, with a marked $100\times100$ region zoomed in on the left.}
    \label{fig:tilted-3D}
\end{figure}
\begin{figure}
\vspace{-0.3in}
    \begin{tikzpicture}
   \node (fig1) at (7,0) {
   	\includegraphics[width=0.63\textwidth]{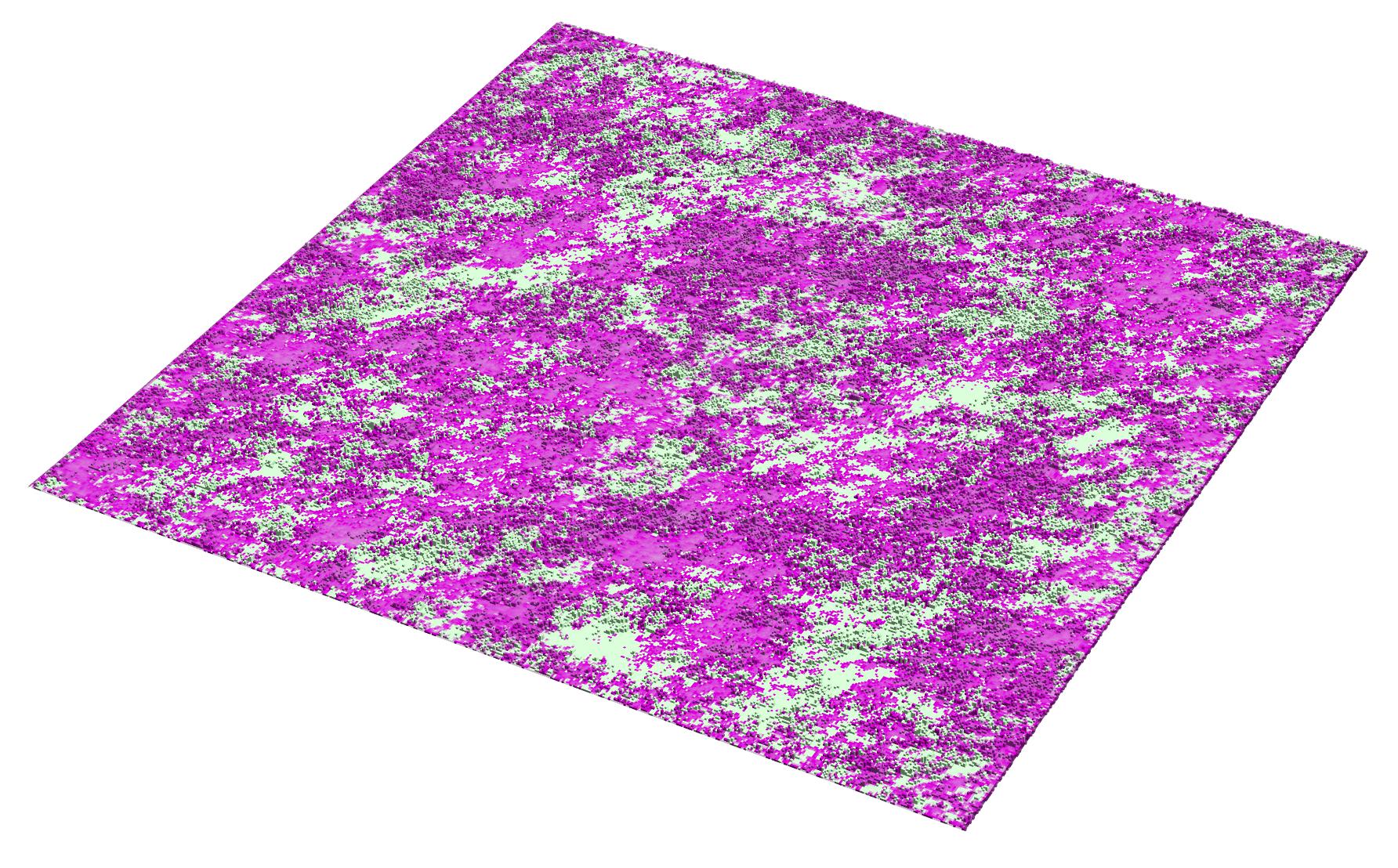}};
    \draw[color=black] (2.45,0.5)--(1.6,1.25);
    \draw[fill=gray!50,opacity=0.35] (2.7,-0.3) circle (0.85);
    \node (fig1z) at (0.2,1.5) {
   	\includegraphics[width=0.4\textwidth]{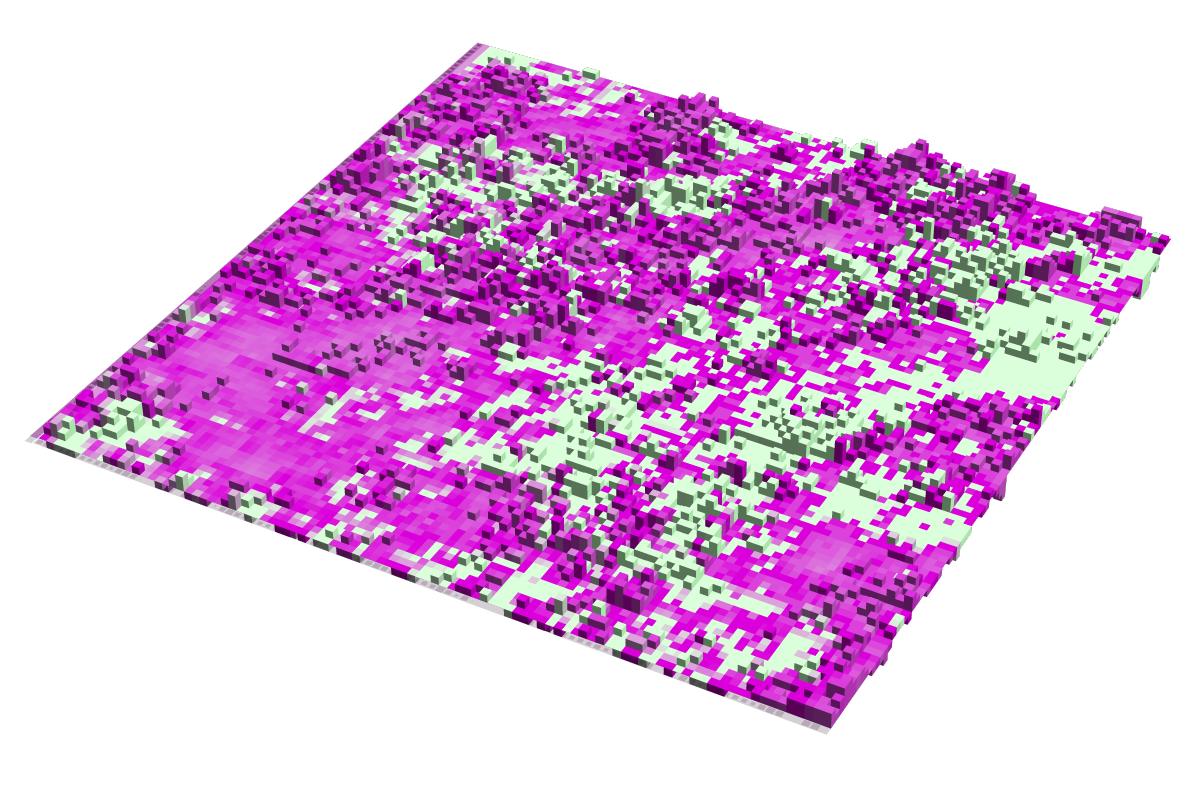}};    
    \end{tikzpicture}
\vspace{-0.3in}
    \caption{The $(2+1)$\Dim SOS surface from \cref{fig:tilted-3D} after centering it about the mean heights.}
    \label{fig:tilted-3D-ctr}
\end{figure}

That the surface in the rough phase should have logarithmic fluctuations/correlations appears in many of these works; as for its continuum scaling limit, Fr\"ohlich, Pfister and Spencer~\cite{FPS81} wrote in 1981 (see p.~186 there) that it should be ``\emph{given by massless Gaussian measures}.'' 

Following the landmark result of Kenyon~\cite{Kenyon00} that for 3\Dim Ising  at zero temperature ($\beta=\infty$) under tilted boundary conditions, the scaling limit of the interface is the Gaussian free field (\GFF) (see also \cite[Theorem 15]{Kenyon97}, \cite[Sec.~3]{Kenyon09} and \cite{KOS06} which covers any planar graph), various works asked if this should be the scaling limit also at finite $\beta$; see, e.g., the discussions in~\cite{BGV01,GiulianiToninelli19} and the following excerpt, pertaining to a slope $\theta=(1,1)$, in the recent work of Giuliani, Renzi and Toninelli~\cite{GRT23}:
``\emph{It is very likely that the \GFF behavior survives the presence of a small but positive temperature; however, the techniques underlying the proof at zero temperature, based on the exact solvability of the planar dimer problem, break down.}''

The roughness of the interface in the tilted $(2+1)$\Dim SOS model, at least for an irrational slope $\theta$, is rigorously known (unlike 3\Dim Ising at low temperature), yet this model remains highly nontrivial; e.g., in 2001, Bodineau, Giacomin and Velenik~\cite{BGV01} considered interfaces that have a slope $\theta=(\theta_1,0)$,
stating ``\emph{...the most natural effective model should have
been the \emph{SOS} model. However, only few results have been obtained about the fluctuations of this model, because the singularity of the interaction does not allow to use the techniques based on strict convexity of the potential.}''  
Velenik~\cite{Velenik06}, in his comprehensive survey from 2006 (that also covers $\R$-valued (continuous) models, where there is no parameter $\beta$; see, e.g., the recent work~\cite{ArmstrongWu23}) wrote that, for the SOS model with any nonzero slope, it is expected that 
``\emph{the large-scale behavior of these random interfaces is identical to that of their continuous counterpart. In particular they should have Gaussian asymptotics. This turns out to be quite delicate... I am not aware of a single rigorous proof for finite $\beta$.}''

\begin{figure}
\vspace{-0.1in}
\begin{tikzpicture}
    \node (fig1) at (0,0) {
    \includegraphics[width=0.32\textwidth]{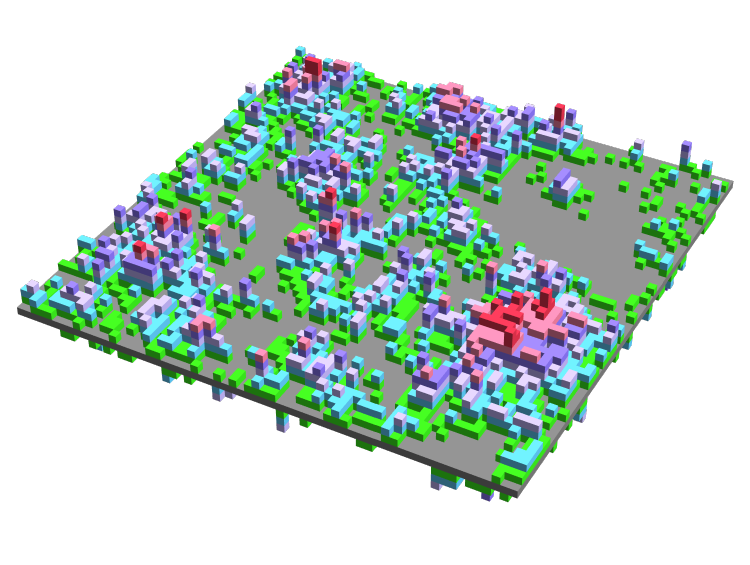}};
    \node [font=\scriptsize] at (-1,-1.5) {$\beta=0.2$}; 
    \node (fig2) at (5,-0.2) {
    \includegraphics[width=0.32\textwidth]{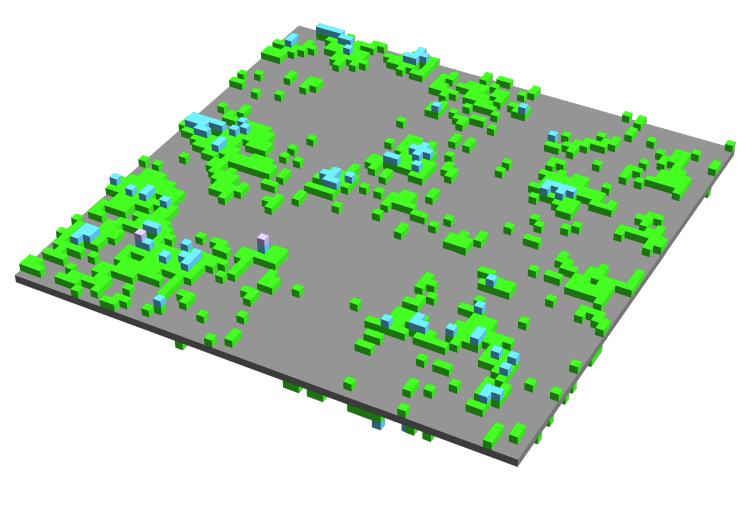}};
    \node [font=\scriptsize] at (4,-1.5)
    {$\beta=0.6$}; 
    \node (fig3) at (10,-0.3) {
    \includegraphics[width=0.32\textwidth]{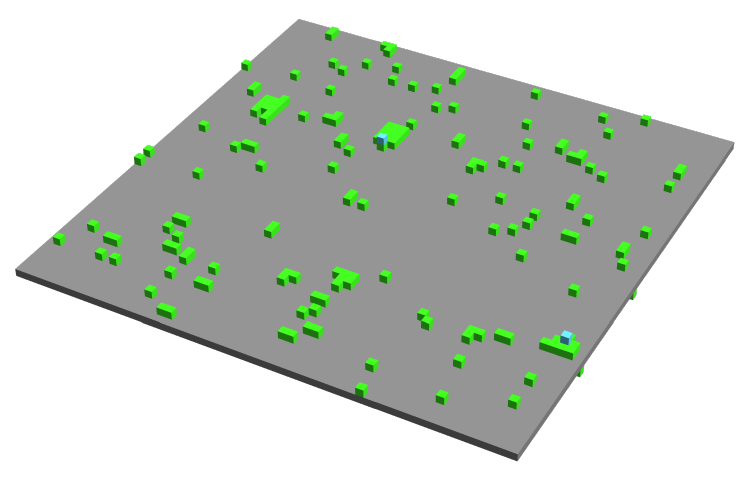}};
    \node [font=\scriptsize] at (9,-1.5) {$\beta=1.0$}; 
\end{tikzpicture}
\vspace{-0.15in}
    \caption{Roughening transition: $(2+1)$\Dim Solid-On-Solid on a $64\times 64$ box at different~$\beta$ with flat boundary conditions; sampled via $10^4$ steps of Glauber dynamics started at all-$0$.}
    \label{fig:roughening}
    \vspace{-0.1in}
\end{figure}
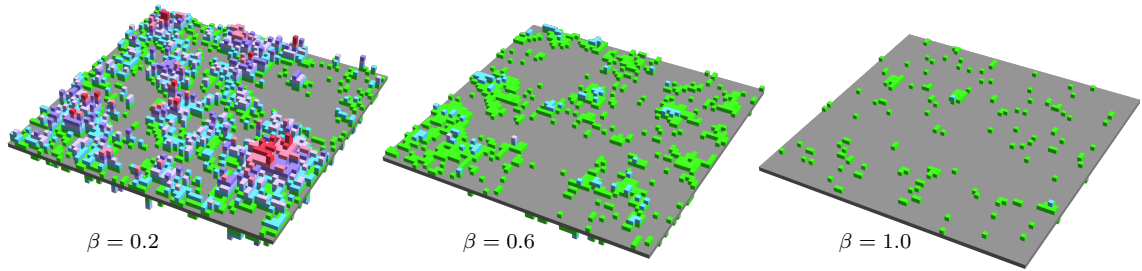

Let us now formally describe what is known on roughening in the SOS model (depicted in \cref{fig:roughening}). 
The roughening transition can be observed in the behavior of $\Var(h(x))$ as $|x|\to\infty$, where~$|x|$ is the Euclidean distance of the site $x$ from the origin $o$ (at which the surface is rooted at $0$). In 1981, Fr\"ohlich and Spencer~\cite{FrohlichSpencer81b} famously proved that $\Var(h(x))\asymp \log |x|$ at small $\beta$, as in the \GFF, the conjectured scaling limit in this regime. 
Conversely, a Peierls-type argument (see Brandenberger and Wayne~\cite{BrandenbergerWayne82}) shows $\Var(h(x))=O(1)$ at large $\beta$. (A recent breakthrough result by Lammers~\cite{Lammers22} proved sharpness of this transition: for some~$\beta_{\textsc r}$, the former holds for $\beta\leq \beta_{\textsc r}$, the latter for $\beta>\beta_{\textsc r}$.)

The $|\nabla\phi|^p$-models ($p\geq 1$) generalize SOS (the case $p=1$) into $\P_\beta(h)\propto \exp[-\beta\sum_{e\in\cE}|\nabla h(e)|^p]$. 
All these models are expected to demonstrate similar behavior, including the roughening transition; and while the Fr\"ohlich--Spencer argument only applies to $p=1,2$, the Peierls argument of \cite{BrandenbergerWayne82} (cf.~\cite{GMM73}) applies to all $p\geq 1$. (See \cite{LMS16} for more on similarities between these models at large~$\beta$.) However, this Peierls argument, albeit quite robust, only addresses contractible level line loops, and thus ceases to imply rigidity at large $\beta$ \emph{when the surface is positioned on a slope}.

\begin{figure}
\vspace{-0.15in}
\begin{tikzpicture}
    \node (fig1) at (0,0.1) {
    \includegraphics[width=0.225\textwidth]{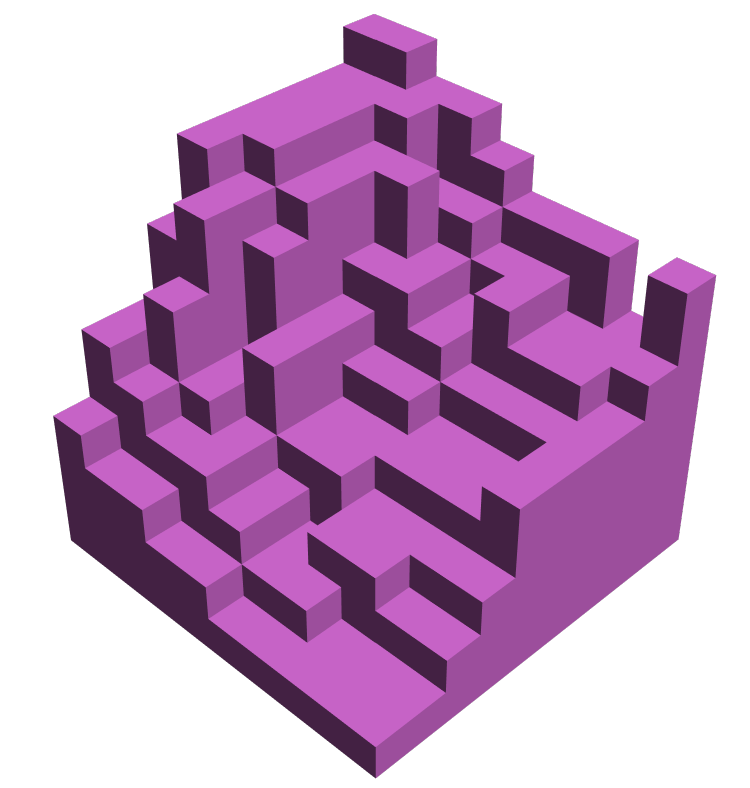}};
    \node (fig2) at (4.7,0) {
    \includegraphics[width=0.25\textwidth]{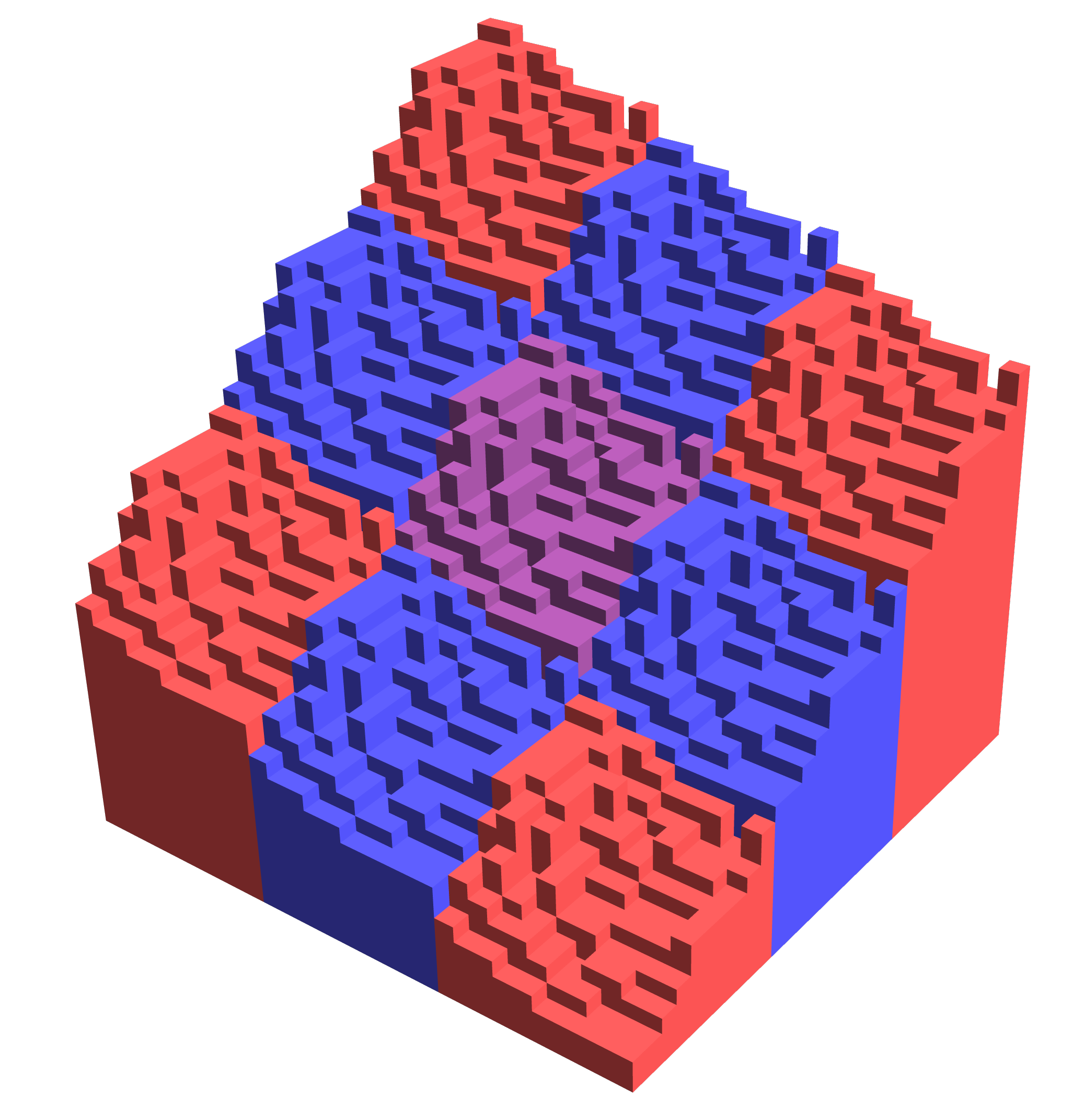}};
    \node (fig3) at (10,-0.2) {
    \includegraphics[width=0.3\textwidth]{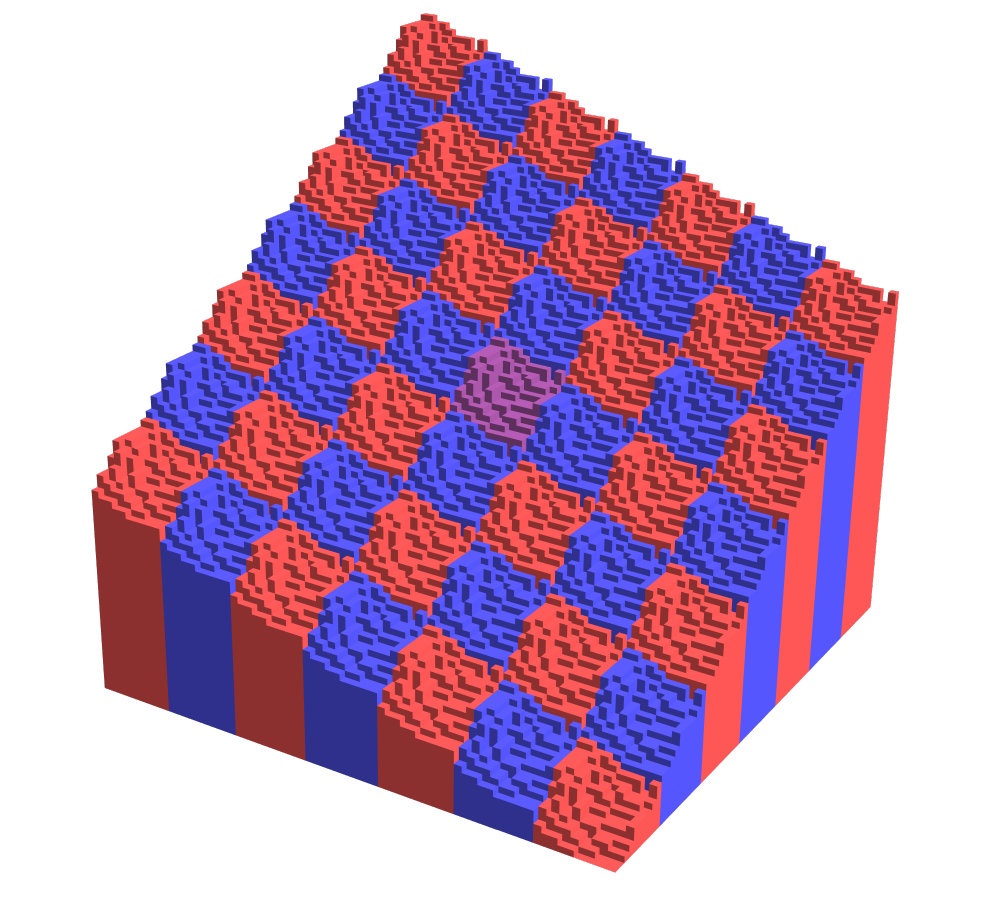}};
\end{tikzpicture}
\vspace{-0.25in}
\caption{SOS configuration on a $10\times 10$ torus, extended periodically with slope $\theta=(1,1)$.}
\label{fig:periodic-bc}
\vspace{-0.15in}
\end{figure}

We now define the precise setup of periodic boundary conditions for $(2+1)$\Dim SOS with slope~$\theta$.
The $(2+1)$\Dim SOS model on the $N\times N$ torus $\Lambda_N$ with slope $\theta=(\theta_1,\theta_2)$, which without loss of generality has $\theta_1,\theta_2\geq 0$, is formally obtained by extending the function $h$ to $\Z^2$ periodically, decreasing it by $\lfloor \theta_1 N\rfloor$ along the $(1,0)$-direction and by $\lfloor \theta_2 N\rfloor$ in the $(0,1)$-direction (see 
\cref{fig:periodic-bc}). That is, one views $h$ as a full plane height function, where for every face $x=(x_1,x_2)$ of $\Z^2$,
\begin{equation}\label{eq:h-periodic-bc}
 h(x_1+N,x_2) = h(x_1,x_2) -\lfloor \theta_1 N\rfloor\quad\mbox{and}\quad h(x_1,x_2+N) = h(x_1,x_2)-\lfloor \theta_2 N\rfloor\,.
\end{equation}
Equivalently, one restricts the height functions $h$ on the torus to those where $\sum \nabla h(\vec e_i) =-\lfloor \theta_1 N\rfloor $ for every non-contractible loop of directed edges $(\vec e_i)$ in the $(1,0)$-direction, and $\sum \nabla h(\vec e_i) =-\lfloor\theta_2 N\rfloor$ for non-contractible loops in the $(0,1)$-direction (and $\sum\nabla h(\vec e_i)=0$ for every contractible loop). (Formally, in this formulation one only considers the $1$-form/vector-field $\nabla h$ as opposed to $h$.)
From~this definition, we see that the SOS measure on the torus with slope $\theta$ is translation-invariant. 

The slope $\theta=(\theta_1,\theta_2)$ deterministically induces $(\theta_1+\theta_2)N$ macroscopic level line loops in the SOS surface (see \cref{subsec:setup} for a formal definition), each one a non-contractible loop, thus unaddressed by the rigidity Peierls argument (which is only applicable to contractible level line loops).
Indeed, these interacting macroscopic level lines (which cannot cross one another) break the surface rigidity and are conjectured to yield a \GFF scaling limit (as also conjectured for the flat setup at small $\beta$).

A beautiful result of Sheffield~\cite{Sheffield05} from 2005 shows that the SOS surface with slope $\theta$ on a torus is rough when $\theta=(\theta_1,\theta_2)$ is \emph{irrational} in one of its coordinates. Precisely, via a highly nontrivial application of the Ergodic Theorem, Sheffield proved that there does not exist an ergodic and translation-invariant Gibbs measure $\mu_\theta$ for the height function $h$ with an irrational slope $\theta$ (where $\theta_i(\mu) := \E_\mu \nabla h(\vec e_i)$ for $\vec e_1=(o,o+(1,0))$, $\vec e_2= (o,o+(0,1))$; both are integrable by assumption). It thus follows that $\Var(h(x))\to\infty$, since otherwise having $\Var(h(x))=O(1)$ would have implied (through tightness and routine reasoning) the existence of an ergodic and translation-invariant subsequential local limit. 
Sheffield's argument is remarkably general: it is applicable to a wide family of nearest-neighbor translation-invariant models of integer height functions on $\Z^2$, including, e.g., all $|\nabla\phi|^p$ ($p\geq 1$) models; thus, the random surfaces in all these models are rough when the slope $\theta$ is irrational. However, this argument gives no bound on the rate of $\Var(h(x))$ as  $|x|\to\infty$.

In recent years, significant progress was made in the study of $(2+1)$\Dim $|\nabla\phi|^p$ models in the flat setup $(\theta=0)$ (cf., the aforementioned work of Lammers~\cite{Lammers22} on the phase transition in SOS; and the recent breakthrough of Bauerschmidt et al.~\cite{BPR22a,BPR22b} on the convergence of the Discrete Gaussian model (the case $p=2$) to a \GFF at high enough temperatures). Lammers and Ott~\cite{LammersOtt24} recently showed that at high enough temperature, Sheffield's result extends to any slope (i.e., including rational ones). More recently, Ott and Schweiger~\cite{OttSchweiger25} showed that (for $|\nabla\phi|^p$ models for $0< p\leq 2$) the height of the origin diverges logarithmically in $N$ at small $\beta$.
However, at low temperature and with a nonzero slope $\theta$, the only known result remains the work of Sheffield~\cite{Sheffield05} described above.

In this work we study the $(2+1)$\Dim SOS model of \cref{eq:SOS-model} on a torus with slope $\theta=(\theta_1,\theta_2)$, for any $\theta_1,\theta_2>0$, with a pinning-type potential $\sV$ (out of a given family of potentials defined below):
\begin{equation}\label{eq:tilted-sos} \P_{\beta,\lambda}(h) =  \frac1{\Zsos_{N,\beta,\lambda}} 
\exp \left[ -\beta|h|- \lambda \sV(h) \right]\,.
\end{equation}
We stress that the potential strength
$\lambda>0$ can be taken as $\epsilon_\beta$, arbitrarily small for large enough~$\beta$ (see \cref{rk:dependency_beta_lambda} for a discussion of the allowed rate, which is inverse-polynomial in $\beta$), and that the potential $\sV(\cdot)$ can be taken to be zero on all but at most an $\epsilon_\beta$-fraction of the sites;
thus, one expects the surface behavior to be governed by its energy $\beta|h|$ and entropy, as in~$\lambda=0$.

The $(2+1)$\Dim SOS model with flat boundary conditions ($\theta=0$) was studied in detail under a pinning potential (see the survey~\cite{IoffeVelenik18} and the recent works by Lacoin~\cite{Lacoin18,Lacoin20}) that rewards every face of $h$ intersecting a preset plane, typically the slab at height $0$---the ground state of the model.
The SOS model with a nonzero slope $\theta$, on the other hand, has exponentially many ground states (monotone surfaces, as we explain next); accordingly, our family of pinning potentials will reward the overlap of $h$ not with a predetermined ground state, but rather with its ``closest'' ground state. 

We say that an SOS height function $h$ is a \emph{monotone surface} if $h(x_1,x_2)\geq h(y_1,y_2)$ whenever $y_1\geq x_1$ and $y_2\geq x_2$. We will often denote a monotone surface by $\varphi$ or $\psi$, and refer to it as a \emph{tiling}, owing to the well-known bijection between monotone surfaces in $\Z^2$ and lozenge tilings of the triangular lattice $\T$ (as well as with dimers in the hexagonal lattice; see below for more details). We impose on the tiling the same periodic conditions with slope $\theta$ as in \cref{eq:h-periodic-bc}. It is easy to see that, out of all periodic SOS surfaces $h$ with slope $\theta$, tilings minimize $|h|$, the number of faces (recalling $|h|=N^2+\sum|\nabla h(e)|$); that is, tilings are the ground states for the model at $\lambda=0$. 

The potentials $\sV$ studied here will pin $h$ to $\psi_0$, the closest tiling to it in terms of common faces. (We do not root $\psi_0$ to be $0$ at the origin, as it will intersect $h$ which is already rooted.) 
We give two concrete examples of such a $\sV$ before describing the general family of potentials considered.
\begin{example}[pinning to a closest tiling] \label{def:pinning-V1}
Let $\psi_0$ be a tiling that minimizes $|\psi_0\setminus h|$ (arbitrarily chosen if multiple exist) out of all those that satisfy \cref{eq:h-periodic-bc}. (Equivalently, $\psi_0$ maximizes $|h\cap \psi_0|$). Let $\{S_i\}$ be the connected components of faces of $\psi_0\setminus h$, and set
\begin{equation}\label{eq:pinning-V}\sV_1(h) = \sum |S_i|^2\,.\end{equation}
\end{example}
\begin{example}[truncated pinning to a tiling]\label{def:pinning-V2}
Let $\psi_0$ be a minimizer of $|\psi_0\setminus h|$ (arbitrarily chosen if multiple exist), and let $\{S_i\}$ be the connected components of faces of $\psi_0 \setminus h$. Set
\begin{equation}\label{eq:trunc-pinning-V}\sV_2(h) = \sum \one_{\{|S_i|\geq 1000\}} |S_i|\log|S_i| \,.\end{equation}
\end{example}
\begin{definition}[family of pinning potentials]\label{def:pinning-V}
Consider a tiling $\psi_0$ minimizing $|\psi_0\setminus h|$ (a canonical way to break ties will be given in \cref{prop:minimizers}). The potential $\sV$ can be any function $\ff$  of a (maximal) connected component $S$ of $\psi_0\setminus h$, which grows faster than linearly in $|S|$. Precisely, 
\begin{equation}\label{eq:gen-potential}\sV(h)=\sum \ff(S_i)\quad\mbox{for some function $\ff\geq 0$ satisfying }\quad \lim_{|S|\to\infty}\frac{\ff(S)}{|S|}=\infty\,,\end{equation}
where the sum goes over all the connected components $S_i$ of $\psi_0\setminus h$, and writing $\lim_{|S|\to\infty} \frac{\ff(S)}{|S|}=\infty$ is short for stating that $\ff(S_k)/|S_k|\to\infty$ holds for any sequence of subsets $(S_k)$ with $|S_k|\to\infty$. 
\end{definition}
(In \cref{def:pinning-V1,def:pinning-V2} we took $\ff(S)=|S|^2$ and $\ff(S)=\one_{\{|S|\geq 1000\}}|S|\log |S|$ resp., though one can take $\sV$ where $\ff$ depends on the geometry of $S$ beyond simply $|S|$, e.g., $\ff(S)=\diam(S)^3$.)
Observe that~$\sV$ may ignore any component $S_i$ of $\psi_0\setminus h$ with fewer than $\sM_0$ faces for some absolute constant~$\sM_0$; i.e., it may penalize $h$ only for large regions missed by the best approximating ground state $\psi_0$.
Such a potential $\sV$ adds just enough noise to the model---even at $\lambda=\epsilon_\beta$---to treat close ground states with atypically large ``frozen'' regions, foiling the analysis of interacting tilings.

\begin{figure}
\vspace{-0.1in}
\begin{tikzpicture}
    \node (fig1) at (0,0) {
    \includegraphics[width=0.28\textwidth]{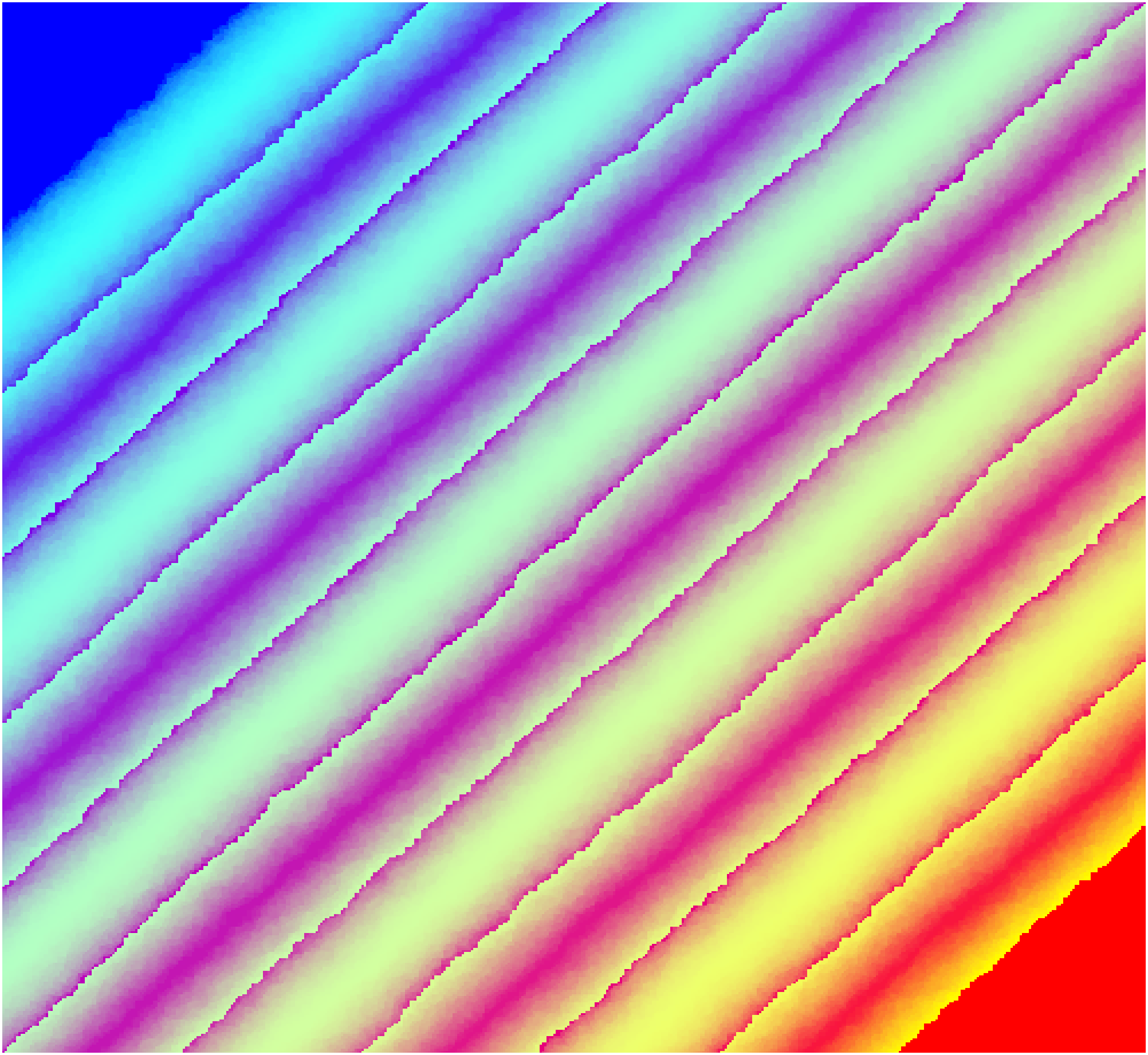}};
    \node (fig2) at (6,0) {
    \includegraphics[width=0.28\textwidth]{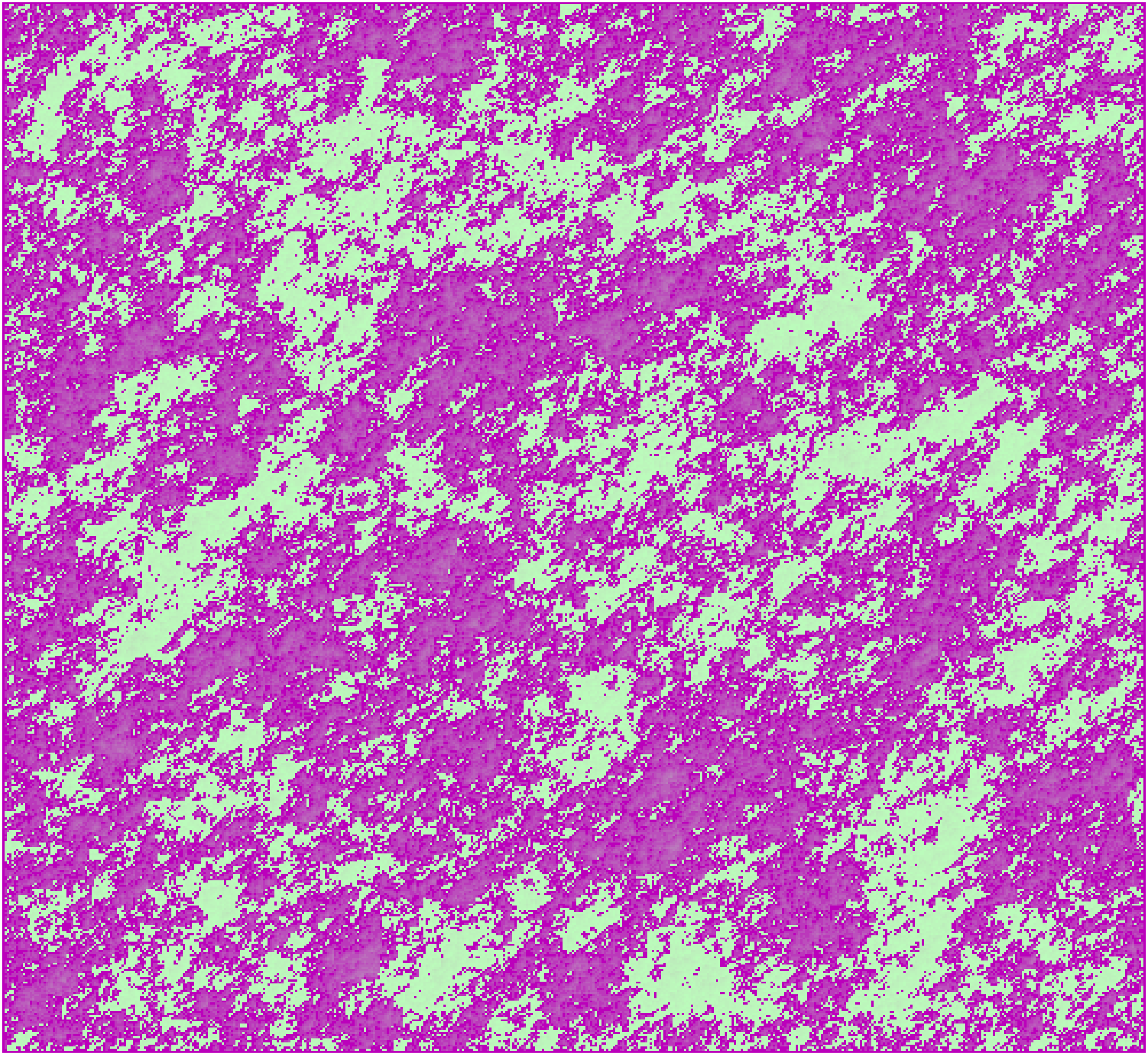}};
\end{tikzpicture}
\vspace{-0.15in}
\caption{Level line comparison of the SOS surface vs.\ its centering from \cref{fig:tilted-3D,fig:tilted-3D-ctr}.}
\label{fig:level-lines}
\vspace{-0.1in}
\end{figure}

\subsection{Main results}
Our motivation was to recover the asymptotic rate of $\Var(h(o))$ in the SOS model of \cref{eq:tilted-sos} at large $\beta$ and a nonzero slope $\theta$ (which, to our knowledge, is not known for any 
$|\nabla\phi|^p$ model with any added potential (of strength $\epsilon_\beta$, so as not to compete with energy/entropy)).
Our main result obtains this, and moreover establishes that (a) the law  on $\nabla h$ around the origin $o$ converges to a full plane limit $\mu_\infty$; and (b) the scaling limit of $\mu_\infty$ converges to a full plane \GFF. 
Recall that the full plane \GFF is defined as a centered Gaussian stochastic process indexed by smooth compactly supported $\C\to\R$ functions with $0$ integral and covariance
\[
\Cov\big( \GFF(f_1), \GFF(f_2)\big) := -\frac{1}{2\pi}\int f_1(u) f_2(v) \log |u-v| \d u \d v\,,
\]
where $\GFF(f)$ is the value of the process at the function $f$, as in the expectation for a distribution\footnote{Informally, one can think of the \GFF as the Gaussian process on $\C$ with $\Cov(\GFF(u), \GFF(v) ) = -\frac{1}{2\pi}\log|u-v|$ but the fact that $\log$ diverges at $0$ means it does not make sense pointwise, hence the definition using test functions.}.
\begin{maintheorem}\label{thm:GFF-convergence}
For every $\lambda>0$ and $\theta_1,\theta_2>0$ there exists $\beta_0$ so the following holds for every $\beta>\beta_0$. Consider the $(2+1)$\Dim \SOS model of \cref{eq:tilted-sos}, for any potential function $\sV$ as per \cref{def:pinning-V}, on $\Lambda_N$, an $N\times N$ torus with slope $\theta=(\theta_1,\theta_2)$ as defined in \cref{eq:h-periodic-bc}. Then
\begin{enumerate}[(a)]
\item \label{it:weak-limit} The law on the \SOS gradients $\nabla h$ converges weakly, as $N\to\infty$, to a full plane limit $\mu_\infty^{\beta,\lambda}$.
\item \label{it:scaling-limit}
Sample $h\sim \mu_\infty^{\beta,\lambda}$.  There exist $\sigma>0$ and a linear map $\fL$ on $\R^2$ (nonrandom) such that
\[ h(n\, \cdot)-\E [h(n\,\cdot)]\xrightarrow[]{\;\mathrm{d}\;} \sigma \GFF \circ \fL\quad\mbox{as $n\to\infty$}\] 
(viewing both sides as stochastic processes indexed by test functions with $0$ integral).
\end{enumerate}
\end{maintheorem}

As the \GFF is a distribution, \cref{thm:GFF-convergence} does not provide asymptotics of finite moments of the \SOS height function $h$. However, building on its proof, one can derive those, as demonstrated next.
\begin{maincoro}\label{cor:variance}
In the setting of \cref{thm:GFF-convergence}, there exists $c>0$ such that 
\[ \Var(h(x))=(c+o(1))\log|x|\quad\mbox{as $|x|\to\infty$}\,.\]
\end{maincoro}
The new results readily yield a law of large numbers (LLN) for $h(x)$, having $\sqrt{\log|x|}$ fluctuations about its mean $\theta \cdot x$ (see the level lines in \cref{fig:level-lines} depicting the LLN). It is interesting to compare these results with the LLN for the shape of the Wulff crystal in 3\Dim Ising near the corner of a box at zero temperature ($\beta=\infty$), due to Cerf and Kenyon~\cite{CerfKenyon01}, where lozenge tilings naturally appear. See \cref{sec:open-prob} for more on this.

The intuition behind the \GFF scaling limit of the centered height function is that, at large $\beta$, the SOS surface $h$ should behave like a randomly perturbed %
``almost uniformly'' chosen ground state  (which, we recall, is a lozenge tiling of the triangular lattice $\T$). As mentioned above, a seminal result of Kenyon~\cite{Kenyon00} established that the scaling limit of a uniform (domino) tiling is the \GFF, whence one would expect that the SOS surface $h$ should inherit the same scaling limit.

Our proof strategy is to increase the probability space via a random tiling $\varphi$ \emph{conditional on~$h$}:
\[ \P(\varphi \mid h) \propto \exp(+\alpha |h\cap\varphi|)\]
for some $\alpha>0$ (e.g., $\alpha=\beta$), rewarding $\varphi$ for faces in $h\cap\varphi$ (see \cref{eq:def-phi-given-h}), and establish that:
\begin{enumerate}[(i)]
    \item\label[part]{pt:weakly-interacting} The marginal on $\varphi$ is a \emph{weakly interacting tiling}, that is to say, it is a tilting of the uniform distribution of the following form:  $\P(\varphi)\propto \exp\big[-\sum_{x}\sum_r \fg_r(\varphi\restriction_{B(x,r)})\big]$, where $B(x,r)$ is the radius-$r$ ball around $x$ and $\fg_r$ are functions whose $L^\infty$-norm decays exponentially with $r$.
\item\label[part]{pt:gmt} Consequently, via the powerful renormalization group machinery of Giuliani, Mastropietro and Toninelli~\cite{GMT17,GMT20}, the random tiling $\varphi$ has a scaling limit given by a \GFF.
\item\label[part]{pt:from-phi-to-h} The SOS surface $h$ is a local perturbation of $\varphi$, thus has the same limit.
\end{enumerate}
The aforementioned renormalization result of \cite{GMT17,GMT20} may be summarized as follows\footnote{%
The setting there is $\Z^2$ with arbitrary weights; the hexagonal lattice is recovered by setting one weight to $0$.}:
\begin{theorem}[\cite{GMT17,GMT20}]\label{thm:gmt}
Fix $\fa,\fb,\fc>0$ to be three sides of a non-degenerate triangle and $R>0$. There exists $\delta>0$ (depending on $\fa,\fb,\fc, R$) so that the following holds for every
function $\fg$ on lozenge tilings of a ball of radius $R$ (with free boundary condition) in the triangular lattice $\T$ satisfying
\begin{equation}\label{eq:single-g} \| \fg \|_\infty\leq \delta\,. \end{equation}
Let $\mu_N$ be distributions over tilings $\varphi$ of the torus $\T_N$ given by
\[
\mu_N( \varphi) \propto \fa^{n_\fa(\varphi)}\fb^{n_\fb(\varphi)} \fc^{n_\fc(\varphi)} 
\exp\Big[\sum _{x} \fg(\varphi \restriction_{B(x,R)})\Big]\quad\mbox{for all $N$}\,,
\vspace{-3pt}
\]
where $n_{\fs}(\varphi)$ counts the number of lozenges of type $\fs\in\{\fa,\fb,\fc\}$ in $\varphi$. 
Then:
\begin{enumerate}[(a)]
    \item \label{it:mu_N_limit} 
$\mu_N \to \mu_\infty$ locally for the gradients $\nabla\varphi$ as $N\to\infty$; and
\item \label{it:phi_GFF_limit} if $\varphi$ is sampled from $\mu_\infty$ and viewed as a height function projected on the plane $\cP_{111}$, then
\[
\varphi( n\, \cdot ) - \E [\varphi( n\, \cdot )] \xrightarrow[]{\;\mathrm{d}\;}  \sigma \GFF \circ \mathfrak L\,,
\]
where $\sigma > 0$, $\mathfrak L$ is an invertible linear map and $\GFF$ is a %
full plane Gaussian free field.
\end{enumerate}
\end{theorem}

The proof of \cref{thm:gmt} given in \cite{GMT17,GMT20} easily extends to the case where, instead of a single $\fg$ as per \cref{eq:single-g}, there is a sequence of functions $(\fg_r)$ on balls of growing radius $r\to\infty$, as long as\footnote{The setting of \cite{GMT17,GMT20} is slightly different: they consider an interaction carried by \emph{patterns}---sets $P$ of lozenges that occur in $\varphi$---that can be disconnected. Their proof extends as long as the weight of a pattern $P$ decays like $e^{-C |P| - \kappa \operatorname{treedist}(P)^\alpha }$ for some $\alpha,\kappa>0$, where $\operatorname{treedist}$ is the size of the minimal connected set (not necessarily a tree inside $P$) that contains $P$. Applied to our case, where the patterns are balls, this corresponds to an $e^{-C r^2}$ decay. Note that the setting of \cite{GMT17,GMT20} includes examples which are not covered by our extension in \cref{thm:gmt-refinement} (which is not geared to handle disconnected patterns): e.g., two body interactions with a stretched exponential decay. }
\[ \|\fg_r\|_\infty \leq \exp(-C r^2)\qquad \mbox{for all $r$, with a large enough constant }C>0\,.\]

Unfortunately, that assumption is much stronger than what our model affords: we can only hope for $\|\fg_r\|_\infty \leq O(e^{-C r})$, i.e., an exponential decay \emph{in the radius of the ball rather than in its volume}.
Several new ingredients were needed to boost \cref{thm:gmt} to this relaxed assumption (\cref{eq:many-g}).
\begin{maintheorem}\label{thm:gmt-refinement}
Fix $\fa,\fb,\fc>0$ to be three sides of a non-degenerate triangle, and let $\pi p_\fa, \pi p_\fb, \pi p_\fc$ denote its angles.
For every $c>0$ there exists $C>0$ (depending on $\fa,\fb,\fc,c$) such that the following holds for every sequence of functions $(\fg_r)_{r=1}^\infty$ on lozenge tilings of balls of radius $r$ in $\T$ satisfying
\begin{equation}\label{eq:many-g}
   \| \fg_r \|_\infty \leq  \exp(-C r)\qquad\mbox{for all $r$}\,.
\vspace{-2pt}   
\end{equation}
Take any sequences of integers $n_{\fs,N}$ divisible by $N$ such that $n_{\fs,N}/N^2 \to p_\fs$ as $N \to \infty$ ($\fs\in\{\fa,\fb,\fc\}$). Let $\mu_N$ be distributions over tilings $\varphi$ of the torus $\T_N$ with lozenge counts $n_\fs(\varphi) = n_{\fs,N}$, satisfying
\[ \Big| \mu_N(\varphi) / \Big( Z_{N}^{-1} \exp\Big[\sum_{x}\sum_{r\leq N} \fg_r(\varphi \restriction_{B(x,r)})\Big] \Big) - 1\Big| \leq c^{-1} e^{-N^c}\quad\mbox{for all $N$}\,,
\vspace{-2pt}
\]
where $Z_N$ is the partition function of $\exp[\sum_{x,r} \fg_r(\cdot)]$. 
Under these assumptions, both conclusions of \cref{thm:gmt} (existence of the limit $\mu_\infty$ and the \GFF scaling limit for it) hold true.
\end{maintheorem}

(For concreteness, in the two theorems above, the notation $\varphi\restriction_{B(x,r)}$ denotes the collection of every lozenge tile of $\varphi$ that intersects a triangle of $B(x,r)$, the ball of radius $r$ about $x$, in~$\T$.) As~a~basic application, if $E_r(\varphi,x)$ is the event that $x$ is connected to the boundary of $B(x,r)$ via a path of lozenges of the same type, then
\Cref{thm:gmt} can be applied to a uniform lozenge tiling~$\varphi$ tilted by $\exp[\delta \sum_x \one_{E_2(\varphi,x)}]$ (or any fixed $r$), but cannot treat a tilt involving $E_r$ for all $r$ unless these come with a prefactor of $\delta_r < e^{-C r^2}$. On the other hand, \cref{thm:gmt-refinement} can be applied to a uniform tiling $\varphi$ tilted by $\exp[\sum_x\sum_r \one_{E_r(\varphi,x)} e^{-C r}]$ (\cref{ex:gmt-appl,ex:gmt-appl-refine}).  (One can similarly tilt by the maximum number of disjoint paths of length $r$ of lozenges of the same type from $x$, etc.) 

\begin{remark}\label{rem:microcanonical}
Whereas the main contribution in \cref{thm:gmt-refinement} is that it allows for interactions that are long-range and decay slowly in the pattern radius (vs.\ \cref{thm:gmt}, dealing with finite interactions), 
another aspect where \cref{thm:gmt-refinement} refines \cref{thm:gmt} is the microcanonical vs.\ canonical ensemble: our slope $\theta=(\theta_1,\theta_2)$ is equivalent to fixing the lozenge counts in $\varphi$ to 
\begin{equation}
    n_\fa = N^2\,,\quad n_\fb = N \lfloor \theta_1 N\rfloor\,,\quad n_\fc = N\lfloor \theta_2 N\rfloor\,;
\end{equation}
thus, we need a version restricted to tilings with deterministic $n_\fa,n_\fb,n_\fc$ (microcanonical ensemble) as opposed to an average over all tilings weighted by $\fa^{n_\fa} \fb^{n_\fb} \fc^{n_\fc}$ which was the setting of \cite{GMT17,GMT20}.
\end{remark}

\begin{remark}\label{rem:phi-proj}
A notable difference between the \GFF scaling limit for $\varphi$, as given in \cref{it:phi_GFF_limit} of \cref{thm:gmt,thm:gmt-refinement}, and the one for $h$, as per \cref{it:scaling-limit} in \cref{thm:GFF-convergence}, is that the latter treats the \SOS surface $h$ as a height function projected on the plane $z=0$, denoted by $\cP_{001}$, while the former considers $\varphi$ as a height function projected on the plane $x+y+z=0$, denoted by $\cP_{111}$. Towards establishing the scaling limit of $h$, we extend this result of \cref{thm:gmt-refinement} and recover the scaling limit also when regarding $\varphi$ as a height function projected on the plane $\cP_{001}$ (see \cref{lem:GGF_SOS_convention}).
\end{remark}
\begin{remark}\label{rem:sigma-L-analytic}
As an output of the renormalization group analysis, we find that $\sigma$ and $\mathfrak L$ in \cref{it:phi_GFF_limit} of the theorems above (the scaling limit of $\varphi$) are analytic functions of  $\fg_r$ from \cref{eq:single-g,eq:many-g}.
We further have asymptotic formulas for all cumulants of the variables $\one_{\{e \in \varphi\}}$ for edges $e$ of the hexagonal lattice (viewing $\varphi$ as a dimer configuration).
\end{remark}
In particular, translating \cref{rem:sigma-L-analytic} to the setting of the SOS surface $h$, one has that $\sigma$ and $\fL$ in the scaling limit of $h$ (\cref{it:scaling-limit} in \cref{thm:GFF-convergence}) are analytic functions of $e^{-\beta}$. As our proof is only applicable to $\beta>\beta_0$, it raises an intriguing open problem whether, for instance, $\sigma$ is analytic for all $\beta$, or if there is a transition, e.g., near the roughening point $\beta_{\textsc r}$. (Note the conjectured \GFF scaling limit arises due to different reasons above and below $\beta_{\textsc r}$: for $\beta<\beta_{\textsc r}$, it is driven by the disorder akin to a discrete \GFF, whereas at $\beta>\beta_{\textsc r}$ it is governed by the law of the ground states.)

\begin{figure}
    \vspace{-0.2in}
    \begin{tikzpicture}
    \node (fig1) at (0,0) {
\includegraphics[width=0.19\textwidth]{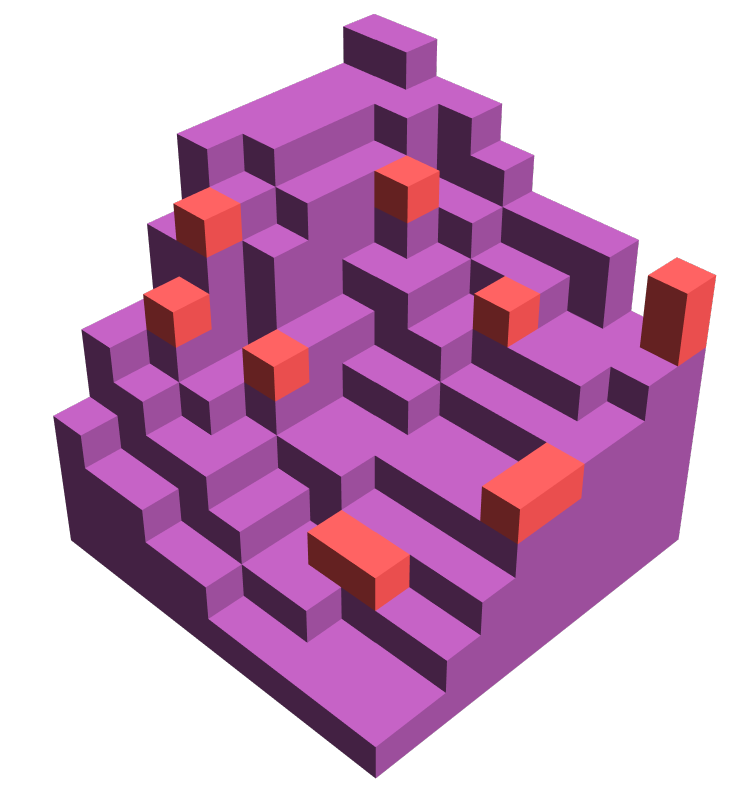}};
    \node (fig2) at (5.5,0) {    \includegraphics[width=0.19\textwidth]{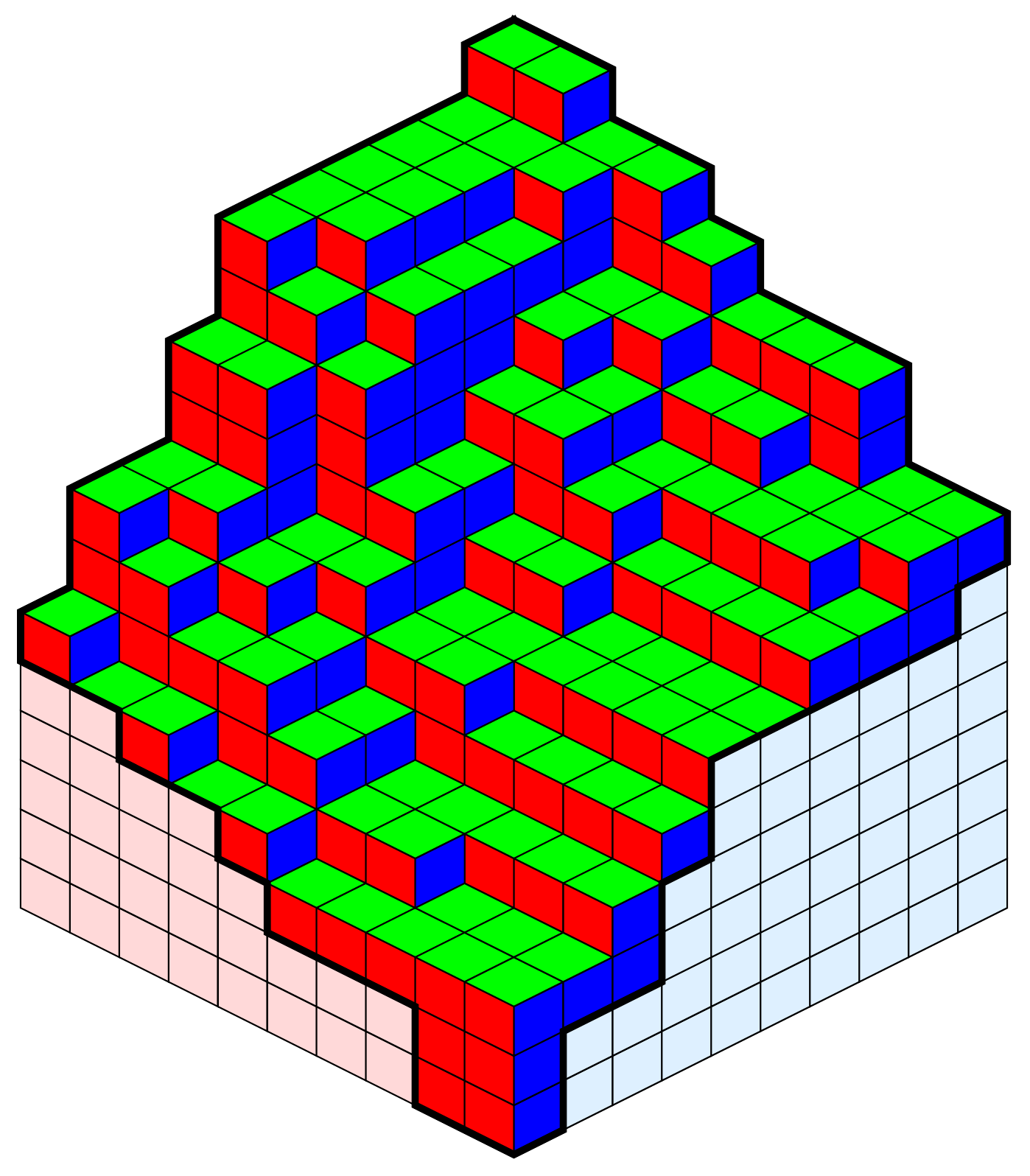}};
    \node (fig3) at (11,0) {    \includegraphics[width=0.19\textwidth]{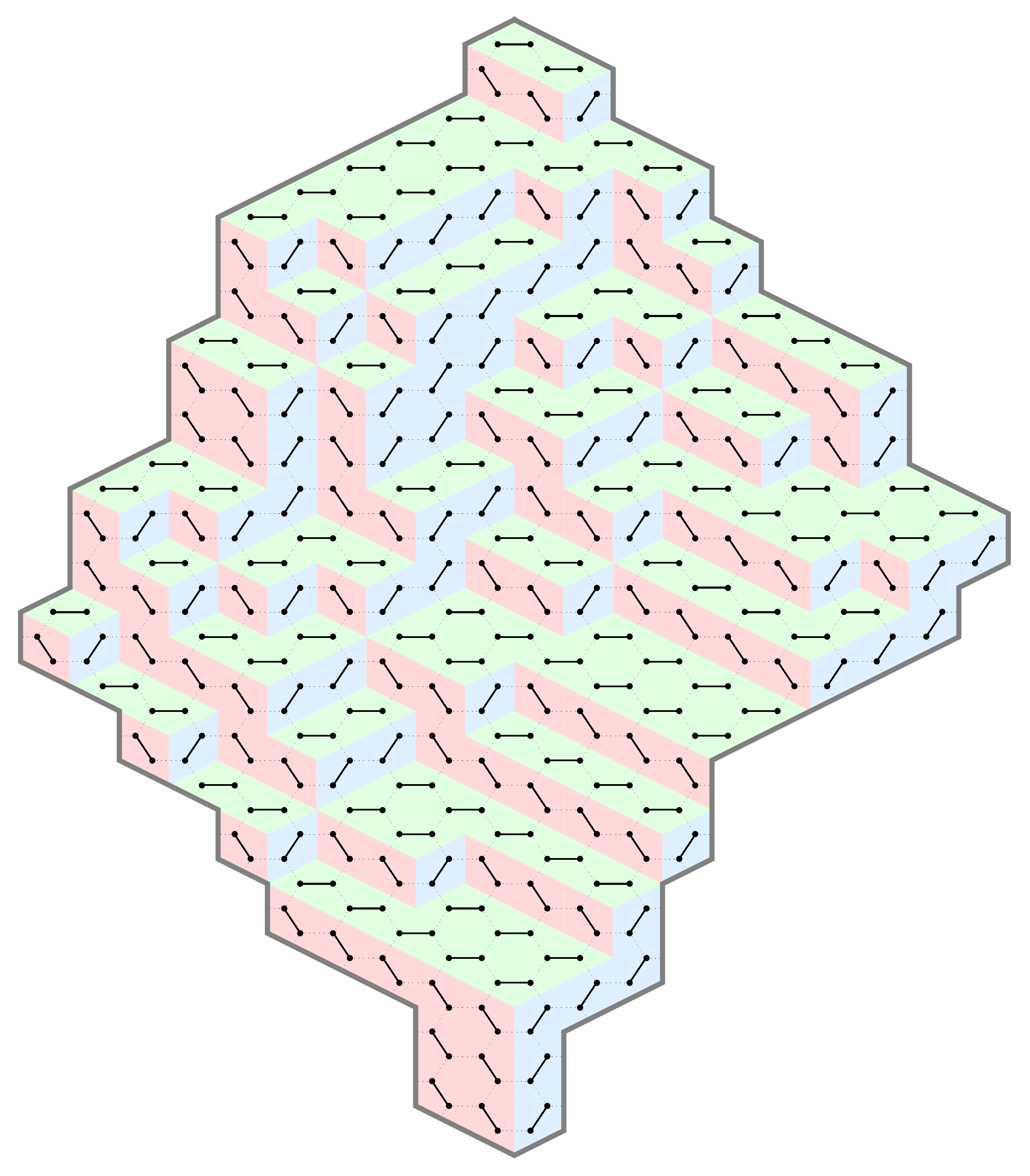}};
    \end{tikzpicture}
    \vspace{-0.24in}
    \caption{An SOS height function $h$ (non-monotone as marked) and a tiling $\varphi$ (a monotone surface) that approximates it, along with its corresponding periodic dimer configuration.}
    \label{fig:sos-approx}
    \vspace{-0.15in}
\end{figure}

\subsection{Proof ideas}
As mentioned above, our approach will be to superimpose a \emph{random tiling}~$\varphi$ approximating the given SOS surface $h$, and prove that $\varphi$ has a \GFF scaling limit. 
Formally, we fix a new parameter~$\alpha > 0$ and, conditional on $h$, sample a tiling~$\varphi$ on the torus (as in \cref{eq:h-periodic-bc})~via:
\begin{equation}\label{eq:def-phi-given-h} \P_{\alpha,\beta,\lambda}(\varphi \mid h) = \frac1{\Ztile_{N,\alpha}(h)} \exp\left(\alpha \left|h \cap \varphi\right|\right)\,,\end{equation}
where $\Ztile_{N,\alpha}(h) = \sum_\psi \exp\left(\alpha \left|h \cap \psi\right|\right)$ sums over tilings $\psi$ of the torus $\Lambda_N$ that  intersect $h$ (see~\cref{fig:sos-approx}).

For the sake of proving \cref{thm:GFF-convergence} one can choose $\alpha = \beta$, but the proofs only need $\alpha$ to be large enough, and keeping it as a free parameter will help identify its effect. This yields the joint law
\begin{equation}\label{eq:big-measure}
\P_{\alpha,\beta,\lambda}(h,\varphi) = \frac{\exp\left[-\beta|h|-\lambda \sV(h)+\alpha|h\cap\varphi|\right]}{\Zsos_{N,\beta,\lambda}\, \Ztile_{N,\alpha}(h)}\,.
\end{equation}
Our main goal is to analyze $\P_{\alpha,\beta,\lambda}(\varphi)$, the marginal probability on the approximating tiling $\varphi$, and show that it is weakly interacting (``nearly uniform''), in that $\P_{\alpha,\beta,\lambda}(\varphi)$ is a tilt of the uniform measure by $\exp[\sum_x \sum_r \fg_r(\varphi\restriction_{B(x,r)})]$  
 per the hypothesis of \cref{thm:gmt-refinement}, yielding the required result.
 
 \subsubsection{Outline of \cref{pt:weakly-interacting}: proving that $\varphi$ is weakly interacting.} (Formally stated in \cref{thm:phi-weakly-interacting}.) This part of the proof will be obtained by the following program, which we believe will be applicable (after some adapting) to various other interface models.
 We first give a brief outline of the program, then expand on each step in \cref{sec:sketch-step-1,sec:sketch-step-2,sec:sketch-step-3}.
 \begin{enumerate}[label=\textbf{Step~\arabic*}:, ref=\arabic*, wide=0pt, itemsep=1ex]
    \item \label[step]{st:1-expand}
 Free energy expansion to flip the conditioning, studying $h$ given~$\varphi$ instead of $\varphi$ given $h$:  
 Simpler routes such as cluster expansion are not applicable here, and this step, done in \cref{sec:expand}, unfortunately turns the tiling $\varphi$ into a fixed environment, inducing complex long-range interactions on~$h$. 
Following this step, the problem is reduced to showing that three measures $\mu,\nu,\pi$---in
 which $\pi$ is nothing but  $\P_{\alpha,\beta,\lambda}(h \mid \varphi)$, and is significantly more challenging to analyze than~$\mu,\nu$---are ``local'' in the sense that we can approximate certain observables for them by functions $\fg_r$ as above.
 \item \label[step]{st:mcmc-mu-nu} Markov chain analysis of Metropolis for sampling $\mu$ and Glauber dynamics for sampling~$\nu$: In both cases we show in \cref{sec:2-out-of-3} that the dynamics for the corresponding measure is \emph{contracting}, i.e., it mixes faster than the time it takes disagreements to propagate, yielding the required locality. Both $\mu$ and $\nu$ are measures on a random tiling $\psi$, conditional on $\varphi$, and are fairly tractable (even though $\nu$ has long-range interactions), in that one can prove contraction for the natural dynamics where each move selects a ``bubble''---a connected component $\sB$ of faces of $\varphi\xor\psi$---to add/delete. 
 \item \label[step]{st:pi} Analysis of $\pi= \P( h \mid \varphi)$, a long-range interacting measure on SOS surfaces $h$ given $\varphi$, whose energy involves the overlap of other tilings $\psi$ with $h\setminus \varphi$. 
 This is one of the main challenges of the paper and covers \cref{sec:geometry-of-minimizers,sec:alg,sec:pi}. 
 The main ingredients in this step are: 
 \begin{enumerate}[(a)]
    \item Studying the set of local minimizers of the energy of $h$ given an arbitrary $\varphi$ (done in \cref{sec:geometry-of-minimizers}), and finding a local operation under which this set remains closed. One can then \emph{group} together bubbles that impact one another, so that the resulting ``bubble groups'' are independent. Adding/deleting an entire bubble group $\fB$ will be the basis for a contracting dynamics for $\pi$.
 \item An algorithm to find an approximation of $h$ by a tiling $\psi$ in locations where $h$ differs from~$\varphi$: The idea here is that, aside from a negligible number of ``frozen'' configurations, either the surface $h$ contains too many faces near said location, or it can be approximated by a tiling. The algorithm in \cref{sec:alg} establishes this, and is key to the definition of bubble groups, implying their energy outweighs their entropy (thus they are amenable to a Peierls argument).
 \item Markov chain analysis of Glauber dynamics for $\pi$  which adds/deletes bubble groups (\cref{sec:pi}): this is one of the most technically challenging parts of the proof, as it aims to control a dynamics where moves (changing a bubble group) occur at all scales, and the long-range interactions are non-explicit (they are given in terms of an expectation of a global variable w.r.t.\ a measure~$\mu_t$, akin to the measure $\mu$ from above, but now depending on the current state $h_t$ of the dynamics). It is here that the potential $\sV$ plays a role, as bubbles in \emph{frozen regions} of $\varphi$ might not contract.
 \end{enumerate}
\end{enumerate}
As a byproduct of our analysis of $\pi$, which we recall is $\P_{\alpha,\beta,\lambda}(h\mid\varphi)$, we find that $h$ is a perturbation (via bubble groups with exponentially decaying sizes) of~$\varphi$, and hence has the same limit (\cref{sec:concluding-thm-1}).

\subsubsection{Sketch of \cref{st:1-expand}}\label{sec:sketch-step-1} To simplify the notation in this part, take $\alpha=\beta$.
A~natural approach for the problem would have been to study $\log \P_{\beta,\lambda}(h,\varphi)$
 via cluster expansion techniques. Unfortunately, these fail for 
 the measure in \cref{eq:big-measure}, due to its long-range interactions and exponentially many ground states. 
 Instead, our first step is to perform a free energy expansion whereby, for a probability distribution of the form $\P_\beta(\sigma) = Z_\beta^{-1} \sum_\sigma \exp[-\beta H(\sigma)]$, one has $
 \log Z_\beta = \log Z_\infty + \int_\beta^\infty \E_{\hat \beta}[H] \d\hat\beta$, 
 under a mild condition on the Hamiltonian $H$ (cf.~\cref{lem:grimmett}). Recall that $\Ztile_{\beta}(h)$ from \cref{eq:def-phi-given-h}, which then appears in the denominator of \cref{eq:big-measure}, is $\sum_\psi \exp(\beta|h\cap\psi|)$. We could apply the free energy expansion to $\log\Ztile_\beta$, but it would be better to shift it by $\exp(\beta|h\cap\varphi|)$: for given $h,\varphi$, define
 \[ G(\psi) := |h\cap\varphi|-|h\cap\psi|\qquad,\qquad \overline G := G - \min_\psi G\,,\]
and 
$Z_\beta := \sum_\psi \exp[-\beta \overline G(\psi)]$. 
Expanding $\log Z_\beta$, one can then show (see \cref{sec:deriving-pi-mu-nu}) that
 \[ \P_\beta (h,\varphi) \propto \exp\bigg[-\beta |h| + \beta \big(\min_\psi G\big) -\lambda \sV-\log Z_\infty -\int_\beta^\infty\E_{\mu_{\hat\beta}}[\overline G]\d\hat\beta\bigg]\,,\]
where $Z_\infty = \#\{\psi : \overline G = 0\}$. The intuition behind the terms in this expansion is as follows:
\begin{itemize}
    \item The term $-\beta |h|$ is negative (hence always in our favor), penalizing SOS surfaces $h$ that are wasteful compared to the optimal (minimum) number of faces achieved by tilings.
    \item The term $\beta \big(\min_\psi G\big)$ is again non-positive (hence in our favor: $\min_\psi G(\psi) \leq G(\varphi) = 0$), penalizing SOS surfaces $h$ that can be better approximated by some tiling $\psi$ compared to~$\varphi$, making the latter less likely to be sampled given $h$.
    \item The term $-\log Z_\infty$ points at the near uniform measure on $\varphi$: when many tilings are equally good approximations of $h$, the choice between them ought to be uniform, i.e., $1/Z_\infty$.
     \item The potential $-\lambda \sV$ will help control $h$ in situations when the environment $\varphi$ is frozen.
    \item The final term $\int_\beta^\infty \E_{\mu_{\hat\beta}}[\overline G]\d\hat\beta$ is the culprit in the long-range non-explicit interactions, the main hurdle for the analysis. E.g., changing $h$ will affect $\overline G$, thereby $\mu_{\hat\beta}$ and the interactions...
\end{itemize}
To study the marginal $\P_\beta(\varphi)$, we must sum the right-hand side above over all $h$, and so we apply the free energy expansion for the second time to rewrite $\log(\sum_h \exp[\cdot])$ as $\int_\beta^\infty\E_{\pi_{\hat\beta}}[\cdot]\d\hat\beta + \log Z'_\infty$
for another partition function term $Z'_\infty$. 
Unfortunately, this new $\log Z'_\infty$ term is highly nontrivial; we resort to a third and final application of the free energy expansion for $\log Z'_\infty$, giving rise to the final measure $\nu$ mentioned above (and a corresponding $\int_\beta^\infty \E_{\nu_{\hat\beta}}[\cdot]\d\hat\beta$ term). Overall, we find that:
\[ \P_\beta(\varphi) \propto \exp\bigg[\int_\beta^\infty \Big(\E_{\pi_{\hat\beta}}|h|+ \frac12 \E_{\nu_{\hat\beta}}|\varphi\xor\psi| 
-
\frac12\E_{\mu_{\hat\beta}}|\varphi\xor\psi|
\Big)\d\hat\beta\bigg]\]
(see \cref{prop:P(phi)-via-3-measures}), where both $\nu$ and $\pi$ incorporate a long-range interaction through a $\int \E[\cdot]\d\hat\beta$ term (though the one in $\nu$ does not involve $h$, and hence is much more tractable). It is only in the analysis of $\pi$ where the need for the potential $\sV$ arises.

\subsubsection{Sketch of \cref{st:mcmc-mu-nu}}\label{sec:sketch-step-2}
The goal here (\cref{sec:2-out-of-3}) is to show that $\int_\beta^\infty \E_{\mu_{\hat\beta}}[\cdot]\d\hat\beta$ and $\int_\beta^\infty \E_{\nu_{\hat\beta}}[\cdot]\d\hat\beta$, the first two integrals in the expression for $\P_\beta(\varphi)$ in the last display, are local functions, in that each can be expressed as $\sum_x \sum_r\fg_r(\varphi\restriction_{B(x,r)})$ for $\fg_r$ supported on a ball of radius $r$ with $\|\fg_r\|_\infty\leq \delta e^{-c r}$.

Establishing this for $\mu$ illuminates the basic approach, which will then also be applicable to $\nu$ after taking into account its long-range interactions that involve $\mu$ (whereas the much more difficult task of proving this for $\pi$ is done in \cref{st:pi}). Given $\varphi$, the measure $\mu$ over tilings $\psi$ is defined as 
\[\mu_{\beta}(\psi)\propto \exp\big[-\tfrac12 \beta|\varphi\xor\psi|\big]\,.\]
We consider Metropolis dynamics for $\mu$ that moves by adding or deleting a $(\varphi,\psi)$-bubble---i.e., a connected component of faces of $\varphi\xor\psi$---and equip the space of tilings $\psi$ with a metric $\dist_\sB(\cdot,\cdot)$ via shortest-paths in the graph of moves of the dynamics.  
One can show this dynamics is contracting: two instances of it $(\psi_t)_{t\geq 0},(\psi'_t)_{t\geq 0}$ can be coupled so that $\E \dist_\sB(\psi_t,\psi'_t) \leq e^{-t/2} \dist_\sB(\psi_0,\psi'_0)$. The sought locality of $\int \E_{\mu_{\hat\beta}}[\cdot]\d\hat\beta$ is now obtained by (i)~letting $\mu_r$ be the restriction of $\mu$ which identifies the tiling outside of the ball $B(o,r)$ with $\varphi$; and (ii)~defining $\fg_r$, for $r=2^k$, as the residual contribution of $\mu_{r}$ to the integral compared to $\mu_{r/2}$. Running two coupled instances of the Metropolis dynamics---$\psi_t$ for $\mu_{r/2}$ in $B(o,r/2)$ and $\psi'_t$ for $\mu_r$ in $B(o,r)$, initially agreeing on $B(o,r/2)$---we look at time $T=c \hat\beta r$: each instance will be close to equilibrium, thus the difference in probabilities of observing a bubble $\sB$ in $\mu_{r/2}$ vs.\ $\mu_{r}$ can be reduced to $|\P(\sB\in\psi_T) - \P(\sB\in\psi'_T)|$. The last term is controlled by quantifying the rate at which disagreements between $\psi_t,\psi'_t$ propagate from $\partial B(o,r/2)$ to its interior by time $T$, using that, even though the sizes of bubbles $\sB$ are unbounded, modifying such a $\sB$ of size $s$ in one of the copies, but not in the other, is exponentially unlikely.

The analysis of $\nu$ is similar, but  the interactions between bubbles are no longer nearest-neighbor:
\[ \nu_{\beta}(\psi)\propto 
\exp\bigg[-\frac12 \beta |\eta\xor\varphi| - \frac12\int_{\beta}^\infty \E_{\mu_{\hat\beta}}|\eta\xor\psi|\d\hat\beta
 \bigg]\,,
 \]
and the $\int\E_{\mu_{\tilde\beta}}[\cdot]\d\tilde\beta$ term causes the probability of witnessing a bubble $\sB$ to be affected by distant bubbles (long-range infections). The locality of $\mu$ (already obtained, as above) now assists in the analysis of the Glauber dynamics to sample~$\nu$ and the speed by which it propagates disagreements. 

\subsubsection{Sketch of \cref{st:pi}}\label{sec:sketch-step-3}

Using the formula for $\P_\beta( h, \varphi)$ and the discussion of the effect of its various terms in \cref{sec:sketch-step-1}, we can discuss more concretely the analysis of $\pi_\beta = \P( h  \mid \varphi)$. 
\begin{enumerate}[(a), ref=(\theenumi\alph*)]
    \item Local minimization (\cref{sec:geometry-of-minimizers}): $\min_\psi G$ and $\log Z_\infty$ are a-priori complex non-local functions of the configuration $h$. The remedy is to identify a local operation under which the energy remains minimized: this allows one to \emph{group} together bubbles that impact one another, so that $\min_\psi G$ and $\log Z_\infty$ can be computed separately over these new ``bubble groups.'' Adding/deleting a full bubble group $\fB$ will be the basis for a contracting dynamics for~$\pi$.
    \item \label[step]{st:sketch-3b} Algorithm (\cref{sec:alg}): Next, we need a bound on the energy cost of one bubble group. One could hope from the discussion above that the $-\beta |h|$ and $\beta \big(\min_\psi G\big)$ terms, combined, would be sufficient. This is actually false, as there exist ``counterexamples'' with an almost optimal number of faces where the best tiling approximation still has only a tiny overlap. We still bound the entropy of such counterexamples by an arbitrary constant via an explicit approximation algorithm. Typically, this is enough for a Peierls type estimate because, on a counterexample, the next term $-\log \Ztile_\infty$ is close to the entropy of tiling, beating the small entropy of counterexamples. However, if additionally $\varphi$ is \emph{locally frozen} then this term disappears and it is to handle these cases that the potential $\sV$ is crucial.
    \item \label[step]{st:sketch-3c} Markov chain (\cref{sec:pi}): The analysis of a Glauber dynamics for $\pi$, which adds/deletes bubble groups, is one of the most technically challenging parts of the proof due to the long-range non-explicit interactions from the integral term. The issue is that when one tries to delete a bubble group, if it lies in the middle of a large region containing many other bubbles, then the integral might dominate the other terms, leaving us with very poor bounds. If we were only crafting a Peierls map, the solution to this issue would be obvious: delete the whole region, reducing the energy even more. Making this idea rigorous, however, is extremely delicate, and involves a global Glauber dynamics where moves (modifying bubble groups) occur at all scales, while it is imperative to control the speed of propagating information. 
    (Updating connected regions of every size at rate $1$ helps the dynamics avoid bottlenecks; however, proving contraction and inferring locality of $\int\E_{\pi}[\cdot]\d\hat\beta$ then become much harder.)
\end{enumerate}

We note that the above strategy for \cref{pt:weakly-interacting} appears fairly robust: The ground states of many tilted models %
can be described using lozenge tilings and, in the analysis, the main parts where we exploited details that are specific to \SOS are \cref{st:sketch-3b,st:sketch-3c}, whose extension seems plausible.

\subsubsection{Outline of \cref{pt:gmt}: Extending the GMT  renormalization analysis to long-range interactions}

As per \cite{GMT17,GMT20},
a prototypical application of \cref{thm:gmt} is when the probability of a tiling is tilted by $e^\delta$ for each pair of adjacent lozenges of the same type (so $\fg$ is supported on a ball of radius $1$). The refined \cref{thm:gmt-refinement} amplifies the framework of \cite{GMT17,GMT20} from looking at a finite neighborhood of every tile $x$ to patterns at \emph{all scales}, as long as the effect of a pattern, per $x$, is exponentially small in its diameter.
We next briefly (and informally) explain the crux of obtaining this improvement.

The strategy in \cite{GMT17,GMT20} to understand the interacting dimer model is to compute its generating function before writing correlations as derivatives of it. Hence, for the sake of this outline, we focus on the base partition function: $\sum_\varphi \exp[ \sum_{x, r} \fg_r( \varphi \restriction_{B(x, r)})]$. The idea is to write the functions $\fg_r$ as sums over finite patterns $\fP_i$ and then expand the exponential to get an expression of the form 
\[
\sum_{ \{ \fP_i \} } \prod_i \sw(\fP_i) \# \{ \varphi \, : \; \fP_i \subset \varphi \mbox{ for all $i$}\}\,.
\]
This can be rewritten using Kasteleyn theory as a large sum over many minors of a fixed matrix~$K$ (not quite in the case of a torus but let us ignore that for now). Said sum is then analyzed using various formulae, first relating minors of $K$ to minors of $K^{-1}$, then, for minors of related matrices, representing some form of restrictions of $K^{-1}$ over a fixed scale. The proof is carried by an induction over scales where one must justify at each step the convergence of several series of determinants. 

In \cite{GMT17,GMT20}, for the first few steps of the induction, the authors verify the convergence somewhat effortlessly because they have the freedom to set their parameters small enough to compensate for relatively rough bounds on the  determinants. After these steps, a contracting property of the induction process (which is hard to prove, hence the difficulty of \cite{GMT17,GMT20}) emerges and eventually guarantees convergence at all scales. In our context, the later contraction will also hold and in fact the setting in \cite{GMT17,GMT20} is general enough that it will apply directly. In our context, however, we cannot afford the same bounds on the determinants. Indeed, the bounds are analogous to a cluster expansion and a decay of $e^{- C r}$ for the weight of a ball of size $r$ is not sufficient to handle associated entropy terms (which are roughly of order $(1+e^{-C})^{r^2}\approx e^{\epsilon r^2}$). Our improvement will be to capitalize on the Markov property of the base model with no weights, en route to rewriting the interaction using only the boundaries of the balls (see \cref{sec:grassmann_formulation}, where the starting point is a different Grassmann formulation compared to that of~\cite{GMT17,GMT20}, and in particular the proof of \cref{prop:grassmann_formulation}). This is the key behind the novel aspects of the Grassmann representation in this part of the proof; see \cref{rem:boundary-trick} for more information.
By doing so, we can bring the problem back to a setting with weights $e^{-C |S|}$ on connected sets $S$, as in \cite{GMT17,GMT20}. In effect, we are using the combinatorial interpretation of the determinants and the Markov property of lozenge tilings to find highly nontrivial exact identities reducing $r^2\times r^2$ blocks to $r\times r$ ones. A remarkable feature of this rewriting is also that the a-priori-difficult-to-handle enumeration over all possible tiling of a given region disappears naturally. 

Unfortunately, since \cite{GMT17,GMT20} use the formalism of Grassmann integrals to encode efficiently the series of determinants mentioned above, the proof requires background on both Kasteleyn theory and Grassmann integrals, even if ultimately the key argument in \cref{prop:grassmann_formulation} is fairly short. Moreover, as mentioned in \cref{rem:microcanonical}, in \cite{GMT17,GMT20}, the authors consider tilings of the torus where the number of tiles of each type is allowed to fluctuate (canonical setting) but for us it is important to fix these numbers as per the slope $\theta$ (microcanonical setting), which requires an additional step after the renormalization argument (see \cref{sec:micro2}). %

\subsection{Open problems and future directions}\label{sec:open-prob}

The first natural open problem is to extend our results to the case $\lambda = 0$ (no  extra potential~$\sV$). As mentioned in \cref{sec:sketch-step-3}, the main place where  $\sV$ enters the analysis is in deriving a uniform upper bound on the probability of a bubble group. There are three main difficulties  in allowing $\lambda=0$: (a) the law of $h$ given $\varphi$ will have bubble groups ``stick'' to a ``locally frozen'' region of $\varphi$. As such a region could in turn have a significant influence though the long-range interaction, this makes the measure far less tractable. We believe that a Markov chain analysis can still be applicable to that case, but it would be significantly more complicated, and there is no hope to get an $L^\infty$ bound on the functions $\fg_r$ as we have when~$\lambda>0$; (b) consequently, one would need to further weaken the assumption in \cref{thm:gmt-refinement}, perhaps replacing the $L^\infty$ bound on $\fg_r$ by an $L^p$ bound under the uniform tiling measure; and (c) one would need to bound, under the uniform tiling measure, the probability that a large ball contains a ``locally frozen region'' (appropriately defined as per the previous two steps). %

Another very natural open problem would be to generalize our argument to other random surface models: first and foremost, 3\Dim Ising interfaces, but also random height functions such as the Discrete Gaussian ($|\nabla\phi|^2$) or restricted SOS (gradients are $0$ or $\pm1$) models. The results of Cerf and Kenyon~\cite{CerfKenyon01} mentioned above support the idea that our approach could work for 3\Dim Ising interfaces since it implies that, at least at the surface tension level, 3\Dim Ising can indeed be seen as a perturbation of uniform lozenge tilings. Of course, there is still a very long way from  approximating the surface tension to the result aimed at here. On the whole, our method seems applicable for a model which can be approximated by a random ground state where an analog of \cref{thm:gmt-refinement} can be established (as in the above three models). The technical difficulties, as we mentioned at the end of \cref{sec:sketch-step-3}, are in the analysis of the energy optimization problem at a deterministic level, both for the definition of bubble groups and for the approximation algorithm.

A question surprisingly related to both previous points is the case of slopes with one nonzero coordinate, studied (for a different Hamiltonian) in \cite{BGV01}. We do not expect our approach to be applicable to that case because it features both the conceptual issue of the previous paragraph and the worst of the technical difficulties from the one before. Indeed, for say $\theta_1 > 0$, $\theta_2 = 0$, the set of ground states (on the torus) is given by ``stair-like'' configurations using only two lozenge types. The uniform law on them has of course no fluctuations in the direction ``parallel to the steps'' but in the orthogonal one it becomes a random walk bridge with $\sqrt{N}$ fluctuations. It is unclear whether this behavior survives at positive temperature, and if it does not then the approach is somewhat doomed. The second difficulty is that ``locally stair-like'' regions are actually the ``counterexample'' from \cref{st:sketch-3b}, so in the case $\theta_1 > 0$, $\theta_2 = 0$ the whole tiling $\varphi$ could be a macroscopic counterexample to the algorithm. Overall, it is unclear whether one should expect \GFF type or degenerate Brownian bridge fluctuations; even the order of $\Var(h(o))$ remains open.

Finally, one could ask about moving from a model on the torus to a model on a box with fixed boundary conditions. This should not create any issue for \cref{pt:weakly-interacting,pt:from-phi-to-h} of our general strategy, i.e., the proof that $\varphi$ is a weakly interacting tiling and that $h$ is a small perturbation of $\varphi$. In fact, many of our statements are for fixed boundary conditions, so the corresponding proofs might become slightly easier in that setting. However, the renormalization argument in \cref{pt:gmt} heavily relies on being in the torus, first because it starts with an explicit diagonalization of (a variant of) the adjacency matrix in Fourier space, and second because the action of the renormalization operation on boundary terms is difficult to handle. %
Even for non-interacting lozenge tilings/dimers, understanding fluctuations with ``generic'' boundary conditions remains a major open problem.  

\section{Setup and energies of tilings that approximate SOS}\label{sec:expand}
In this section, following a brief account of preliminaries and setup, we carry out \cref{st:1-expand} of \cref{pt:weakly-interacting} of the proof program outlined above. Recall that our goal 
in that part is to establish that $\varphi$---the random tiling which approximates our SOS surface as per \cref{eq:def-phi-given-h}---is weakly interacting (so as to fulfill the hypothesis of \cref{thm:gmt-refinement}). This result is formalized as follows, where we recall
from the introduction that $\varphi\restriction_{B(x,r)}$ denotes the set of lozenges of $\varphi$ intersecting the ball $B(x,r)$ in $\T$.

\begin{theorem}\label{thm:phi-weakly-interacting}
Let $h$ be a $(2+1)$\Dim \SOS surface as per \cref{eq:tilted-sos} for a potential $\sV$ as in \cref{def:pinning-V}. For every $C^\star>0$ there exists $C>0$ so the following holds. 
Let $\varphi$ be a tiling given $h$ as per \cref{eq:def-phi-given-h}.
If $0<\lambda< \frac{1}{C}$ and $\alpha \wedge \beta \geq C\lambda^{-21}$, then there exist functions $\fg_r$, for $r=2^k$ ($k=0,1,\ldots$), satisfying
\[ \|\fg_r\|_\infty \leq C e^{-(\alpha \wedge \beta)/2-C^\star r}\quad\mbox{for all $r$}\,,\]
and a normalizer $Z_{N}$, such that for every $N$,
\[ \bigg|\frac{\P_{\alpha,\beta,\lambda}(\varphi)}{Z_{N}^{-1} \exp\left[ \sum_{x\in\T_N}\sum_{0\leq r< N/2} \fg_{r}(\varphi\restriction_{B(x,r)})\right]} - 1 \bigg| \leq C e^{-(\alpha \wedge \beta)/2-C^\star N}
\,.
\]
In addition, for a potential $\sV$ satisfying the relaxed condition 
\begin{equation}\label{eq:gen-potential-weaker}\sV(h)=\sum \ff(S_i)\quad\mbox{for some function $\ff$ with }\quad \ff(S)\geq |S|\one_{\{|S|\geq \sM_0\}}\end{equation}
(the sum is over the connected components $S_i$ of $\psi_0\setminus h$; cf.~\cref{eq:gen-potential}) for an arbitrarily large $\sM_0>0$, there exists an absolute constant $C>0$ such that, if $0<\lambda<\frac1{C}$ and $\alpha\wedge\beta\geq C\lambda^{-30}+\sM_0$ then the  conclusion holds true with the weaker bounds where $C^\star$ is replaced by $\lambda^3/(C\sM_0)$.
\end{theorem}
(Notice that, in the above, the radius of the ball $B(x,r)$ at the final largest scale is $\frac{N}4 \leq r < \frac{N}2$.)
The analysis in this section will express $\P_{\alpha,\beta,\lambda}(\varphi)$ in terms of measures $\mu,\nu,\pi$ (see \cref{prop:P(phi)-via-3-measures}). This will reduce \cref{thm:phi-weakly-interacting} into proving these measures are local (in increasing order of difficulty):
\Cref{thm:mu-nu} establishes this for $\mu,\nu$, whereas \cref{thm:pi} gives the analogous statement for $\pi$. Combining these two theorems with \cref{prop:P(phi)-via-3-measures} will thus imply the above result on $\P_{\alpha,\beta,\lambda}(\varphi)$.

\subsection{Preliminaries and setup}\label{subsec:setup}
We now import some background on the relation between height functions on the square lattice $\Z^2$ and lozenge tilings of the triangular lattice $\T$, as well as the \GFF, adding context to the results in \cref{thm:GFF-convergence,thm:gmt-refinement} (e.g., the mode of convergence to the \GFF and the notion of viewing lozenge tilings of $\T$ as height functions projected on $\cP_{001}$ as opposed to $\cP_{111}$). 

\begin{figure}
\vspace{-0.1in}
\begin{tikzpicture}
   \node (fig0) at (0,0) {
   	\includegraphics[width=0.28\textwidth]{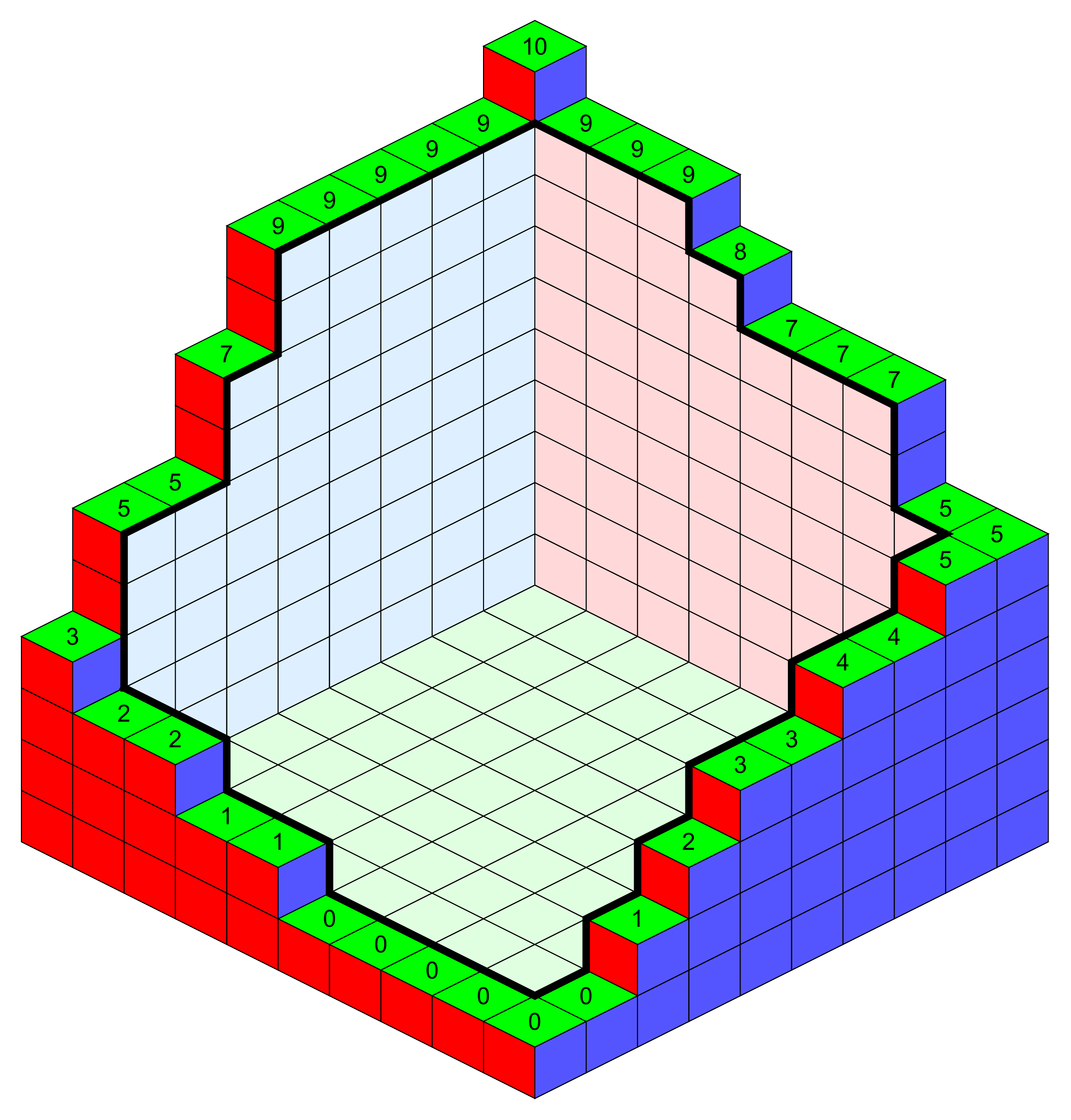}};
    \node (fig1) at (5,0) {
    \includegraphics[width=0.28\textwidth]{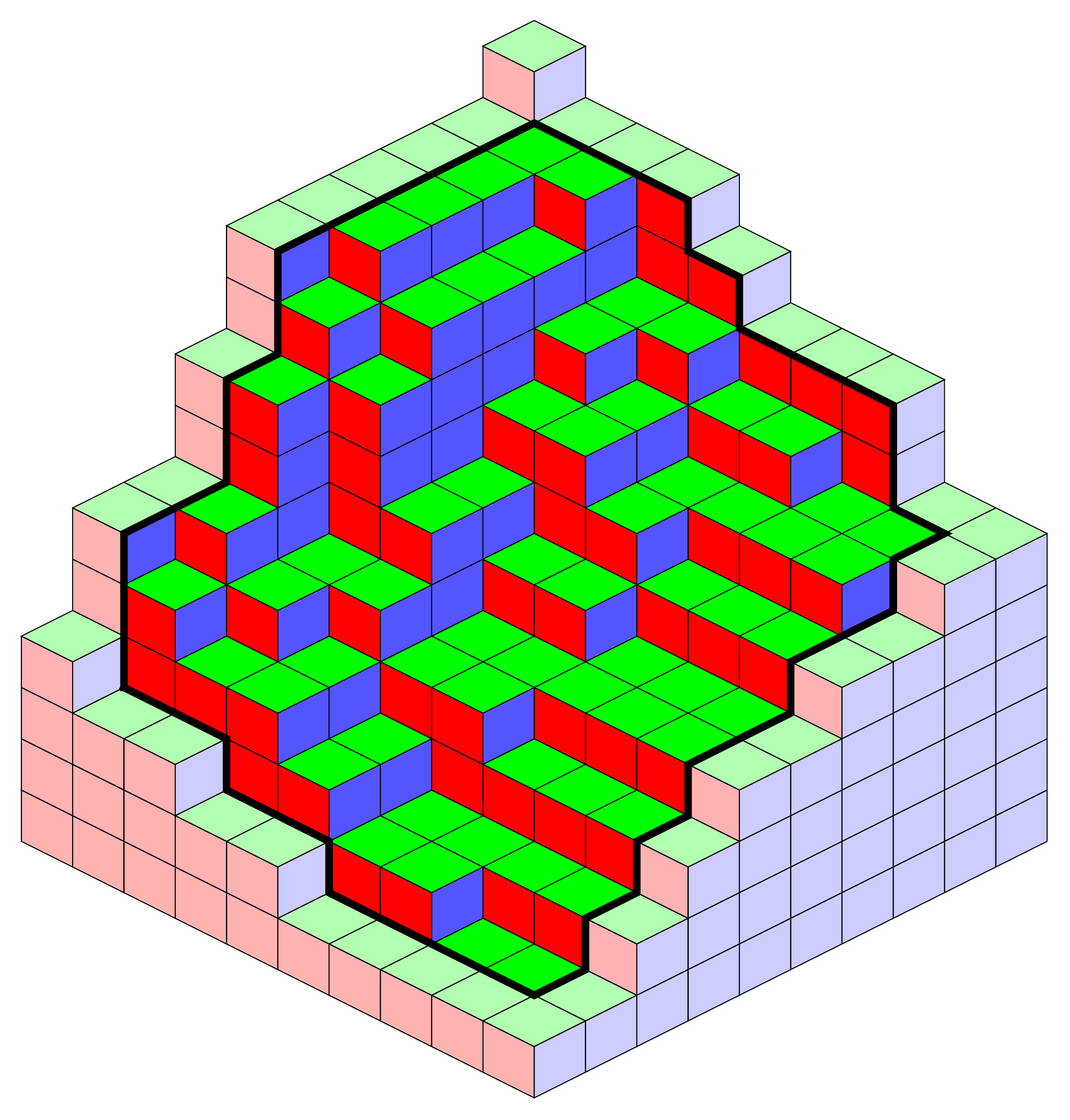}
    };
    \node (fig2) at (10,0) {
   	\includegraphics[width=0.28\textwidth]{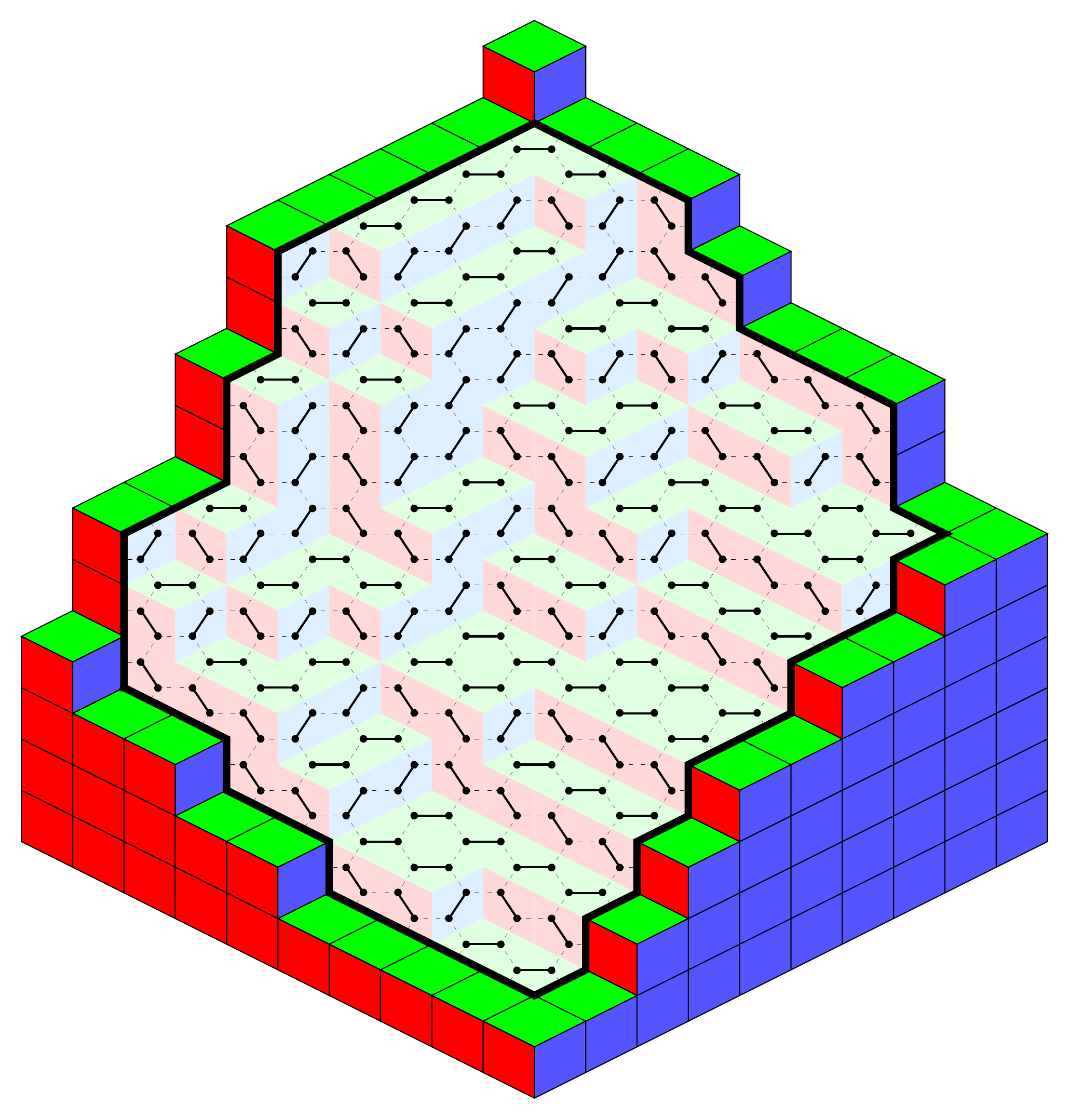}};
    \end{tikzpicture}
    \vspace{-0.2in}
    \caption{
    Monotone surface in $\Z^2$ vs.\ lozenge tiling of $\T$ (dimers in the hexagonal lattice). The fixed boundary heights (on left) determine the boundary of the tiled region (on right).
    (Unlike \cref{fig:sos-approx}, the lozenge tiling here is not a periodic tiling of the plane.)
    }
    \label{fig:surface-bc-tiling-dimers}
\end{figure}

\subsubsection{Surfaces and projections}
A plaquette, or face, in $\Z^3$ is a unit square that is either horizontal (with opposing corners $x,x+(1,1,0)$) or vertical (opposing corners $x,x+(1,0,1)$ or $x,x+(0,1,1)$). The SOS height function $h$ assigns a horizontal face at height $h(x)$ to each of the horizontal faces $x=(x_1,x_2)$ of the $N\times N$ square grid. It is viewed as a surface via a minimum completion of vertical faces to make it simply connected in $\R^3$, i.e., $|h(x)-h(y)|$ vertical faces between two neighboring faces $x,y$.
We will routinely move between viewing $h$ as a height function and viewing it as the set of faces comprising its interface, a subset of the set of all possible plaquettes in $\Z^3$. As there are always exactly $N^2$ horizontal faces in $h$, the leading term $\beta |h|$ in the SOS Hamiltonian aims to minimize the number of vertical faces. As such, under any fixed boundary conditions that are monotone decreasing along the $(x_1,x_2)$ coordinates, the ground state of the SOS model would be a monotone (decreasing) surface. The three types of faces then correspond to a lozenge tiling of the triangular lattice $\T$ (see \cref{fig:surface-bc-tiling-dimers}, where  vertical faces with opposing corners $x,x+(1,0,1)$ are in blue, vertical faces with opposing corners $x,x+(0,1,1)$ are in red, and horizontal faces in green). 

When the boundary condition is periodic with slope $\theta$, as per \cref{eq:h-periodic-bc}, the number of blue vertical faces per row $(x_1,\cdot)$ and the number of red vertical faces per given column $(\cdot,x_2)$ are each predetermined $(\lfloor \theta_2 N\rfloor$ and $\lfloor \theta_1 N\rfloor$, respectively), in which case a full plane periodic monotone surface in $\Z^2$ corresponds to a periodic lozenge tiling of $\T$; see \cref{fig:sos_tiling_periodic_bc}.

\begin{figure}
\begin{tikzpicture}
   \node (fig1) at (0,0) {
    \includegraphics[width=0.6\textwidth]{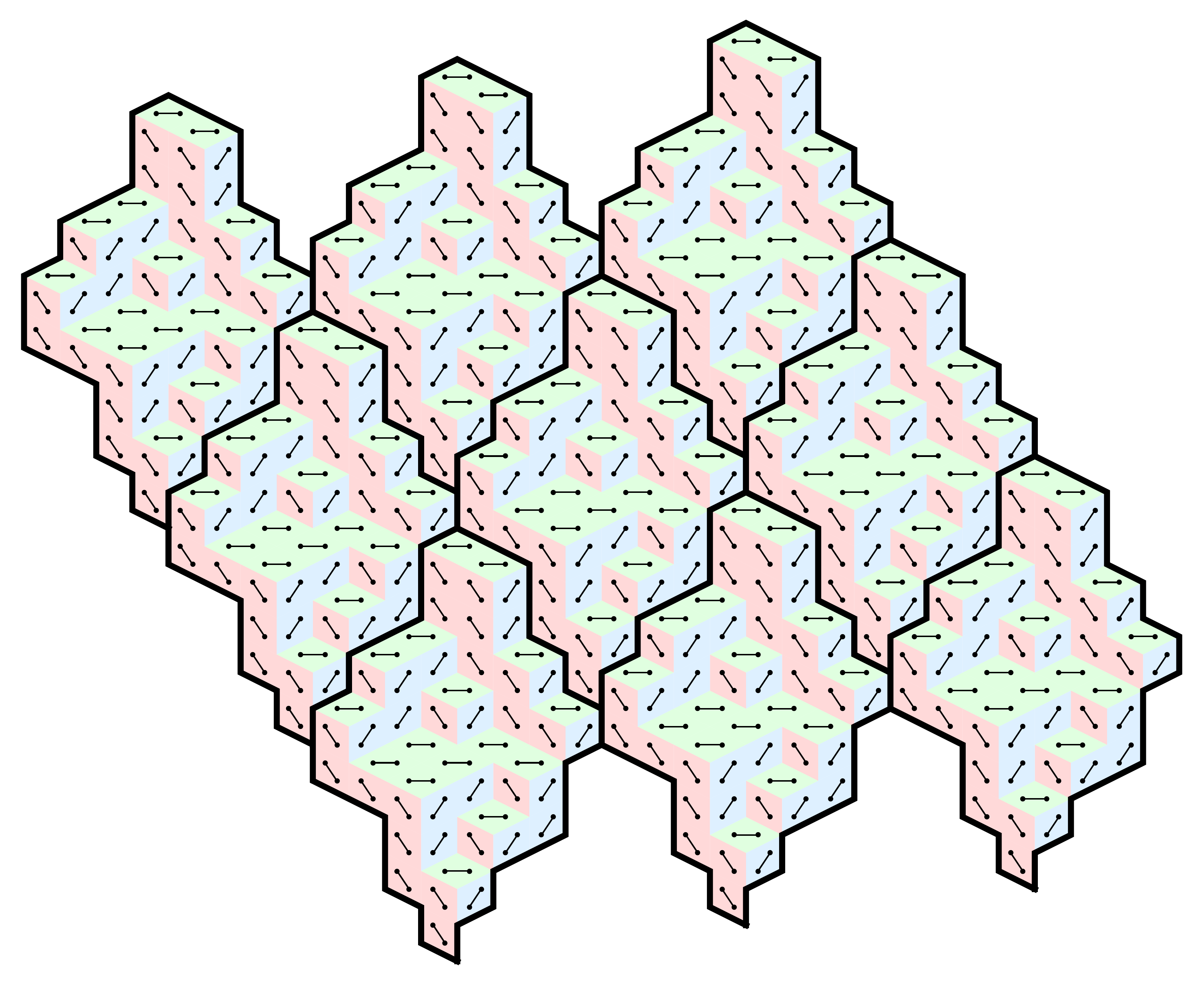}
    };
    \node (fig2) at (7,1.) {
   	\includegraphics[width=0.4\textwidth]{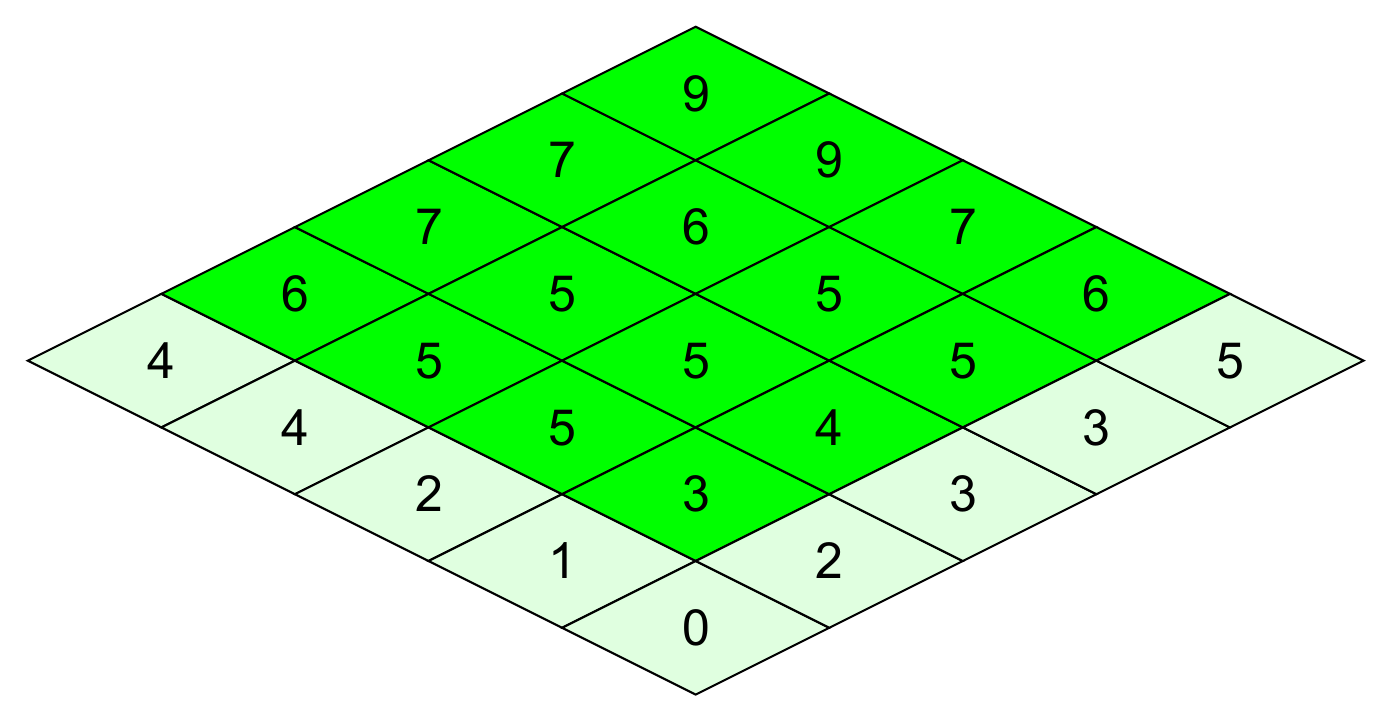}};
    \end{tikzpicture}
    \vspace{-0.25in}
    \caption{
    Monotone height function on a $4\times 4$ torus with $\sum \nabla h(\vec e_i) =-5$ for loops along the $(1,0)$-direction and $-4$ along the $(0,1)$-direction, and the periodic tiling of $\T$ it induces.
    }
\label{fig:sos_tiling_periodic_bc}
\end{figure}

In the above notion of height functions (describing SOS configurations $h$ as well as tilings $\varphi$, viewed as their special case of monotone surfaces), the height of faces was measured via a projection onto the $x_3=0$ plane. We now discuss the relation between this notion and the one where the heights are projected onto the $x_1+x_2+x_3=0$ plane, as is common in the study of dimers. 
At the discrete level, in the full plane, the link between the two descriptions is as follows:

\begin{definition}
Let $\varphi$ be a discrete monotone surface, seen as a union of plaquettes of $\Z^3$ (with a chosen root). Let $\cP_{001}$ denote the plane with the equation $x_3 = 0$, let $\cP_{111}$ be the plane with the equation $x_1 + x_2+x_3 = 0$ and let $\Upsilon_{001}$ and $\Upsilon_{111}$ denote the orthogonal projections on these planes. Note that $\Upsilon_{001} (\Z^3)$ is the square lattice $\Z^2$ while $\Upsilon_{111}(\Z^3)$ is the triangular lattice $\T$.
    \begin{enumerate}[(i)]
        \item The $\cP_{001}$ height function, or height with SOS convention, assigns heights to  the faces of~$\Z^2$. %
        Given a face $u$ of $\Z^2$, there exists a unique plaquette $f$ of $\Z^3$ with $f \in \varphi$ and $\Upsilon_{001}(f) = u$, and $\varphi_{001}(u)$ is defined as the third coordinate of $f$ (well defined since $f$ is parallel to $\cP_{001}$).
        \item The $\cP_{111}$ height function, or height with tiling convention, assigns heights to the vertices of~$\T$ (or equivalently, the faces of the hexagonal lattice). Given a vertex $v$ of $\T$, there exists a unique vertex $x$ with $\Upsilon_{111}(x) = v$ and $x \in \varphi$. We let $\varphi_{111}(v)$ be the third coordinate of $x$.
    \end{enumerate}
\end{definition}

See \cref{fig:P111-P001} for an illustration of the $\cP_{001}$ and $\cP_{111}$ height functions.
On the torus, one defines these simply by first mapping the configuration from the torus to the full plane in the natural way. 
Similarly, the notions of $\cP_{001},\cP_{111}$ height functions extend to continuous surfaces (in $\R^3$).

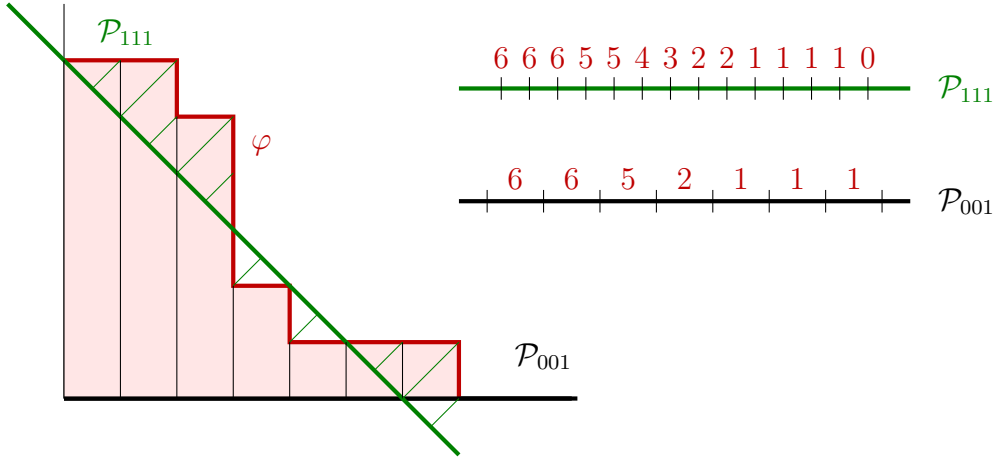
\begin{figure}
      \begin{tikzpicture}
    \begin{scope}[scale=0.75]
        \draw (0,0)--(0,7);
        \draw[ultra thick] (0,0)--(9.1,0);
        \filldraw[draw=none,red!10] (0,7)--(0,6)--(2,6)--(2,5)--(3,5)--(3,2)--(4,2)--(4,1)--(7,1)--(7,0)--(0,0)--cycle;
        \draw[ultra thick,red!75!black] (0,6)--(2,6)--(2,5)--(3,5)--(3,2)--(4,2)--(4,1)--(7,1)--(7,0);
        \node[font=\large,red!75!black] at (3.5,4.5) {$\varphi$};
        
        \draw[ultra thick,black] (0,0)--(9,0);
        \node[font=\large] at (8.5,0.75) {$\mathcal{P}_{001}$};

        \draw[ultra thick,green!50!black] (-1,7)--(7,-1);
        \node[font=\large,green!50!black] at (1.1,6.5) {$\mathcal{P}_{111}$};

        \coordinate (p001) at (7,3.5);
        \draw[ultra thick, black] (p001)--($(p001)+(8,0)$);
        \node[font=\large] at ($(p001)+(9,0)$) {$\mathcal{P}_{001}$};
        
        \draw ($(p001)+(.5,-0.2)$)--($(p001)+(.5,0.2)$);
        
        \foreach \x/\y in {0/6,1/6,2/5,3/2,4/1,5/1,6/1} { \draw (\x,0)--(\x,\y);
          \draw ($(p001)+(\x+1.5,-0.2)$)--($(p001)+(\x+1.5,0.2)$);
          \node[font=\large,red!75!black] at ($(p001)+(\x+1,0.4)$) {$\y$};
          };

        \coordinate (p111) at (7,5.5);
        \draw[ultra thick, green!50!black] (p111)--($(p111)+(8,0)$);
        \node[font=\large,green!50!black] at ($(p111)+(9,0)$) {$\mathcal{P}_{111}$};

        \foreach \x/\y in {0/6,1/6,2/6,3/5,4/5,5/4,6/3,7/2,8/2,9/1,10/1,11/1,12/1,13/0} { 
          \draw ($(p111)+(\x/2+0.75,-0.2)$)--($(p111)+(\x/2+0.75,0.2)$);
          \node[font=\large,red!75!black] at ($(p111)+(\x/2+0.75,0.55)$) {$\y$};
          };

        \foreach \x/\y/\z in {1/6/5.5,2/6/5,2/5/4.5,3/5/4,3/4/3.5,3/2/2.5,4/1/1.5,6/1/0.5,7/1/0,7/0/-.5} {
          \draw[green!50!black] (\x,\y)--(6-\z,\z);
          };

    \end{scope}
\end{tikzpicture}
    \caption{A schematic 2D representation of the two height function conventions. Left: A discrete monotone function $\varphi$ in red with the projections on $\cP_{001}$ and $\cP_{111}$ represented by black or green fine lines. Right: The corresponding height functions, top with tiling convention and bottom with SOS convention. The non-linearity of the transformation from one convention is apparent, e.g., in the drop from $5$ to $2$ in the $\cP_{001}$ convention, which corresponds to the interval $(5,4,3,2)$ in the $\cP_{111}$ convention.}
    \label{fig:P111-P001}
\end{figure}

Let us next discuss the pinning conventions. The joint law on $(h,\varphi)$ as per \cref{eq:tilted-sos,eq:def-phi-given-h}, regardless of the pinning $h_{001}(o)=0$ that we specified in the torus (where $o$ is the origin face of $\Z^2$), is a law on $(\nabla h_{001},h_{001}-\varphi_{001})$ which is invariant under translations in the plane $\cP_{001}$. Note that, equivalently, it can also be seen as a joint law on $(\nabla \varphi_{001},h_{001}-\varphi_{001})$, which will be more convenient in our analysis after we invert the order of the conditioning, focusing first on the marginal on $\varphi$ and then viewing $h$ as a perturbation given $\varphi$. This already introduces a complication because, from that point of view, pinning $h(o)$ to $0$ is no longer natural (said pinning, as seen by $\varphi$, would be carried out onto a random hidden SOS configuration, as opposed to a deterministic pinning). We will thus need to change our pinning convention in that context.

Another complication is that we will need results on the dimer model which require translation invariance with respect to $\cP_{111}$ (the usual setting in the dimer literature). However, if we pin $\varphi_{001}(o)$ to $0$, we break this translation invariance since this amounts to forcing the origin to be covered by a green tile. (For this reason we did not, for instance, specify that $h_{001}(o)=\varphi_{001}(o)=0$ when introducing $\varphi$ in \cref{eq:def-phi-given-h}, and instead just asked that $\varphi \cap h \neq \emptyset$. We could have asked for the former, if we were to then re-root a uniformly chosen face of $\Lambda_N$ to be at height $0$ instead of $o$.)

To address these two issues, we instead pin $\varphi_{111}(o)$ to $0$, where $o$ is now the origin of the triangular lattice $\T$, or equivalently in terms of monotone surfaces, require that $(0,0,0) \in \varphi$ (i.e., it is a corner of a plaquette in $\varphi$). It can then be checked (e.g., by considering the $\sigma$-finite measure on $(h, \varphi)$ obtained by giving measure $1$ to every possible height shift, which is invariant under all $\Z^3$ translations) that the resulting law of $\varphi$ is indeed invariant under $\cP_{111}$ translations as needed. See \cref{fig:projection_lozenge} and its caption for details on how to read both height functions ($\cP_{001},\cP_{111}$) from a tiling.

\begin{figure}
\begin{tikzpicture}
    \node (fig) at (0,0) {
    \includegraphics[width=.35\textwidth]{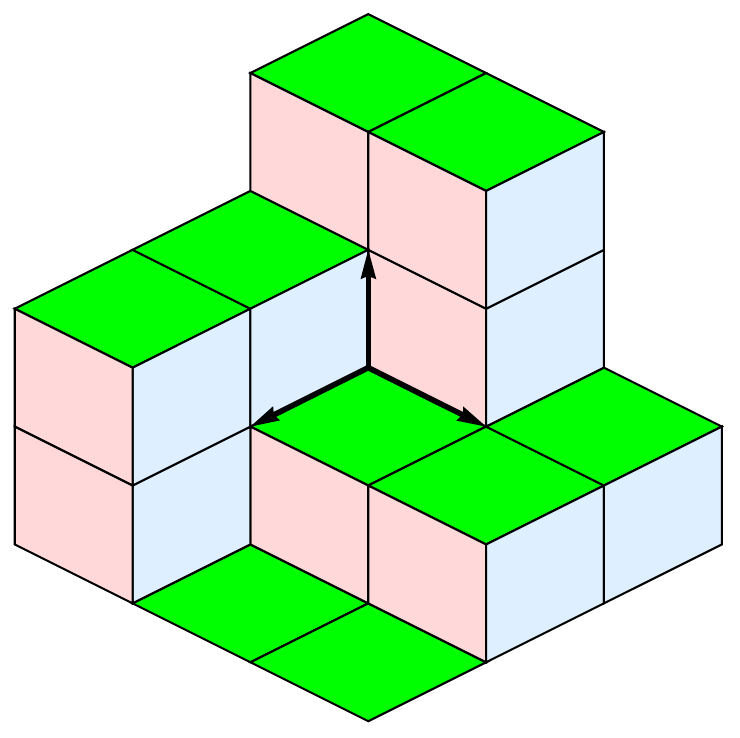}
    };
    \node[font=\scriptsize] at (-1,0.95) {$(\frac12,-\frac12)$};
    \node[font=\scriptsize] at (0,-.42) {$(\frac12,\frac12)$};
    \node[font=\scriptsize] at (0.9,1.86) {$(-\frac12,\frac12)$};
    \node[font=\scriptsize] at (-0.0,2.33) {$(-\frac12,-\frac12)$};
    \node[font=\large] at (.25,.35) {$e_\uparrow$};
\end{tikzpicture}
    \caption{The effect of the $\Upsilon_{001}$ projection on a  tiling. The root $o$ is the origin of the vector~$e_{\uparrow}$. Each green horizontal lozenge corresponds to a face in $\cP_{001}$ and is labeled with the first two coordinates of its center. Successive green lozenges along the $e_{\uparrow}$ direction correspond to the faces along the main diagonal of $\cP_{001}$ and, for $u$ on that diagonal, $\varphi_{111}(u)$ is the number of vertical edges not covered by a lozenge between the root $o$ and corresponding tile. In other words, for $k , \ell \geq 0$, one has $\varphi_{111}(o + k e_{\uparrow}) \leq \ell$ when $\varphi_{001}( - (k-\ell) + \frac{1}{2}, - (k-\ell) + \frac{1}{2} ) \leq \ell$ and $\varphi_{111}(o + k e_{\uparrow}) \geq \ell$ when $\varphi_{001}( - (k-\ell) - \frac{1}{2}, - (k-\ell) - \frac{1}{2} ) \geq \ell$.}
    \label{fig:projection_lozenge}
\end{figure}
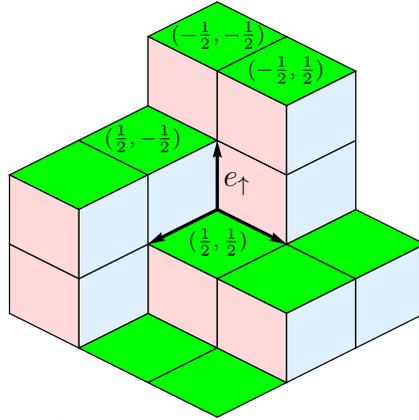

Note that while the $\cP_{001}$ height function is natural for an SOS discrete surface, the $\cP_{111}$ height function is far less so: for instance, a ``spike'' in $h_{001}$ (an isolated column, e.g., $h_{001}(x)=2$ and $h_{001}(y)=0$ for all $y\sim x$) becomes an overhang from the $\cP_{111}$ point of view (no longer well-defined).

It will be convenient to view our discrete surfaces ($h$ and $\varphi$) as continuous surfaces in $\R^3$, whereby the notions of height functions $\varphi_{111}(x)$, $h_{001}(y)$ will make sense for any $x \in \cP_{111}$ and $y$ in $\cP_{001}$ minus the edges of $\Z^3$. In the following, when integrating a discrete height function, we will always consider it as extended as above and we will use $\langle \cdot ,  \cdot \rangle$ to denote the $L^2$ inner product.

\subsubsection{The Gaussian free field}
The \GFF can be viewed as a natural extension of Brownian motion (or a Brownian bridge) to a parameter space more general than $\R_+$. There is an extensive literature on it, well beyond the scope of this paper; here we will only discuss the basic definition of the \GFF on $\R^2$ and simply connected domains of $\R^2$, and the meaning of the convergence in \cref{thm:GFF-convergence}.

In analogy to the case of Brownian motion, the simplest definition of the \GFF on a bounded open domain $D$ (say with a smooth $\partial D$) should be as a centered Gaussian process indexed by points of $D$ where one only needs to fix the covariance matrix, a natural choice being to take the Green function with Dirichlet boundary conditions (the choice of normalization of the Green function is unfortunately not completely canonical here). Unfortunately, this does not make sense directly because the Green function diverges on the diagonal. The solution is to see the \GFF as a stochastic process indexed by (regular enough) test functions $f$; in what follows we denote its value at $f$ by $\langle \GFF, f\rangle$ instead of $\GFF(f)$ to emphasize that it is viewed as a Schwartz distribution. The natural choice of covariance is then
\[
\Cov( \langle \GFF_D, f\rangle, \langle \GFF_D, g\rangle ) := \int_{D} \int_D f(u) g(v) G_D(u,v)\, \d u \d v
\]
where $G_D$ is the Green function of the domain $D$ with Dirichlet boundary conditions.

We would like to just take $D = \R^2$ but again this does not work quite directly since the full plane Green function does not define a valid covariance. In fact, this issue is already present in the 1\Dim case when one wants to define a full plane Brownian motion: the law of all increments is perfectly well defined (and has the symmetries of $\R$) but one needs to pin the process at a point to turn these increments into an actual process. In the \GFF case, since the value at a point is undefined, one could similarly pin the value of $\langle \GFF, f\rangle$ for some fixed $f$, e.g., the indicator of a ball, but an arbitrary convention of such a sort complicates the analysis. It is more elegant to stick with only the law of all increments, which in the generalized function point of view is equivalent to restricting the set of test functions to mean $0$ ones. This leads to a definition as given above \cref{thm:GFF-convergence}:
\begin{definition}\label{def:gff-stoch-process}
    The full plane \GFF is the centered Gaussian process indexed by smooth test functions from $\R^2$ to $\R$ with $0$ mean and covariance
    \[
    \Cov(  \langle \GFF_{\R^2}, f\rangle, \langle \GFF_{\R^2}, g\rangle ) := -\frac{1}{2\pi}\int_{\R^2} \int_{\R^2} f(u) g(v) \log |u-v|\, \d u \d v .
    \]
\end{definition}
 The convention for the constant in front of the $\log$ is not completely standard, we use the one from \cite{BerestyckiPowell25}. With this convention, the limit of the height function for uniform dimers is $\sqrt{\frac{2}{\pi}} \GFF$.
 
In the Brownian motion case, the next step after the definition as a stochastic process is to prove the existence of continuous versions, that is, to establish regularity estimates. The analog for the 2\Dim \GFF is to say that it can be realized in a ``concrete'' space of distribution. The full plane case (as opposed to a bounded domain $D$) is somewhat awkward here since one would need to use a weighted space (cf., e.g.,~\cite[\S6.4]{BerestyckiPowell25}), but for a bounded domain one has the following:
\begin{proposition}[{\cite[Thm.~1.45]{BerestyckiPowell25}}]
    For any $\epsilon>0$, for any simply connected bounded open set $D$, there exists a version of $\GFF_D$ which is a random variable in the Sobolev space $H^{-\epsilon}(D)$.
\end{proposition}
In fact, this version of the \GFF can be written explicitly as the almost surely convergent series $\GFF_D = \sum_{n} X_n f^{D}_n$ where $\{X_n\}$ are i.i.d.\ standard normal variables and $\{f^{D}_n\}$ are the eigenvectors for the Laplacian in $D$ with Dirichlet boundary condition, normalized to have unit $H^1$-norm. Note that the almost sure convergence in $H^{-\epsilon}$ of the series is part of the proposition and non-obvious but in $H^{-1-\epsilon}$ it is not hard to check that the series becomes absolutely convergent almost surely.

When discussing convergence to the \GFF of some discrete random function $h_n$, the simplest approach is to view the \GFF as a stochastic process (as in \cref{def:gff-stoch-process}), i.e., to show that for every smooth $f$ (possibly with $0$ mean in the full plane case), $\langle h_n, f \rangle$ converges to a centered Gaussian with variance $\iint f(u)f(v) G(u,v)\d u \d v$. Indeed, this is the notion of convergence in \cref{thm:GFF-convergence}. For a concrete example of a property of the \SOS function $h$ that one can read from the \GFF limit via this mode of convergence, take any $x,y\in \R^2$, and set
\[ f(u) = \delta_\epsilon(u-x) - \delta_\epsilon(u-y)\,, \]
where $\delta_\epsilon(u)=\epsilon^{-2}\delta_1(u/\epsilon)$ for a smooth probability density $\delta_1$ supported on $B(o,1)$. \Cref{thm:GFF-convergence} shows that this height statistic for $h$, comparing the local heights near $\lfloor n x\rfloor $ vs.\ $\lfloor n y\rfloor $, converges as $N\to\infty$ followed by $n\to\infty$ to a centered Gaussian whose variance, for $x\ne y$, is
\[
\frac{\sigma^2}{\pi}\left(\log\frac{|\fL(x-y)|}{\epsilon}+c_{\delta,\fL}\right)+o(1)\,,\qquad
c_{\delta,\fL}:=-\iint\delta_1(u)\delta_1(v)\log|\fL(u-v)|\d u\d v\,,
\]
as $\epsilon\to0$, with $\sigma,\fL$ as in \cref{thm:GFF-convergence}.

Regarding other notions of convergence, at the discrete level $h_n$ is an actual function and, in many cases (including in the proofs of \cref{thm:gmt} and its refinement \cref{thm:gmt-refinement} derived here), it is natural along the way to compute pointwise correlations (or higher moments of $h_n$), such as proving that (in a domain~$D$)
\[
\Cov( h_n(u), h_n(v) ) \to G_D(u,v)
\]
(in fact, in Kenyon's paper~\cite{Kenyon00}, this is the targeted notion of convergence) and similar estimates with more points. These pointwise bounds are actually conceptually stronger than the convergence as a stochastic process since, in a sense, this mode of convergence controls the joint law of the values of the \GFF at different points even though this law has no formal existence. Indeed, in the proof of \cite[Thm.~5.1]{BLR20}, it is shown that the convergence of all $n$-point functions with fixed distinct points coupled with even a fairly rough bound on the divergence as $u -v \to 0$ is enough to prove not only convergence in $H^{-1-\epsilon}$ but even to control all moments of the $H^{-1-\epsilon}$-norm of the field. For us, a similar statement would also hold, up to some small extra difficulties involving weighted spaces to accommodate the full plane picture.%

\subsection{Deriving the three measures 
\texorpdfstring{$\mu,\nu,\pi$}
{\textmu,\textnu,\textpi}
}\label{sec:deriving-pi-mu-nu}
We will decompose $\P_{\alpha,\beta,\lambda}(\varphi)$ as follows:
\begin{proposition}
    \label{prop:P(phi)-via-3-measures}
    In the setting of \cref{thm:GFF-convergence}, the measure $\P_{\alpha,\beta,\lambda}$ from \cref{eq:def-phi-given-h} satisfies
    \begin{equation}
\P_{\alpha,\beta,\lambda}(\varphi) \propto \exp\bigg[ -\int_{\alpha}^\infty \mu_{\varphi, \hat \alpha}(\tfrac12|\varphi\xor\psi|)\d\hat\alpha
+\int_{\alpha}^\infty\nu_{\varphi,\bar\alpha}(\tfrac12|\eta\xor\varphi|)\d\bar\alpha
+\int_{\beta}^{\infty}\pi_{\varphi,\hat\beta}(H_h) \d\hat\beta \bigg] \,,
\label{eq:phi-3-measures}\end{equation}
for measures $\mu_{\varphi,\hat\alpha}$, $\nu_{\varphi,\bar\alpha}$ and $\pi_{\varphi,\hat\beta}$ defined below in \cref{eq:mu-def,eq:nu-def,eq:pi-def} respectively.
\end{proposition}

\begin{proof}
We begin with the free energy expansion identity that was briefly described in \cref{sec:sketch-step-1}---a folklore approach of expressing the log-partition function as a certain integral over the temperature (e.g., see \cite[Lemma 7.90]{Grimmett_RC} for a version of it specialized to the random cluster model); we include its short proof here for completeness.
\begin{lemma}\label{lem:grimmett}
Let $Z_\beta = \sum_{x\in\fX} e^{-\beta H(x)}$ for $\beta>0$ and a function $H:\fX\to\R$ on a finite set $\fX$.
Let $F:\fX\to\R$ be a function  such that $\min_{x\in X}F(x) = 0$, and for any $\hat\beta\geq 0$, set
\[  Z_{\beta}^{\hat\beta} = \sum_{x\in \fX} \exp\left[-\beta H(x) - \hat\beta F(x)\right]\,.\]
With $\left<\cdot\right>_{\beta,\hat\beta}$ denoting expectation w.r.t.\ $\P_{\beta}^{\hat\beta}(x):=
 (Z_{\beta}^{\hat\beta})^{-1}\exp[-\beta H(x) - \hat\beta F(x)]$, one has
\[ \log Z_\beta = \int_0^\infty \left<F\right>_{\beta,\hat\beta} \d \hat\beta + \log Z_{\beta}^{\infty}\,.\]
In the special case where $H=F$ we have that 
\[ \log Z_\beta = \int_\beta^\infty  \left<H\right>_{\hat\beta}\d \hat\beta + \log Z_{\infty}\,,\]
where 
$\left<\cdot\right>_{\hat\beta}$ denotes expectation w.r.t.\ $\P_{\hat\beta}(x):=
 (Z_{\hat\beta})^{-1} \exp\big[-\hat\beta H(x)\big]$.
\end{lemma}
\begin{proof}
The conclusion follows from the elementary fact  $\frac{\d}{\d\hat\beta}\log Z_{\beta, \hat \beta}^{\hat\beta} = \left<-F\right>_{\hat\beta}$, using that $\lim_{\hat\beta\to\infty} Z_{\beta}^{\hat\beta} $ exists and is given by $\sum_{x\in\fX} \exp\big[-\beta H(x)\big]\one_{\{F(x)=0\}}>0 $ as $F(x)\geq 0$ and $F$ attains its minimum $0$ on $\fX$. The special case follows by a change of variables, since $\P_{\beta}^{\hat\beta} = \P_{\beta+\hat\beta}$ in that case, and $F$ does not depend on $\beta,\hat\beta$.
\end{proof}
Define
\begin{equation}\label{eq:G-def}
G_{h,\varphi}(\psi) =  |h \cap \varphi| - |h\cap\psi| \,,\end{equation}
so that \cref{eq:big-measure} gives
\begin{align}\label{eq:towards-P(phi)-via-G}
\P_{\alpha,\beta,\lambda}(\varphi) 
&= \frac1{\Zsos_{N,\beta,\lambda} } \sum_h \frac{e^{-\beta |h| - \lambda \sV(h)}}{\sum_\psi e^{-\alpha G_{h,\varphi}(\psi)}}\,.
\end{align}

We also define
\begin{equation}\label{eq:G-g-def} G^\gr_{h,\varphi} = \min_\psi G_{h,\varphi}(\psi) \, ,
\end{equation}
 and, observing that $G^\gr_{h,\varphi}\leq G_{h,\varphi}(\varphi) = 0$, further set
\begin{equation}\label{eq:G-bar-def}
\overline{G}_{h}(\psi) = G_{h, \varphi}(\psi) - G^\gr_{h, \varphi}\,.
\end{equation}
Note that we dropped $\varphi$ in $\overline{G}_h$ because (unlike $G_{h,\varphi}^\gr$) it does not actually depend on $\varphi$: indeed, with $|h\cap\varphi|$ canceling out, we are left with
$ \overline G_{h}(\psi) = - |h\cap\psi| + \max_{\psi_0} |h\cap\psi_0|$.

We may now apply the (special case $H=F$ of) \cref{lem:grimmett} onto 
\[ Z_{h}^{\alpha}= \sum_\psi e^{-\alpha \overline{G}_{h}(\psi)}\] 
for $F(\psi)=\overline{G}_{h}(\psi)$, noting that $F\geq 0$ and $F(\varphi)=0$ as mentioned above. Since we will need to apply \cref{lem:grimmett} multiple times and keep track of the different measures, let us introduce the following notations for the quantities involved in that lemma. Define the probability measure $\mu_{h,\hat\alpha}$ on tilings $\psi$ by
\begin{equation}
    \label{eq:mu-def}
    \mu_{h,\hat\alpha}(\psi) = \frac1{Z_\mu^{\hat\alpha}(h)} \exp\left[-\hat\alpha \overline G_{h}(\psi)\right]\,,
\end{equation}
where the partition function $Z_\mu^{\hat\alpha}(h)$ is the normalizing constant. 
Note that
\begin{equation}\label{eq:Z-mu-g-def}
 Z^{\infty}_{\mu}(h)=\#\left\{\psi \,:\; G_{h,\varphi}(\psi)= G^\gr_{h,\varphi}\right\}\,,
\end{equation}
and let $\mu_{h,\hat\alpha}(F)=\sum_\psi  F(\psi) \mu_{h,\hat\alpha}(\psi)$ be the expectation of a function $F$ on tilings $\psi$.

With these notations, we can rewrite $\P_{\alpha,\beta,\lambda}(\varphi)$ from \cref{eq:towards-P(phi)-via-G} as
\begin{align*}\P_{\alpha, \beta, \lambda}(\varphi) &= \frac1{\Zsos_{N,\beta,\lambda}} \sum_h \exp\Big[ -\beta |h|  + \alpha G^\gr_{h,\varphi} - \log Z_\mu^{\alpha}(h) -
\lambda \sV(h) \Big]\,, \end{align*}
and decomposing $\log Z_\mu^{\alpha}$ as per \cref{lem:grimmett} (with \cref{eq:Z-mu-g-def} in mind) yields 
\begin{align*}\P_{\alpha, \beta, \lambda}(\varphi) &= \frac1{\Zsos_{N,\beta,\lambda}} \sum_h \exp\Big[ -\beta |h|+ \alpha G^\gr_{h, \varphi} - \int_\alpha^\infty \mu_{h,\hat\alpha}(\overline{G}_{h}) \d \hat{\alpha}  - \log Z^\infty_{\mu}(h) -\lambda \sV(h) \Big]\,.
\end{align*}
It will be further useful to define
\begin{equation}\label{eq:H-def}
 H_{h} = |h|-|\varphi| = |h\setminus \varphi| - |\varphi \setminus h|\,,\end{equation}
and $\Zsosbar := e^{ \beta |\varphi|} \Zsos$ (note that $|\varphi|$ is in fact independent of $\varphi$ given the boundary conditions), using which we can rewrite the above expression as
\begin{align} \P_{\alpha,\beta,\lambda}(\varphi) = \frac1{\Zsosbar_{N,\beta,\lambda}}  \sum_h \exp\Big[ &-\beta  H_{h}+ \alpha G^\gr_{h, \varphi} - \int_\alpha^\infty \mu_{h,\hat\alpha}(\overline{G}_{h}) \d \hat{\alpha}  - \log Z^\infty_\mu(h) -\lambda \sV(h)\Big]\,.
\label{eq:phi-expression-gr1}
\end{align}
\begin{remark}
    As mentioned in the introduction, the term $-\beta H_h$ penalizes configurations with too many faces while the term $+\alpha G^\gr_{h,\varphi}$ (which is non-positive) penalizes configurations which can be well approximated by a tiling albeit by one different from $\varphi$ (because given that $h$, it makes $\varphi$ less likely to be chosen). One might expect that these two terms would be sufficient to treat all deviations from $\varphi$ but that is actually not the case: see \cref{prop:alg} and the discussion there for more details. The term $-\log Z^\infty_\mu(h)$ also has a relatively straightforward interpretation: when there are many tilings which are all equally good approximations of $h$, we need to choose uniformly over them and therefore the probability of a fixed $\varphi$ is of order $1/Z^\infty_\mu$. It is the term $\int \mu_{h,\hat\alpha}(\overline G_h)\d\hat\alpha$ that embodies the complicated long-range interactions of this distribution. 
\end{remark}

To treat the summation over $h$ in \cref{eq:phi-expression-gr1}, note that $H_{h}\geq 0$ for every $h$ since tilings minimize the number of faces by definition. Since $H=0$ is attained at $h=\varphi$ (note that $H=0$ can occur not just for $h=\varphi$), we may apply \cref{lem:grimmett} for the second time, this time onto $F(h) = H_{h}$ and 
\[
Z_{\varphi}^{\hat\beta} =\sum_h \exp\left[-\hat\beta H_h + \alpha G_{h,\varphi}^\gr-\log Z_{\mu}^\infty(h) -\int_\alpha^\infty \mu_{h,\hat\alpha}(\overline{G}_{h}) \d \hat{\alpha}- \lambda \sV(h)\right]\,,
\]
translating \cref{eq:phi-expression-gr1} into %
\begin{equation}\label{eq:phi-expression-gr2}
\P_{\alpha,\beta,\lambda}(\varphi) = \frac1{\Zsosbar_{N,\beta,\lambda}}\exp\bigg[ \int_{\beta}^{\infty}\pi_{\varphi,\hat\beta}(H_h) \d\hat\beta +\log Z_{\pi}^\infty(\varphi) \bigg]\,,
\end{equation}
where, analogously to $\mu$ from \cref{eq:mu-def}, if $F$ is a function on SOS height functions $h$, we let
\begin{equation}\label{eq:pi-def} 
\pi_{\varphi,\hat\beta}(F) := \frac1{Z_{\pi}^{\hat\beta}(\varphi)} \sum_h  F(h) \exp\Big[-\hat\beta H_h+ \alpha G_{h,\varphi}^\gr-\log Z_\mu^\infty(h) -\int_\alpha^\infty \mu_{h,\hat\alpha}(\overline{G}_{h}) \d \hat{\alpha} - \lambda \sV(h)\Big]\,.
\end{equation}
(The normalizer $Z_\pi^{\hat\beta}(\varphi)$ is exactly $Z_{\varphi}^{\hat\beta}$ from above, and we include $\pi$ in the notation to help recall the measure that it is associated to.)
\begin{remark}\label{rk:h_conditionnal}
Notice that, in the special case $\hat\beta=\beta$, the measure $\pi_{\varphi,\beta}$ is nothing but the conditional law of $h$ given $\varphi$:
   \[ \P_{\alpha,\beta,\lambda}(h\mid\varphi) = \pi_{\varphi,\beta}(h)\,.\]
Our analysis of $\pi_{\varphi,\hat\beta}(h)$ will therefore serve both the goal of establishing the required properties on $\varphi$ (towards showing a convergence to a \GFF), and the characterization of $h$ as a perturbation of~$\varphi$.
\end{remark}
It remains to expand the nontrivial term $Z_\pi^\infty(\varphi)$ in \cref{eq:phi-expression-gr2}---which is the following sum over tilings~$\eta$:
\begin{equation}\label{eq:Z-phi-gr-def}
Z_{\pi}^\infty(\varphi) := \sum_\eta \exp\Big[\alpha G_{\eta,\varphi}^\gr - \log Z_{\mu}^\infty(\eta) - \int_\alpha^\infty \mu_{\eta,\hat\alpha}(\overline G_{\eta})\d \hat\alpha - \lambda \sV(\eta)\Big]\,.
\end{equation}
Observe that $Z_\mu^\infty(\eta)=1$ since only $\psi=\eta$ minimizes $G_{\eta,\varphi}(\psi)=|\eta\cap\varphi|-|\eta\cap\psi|$ (whereby its value is $G_{\eta,\varphi}^\gr = |\eta\cap\varphi|-|\eta|= -\frac12 |\eta\xor\varphi|$).
It will be convenient to further observe that
\begin{equation}\label{eq:Gbar-simplification} \overline G_{\eta}(\psi) = |\eta\cap\varphi|-|\eta\cap \psi| - (|\eta\cap\varphi|-|\eta|) = |\eta\setminus \psi| = \tfrac12 |\eta\xor\psi|\,,\end{equation}
and that $\sV(\eta)= 0$ for every tiling $\eta$ (the potential $\sV(h)$ in \cref{eq:gen-potential} sums $\ff$ over the connected components of $\psi_0\setminus h$; when $h$ is itself 
a tiling $\eta$, one has $\psi_0=h$).
Combined, \cref{eq:Z-phi-gr-def} becomes
\begin{equation}\label{eq:Z-phi-gr'}
 Z_{\pi}^\infty(\varphi) = \sum_\eta \exp\Big[-\tfrac12 \alpha |\eta\xor\varphi| -   \int_\alpha^\infty \mu_{\eta,\hat\alpha}(\tfrac12 |\eta\xor\psi|)\d \hat\alpha\Big]\,.
\end{equation}
We now apply \cref{lem:grimmett} for the third time, to $Z_{\pi}^\infty(\varphi)$ from \cref{eq:Z-phi-gr'} and $F(\eta) = \frac12|\eta\xor\varphi|\geq 0$ (noting that $F(\varphi)=0$ and that in this application the only ground state is $\eta=\varphi$), and obtain that
\begin{equation}\label{eq:Z-phi-gr''}
\log Z_{\pi}^\infty(\varphi) = \int_{\alpha}^\infty \nu_{\varphi,\bar\alpha}(\tfrac12 |\eta\xor\varphi|) \d \bar\alpha - \int_\alpha^\infty \mu_{\varphi,\hat\alpha}(\tfrac12 |\varphi\xor\psi|)\d \hat\alpha\,,
\end{equation}
where, 
\begin{equation}\label{eq:nu-def} \nu_{\varphi,\bar\alpha}(\eta) := \frac1{Z_{\nu}^{\bar\alpha}(\varphi)} \exp\Big[-\tfrac12 \bar\alpha |\eta\xor\varphi| -\int_\alpha^\infty \mu_{\eta,\hat\alpha}(\tfrac12|\eta\xor\psi|) \d \hat{\alpha}\Big]\,,\end{equation}
with, as before, $Z_{\nu}^{\bar\alpha}(\varphi)$ the normalizer of this distribution, and $\nu_{\varphi,\bar\alpha}(F) = \sum_\eta F(\eta)\nu_{\varphi,\bar\alpha}(\eta)$.

Plugging \cref{eq:Z-phi-gr''} into \cref{eq:phi-expression-gr2} yields \cref{eq:phi-3-measures}, concluding the proof of \cref{prop:P(phi)-via-3-measures}.
\end{proof}
\section{Local decomposition of \texorpdfstring{$\mu_{\varphi, \hat \alpha}$}{\textmu} and \texorpdfstring{$\nu_{\varphi, \bar \alpha}$}{\textnu}}
\label{sec:2-out-of-3}
Our goal in this section is to decompose $\int \mu_{\varphi,\hat\alpha}(\cdot)\d\hat \alpha$ and $\int \nu_{\varphi,\bar\alpha}(\cdot)\d\bar\alpha$ as follows.
\begin{theorem}\label{thm:mu-nu}
There is an absolute constant $C$ such that if $\alpha\wedge \beta>C$, then there exist functions $\fg^\mu_r,\fg^\nu_r$ for $r=2^k$ with $k=0,1,\ldots$, defined on lozenge tilings of $B(o,r)\subset \T$, such that for every $N$,
\begin{align}\label{eq:orig-prop1}
\sup_\varphi \bigg|\int_{\alpha}^\infty \mu_{\varphi, \hat \alpha}(\tfrac12|\varphi\xor\psi|)\d \hat\alpha - \sum_{x\in\T_N} \sum_{\substack{0\leq r < N/2 \\ r=2^k}} \fg^\mu_r (\varphi\restriction_{B(x,r)}) \bigg| &\leq C e^{-\alpha -\alpha N/C} \,,\\
 \label{eq:orig-prop2}
\sup_\varphi \bigg|\int_{\alpha}^\infty \nu_{\varphi,\bar\alpha}(\tfrac12|\eta\xor\varphi|)\d \bar\alpha 
- \sum_{x\in\T_N} \sum_{\substack{0\leq r < N/2 \\ r=2^k}}
\fg^\nu_r(\varphi\restriction_{B(x,r)}) \bigg|&\leq C e^{-\alpha -\alpha N/C}\,,
\end{align}
and for every integer $r=2^k$ ($k\geq 0$) one has $\|\fg_r^\mu \|_\infty \vee \|\fg_r^\nu \|_\infty \leq C e^{- \alpha - \alpha r/C } $.
\end{theorem}

The theorem above will be established via a dynamical analysis, comparing the speed of propagating information along Metropolis/Glauber dynamics, vs.\ their mixing time. At the center of these Markov chains is the following object.
\begin{definition}[Bubble]
Given a tiling $\varphi$ and a height function $h$, 
a \emph{bubble} $\sB$ is a connected component of faces of $\varphi\xor h$ (where faces are adjacent if they share an edge). We further color each face blue if it is part of $\varphi$ and red otherwise.
\end{definition}
(NB.\ the colors cannot simply be read from the geometry of $\sB$ alone, even when $h$ is \emph{not} a tiling.)
To each bubble $\sB$, one may associate the portion of the energy $H_h$ it accounts for:
\begin{equation}\label{eq:bubble-H}
H (\sB) = |\sB| - 2 | \sB \cap \varphi|\,,
\end{equation}
noting that there is no actual dependence on $\varphi$ since $|\sB\cap\varphi|$ can be read from $\sB$.

\subsection{Weak locality of the integral over \texorpdfstring{$\mu$}{\textmu}}

Recall from \cref{eq:mu-def,eq:Gbar-simplification} that
\begin{equation}\label{eq:mu-phi-phi-simplified}
\mu_{\varphi,\hat\alpha}(\psi) \propto e^{-\frac12 \hat \alpha|\varphi\xor\psi|}= e^{-\frac12\hat\alpha \sum_\sB |\sB|}\,,
\end{equation}
where the summation on the right-hand side is over all $(\varphi,\psi)$-bubbles $\sB$.

Our strategy for proving \cref{thm:mu-nu}, \cref{eq:orig-prop1} is the following. We construct a Gibbs sampler for $\mu_{\varphi,\hat\alpha}$, namely Metropolis on $(\varphi,\psi)$-bubbles $\sB$, and show that it is a contraction w.r.t.\ the natural distance on configurations defined via moves of that dynamics (\cref{prop:mu-contraction} below). This shows that, in a ball of radius $r$, after running the sampler for a time $c r$ the remaining error (say in total variation) is at most $e^{-c' r}$ for small constants $c, c'$. However, information only propagates at a linear speed in a local dynamics such as this one, so at time $cr$, the configuration in the center is yet unaffected by the initial boundary values at distance $r$ or $2r$.

\begin{remark}
    \label{rem:mu-dynamical-approach}
The measure $\mu$ is sufficiently straightforward so that simpler approaches could have been used to derive \cref{thm:mu-nu}, \cref{eq:orig-prop1}. E.g., a standard approach would have been to apply cluster expansion (which would have given a more explicit conclusion). However, such methods do not extend to the more complicated measure $\pi$. We chose to use the same approach we will eventually use for $\pi$ (which is still fairly simple for $\mu$), so as to (a) have a uniform approach, and (b) allow the reader to see this argument in the simpler setting of $\mu$, and then somewhat more complex setting of $\nu$ (with non-local interactions), before the final (quite technical) analysis of $\pi$.
See also \cref{rem:bubble-nu-no-ce,rem:bubble-groups-no-ce} for more on this.
\end{remark}

\subsubsection*{Metropolis dynamics on bubbles for $\mu$}

Given a fixed reference tiling $\varphi$ of $\T_N$, define the following dynamics $(\psi_t)$ on tilings of $\T_N$:
\begin{enumerate}[(i)]
    \item \label{it:metropolis-clocks-mu} Attach a rate-1 Poisson clock to every pair $(\sB,f)$ where $\sB$ is a candidate for a $(\varphi,\psi)$-bubble and $f$ is a marked face of $\sB$.
    \item If $\sB$ is a $(\varphi,\psi_t)$-bubble, erase it from $\psi_t$.
    \item \label{it:metropolis-rand-mu} If $\sB$ can be added to $\psi_t$ as a new bubble (that is, it neither intersects nor is  adjacent to any existing $(\varphi,\psi_t)$-bubble) then do so with probability $\exp(-\frac12\hat\alpha|\sB|)$.
\end{enumerate}
Note that $\psi_t$ is irreducible and reversible w.r.t.\ $\mu_{\varphi,\hat\alpha}$ as per \cref{eq:mu-phi-phi-simplified}.
Define $\dist_{\sB}(\psi,\psi')$ to be the length of the geodesic in the graph defined by legal moves of the Metropolis dynamics---namely, the minimum number of moves (each one adding or erasing a bubble) needed to reach $\psi'$ from $\psi$.

\begin{proposition}\label{prop:mu-contraction}
The dynamics $(\psi_t)$ is contracting w.r.t.\ $\dist_\sB$; that is, if $\hat\alpha$ is large enough (independently of $\varphi$), then, under the coupling where we synchronize the Poisson clocks from \cref{it:metropolis-clocks-mu} and the $\mathrm{Uniform}([0,1])$ variables used for \cref{it:metropolis-rand-mu}, we have, for all $t>0$, \[\E \dist_\sB(\psi_t,\psi'_t) \leq e^{-t/2}\dist_\sB(\psi_0,\psi'_0)\,.\]
\end{proposition}
\begin{proof}
By the triangle inequality, it suffices to consider $\psi_0,\psi'_0$ that differ on a single bubble $\sB_0$ and show that $\frac{\d}{\d t}\E \dist_\sB(\psi_t,\psi'_t)\big|_{t=0} \leq -\frac12$ under the aforementioned coupling. Assume by symmetry that $\psi_0'$ contains $\sB_0$ and $\psi_0$ does not, and observe that, since the only condition for being allowed to add a bubble is that it cannot intersect a previous one, only the following scenarios are possible:
\begin{enumerate}
    \item {}[\emph{healing}] At rate $|\sB_0|$ we select $(\sB_0,f)$ for some $f$. In that case we always remove $\sB_0$ from $\psi'$. If we do not add $\sB_0$ to $\psi$, which happens with probability $1-\exp(-\frac12\hat\alpha|\sB_0|)$, we get $\dist_\sB(\psi,\psi')=0$, otherwise we keep $\dist_\sB(\psi,\psi')=1$.
    \item {}[\emph{neutral}] We select some $(\sB, f)$ with $\sB \cap \sB_0 = \emptyset$. This will leave the distance unchanged.
    \item {}[\emph{infection}] We select $(\sB, f)$ with $\sB \cap \sB_0 \neq \emptyset$. In that case the distance increases by 1 when we add $\sB$ to $\psi$ which happens with probability $\exp( - \tfrac12 \hat \alpha |\sB|)$ and otherwise it is unchanged.
\end{enumerate}
Overall, this gives
\[
\frac{\d}{\d t}\E[ \dist_\sB(\psi_t,\psi'_t) ]\big|_{t=0} \leq -(1-e^{-\frac12\hat\alpha|\sB_0|}) |\sB_0| + \sum_{\sB \cap \sB_0 \neq \emptyset } |\sB| e^{- \frac12 \hat \alpha |\sB|}.
\]
Since the number of bubbles of size $s$ only grows exponentially in $s$ with a constant which can be made uniform over $\varphi$, for $\hat \alpha$ large enough the right-hand side of the equation is bounded from above by $-\frac12$, which concludes the proof.
\end{proof}

\begin{proof}[Proof of \cref{thm:mu-nu}, \cref{eq:orig-prop1}]
Let $r=2^k$ with $k\geq 0$, and let $\Lambda_r^\varphi$ denote $\Upsilon(\varphi\restriction_{B(x,r)})$, i.e., the tiles of $\varphi$ whose projection to $\T_N$ intersects $B(x,r)$. Denote by $\mu_r$ the measure defined as $\mu_{\varphi,\hat\alpha}$ but on tilings of $\Lambda_r^\varphi$, that is
\[
 \mu_r(\psi)\propto \exp(-\tfrac12\hat\alpha |\varphi\xor\psi|)\,.
\]
In case $r \geq N/2$, we replace $B(x,r)$ in the definition of $\mu_r$ by the full torus $\T_N$, i.e., $\mu_r = \mu_{\varphi,\hat\alpha}$. (With this definition, every $B(x,r)$ that is strictly contained in $\T_N$ is also simply connected.)
\begin{remark}\label{rem:mu-r-def-vs-mu}
Since the interactions in $\mu_{\varphi,\hat\alpha}$ are nearest-neighbor between bubbles, we could have also defined the above via the original measure on $\T_N$ conditional on $\psi$ identifying with $\varphi$ on tiles which are outside $B(x,r)$. However, the above definition will be the correct one for the sake of \cref{eq:orig-prop2}, in which the long-range interactions will distinguish it from the alternative one.
\end{remark}

\subsubsection*{Constructing the local function}
With the above definition, denote by $\{\sB\in\psi\}$ for some bubble $\sB$ the event that $\sB$ appears in $\psi$ as a (complete) bubble, and let
\begin{align*}
    f_{1,\sB}^\mu(\varphi\restriction_{B(x,1)}) &:= \int_\alpha^\infty \mu_1(\sB\in\psi)\d\hat\alpha\,,\\
f^\mu_{2r, \sB}(\varphi\restriction_{B(x,2r)}) &:= \int_{\alpha}^\infty \big[\mu_{2r}(\sB\in\psi) - \mu_r(\sB\in\psi)\big]\d\hat\alpha\qquad \mbox{for $r=2^k$, $k\geq 0$}\,. 
\end{align*}
We claim that, by definition,
\begin{equation}\label{eq:mu-integral-with-f} \int_\alpha^\infty \mu_{\varphi,\hat\alpha}(\tfrac12 |\varphi\xor\psi|)\d \hat\alpha = \frac12\sum_{x\in\T_N} \sum_{\substack{\text{$(\varphi,\psi)$-bubble $\sB$} \\ \Upsilon(\sB)\ni x}} \sum_{\substack{r=2^k \\ \text{for $k\geq 0$}}} f_{r,\sB}^\mu (\varphi\restriction_{B(x,r)}) \,.\end{equation}
Indeed, the telescopic sum over $\mu_r$ for $r=2^k$ ($k\geq 1$) induced by the sum over $f_{r,\sB}^\mu (\varphi\restriction_{B(x,r)})$ will leave only the last term $\int_\alpha^\infty \mu_{\varphi,\hat\alpha}(\sB\in\psi)\d\hat\alpha$. (Recall that we are in a finite domain, hence have a finite set of possible $\sB$'s, and there are no issues when exchanging the integral and sum.) The right-hand side of \cref{eq:mu-integral-with-f} is therefore 
\[ \frac12\sum_{x\in\T_N} \sum_{\sB:\,\Upsilon(\sB)\ni x} \int_\alpha^\infty \mu_{\varphi,\hat\alpha}(\sB\in\psi)\d\hat\alpha\,.\]
Looking at the left hand of \cref{eq:mu-integral-with-f}, we recall that each lozenge $f\in\sB$ (of which there are $|\sB|/2$ in $\varphi$ and $|\sB|/2$ in $\psi$, as both $\varphi,\psi$ are tilings) contains two triangular faces $x\in\Upsilon(f)\subset \T_N$, so that $|\Upsilon(\sB)|=|\sB|$. Thus,
\[ \frac12 |\varphi\xor\psi| = \frac12\sum_{\sB} |\sB|\one_{\{\sB\in\psi\}} = 
\frac12 \sum_{x\in\T_N}\sum_{\sB:\,\Upsilon(\sB)\ni x} \one_{\{\sB\in\psi\}}\,,
\]
which, after taking expectation under $\mu_{\varphi,\hat\alpha}$ and integrating over $\hat\alpha$ results in the preceding equation, thereby
establishes \cref{eq:mu-integral-with-f}.

We will now argue that for every $\sB$ with $\Upsilon(\sB)\ni o$, for $\alpha$ large enough one has \begin{equation}\label{eq:f-mu-bound}\|f_{2r,\sB}^\mu\|_\infty\leq 
\exp\Big[- \Big((\frac{\alpha}{16} - C)\one_{\{r\geq 2|\sB|\}} r \Big)
\,\vee\, \Big( \alpha \frac{|\sB|}2
\Big)\Big]\leq
\exp\left[-\alpha|\sB|/4 - \alpha r/64\right]\,.\end{equation} The term $ \alpha |\sB|/2$ in the first exponent in \cref{eq:f-mu-bound} follows from a Peierls-type argument: for every~$\sB$,
\[ \mu_{r}(\sB\in\psi) \leq \exp(-\tfrac12 \hat\alpha |\sB|)\,,\]
as one derives from the ratio of $\mu_{r}(\psi')$ for $\psi'$ that has $\sB\in\psi'$ and $\mu_{r}(\psi)$ for $\psi=\psi'\xor\sB$. This shows $\|f_{2r,\sB}^\mu\|_\infty \leq \frac{4}{|\sB|}\exp[-\frac12\alpha|\sB|]$ (bounding $|\mu_r-\mu_{2r}|$ by $\mu_r+\mu_{2r}$ and then integrating over $\hat\alpha$).

The term $(\alpha /16-C)r$ in the first exponent in \cref{eq:f-mu-bound} under the assumption that $|\sB|\leq r/2$ (as per that indicator) is more delicate, and here we will use the contracting Metropolis dynamics for~$\mu$. Denote by $(\psi_t,\psi'_t)$ two coupled instances of the dynamics, with domains $B(o,r)$ and $B(o,2r)$ respectively, from an initial configuration which agrees on $B(o,r)$, where every update of $\psi'_t$ that is confined to $B(o,r)$ uses the joint law as per \cref{prop:mu-contraction} (whereas updates of $\psi'_t$ in the annulus $B(o,2r)\setminus B(o,r)$ will be sampled via the product measure of $\psi_t$ and $\psi'_t$ on this event). Run the dynamics for time
    \[ T = \frac1{8} \hat\alpha \, r\,,\]
and write
\[\left|\mu_{2r}(\sB\in\psi) - \mu_r(\sB\in\psi)\right| \leq \Xi_1 + \Xi_2 + \Xi_3\,,
\]
where
\begin{align*} \Xi_1 &:= \left| \mu_{r}(\sB\in\psi) - \P(\sB\in \psi_{T})\right|\,,\\
\Xi_2 &:= \left| \mu_{2r}(\sB\in\psi) - \P(\sB\in \psi'_{T})\right|\,,\\
\Xi_3 &:= \left| \P(\sB\in \psi_{T}) - \P(\sB\in \psi'_{T})\right|\,.
\end{align*}
By \cref{prop:mu-contraction} for $t=T$ (bounding $\dist_\sB(\psi_T,\tilde\psi_T)$ where $\tilde\psi_0\sim\mu_r$, thus also $\tilde\psi_T\sim\mu_r$), we have
\[ \Xi_1 \leq e^{-T/2} |B(o,r)| \leq e^{-(\hat\alpha/16-o(1))r}\,,\]
where the $o(1)$-term goes to $0$ as $r\to\infty$, and similarly
$ \Xi_2 \leq e^{-T/2} |B(o,2r)| \leq e^{-(\hat\alpha/16-o(1))r}$.

For $\Xi_3$, we must bound the rate of propagation of information in the Metropolis dynamics, which, unlike single-site dynamics, has the small complication of allowing long-range interactions by virtue of moves that use arbitrarily large bubbles $\sB$. To this end, we take a union bound over any potential sequence of updates starting from a disagreement (necessarily outside of $B(o,r)$) and making its way to $\sB$:
such a sequence necessarily contains a shortest path of intersecting bubbles $\sB_1,\ldots,\sB_m$ ($m\geq 1$) such that $\sB_{i}\cap\sB_{i+1} \neq\emptyset$, the first bubble
$\sB_1$ intersects the boundary of $B(o,r)$ and the last bubble $\sB_m$ intersects $\sB$.
Let $v_i$ be a vertex of $\sB_{i}\cap\sB_{i-1}$ (say, a minimal one according to some lexicographic ordering), and let $r_i$ be the length of the shortest path between $v_i$ and $v_{i+1}$, noting that $s_i :=|\sB_i| \geq r_i$. With these notations,
\begin{align*}
\Xi_3 \leq \sum_{m\geq 1}\sum_{\substack{\{r_i\}\\ \sum r_i \geq r/2}} \sum_{\substack{\{s_i\}\\ s_i\geq r_i}} \int_{0 \leq t_1 \leq \ldots \leq t_m \leq T} \prod_i\Big( c_0^{s_i} s_{i-1} \cdot s_i e^{-\frac12 \hat\alpha s_i} e^{-(t_i - t_{i-1})  s_i e^{-\frac12 \hat\alpha  s_i}}\Big)\d t_1\ldots \d t_m\,,
\end{align*}
since $\dist(\Upsilon(\sB),\partial B(o,r))\geq r-\diam(\sB)\geq r/2$ (using $|\sB|\leq r/2$ and $\Upsilon(\sB)\ni o$), and
 the update of $\sB_i$ is via an exponential clock with rate $s_i\exp(-\frac{1}{2}\hat\alpha s_i)$, for ringing the appropriate bubble and then accepting it (where the term $\exp(-(t_i-t_{i-1})s_ie^{-\frac12 \hat\alpha s_i})$ accounts for $t_i$ being the first update of this kind over all $t>t_{i-1}$, and the term $c_0^{s_i}$ bounds the number of bubbles of size $s_i$ rooted at a marked point). 
 Bounding $t_i-t_{i-1}\geq 0$, the resulting integral is explicit and gives
\begin{align*}
\Xi_3 &\leq \sum_{m\geq 1}\sum_{\{r_i\}:\sum r_i \geq r/2} \sum_{\{s_i\}: s_i\geq r_i} \big(c_1 e^{-\frac12\hat\alpha}\big)^{\sum_{i=1}^m s_i} \,\frac{T^m}{m!} \leq e^{\frac32T} e^{-\frac14\hat\alpha r +Cr}\,,
\end{align*}
where the last inequality holds for some $C$ provided $\hat\alpha$ is large enough. Since $T = \frac{1}{8} \hat \alpha r$, this is at most $e^{-(\hat\alpha/16-C)r}$. Combining the bounds on $\Xi_1,\Xi_2,\Xi_3$ and integrating over $\hat\alpha$, we obtain \cref{eq:f-mu-bound}.

To conclude the proof of \cref{eq:orig-prop1}, let 
\[ \fg_r^\mu := \frac12\sum_{\sB:\;\Upsilon(\sB)\ni o} f_{r,\sB}^\mu\,.\]
The tail on $|\sB|$ in the bound of \cref{eq:f-mu-bound} on $\|f_{r,\sB}^\mu\|_\infty$ then implies the required bound on $\|\fg_r^\mu\|_\infty$ as the minimal $\sB$ has $6$ faces (a single cube), concluding the proof of \cref{eq:orig-prop1}.
\end{proof}

The argument just given yields the following locality estimate for the measures $\mu_{D,\eta,\hat\alpha}$, which we record for later use (in \cref{lem:nu-domain-change} below). For a tiling $\eta$ of $\T_N$ and a set $D$ of triangles of $\T_N$, let $\mu_{D,\eta,\hat\alpha}$ denote the measure proportional to $e^{-\frac12\hat\alpha|\eta\xor\psi|}$ on tilings $\psi$ of $\Lambda^\eta_D$ (that is, $\mu_r$ with $\varphi$ replaced by $\eta$ and $B(x,r)$ replaced by $D$; for $D=\T_N$ this is $\mu_{\eta,\hat\alpha}$). For $x\in\T_N$ and a tiling $\psi$, set
\begin{equation}\label{eq:F-eta-x}
F_{\eta,x}(\psi):=\frac12\sum_{\substack{\text{$(\eta,\psi)$-bubble $\sB'$}\\ \Upsilon(\sB')\ni x}}\frac{|\sB'|}{|\Upsilon(\sB')|}\one_{\{\sB'\in\psi\}}\,,
\end{equation}
so that $\frac12|\eta\xor\psi|=\sum_{x\in\T_N}F_{\eta,x}(\psi)$, and $0\leq F_{\eta,x}\leq1$ (at most two bubbles contain a face projecting onto $x$, namely those containing the lozenges of $\eta$ and of $\psi$ at $x$, and each bubble $\sB'$ has $|\sB'|=2|\eta\cap\sB'|\leq|\Upsilon(\sB')|$). Note that $F_{\eta,x}(\psi)$ depends on $\psi$ only through the bubbles containing a face above $x$, and that adding or removing a single bubble changes it by at most $\frac12$.

\begin{lemma}\label{lem:mu-locality}
There exist absolute constants $C,c>0$ such that the following holds for all $\hat\alpha\geq C$. Let $\eta$ be a tiling of $\T_N$, let $x\in\T_N$, let $d\geq 1$, and let $D\subseteq\T_N$ be any set of triangles, and set $D_x:=D\cap B(x,d/2)$. Then
\[ \big|\mu_{D,\eta,\hat\alpha}(F_{\eta,x})-\mu_{D_x,\eta,\hat\alpha}(F_{\eta,x})\big|\leq Ce^{-c\hat\alpha d}\,.\]
\end{lemma}
\begin{proof}
Write $B:=D_x$. Both measures impose the same boundary condition along $\partial D$; only the artificial boundary near $\partial B(x,d/2)$ can create a disagreement. The one-lozenge boundary layer in $\Lambda_D^\eta$ changes distances by at most $2$, which is harmless after adjusting $C,c$ (bounded $d$ follows directly from the Peierls bound). Let $(\psi'_t)$ be the Metropolis dynamics for $\mu_{D,\eta,\hat\alpha}$ (defined as in \cref{it:metropolis-clocks-mu,it:metropolis-rand-mu} with $\varphi$ replaced by $\eta$ and candidate bubbles restricted to those with projection intersecting $D$), started from $\psi'_0\sim\mu_{D,\eta,\hat\alpha}$, so that $\psi'_T\sim\mu_{D,\eta,\hat\alpha}$ for every $T$. Let $(\psi_t)$ be the Metropolis dynamics for $\mu_{B,\eta,\hat\alpha}$ started from the tiling $\psi_0$ whose $(\eta,\psi_0)$-bubbles are those $(\eta,\psi'_0)$-bubbles $\sB$ supported in $\Lambda_B^\eta$, and couple the two chains by synchronizing the Poisson clocks and the uniform variables of all candidate bubbles supported in $\Lambda_B^\eta$. Since the acceptance probability $e^{-\frac12\hat\alpha|\sB|}$ of a bubble does not depend on the current configuration, the two chains attempt the same moves with the same outcomes as long as these are legal in both. Consequently, inside $B$ at time $0$ the two chains differ only on the $(\eta,\psi'_0)$-bubbles crossing the artificial boundary, and thereafter a new difference can only be created by the update of a bubble $\sB$ which either intersects (or is adjacent to) an existing difference, or is allowed only in the larger domain and meets the artificial boundary; in particular, as in the proof of \cref{eq:f-mu-bound}, if the bubbles above $x$ differ between $\psi_T$ and $\psi'_T$ then there is a chain of intersecting bubbles $\sB_1,\ldots,\sB_m$, created successively before time $T$, from the artificial boundary to $x$, or else some $(\eta,\psi'_0)$-bubble of diameter at least $d/4$ crosses the artificial boundary. The latter event has probability at most $Cd^2\sum_{s\geq d/4}c_0^se^{-\frac12\hat\alpha s}\leq e^{-c\hat\alpha d}$ by a Peierls estimate under $\mu_{D,\eta,\hat\alpha}$, and the union bound of the proof of \cref{eq:f-mu-bound} (with $r/2$ replaced by $d/4$) bounds the probability of the former by $e^{\frac32T}e^{-\frac18\hat\alpha d+Cd}$. Taking $T:=\hat\alpha d/32$, we conclude that $|\E F_{\eta,x}(\psi'_T)-\E F_{\eta,x}(\psi_T)|\leq 2e^{-c\hat\alpha d}$ for $\hat\alpha\geq C$. Finally, by \cref{prop:mu-contraction} (applied to the chain on $B$, with $\tilde\psi_0\sim\mu_{B,\eta,\hat\alpha}$ so that $\tilde\psi_T\sim\mu_{B,\eta,\hat\alpha}$),
\[|\E F_{\eta,x}(\psi_T)-\mu_{B,\eta,\hat\alpha}(F_{\eta,x})|\leq \tfrac12\E\dist_\sB(\psi_T,\tilde\psi_T)\leq \tfrac12 e^{-T/2}\,|B|\leq Cd^2e^{-\hat\alpha d/64}\,,\]
and combining the two estimates concludes the proof.
\end{proof}

\subsection{Weak locality of the integral over \texorpdfstring{$\nu$}{\textnu}}

Having established \cref{eq:orig-prop1}, and recalling the definition of $\nu$ from \cref{eq:nu-def}, we can now substitute the expression for $\int \mu_{\eta,\hat\alpha}(\frac12|\eta\xor\psi|)\d\hat\alpha$ in terms of $\fg_r^\mu$ to find that
\begin{equation}\label{eq:nu-phi-simplified}
\nu_{\varphi,\bar\alpha}(\eta) \propto \exp\Big[-\tfrac12 \bar\alpha |\eta\xor\varphi| - \sum_{x\in \T_N}\sum_r \fg_r^\mu(\eta\restriction_{B(x,r)})
 \Big]\,.\end{equation} 
We will prove \cref{eq:orig-prop2} in a similar manner to \cref{eq:orig-prop1} above, albeit with some added difficulty due to the interaction between bubbles carried through the functions $\fg_r^\mu$ (representing the integral over $\hat\alpha\in(\alpha,\infty)$ of $\E[\frac12|\eta\xor\psi|]$ for $\psi\sim\mu_{\eta,\hat\alpha}$).

\begin{remark}\label{rem:bubble-nu-no-ce}
    We can comment here on why we choose the dynamical approach (inferring spatial mixing from temporal mixing) to prove these locality results. As will be apparent in the rest of the section, in our approach there is no significant difficulty going from $\mu$ to $\nu$: the proof below is more technical because there are more terms to handle (more cases in the dynamics, a more subtle enumeration of paths in the propagation of information), but no new ideas are necessary. 
    
    On the other hand, while the locality of $\mu$ could have been also established via cluster expansion techniques, in the case of $\nu$ it already becomes unclear how to do so (let alone in the case of~$\pi$).
    For example, in the case of $\mu$, the standard cluster expansion approach would be to have a cluster be a bubble, and two bubbles are compatible iff they are disjoint. The standard  condition to ensure convergence (e.g.,~\cite[\S5.4]{FriedliVelenik18}) is to compare the energy of a given bubble $\sB_0$ with the total weights of all bubbles incompatible with it: $\sum_{\sB : \sB\cap \sB_0 \neq \emptyset} e^{- a |\sB|} \leq a |\sB_0|$ for some $a>0$ (e.g., $a=\hat\alpha/2$ for~$\mu$).
    For~$\nu$, beyond the compatibility of bubbles (a pairwise interaction), we have a many-body interaction due to the functions $\fg_r$. The classical cluster expansion generalizes to $n$-body interactions for  fixed~$n$ (replacing the dependency graph by an $n$-uniform hypergraph), but here the type of interaction is different (it cannot  even be written as a sum over $n$-body interactions for all $n$). A natural approach would be to introduce ``ball clusters''---specifying  $\eta\restriction_{B}$ in some ball $B$---which fails as it would need $\fg_r$ to decay as $e^{-C r^2}$ instead of $e^{-Cr}$. It is plausible that, for $\nu$, one could rewrite the long-range interaction term $\int_\alpha^\infty\mu_{\eta,\hat\alpha}(\frac12|\eta\xor\psi|)\d\hat\alpha$ in a way that would make a suitable cluster expansion framework doable; however, this does not seem feasible for $\pi$, in light of the even more delicate non-explicit long-range interaction between the objects of interest there (bubble groups).
\end{remark}

\subsubsection*{Glauber dynamics on bubbles for $\nu$}

Given a fixed reference tiling $\varphi$ of $\T_N$, define the following dynamics $(\eta_t)$ on tilings of $\T_N$:
\begin{enumerate}[(i)]
    \item \label{it:glauber-clocks-nu} Attach a rate-1 Poisson clock to every pair $(\sB,f)$ where $\sB$ is a candidate for a $(\varphi,\eta)$-bubble and $f$ is a marked face.
    \item \label{it:glauber-move-nu} If $\sB$ is either a (full) $(\varphi,\eta_t)$-bubble, or can be fully added to $\eta_t$ (i.e., neither intersecting nor adjacent to another bubble), let $\{\eta,\hat\eta\}$ denote the configurations $\{\eta_t,\eta_t\xor\sB\}$ such that $\sB$ is a $(\hat\eta,\varphi)$-bubble. The dynamics moves either to $\eta$ or to $\hat\eta$ with weights $w_\eta$ and $w_{\hat\eta}$ given by %
    \[ 
    w_\eta = \exp\Big[ - \sum_{x\in\T_N}\sum_{r} \fg_r^\mu(\eta\restriction_{B(x,r)})\Big]
    \quad,\quad
    w_{\hat\eta} = \exp\Big[ -\tfrac{1}{2}\bar \alpha |\sB| - \sum_{x\in\T_N}\sum_r \fg_r^\mu(\hat\eta\restriction_{B(x,r)})\Big]    
    \,.    
    \]
\end{enumerate}
As usual with Glauber dynamics, $\eta_t$ is irreducible and reversible w.r.t.\ $\nu_{\varphi,\bar\alpha}$.

\begin{proposition}\label{prop:nu-contraction}
The dynamics $(\eta_t)$ is contracting w.r.t.\ $\dist_\sB$; that is, if $\alpha$ is large enough (independently of $\varphi$), then, under the coupling where we synchronize the Poisson clocks from \cref{it:glauber-clocks-nu} and $\mathrm{Uniform}([0,1])$ variables used for the move in \cref{it:glauber-move-nu}, we have, for all $t>0$,  \[\E \dist_\sB(\eta_t,\eta'_t) \leq e^{-t/2}\dist_\sB(\eta_0,\eta'_0)\,.\]
\end{proposition}
\begin{proof}
Throughout this proof, since the domain of $\fg_r^\mu(\cdot)$ is $B(o,r)$ and we will always apply it to terms of the form $\eta\restriction_{B(x,r)}$, with $x$ clear from the context,
we write  $\fg_r^\mu(\eta)$ to abbreviate $\fg_r^\mu(\eta\restriction_{B(x,r)})$. 

It again suffices to consider $\eta_0,\eta'_0$ that differ on a single bubble $\sB_0$ (assume by symmetry that $\eta_0'$ contains $\sB_0$ and $\eta_0$ does not) and show that $\frac{\d}{\d t}\E \dist_\sB(\eta_t,\eta'_t)\big|_{t=0} \leq -\frac12$ under the aforementioned coupling. 
Now that bubbles have interactions, the possible scenarios are as follows:
\begin{enumerate}
    \item {}[\emph{healing}] \label{it:healing-nu} At rate $|\sB_0|$ we select $(\sB_0,f)$ for some $f$. This has the effect of reducing $\dist_\sB(\eta,\eta')$ to $0$ deterministically as the coupling will select, for both tilings, the same weight (yielding either  $\hat\eta$, containing $\sB_0$, or $\eta$, without it, in both instances).
    \item {}[\emph{long-range infection}]\label{it:long-range-nu} We select $(\sB, f)$ with $\sB \cap \sB_0 = \emptyset$. This may increase $\dist_\sB(\eta,\eta')$ by $1$.
    \item {}[\emph{contact infection}]\label{it:contact-nu} We select $(\sB, f)$ with $\sB \cap \sB_0 \neq \emptyset$. In that case $\dist_\sB(\eta,\eta')$  increases by 1 when we add $\sB$ to $\eta$, and otherwise it is unchanged.
\end{enumerate}
The effect of the healing step (\cref{it:healing-nu}) is straightforward: it contributes $-|\sB_0|$ to $\frac{\d}{\d t}\E\dist_\sB(\eta_t,\eta_t')\big|_{t=0}$.
Let us move to the effect of the long-range infections (\cref{it:long-range-nu}). Recall that neither $\eta$ nor $\eta'$ contains~$\sB$ (as they only differ on the disjoint bubble $\sB_0$), and define
\[\hat\eta := \eta \xor \sB\quad,\quad \hat \eta' := \eta'\xor\sB\] to be the configurations obtained by adding $\sB$ to $\eta$ and $\eta'$, resp. (That is, here and in what follows, we will typically use ${\cdot}'$ to denote the presence of $\sB_0$ and $\hat{ \cdot}$ to denote the presence of~$\sB$.)
The probability to increase the distance by $1$ is then at most $|p-p'|$ where $p$ corresponds to the probability of the move $\eta\mapsto \hat\eta$ and $p'$ to the move $\eta'\mapsto \hat\eta'$, each given by
\begin{align*} p &= \frac{w_{\hat\eta}}{w_{\eta}+w_{\hat\eta}} =\frac{\exp[ -\frac{1}{2}\bar \alpha |\sB| - \sum_{x,r} \fg_r^\mu(\hat\eta)]}{ \exp[ - \sum_{x,r} \fg_r^\mu(\eta)]  + \exp[ -\frac{1}{2}\bar \alpha |\sB| - \sum_{x,r} \fg_r^\mu(\hat\eta)]}  \,,\\
 p' &= \frac{w_{\hat\eta'}}{w_{\eta'}+w_{\hat\eta'}} = \frac{\exp[ -\frac{1}{2}\bar \alpha |\sB| - \sum_{x,r} \fg_r^\mu(\hat\eta')]}{ \exp[ - \sum_{x,r} \fg_r^\mu(\eta')]  + \exp[ -\frac{1}{2}\bar \alpha |\sB| - \sum_{x,r} \fg_r^\mu(\hat\eta')]} \,.
\end{align*}
We can rewrite
\begin{equation}
p = \frac{\exp[ - \sum_{x,r} (\fg_r^\mu(\hat\eta) - \fg_r^\mu(\eta))] }{\exp[\frac{1}{2}\bar \alpha |\sB| ] + \exp[- \sum_{x,r} (\fg_r^\mu(\hat\eta) - \fg_r^\mu(\eta))]}\,,\label{eq:peierls_nu}
\end{equation}
and similarly for $p'$. Then, as $x \to \frac{e^{x}}{c+  e^x}$ is $\frac14$-Lipschitz for any $c>0$ (applied here for  $c=\exp[\frac12\bar\alpha|\sB|]$),
\begin{equation}\label{eq:p-p'}
|p - p'| \leq \frac{1}{4} \Big|\sum_{x, r} \big(\fg_r^\mu(\eta) + \fg_r^\mu(\hat\eta') - \fg_r^\mu(\hat\eta) - \fg_r^\mu(\eta')\big)\Big|\,.
\end{equation}
The only non-zero terms in the above sum  correspond to pairs $(x,r)$ such that $B(x,r)$ intersects both $\Upsilon (\sB_0)$ and $\Upsilon(\sB)$; so, using that
$\|\fg^\mu_r \|_\infty \leq C e^{- \alpha r/C } $,
we get
\begin{equation}\label{eq:jump_nu}
|p- p'| \leq 
\sum_{\substack{x,r\\ B(x,r)\cap\Upsilon(\sB_0)\neq\emptyset\\
B(x,r)\cap\Upsilon(\sB)\neq\emptyset}} Ce^{-\alpha r/C} \leq
C( |\sB_0|\,\wedge\,|\sB|) e^{ - \alpha \dist(\sB, \sB_0) /C} \end{equation}
for some $C > 0$. The above bound is useful for all $\sB$ such that $|\sB| \leq \dist(\sB, \sB_0)$ but not for much larger bubbles. For those we will bound $|p-p'|$ by $p+p'$ and bound each probability separately.
Going back to the expression for $p$ in \cref{eq:peierls_nu}, we see that 
\[
\sum_{x,r} |\fg_r^\mu( \hat \eta) - \fg_r^\mu(  \eta)| \leq \sum_{\substack{x,r \\ B(x,r)\cap\sB\neq\emptyset}} \big(|\fg_r^\mu(\eta)| + |\fg_r^\mu(\hat\eta)| \big)\,,
\]
because the balls which do not intersect $\sB$ do not contribute to the sum on the left. The right-hand side is bounded by $C |\sB|$ if $\alpha$ is large enough and hence
\begin{equation}\label{eq:peierls_nu2}
p \leq C e^{- \frac{1}{2} (\bar \alpha - C) |\sB|} \, ,
\end{equation}
and the same holds true for $p'$.
For $\alpha$ (and by extension, also $\bar\alpha$) large enough, we can then sum the contribution to $|p-p'|$ over all $\sB$ via \cref{eq:jump_nu} for $|\sB|\leq \dist(\sB,\sB_0)$ and otherwise via \cref{eq:peierls_nu2} (and its analog for $p'$). Overall, we obtain that long-range infections contribute at most 
\[
\sum_{\substack{\sB \\ \Upsilon(\sB)\cap\Upsilon(\sB_0) = \emptyset}} |\sB| Ce^{-\frac{1}{2}\bar \alpha |\sB| - \alpha \dist(\sB, \sB_0)/C} \leq \tfrac{1}{4} |\sB_0|\,.
\]

It remains to treat the effect of  contact infections (\cref{it:contact-nu}). We in fact already bounded the probability to add a bubble $\sB$ in \cref{eq:peierls_nu2}, so simply summing this bound concludes the proof.
\end{proof}

\begin{proof}[Proof of \cref{thm:mu-nu}, \cref{eq:orig-prop2}]
As in the proof of \cref{eq:orig-prop1} for $\mu$, let $r=2^k$ with $k\geq 0$, and let $\Lambda_r^\varphi$ denote $\Upsilon(\varphi\restriction_{B(o,r)})$, i.e., the tiles of $\varphi$ whose projection to $\T_N$ intersects $B(o,r)$. Denote by $\nu_r$ the measure defined as $\nu_{\varphi,\bar\alpha}$ but on tilings of $\Lambda_r^\varphi$, that is
\[
 \nu_r(\eta)\propto \exp\bigg[-\frac12\bar\alpha |\varphi\xor\eta| - \int_{\alpha}^\infty \mu_{r, \eta}(\tfrac12| \eta \xor \psi|) \d \hat \alpha\bigg] \,,
\]
where $\mu_{r, \eta}$ is defined like $\mu_r$ but with $\varphi$ replaced by $\eta$.

(Again, in case $r \geq N/2$, we replace $B(o,r)$ in the definition of $\nu_r$ by the full torus $\T_N$, that is, we take $\nu_r = \nu_{\varphi,\bar\alpha}$; thus, every $B(o,r)$ that is strictly contained in $\T_N$ is also simply connected.)
\begin{remark}
As a follow-up to \cref{rem:mu-r-def-vs-mu}, in the above definition we can now see the difference between $\nu_r$ and $\nu$ conditioned on $\eta = \varphi$ outside of $B(o,r)$: the former uses $\mu_r$ and is therefore measurable with respect to $\varphi \restriction_{B(o,r)}$ while the latter would still involve the original $\mu$.
\end{remark}

The measure $\nu_r$ differs from $\nu_{\varphi,\bar\alpha}$ conditioned on $\eta=\varphi$ outside $B(o,r)$ not only in its set of allowed configurations but also in its interaction, which involves $\mu_{r,\eta}$ rather than $\mu_\eta$; in particular, the Glauber dynamics for $\nu_r$ and for $\nu_{2r}$ assign different acceptance probabilities to the same move. To handle this we need the following locality estimates, in which, for a domain $D$ as in \cref{lem:mu-locality} and a tiling $\eta$ of $\T_N$, we write
\begin{equation}\label{eq:U-D-def}
U_D(\eta):=\int_\alpha^\infty\mu_{D,\eta,\hat\alpha}\big(\tfrac12|\eta\xor\psi|\big)\d\hat\alpha\,,
\end{equation}
so that $\nu_r(\eta)\propto\exp[-\frac12\bar\alpha|\eta\xor\varphi|-U_{B(o,r)}(\eta)]$, and for a bubble $\sB$ that can be added to $\eta$ we write $\Delta_\sB U_D(\eta):=U_D(\eta\xor\sB)-U_D(\eta)$. The Glauber dynamics for $\nu_r$ is defined as in \cref{it:glauber-clocks-nu,it:glauber-move-nu} with candidate bubbles restricted to those with projection intersecting $B(o,r)$, and with $\sum_{x,r}\fg^\mu_r(\cdot)$ replaced by $U_{B(o,r)}(\cdot)$ in the weights $w_\eta,w_{\hat\eta}$; it is reversible w.r.t.\ $\nu_r$.

\begin{lemma}\label{lem:nu-domain-change}
There exist absolute constants $C,c>0$ such that the following holds for all $\alpha\geq C$, all domains $D\subseteq D'$ as in \cref{lem:mu-locality}, all tilings $\eta$ of $\T_N$ and all bubbles $\sB$ that can be added to $\eta$, with $\Upsilon(\sB)\subseteq D$.
\begin{enumerate}[(i)]
\item \label{it:U-size} $|\Delta_\sB U_D(\eta)|\leq C|\sB|$.
\item \label{it:U-reference} If $\sB_0$ is a $(\varphi,\eta)$-bubble disjoint from $\sB$ and $\eta':=\eta\xor\sB_0$, then $|\Delta_\sB U_D(\eta)-\Delta_\sB U_D(\eta')|\leq C(|\sB|\wedge|\sB_0|)e^{-c\alpha\dist(\Upsilon(\sB),\Upsilon(\sB_0))}$.
\item \label{it:U-domain} $|\Delta_\sB U_{D'}(\eta)-\Delta_\sB U_D(\eta)|\leq C|\sB|e^{-c\alpha\dist(\Upsilon(\sB),\partial D)}$ (where the right-hand side is read as $0$ if $D=D'$).
\end{enumerate}
\end{lemma}
\begin{proof}
Write $\hat\eta:=\eta\xor\sB$ and $\hat\eta':=\eta'\xor\sB$, and for $x\in\T_N$ let $a_x:=\dist(x,\Upsilon(\sB))$, $b_x:=\dist(x,\Upsilon(\sB_0))$ and $e_x:=\dist(x,\partial D)$ for $x\in D$. By the decomposition $\frac12|\eta\xor\psi|=\sum_xF_{\eta,x}(\psi)$ of \cref{eq:F-eta-x},
\begin{equation}\label{eq:Delta-U-decomp}
\Delta_\sB U_D(\eta)=\sum_{x\in\T_N}\int_\alpha^\infty\Big[\mu_{D,\hat\eta,\hat\alpha}(F_{\hat\eta,x})-\mu_{D,\eta,\hat\alpha}(F_{\eta,x})\Big]\d\hat\alpha\,.
\end{equation}
We claim that, for every $\hat\alpha\geq C$,
\begin{equation}\label{eq:F-x-reference}
\big|\mu_{D,\hat\eta,\hat\alpha}(F_{\hat\eta,x})-\mu_{D,\eta,\hat\alpha}(F_{\eta,x})\big|\leq Ce^{-c\hat\alpha(a_x\vee1)}\,.
\end{equation}
Indeed, if $a_x=0$ then each of the two expectations is at most the probability that some bubble contains a face above $x$, which is at most $Ce^{-3\hat\alpha}$ by a Peierls estimate (every bubble has at least six faces). If $a_x\geq1$, then $\eta=\hat\eta$ on $B(x,a_x)$, so the measures $\mu_{D\cap B(x,a_x/2),\eta,\hat\alpha}$ and $\mu_{D\cap B(x,a_x/2),\hat\eta,\hat\alpha}$ coincide and $F_{\eta,x}=F_{\hat\eta,x}$ on their support, whence \cref{eq:F-x-reference} follows from two applications of \cref{lem:mu-locality} (with $d=a_x$) together with the Peierls estimate just mentioned (which gives the factor $e^{-c\hat\alpha}$ for $a_x\leq 1/c$). Plugging \cref{eq:F-x-reference} into \cref{eq:Delta-U-decomp}, summing over $x$ (the number of $x$ with $a_x=a$ is at most $C|\Upsilon(\sB)|(a+1)$) and integrating over $\hat\alpha\geq\alpha$ gives \cref{it:U-size}, as $|\Upsilon(\sB)|\leq2|\sB|$.

For \cref{it:U-reference}, write the difference $\Delta_\sB U_D(\eta)-\Delta_\sB U_D(\eta')$ via \cref{eq:Delta-U-decomp} as a sum over $x$ of integrals of $\mu_{D,\hat\eta}(F_{\hat\eta,x})-\mu_{D,\eta}(F_{\eta,x})-\mu_{D,\hat\eta'}(F_{\hat\eta',x})+\mu_{D,\eta'}(F_{\eta',x})$. Grouping the first two and the last two terms, \cref{eq:F-x-reference} bounds the summand by $2Ce^{-c\hat\alpha(a_x\vee1)}$; grouping instead the first and third terms (which differ by the bubble $\sB_0$, so that \cref{eq:F-x-reference} applies with $\sB_0$ in the role of $\sB$) and the second and fourth, it is at most $2Ce^{-c\hat\alpha(b_x\vee1)}$. Letting $\fd:=\dist(\Upsilon(\sB),\Upsilon(\sB_0))$ and assuming by symmetry $|\sB|\leq|\sB_0|$, the faces $x$ with $a_x\geq\fd/2$ contribute at most $C|\Upsilon(\sB)|e^{-c\hat\alpha\fd/2}$ (by summability of the exponential), whereas the at most $C|\Upsilon(\sB)|\fd^2$ faces with $a_x<\fd/2$ have $b_x\geq\fd/2$ and contribute at most $C|\Upsilon(\sB)|\fd^2e^{-c\hat\alpha\fd/2}$. Integrating over $\hat\alpha\geq\alpha$ gives \cref{it:U-reference}.

For \cref{it:U-domain}, write $\Delta_\sB U_{D'}(\eta)-\Delta_\sB U_D(\eta)$ as a sum over $x$ of integrals of $\mu_{D',\hat\eta}(F_{\hat\eta,x})-\mu_{D',\eta}(F_{\eta,x})-\mu_{D,\hat\eta}(F_{\hat\eta,x})+\mu_{D,\eta}(F_{\eta,x})$ (where the last two terms vanish if $\dist(x,D)\geq2$, as then no bubble of a tiling of $\Lambda^\eta_D$ has a face above $x$). As before, \cref{eq:F-x-reference} bounds the summand by $2Ce^{-c\hat\alpha(a_x\vee1)}$. On the other hand, for $x\in D$ with $e_x\geq2$, grouping the first and third terms and the second and fourth, two applications of \cref{lem:mu-locality} with $d=e_x$ (both $D$ and $D'$ contain $B(x,e_x)$, and the functions being integrated are the same) bound the summand by $4Ce^{-c\hat\alpha e_x}$. With $\fd:=\dist(\Upsilon(\sB),\partial D)$, the faces with $a_x\geq\fd/2$ contribute at most $C|\Upsilon(\sB)|e^{-c\hat\alpha\fd/2}$, while the at most $C|\Upsilon(\sB)|\fd^2$ faces with $a_x<\fd/2$ lie in $D$ and have $e_x\geq\fd/2$, so they contribute at most $C|\Upsilon(\sB)|\fd^2e^{-c\hat\alpha\fd/2}$; integrating over $\hat\alpha\geq\alpha$ gives \cref{it:U-domain}.
\end{proof}
The bounds \cref{eq:jump_nu,eq:peierls_nu2}, and with them \cref{prop:nu-contraction}, hold verbatim for the Glauber dynamics of $\nu_r$: in their derivations, the expansion of the interaction in terms of the $\fg^\mu_r$ was used only through the bound $\sum_{x,r}|\fg_r^\mu(\hat\eta)-\fg_r^\mu(\eta)|\leq C|\sB|$ and through the estimate of $|\sum_{x,r}(\fg^\mu_r(\eta)+\fg^\mu_r(\hat\eta')-\fg^\mu_r(\hat\eta)-\fg^\mu_r(\eta'))|$, which are replaced by \cref{lem:nu-domain-change}\cref{it:U-size,it:U-reference} respectively.

\subsubsection*{Constructing the local function}
Let us denote by $\{\sB\in\eta\}$ for some bubble $\sB$ the event that $\sB$ appears in $\eta$ as a (complete) bubble, and let
\begin{align*}
f_{1,\sB}^\nu(\varphi\restriction_{B(x,1)}) &:= \int_\alpha^\infty \nu_1(\sB\in\eta)\d\bar\alpha\,,\\ 
f^\nu_{2r, \sB}(\varphi\restriction_{B(x,2r)}) &:= \int_{\alpha}^\infty \big[\nu_{2r}(\sB\in\eta) - \nu_r(\sB\in\eta)\big]\d\bar\alpha \qquad \mbox{for $r=2^k$, $k\geq 0$}\,.
\end{align*}
As was the case for $\mu$ in \cref{eq:mu-integral-with-f}, we have that
\begin{equation}\label{eq:nu-integral-with-f} \int_\alpha^\infty \nu_{\varphi,\bar\alpha}(\tfrac12 |\varphi\xor\eta|)\d \bar\alpha = \sum_{x\in\T_N} \sum_{\substack{\text{$(\varphi,\eta)$-bubble $\sB$} \\ \Upsilon(\sB)\ni x}} \frac12\sum_{\substack{r=2^k \\ \text{\emph{for }$k\geq 0$}}} f_{r,\sB}^\nu (\varphi\restriction_{B(x,r)}) \,,\end{equation}
and now wish to argue that \begin{equation}\label{eq:f-nu-bound}\|f_{2r,\sB}^\nu\|_\infty\leq C \exp[-\alpha |\sB|/4-\alpha r/C]\,.\end{equation} Since $\varphi$ and $\eta$ are tilings, $|\Upsilon(\sB)|=|\sB|$, giving the same factor $1/2$ as for $\mu$. The term corresponding to $\alpha|\sB|$ in the exponent again follows from a routine Peierls argument, already given in our proof of the contraction. Indeed, $p$ from  \cref{eq:peierls_nu} is bounded from above as per \cref{eq:peierls_nu2}, so $\nu_r(\sB\in\eta)$ and $\nu_{2r}(\sB\in \eta)$ are both at most $C\exp(-\frac12(\bar\alpha -C)|\sB|)$. However, unlike the analogous proof for the measure $\mu$, this time we would not want to integrate over $\bar\alpha$ just yet.

As before, the dependence in $r$ will be derived from the contracting Glauber dynamics. Let $(\eta_t,\eta'_t)$ be two coupled instances of the dynamics, with domains $B(o,r)$ and $B(o,2r)$ respectively, from an initial configuration which agrees on $B(o,r)$, where every update of $\eta'_t$ that is confined to $B(o,r)$ uses the joint law as per \cref{prop:nu-contraction} (whereas updates of $\eta'_t$ in the annulus $B(o,2r)\setminus B(o,r)$ will be sampled via the product measure of $\eta_t$ and $\eta'_t$ on this event). Run the dynamics for time
    \[ T = \epsilon \alpha \, r\,,\]
for an $\epsilon$ to be chosen later. We emphasize that in the previous proof we considered time $T\asymp \hat\alpha r$ as opposed to $T\asymp \alpha r$; this is due to the fact that the long-range interactions (which were not present in the measure $\mu$) decay only according to an $\alpha$-term, and we cannot afford to analyze the dynamics at larger scales.
Next write, still as in the $\mu$ case,
\[\left|\nu_{2r}(\sB\in\eta) - \nu_r(\sB\in\eta)\right| \leq \Xi_1 + \Xi_2 + \Xi_3\,,
\]
where
\begin{align*} \Xi_1 &:= \left| \nu_{r}(\sB\in\eta) - \P(\sB\in \eta_{T})\right|\,,\\
\Xi_2 &:= \left| \nu_{2r}(\sB\in\eta) - \P(\sB\in \eta'_{T})\right|\,,\\
\Xi_3 &:= \left| \P(\sB\in \eta_{T}) - \P(\sB\in \eta'_{T})\right|\,.
\end{align*}
Again, \cref{prop:nu-contraction} for $t=T$ gives us
\[ \Xi_1 \leq e^{-T/2} |B(o,r)| \leq e^{-(\frac12\epsilon\alpha-o(1))r}\,, \quad \quad \Xi_2 \leq e^{-T/2} |B(o,2r)| \leq e^{-(\frac12\epsilon\alpha-o(1))r}\,,\]
where the $o(1)$-term goes to $0$ as $r\to\infty$.

For $\Xi_3$, we must bound the rate of information propagation in the Glauber dynamics. 
Compared to the $\mu$ case, this has the further complication that we need to consider sequences of bubbles that do not intersect. We can still enumerate over sequences
 $\sB_1,\ldots,\sB_m$ ($m\geq 1$) such that the distance from $\sB_i$ to $\bigcup_{\ell\leq i} \sB_\ell$ is reached at a point of $\sB_{i-1}$ for all $i$. By analogy with the previous case, let $r_i$ be the diameter of $\sB_i$ and let $j_i = \dist( \sB_{i-1}, \sB_i)$ ($j_i = 0$ is possible if $\sB_i \cap \sB_{i-1} \neq \emptyset$), with the convention that $j_1$ is the distance to $\partial B(o,r)$ instead. We must also account for the fact that the two chains have different interactions: a bubble $\sB_i$ that can be added to both $\eta_t$ and $\eta'_t$ may be accepted in only one of them even when the two configurations agree in a neighborhood of $\sB_i$, because the weights involve $U_{B(o,r)}$ in one chain and $U_{B(o,2r)}$ in the other (while the clocks and the uniform variables are synchronized). Since $x\mapsto\frac{e^x}{c+e^x}$ is $\frac14$-Lipschitz, the probability of this event is at most $\frac14|\Delta_{\sB_i}U_{B(o,2r)}(\eta_t)-\Delta_{\sB_i}U_{B(o,r)}(\eta_t)|\leq C|\sB_i|e^{-c\alpha\dist(\Upsilon(\sB_i),\partial B(o,r))}$ by \cref{lem:nu-domain-change}\cref{it:U-domain}, which is at most $C|\sB_i|e^{-c\alpha j_i}$ since the annulus $B(o,2r)\setminus B(o,r)$ counts as a disagreement. Combined with the Peierls bound \cref{eq:peierls_nu2} for the acceptance in at least one of the two chains, via $\min(a,b)\leq\sqrt{ab}$, this is of the same form $Cs_ie^{-\frac14(\bar\alpha-C)s_i-c\alpha j_i/2}$ as the bound \cref{eq:jump_nu} for a disagreement generated by a difference between the configurations; it is therefore covered, after adjusting the constants, by the union bound below, in which the first factor $Cs_i e^{-\frac14(\bar\alpha-C) s_i-\alpha j_i/C}$ now accounts for both mechanisms. (Note that for $\mu$ no such term arises, since the acceptance probabilities of the Metropolis dynamics do not depend on the domain.)
 Using \cref{eq:jump_nu,eq:peierls_nu2} (and \cref{lem:nu-domain-change}\cref{it:U-domain}) to bound the rate at which the dynamics can create differences, the union bound now shows that $\Xi_3$ is at most 
\begin{align*}
& \sum_{m\geq 1}\sum_{\{r_i, j_i\}:\sum r_i + j_i \geq r/2} \sum_{\{s_i\}: s_i\geq r_i} \int_{0 \leq t_1 \leq \ldots \leq t_m \leq T}\\
&  \prod_i \bigg( c_0^{s_i} C s_{i-1} (j_i +1) \cdot Cs_ie^{-\frac14(\bar\alpha-C)s_i - \alpha j_i/C}e^{-(t_i - t_{i-1}) C s_ie^{-\frac14(\bar\alpha-C)s_i- \alpha j_i/C}}\bigg)\d t_1\ldots \d t_m,
\end{align*}
with the extra factor $C (j_i+1)$ associated with the choice of a root for $\sB_i$ given $\sB_{i-1}$ and the bound on the probability to create a defect with size $s$ at distance $j$ corresponding to an exponential clock with rate at most $C s\exp(-\frac14(\bar\alpha-C) s - \alpha j/C)$.
Using the same bound on the integral as for $\mu$, we get
\begin{align*}
\Xi_3 &\leq e^{\frac32 T} C e^{-(\alpha/C)r/2}\,,
\end{align*}
where the last inequality holds for some $C$ provided that $\bar\alpha$ is large enough. For $\epsilon$ chosen small enough, this is less than $C e^{- \alpha r/C}$ since $T = \epsilon \alpha r$.

Combining this with the bound $C\exp[-\frac12(\bar\alpha-C)|\sB|]$ established above, we see that
\[\left|\nu_{2r}(\sB\in\eta) - \nu_r(\sB\in\eta)\right| \leq C e^{- \alpha r/C - \frac{1}{4} (\bar \alpha - C) |\sB|} \,,
\]
which we can integrate in $\bar \alpha$ (as we kept the $\bar\alpha|\sB|$ term until now)
to obtain the sought \cref{eq:f-nu-bound}.

We now define $\fg_r^\nu$ (inheriting the sought $L^\infty$ bound from $f_r^\nu$) as
\[ \fg_r^\nu := \sum_{\sB:\;\Upsilon(\sB)\ni o} \frac12 f_{r,\sB}^\nu\,, \]
which, in view of \cref{eq:f-nu-bound}, establishes \cref{eq:orig-prop2}, because the smallest bubble $\sB$, a single cube, consists of $|\sB|=6$ faces.
\end{proof}

\section{Geometry of the energy minimizers}\label{sec:geometry-of-minimizers}

By now we have proved a local decomposition for two of the three measures from \cref{prop:P(phi)-via-3-measures}, but as mentioned in the introduction, the last one is much more complicated and its analysis covers \cref{sec:geometry-of-minimizers,sec:alg,sec:pi}. According to the outline in \cref{sec:sketch-step-3}, the first step is to turn the minimization of $G
_{h,\varphi}(\cdot)$ into a local observable through the definition of an appropriate notion of \emph{bubble groups}.

Recall from \cref{subsec:setup} that $\cP_{111}$ is the plane of equation $x_1+x_2+x_3 = 0$ and that $\Upsilon_{111}$ denotes the orthogonal projection onto that plane. Recall further that $\Upsilon_{111}$ establishes a bijection between $\varphi$ and $\cP_{111}$ and that any face of $\Z^3$ projects to a lozenge which we see as covering a black triangular face and a white one. We will very rarely use the projection on the $\cP_{001}$ plane before \cref{sec:concluding-thm-1} so for the ease of notation we will drop the index from $\Upsilon_{111}$.

\subsection{Ordering the energy minimizers by their heights}
A key observation in this section is the following result, providing a partial order over minimizers of the energy $G_{h,\varphi}$ from \cref{eq:G-def}. Note that, whereas $G_{h,\varphi}$ depends on $\varphi$, its set of minimizers is determined solely by $h$.

\begin{proposition}\label{prop:minimizers}
For every $h$, if $\Psi^\gr_{h}$ is the set of tilings $\psi$ that minimize the energy $G_{h,\varphi}$ from \cref{eq:G-def}, then $\Psi^\gr_{h}$ is closed under taking a maximum (viewing $\psi_1,\psi_2\in\Psi^\gr_{h}$ as height functions on $\Lambda_N$ w.r.t.\ $\cP_{001}$ in order to define their maximum $\psi_1\vee\psi_2$) as well as closed under taking a minimum.
In particular, there is a unique maximal element $\psi^\sqcup\in\Psi^\gr_{h}$ and a unique minimal element~$\psi^\sqcap\in\Psi^\gr_{h}$.
\end{proposition}
\begin{proof}
Recall that $\psi\in\Psi_h^\gr$ if and only if it maximizes $|h\cap\psi|$.
    It is easy to verify that if $\psi_1, \psi_2$ are monotone surfaces then so are $\psi_1 \vee \psi_2$ and $\psi_1 \wedge \psi_2$. Take $\psi_1,\psi_2\in \Psi^\gr_{h}$, and note that
    \begin{align} \quad h\cap \psi_1 &= \left(h\cap (\psi_1 \cap \psi_2 )\right) \cupdot \left(h\cap ((\psi_1 \wedge \psi_2) \setminus \psi_2)  \right) \cupdot \left(h\cap ((\psi_1 \vee \psi_2) \setminus \psi_2) \right)\label{eq:G-h-phi-psi1}\end{align}
    (where $\cupdot$ denotes a disjoint union), and similarly for $h\cap\psi_2$. Also,
    \begin{align}
    \left(h\cap(\psi_1 \wedge \psi_2)\right) &= \left(h\cap (\psi_1 \cap \psi_2 )\right) \cupdot \left(h\cap ((\psi_1 \wedge \psi_2) \setminus \psi_2  )\right)\cupdot\left(h\cap((\psi_1 \wedge \psi_2) \setminus \psi_1 )\right)\,,\label{eq:G-h-phi-psi1-min-psi2}
\end{align}
    and similarly for $h\cap(\psi_1 \vee \psi_2)$.

Since $\psi_1 \in \Psi^\gr_{h}$, we have $|h\cap \psi_1| \geq \left|h\cap(\psi_1\wedge \psi_2)\right|$, so comparing \cref{eq:G-h-phi-psi1,eq:G-h-phi-psi1-min-psi2} yields 
\begin{align}
\left|h\cap((\psi_1 \vee \psi_2) \setminus \psi_2  )\right| 
        &\geq \left|h\cap ((\psi_1 \wedge \psi_2) \setminus \psi_1  )\right|\,.\label{eq:G-M-psi2}
        \end{align}
Similarly, $|h\cap\psi_1| \geq |h\cap(\psi_1\vee\psi_2)|$, so \cref{eq:G-h-phi-psi1} and the analog of \cref{eq:G-h-phi-psi1-min-psi2} for $h\cap(\psi_1\vee\psi_2)$ give
\begin{align}
    \left|h\cap((\psi_1 \wedge \psi_2) \setminus \psi_2  ) \right|
        &\geq \left|h\cap ((\psi_1 \vee \psi_2) \setminus \psi_1  )\right|\,. \label{eq:G-m-psi2}
\end{align}    
Reversing the roles of $\psi_1,\psi_2$ we find, in the same manner, that
\begin{align}
\left|h\cap ((\psi_1 \vee \psi_2) \setminus \psi_1  ) \right|
        &\geq \left|h\cap((\psi_1 \wedge \psi_2) \setminus \psi_2  )\right|\,,\label{eq:G-M-psi1} \\        
\left|h\cap ((\psi_1 \wedge \psi_2) \setminus \psi_1  )\right| 
        &\geq \left|h\cap ((\psi_1 \vee \psi_2) \setminus \psi_2  )\right|\,.\label{eq:G-m-psi1}
 \end{align}
Combining \cref{eq:G-M-psi2,eq:G-m-psi1} yields $\left|h\cap ((\psi_1\vee\psi_2)\setminus\psi_2)\right|=\left|h\cap ((\psi_1\wedge\psi_2)\setminus\psi_1)\right|$, which implies (again through \cref{eq:G-h-phi-psi1,eq:G-h-phi-psi1-min-psi2})
that $|h\cap\psi_1| = |h\cap(\psi_1\wedge\psi_2)|$ and so $\psi_1\wedge\psi_2\in\Psi_{h}^\gr$. Similarly, via \cref{eq:G-m-psi2,eq:G-M-psi1} we find $|h\cap\psi_1|=|h\cap(\psi_1\vee\psi_2)|$, whence $\psi_1\vee\psi_2\in\Psi_{h}^\gr$, as required.
\end{proof}
\begin{remark} \label{rem:gen-Psi-h} More generally, the proof above holds if $\Psi_h^\gr$ is the set of minimizers of $\psi\mapsto \sum_{f\in\psi}a_{f} $ for any $a_{f}\in\R $ per plaquette $f$ in $\Z^3$   (minimizing $G_{h,\varphi}$ corresponds to taking  $a_{f}=-\one_{\{f\in h\}}$).
\end{remark}
\begin{corollary}\label{cor:psi-min-max-construction}
For every $h$, the set $\Psi_{h}^\gr$ of tilings $\psi$ minimizing $G_{h,\varphi}(\psi)$ can be obtained as follows:
\begin{enumerate}[(1)] 
\item Let $\psi_0 = \psi^\sqcap \cap \psi^\sqcup$. Every $\psi\in\Psi_{h}^\gr$ will include $\psi_0$.
\item \label{it:ind-component} Viewing $\psi^\sqcap$ as a tiling, for every connected component $\cC_i$ ($i\geq 1)$ of $\Upsilon(\psi^\sqcap \setminus \psi_0)$, choose $\psi_i$ independently among the lozenge tilings of $\cC_i\subset \T_N$ which, once lifted to $\Z^3$ using the heights of $\psi_0$ along $\partial\cC_i$, maximize $|h\cap\psi_i|$.
\item Complete~$\psi$ by gluing $\psi_0$ with all the $\psi_i$.
\end{enumerate}
\end{corollary}

The following local representation of the energy $G_{h,\varphi}$ will also have a useful role in the proofs.
\begin{lemma}[Zero range energy]\label{lem:zero-range}
For every $h,\varphi$ there exists an integer-valued function $g_{h,\varphi}$ on plaquettes of $\Z^3$ such that the energy function $G_{h,\varphi}$ from \cref{eq:G-def} can be given by 
\[G_{h, \varphi}( \psi) = \sum_{f\in\psi}  g_{h, \varphi}(f)\,.\]
\end{lemma}
\begin{proof}
Define $g_{h,\varphi}(f)$ as follows:
\[
g_{h, \varphi}(f) := \one_{\{\left(\Upsilon^{-1}( \Upsilon(f) \restriction_{\text{black}}) \cap \varphi \right)\subset h \}}  - \one_{\{f \in h\}} \, .
\]
In other words, the first term is obtained by projecting $f$ to a lozenge $\Upsilon(f)$ of $\cP_{111}$, looking at the restriction to the black triangle in $\Upsilon(f)$, then lifting this triangle back to $\varphi$ and finally checking whether this is part of the intersection of $\varphi$ and $h$. Verifying \cref{eq:G-def} is then immediate.
\end{proof}

The following is the analog of \cref{lem:zero-range} for the normalized energy function $\overline G_{h}$ from \cref{eq:G-bar-def}.
\begin{lemma}\label{lem:zero-range-G-bar}
For every $h$ there exists an integer-valued function $\overline{g}_{h}$ on plaquettes of $\Z^3$ such that the energy function $\overline{G}_h$ from \cref{eq:G-bar-def} can be given by 
\[\overline{G}_{h}( \psi) = \sum_{f\in\psi} \overline{g}_{h}(f)\,.\]
\end{lemma}
\begin{proof}
With \cref{prop:minimizers} in mind, for every face $f\in\psi$, let $t$ be the black portion of $\Upsilon(f)$ and let $f^\sqcup$ be the unique face of $\psi^\sqcup\in \Psi^\gr_{h}$ such that $t \subset \Upsilon(f^\sqcup)$. We set
\begin{equation}\label{eq:g-bar-def} 
\overline g_{h}(f) := g_{h, \varphi}(f) - g_{h, \varphi}(f^\sqcup) = \one_{\{f^\sqcup \in h\}} - \one_{\{f \in h\}}\,.\end{equation}
This provides a valid decomposition of $\overline G_h$ because, as in the proof of \cref{lem:zero-range}, every tiling has exactly one face projecting on each black triangle.
\end{proof}

\begin{remark}
As an alternative to the local representation of $G_{h,\varphi}$ from \cref{lem:zero-range}, one could have used an equivalent expression for $G_{h,\varphi}(\psi) =|h\cap \varphi|-|h\cap\psi|$ in terms of symmetric differences:
\begin{equation}\label{eq:G-def-2}
G_{h,\varphi}(\psi) = \frac12 |\varphi\xor\psi| - |(\varphi\xor\psi)\cap(\varphi \xor h)|\,.\end{equation}
To verify this equivalent formulation, observe that, for any sets of faces $h,\varphi,\psi$,
\begin{align}
|h\cap\psi| &= |\varphi\cap\psi| - |\varphi\setminus h| + |(\varphi\xor h)\cap(\varphi\xor\psi)| \nonumber\\
&= \frac12 (|\varphi| + |\psi|) - \frac12 |\varphi\xor\psi| - |\varphi\setminus h| + |(\varphi\xor h)\cap (\varphi\xor\psi)|\,.\label{eq:h-cap-psi}
\end{align}
(The first equality expressed $h\cap\psi$ as the union of $h\cap\varphi\cap\psi$ and $h\cap(\psi\setminus\varphi)$; the latter is $(\varphi\xor h)\cap(\psi\setminus\varphi)$, the former is $(\varphi\cap\psi) \setminus ((\varphi\cap\psi)\setminus h)$, i.e.,
$(\varphi\cap\psi) \setminus (\varphi\setminus h) \cup (\varphi\xor h)\cap (\varphi\setminus \psi)$. The second equality put $|\varphi\cap\psi|$ as $\frac12(|\varphi|+|\psi|-|\varphi\xor\psi|)$.)
When $|\psi|=|\varphi|$ --- the case at hand when $\varphi,\psi$ are both tilings --- the expression $\frac12(|\varphi|+|\psi|)-|\varphi\setminus h|$ in 
the right hand of \cref{eq:h-cap-psi} is nothing but $|\varphi\cap h|$, whence comparing it to \cref{eq:G-def} establishes \cref{eq:G-def-2}.
\end{remark}

\subsection{Bubble groups}\label{sec:bubble-groups}
\Cref{cor:psi-min-max-construction} allows us to elevate the notion of bubbles to bubble groups---addressing long-range interactions between bubbles---which will be used to show the locality of~$\pi$.
Recall the definition of $\pi$ from \cref{eq:pi-def}, and that the first term in the exponent, $H_h$, breaks into a sum over each of the bubbles: $H_h = \sum_\sB H(\sB)$ as per \cref{eq:bubble-H}. An analog of this for the next two terms in that exponent, $G_{h,\varphi}^\gr$ and $\log Z_\mu^\infty(h)$, readily follows from \cref{cor:psi-min-max-construction}, as $\psi^\sqcap \cap \psi^\sqcup$ belongs to every minimizer $\psi\in\Psi_h^\gr$, and elsewhere the minimization can be solved independently:

\begin{observation}\label{obs:bubble-grp-relations}
Given $h$ and $\varphi$, consider an equivalence relation on $(h,\varphi)$-bubbles where $\sB \sim\sB'$ if (perhaps not only if) $\Upsilon(\sB)$ and $\Upsilon(\sB')$ intersect a common connected component~$\cC$ of $\Upsilon(\psi^\sqcap \xor\psi^\sqcup)$. 
Then, referring to the resulting equivalence class as a \emph{bubble group} $\fB$ (consisting of its bubbles $\sB_i$ as well as the ``connector'' components $\cC_j$), by \cref{cor:psi-min-max-construction} one can write $G_{h,\varphi}^\gr=\sum_\fB G^\gr(\fB)$ for an appropriate function $G^\gr(\cdot)$ on bubble groups, and the same holds for $\log Z_\mu^\infty(h)$ and $\sV(h)$.
\end{observation}

In \cref{sec:2-out-of-3}, we proved the weak locality of $\mu$ and $\nu$ after evaluating the effect of deleting a single bubble $\sB$ as part of a Metropolis/Glauber dynamics. If we are to mimic this for a bubble group~$\fB$ via \cref{obs:bubble-grp-relations}, we can only hope to gain energy from its bubbles (rather than other faces of the connector components $\cC_j$); it is then natural to adopt the minimal choice of having $\sB\sim\sB'$ if \emph{and only if} they intersect a common component of $\Upsilon(\psi^\sqcap \xor \psi^\sqcup)$. Unfortunately, that equivalence relation is not monotone w.r.t.\ the deleting of bubble groups (see \cref{fig:bubble-grp-non-monotone})---so deleting~$\fB_1$ and then $\fB_2$ in the dynamics might be permitted, whereas deleting $\fB_2$ followed by $\fB_1$ might not be...

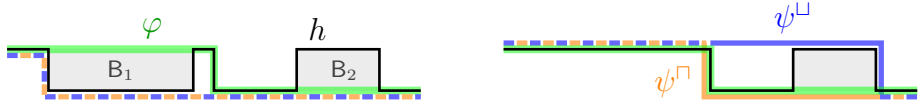
\begin{figure}
    \begin{tikzpicture}
    \begin{scope}[scale=0.55]

    \filldraw[fill=gray!15,draw=none] (1,1)--(4.5,1)--(4.5,0)--(1,0)--cycle;
    \filldraw[fill=gray!15,draw=none] (7,1)--(9,1)--(9,0)--(7,0)--cycle;
    
    \draw [line width=3pt,green,opacity=0.5] (0,1)--(5,1)--(5,0)--(10,0);
    
    \draw [line width=1pt,black] (0,1)--(1,1)--(1,0)--(4.5,0)--(4.5,1)--(5,1)--(5,0)--(7,0)--(7,1)--(9,1)--(9,0)--(10,0);

    \draw [line width=2pt,blue!60,
    dash pattern= on 4pt off 6pt,postaction={draw,orange!60,dash pattern= on 4pt off 6pt,dash phase=5pt,line width=2pt}] (0,.85)--(.9,.85)--(.9,-0.15)--(10,-0.15);

    \node[green!60!black,font=\large] at (3.5,1.5) {$\varphi$};
    \node[font=\large] at (7.5,1.5) {$h$};

    \node[gray!50!black,font=\footnotesize] at (2.75,0.5) {$\sB_1$};
    \node[gray!50!black,font=\footnotesize] at (8.15,0.5) {$\sB_2$};
    
    \end{scope}

    \begin{scope}[scale=0.55,shift={(12,0)}]

    \filldraw[fill=gray!15,draw=none] (7,1)--(9,1)--(9,0)--(7,0)--cycle;

    \draw [line width=2pt,blue!60,
    dash pattern= on 4pt off 6pt,postaction={draw,orange!60,dash pattern= on 4pt off 6pt,dash phase=5pt,line width=2pt}] (0,1.15)--(4.9,1.15);

    \draw [line width=2pt,blue!60,
    dash pattern= on 4pt off 6pt,postaction={draw,orange!60,dash pattern= on 4pt off 6pt,dash phase=5pt,line width=2pt}] (9.1,-0.15)--(10,-0.15);
    
    \draw [line width=2pt,blue!60] (5,1.15)--(9.13,1.15)--(9.13,-0.15);
    
    \draw [line width=2pt,orange!60] (4.85,1.15)--(4.85,-0.15)--(9.1,-0.15);
        
    \draw [line width=3pt,green,opacity=0.5] (0,1)--(5,1)--(5,0)--(10,0);
    
    \draw [line width=1pt,black] (0,1)--(5,1)--(5,0)--(7,0)--(7,1)--(9,1)--(9,0)--(10,0);

    \node[orange,font=\large] at (4.1,0.2) {$\psi^\sqcap$};
    \node[blue,font=\large] at (7,1.75) {$\psi^\sqcup$};
    
    \end{scope}
\end{tikzpicture}
\vspace{-0.15in}\caption{Non-monotonicity of $\psi^\sqcap \xor \psi^\sqcup$ under bubble deletion: deleting the $(h,\varphi)$-bubble $\sB_1$ on the left (where $\psi^\sqcup = \psi^\sqcap$) introduces new faces in $\psi^\sqcup \setminus (\psi^\sqcap \cup \varphi)$ on the right.
}
\label{fig:bubble-grp-non-monotone}
\end{figure}

We emphasize that, unlike the definition of a bubble $\sB$, which relied on one fixed equivalence relation (adjacency of plaquettes in $\Z^3$), the difficulty here is that the relation $\sim$ in the definition of a bubble group $\fB$ must be a function of $h$ (taking into account $\psi^\sqcap \cap \psi^\sqcup$); a given collection of bubbles may, or may not, be categorized as one complete bubble group, depending on their exterior. 
We wish to be in a position where, if $h$ is obtained from $\hat h$ by deleting a set of bubbles $\{\sB_i\}$, which formed one or more complete bubble groups as per the notion of bubble groups in $(\hat h,\varphi)$, then the same holds also for the notion of bubble groups in $(h,\varphi)$. To this end, we must increase the set of plaquettes used in \cref{obs:bubble-grp-relations} to relate $\sB,\sB'$ beyond the minimal requirement $\psi^\sqcap\xor\psi^\sqcup$ (and, by doing so, inevitably increase the entropy of bubble groups and their long-range interactions). One may take $(\varphi\xor\psi^\sqcap) \cup (\varphi\xor\psi^\sqcup)$ (a super-set of $\psi^\sqcap \xor\psi^\sqcup$), but this fails too (again  see \cref{fig:bubble-grp-non-monotone}).
As it turns out, a closer inspection of $(h,\varphi)$-bubbles, accounting for their ``positive'' and ``negative'' portions (``above~$\varphi$''/``below $\varphi$'') will support a notion of bubble groups with the sought properties. 

\begin{figure}
    \begin{tikzpicture}
    \begin{scope}[scale=0.55,shift={(8,0)}]

    \filldraw[fill=gray!15,draw=none] (1,1)--(4.5,1)--(4.5,0)--(1,0)--cycle;
    \filldraw[fill=gray!15,draw=none] (7,1)--(9,1)--(9,0)--(7,0)--cycle;
    
    \draw [line width=3pt,green,opacity=0.5] (0,1)--(5,1)--(5,0)--(10,0);
    
    \draw [line width=1pt,black] (0,1)--(1,1)--(1,0)--(4.5,0)--(4.5,1)--(5,1)--(5,0)--(7,0)--(7,1)--(9,1)--(9,0)--(10,0);

    \node[green!60!black,font=\large] at (3.5,1.5) {$\varphi$};
    \node[font=\large] at (7.5,1.5) {$h$};

    \node[gray!50!black,font=\footnotesize] at (2.75,0.5) {$\sB_1$};
    \node[gray!50!black,font=\footnotesize] at (8.15,0.5) {$\sB_2$};
    
    \end{scope}

    \begin{scope}[scale=0.55,shift={(0,-2.75)}]

    \draw [line width=2pt,blue!60,
    dash pattern= on 4pt off 6pt,postaction={draw,orange!60,dash pattern= on 4pt off 6pt,dash phase=5pt,line width=2pt}] (0,1.15)--(1.13,1.15)--(1.13,-0.15)--(10,-0.15);
        
    \draw [line width=3pt,green,opacity=0.5] (0,1)--(5,1)--(5,0)--(10,0);
    
    \draw [line width=1pt,black] (0,1)--(1,1)--(1,0)--(4.5,0)--(4.5,1)--(4.5,1)--(5,1)--(5,0)--(7,0);
    \draw [line width=1pt,black] (9,0)--(10,0);

    \node[font=\large] at (5,-0.8) {$\psi^{-,\sqcap}=\psi^{-,\sqcup}$};
    \node[black,font=\large] at (8.75,1) {$h\cap \mathbb{H}^-_\varphi$};
    
    \end{scope}

    \begin{scope}[scale=0.55,shift={(13,-2.75)}]

    \draw [line width=2pt,blue!60,
    dash pattern= on 4pt off 6pt,postaction={draw,orange!60,dash pattern= on 4pt off 6pt,dash phase=5pt,line width=2pt}] (0,1.15)--(4.9,1.15);

    \draw [line width=2pt,blue!60,
    dash pattern= on 4pt off 6pt,postaction={draw,orange!60,dash pattern= on 4pt off 6pt,dash phase=5pt,line width=2pt}] (9.1,-0.15)--(10,-0.15);
    
    \draw [line width=2pt,blue!60] (5,1.15)--(9.13,1.15)--(9.13,-0.15);
    
    \draw [line width=2pt,orange!60] (4.85,1.15)--(4.85,-0.15)--(9.1,-0.15);
        
    \draw [line width=3pt,green,opacity=0.5] (0,1)--(5,1)--(5,0)--(10,0);
    
    \draw [line width=1pt,black] (0,1)--(1,1);
    \draw [line width=1pt,black] (4.5,1)--(5,1)--(5,0)--(7,0)--(7,1)--(9,1)--(9,0)--(10,0);

    \node[orange,font=\large] at (3.75,0.2) {$\psi^{+,\sqcap}$};
    \node[blue,font=\large] at (6.25,1.8) {$\psi^{+,\sqcup}$};
    \node[black,font=\large] at (10.75,1) {$h\cap \mathbb{H}^+_\varphi$};

    \end{scope}
\end{tikzpicture}
\vspace{-0.1in}\caption{Illustration of $\psi^{\pm,\sqcup}$ and $\psi^{\pm,\sqcap}$  in \cref{def:bubble-group} for the pair $(h, \varphi)$ from \cref{fig:bubble-grp-non-monotone}. Note the hole in $h \cap \mathbb{H}_\varphi^-$ (resp.\ $h\cap \mathbb{H}_\varphi^+$) instead of $\sB_2$ (resp.\ $\sB_1$). The corresponding bubble group then goes from the leftmost point of $\sB_1$ to the rightmost point of $\sB_2$.
}
\label{fig:bubble-grp-illustration}
\end{figure}

\begin{definition}[Bubble group]\label{def:bubble-group}
Given $h$ and $\varphi$, let $\mathbb{H}_\varphi^+$ denote the upper half-space of $\Z^3$ delimited by the faces of $\varphi$ viewed as a full-plane periodic height function (including the faces of $\varphi$ itself), and define the lower half-space $\mathbb{H}_\varphi^-$ analogously (again, including the faces of $\varphi$ itself). Set
\[ G_{h,\varphi}^+(\psi) := |h\cap \varphi| - |h\cap\psi \cap \mathbb{H}_\varphi^+|\quad,\quad 
G_{h,\varphi}^-(\psi) := |h\cap \varphi| - |h\cap\psi \cap \mathbb{H}_\varphi^-|\,.\]
Let $\psi^{\pm,\sqcup}$ and $\psi^{\pm,\sqcap}$ be the unique maximal and minimal elements of the minimizers of $G_{h,\varphi}^\pm$, as in \cref{prop:minimizers} (illustrated in \cref{fig:bubble-grp-illustration}) with the generalized form of \cref{rem:gen-Psi-h} in mind, and define
\[ \delta=\delta(h,\varphi):= \psi^{+,\sqcup}\xor \psi^{-,\sqcap}\,.\]
We say $\sB \sim \sB'$ if both $\Upsilon(\sB)$ and $\Upsilon(\sB')$ intersect a common connected component~$\cC$ of $ \Upsilon(\delta)$. A bubble group $\fB=
(\{\sB_i\},\{\cC_j\})$ is a maximal connected component of bubbles $\sB_i$ w.r.t.\ this adjacency relation, joined by every connected component $\cC_j$ of $\Upsilon(\delta)$ intersecting any of them.
\end{definition}

This definition is fairly intricate; to establish that it meets our requirements we must verify:
\begin{enumerate}[(i)]
\item{}[\emph{valid equivalence relation as \cref{obs:bubble-grp-relations} requires}] \label{it:legal-relation} ~$\Upsilon(\psi^\sqcap\xor\psi^\sqcup)\subseteq \Upsilon(\delta)$.
\item{}[\emph{entropy is controlled via the energy}] \label{it:entropy-relation} ~$\sum |\cC_j| \leq c_0 \sum |\sB_i|$ for every bubble group $\fB$.
\item{}[\emph{monotonicity}] \label{it:monotone-relation} ~If $h$ is obtained from $\hat h$ by deleting a bubble group $\fB$, then $\Upsilon(\delta) \subseteq \Upsilon(\hat\delta)$.
\end{enumerate}

We begin by showing  $\psi^{+,\sqcup}$ and $\psi^{-,\sqcap}$ always sandwich $\varphi,\psi^\sqcap,\psi^\sqcup$ between them. Note that in our notation, $\psi_1 \subset \mathbb{H}_{\psi_2}^+$ as a collection of plaquettes iff $\psi_1 \geq \psi_2$ as a height function on $\Lambda_N$ w.r.t.\ $\cP_{001}$. 

\begin{claim}\label{clm:psi+cup-above-phi}
For every $h$ and $\varphi$	one has 
$\psi^{+,\sqcup} \geq \varphi$ and, symmetrically, $\psi^{-,\sqcap} \leq \varphi$.
\end{claim}
\begin{proof}
Suppose that $\cC$ is a nonempty (maximal) connected component of faces of $\psi^{+,\sqcup} \cap (\Z^3 \setminus \mathbb{H}_{\varphi}^+)$. Since $\psi^{+,\sqcup}$ and $\varphi$ induce two lozenge tilings of the region $S=\Upsilon(\cC)$ with boundary tiles agreeing with~$\varphi$, we can swap the two tilings in $S$: let $\psi'$ be $\psi^{+,\sqcup}$ outside of $S$ and $\varphi$ inside. By definition, $G_{h,\varphi}^+$ does not reward faces in~$\cC$, so $G_{h,\varphi}^+(\psi') \leq G_{h,\varphi}^+(\psi^{+,\sqcup})$. If the set of faces $\cC'$ replacing $\cC$ in $\psi'$ contains a face of~$h$, then the inequality is strict, contradicting the fact that $\psi^{+,\sqcup}$ is a minimizer. Regardless, $ \psi' \not\leq \psi^{+,\sqcup}$,
contradicting the fact that $\psi^{+,\sqcup}$ is a maximal element of the minimizers.
\end{proof}

\begin{claim}\label{clm:psi+cup-above-psicup}
For every $h$ and $\varphi$	one has $\psi^{+,\sqcup} \geq \psi^\sqcup$ and, symmetrically, $\psi^{-,\sqcap} \leq \psi^\sqcap$.
\end{claim}
\begin{proof}
As in the proof of \cref{clm:psi+cup-above-phi}, suppose that $\cC$ is a nonempty (maximal) connected component of faces of $\psi^{+,\sqcup} \cap (\Z^3 \setminus \mathbb{H}_{\psi^\sqcup}^+)$, and let $\psi'$ be the tiling obtained by swapping $\cC$ by $\cC'$ corresponding to the tiling of $S=\Upsilon(\cC)$ by $\psi^\sqcup$. Since we know by \cref{clm:psi+cup-above-phi} that $\psi^{+,\sqcup}\geq \varphi$, we also have $\cC'\subset \mathbb{H}_\varphi^+$; thus, every face of $h\cap \psi'$ is rewarded by $G_{h,\varphi}^+$, and since it is the maximum (by definition of $\psi^\sqcup$) we find that $\psi'$ is a minimizer of $G_{h,\varphi}^+$, a contradiction to $\psi^{+,\sqcup}$ being a maximal element. 
\end{proof}

Combining the preceding two claims, every face of $\psi^{+,\sqcup}\cap \psi^{-,\sqcap}$ must also belong to $\psi^\sqcap,\psi^\sqcup,\varphi$ (e.g., it belongs to $\mathbb{H}_\varphi^+\cap \mathbb{H}_\varphi^- = \varphi$, and one argues similarly for $\psi^\sqcup \geq \psi^\sqcap$), giving \cref{it:legal-relation} as follows:

\begin{corollary}\label{cor:legal-equiv-rel}
For every $h$ and $\varphi$ one has $\Upsilon(\psi^\sqcup \xor \psi^\sqcap) \subseteq \Upsilon( \psi^{+,\sqcup} \xor \psi^{-,\sqcap})$.
\end{corollary}
(In fact, we established the stronger statement $\Upsilon\left((\varphi \xor \psi^\sqcap) \cup (\varphi \xor \psi^\sqcup)\right)\subseteq \Upsilon(\psi^{+,\sqcup} \xor \psi^{-,\sqcap})$.)
The following simple claim will readily establish \cref{it:entropy-relation}.

\begin{claim}\label{clm:area-of-bubble-grp}
For every bubble group $\fB=(\{\sB_i\},\{\cC_j\})$ we have $\sum_j|(\varphi\xor\psi^{+,\sqcup})\cap\Upsilon^{-1}(\cC_j)| \leq 2 \sum_i |\sB_i|$, and the same bound holds (symmetrically) for $\psi^{-,\sqcap}$.
\end{claim}
\begin{proof}
Let $K$ be a connected component of $\Upsilon(\varphi\xor\psi^{+,\sqcup})$, the set of triangles above which the plaquettes of the two tilings differ. If a lozenge of $\varphi$ or of $\psi^{+,\sqcup}$ contains a triangle of $K$, then the plaquettes of the two tilings above its other triangle differ as well, so the lozenge is contained in $K$; thus both tilings tile $K$, they agree above every lozenge adjacent to $K$, and the set $\psi'$ consisting of the plaquettes of $\varphi$ above $K$ and of the plaquettes of $\psi^{+,\sqcup}$ above the complement of $K$ is a tiling. Moreover, by \cref{clm:psi+cup-above-phi}, a plaquette common to $\psi^{+,\sqcup}$ and $\psi^{-,\sqcap}$ belongs to $\varphi$, so every plaquette of $\psi^{+,\sqcup}$ above $K$ belongs to $\delta$, whence $K\subseteq\Upsilon(\delta)$ and, being connected, $K$ lies in a single component of $\Upsilon(\delta)$. Since $\psi^{+,\sqcup}\subseteq\mathbb H^+_\varphi$ by \cref{clm:psi+cup-above-phi}, the minimality of $\psi^{+,\sqcup}$ gives
\[0\leq G^+_{h,\varphi}(\psi')-G^+_{h,\varphi}(\psi^{+,\sqcup})=\big|h\cap\psi^{+,\sqcup}\cap\Upsilon^{-1}(K)\big|-\big|h\cap\varphi\cap\Upsilon^{-1}(K)\big|\,.\]
As $\varphi$ and $\psi^{+,\sqcup}$ have the same number of plaquettes above $K$, none of which is common to both, and as $h\cap\psi^{+,\sqcup}\cap\Upsilon^{-1}(K)\subseteq h\setminus\varphi$, we get
\[\big|(\varphi\xor\psi^{+,\sqcup})\cap\Upsilon^{-1}(K)\big|=2\big|(\varphi\setminus h)\cap\Upsilon^{-1}(K)\big|+2\big|h\cap\varphi\cap\Upsilon^{-1}(K)\big|\leq2\big|(\varphi\xor h)\cap\Upsilon^{-1}(K)\big|\,.\]
Summing over the components $K$ contained in $\bigcup_j\cC_j$, and noting that every face of $\varphi\xor h$ above $\bigcup_j\cC_j$ belongs to a bubble of $\fB$ by \cref{def:bubble-group}, yields the claim.
\end{proof}
It remains to establish the monotonicity of the bubble groups.
\begin{claim}\label{clm:psi+-mon}
For every $h,\hat h$ and $\varphi$, if $h$ is obtained by deleting an $(\hat h,\varphi)$-bubble group $\fB$ from $\hat h$ and we denote $\psi^\cdot(\hat h)$ by $\hat\psi^\cdot$, then we have
$\hat\psi^{+,\sqcup} \geq \psi^{+,\sqcup}$ and, symmetrically, $\hat\psi^{-,\sqcap} \leq \psi^{-,\sqcap}$.
\end{claim}
\begin{proof}
Throughout this proof, write $\psi^+ = \psi^{+,\sqcup}$ and $\hat\psi^+=\hat\psi^{+,\sqcup}$ for brevity, and let $h^+ = h \cap \mathbb{H}_\varphi^+$ and $\hat h^+ = \hat h \cap \mathbb{H}_\varphi^+$. 
With the aim of showing $\hat\psi^+ \geq \psi^+$, suppose that $\cC$ is a nonempty (maximal) connected component of $\hat\psi^+ \cap (\Z^3\setminus \mathbb{H}_{\psi^+}^+)$, and let $\cC'$ be the tiling induced by $\psi^+$ on $\Upsilon(\cC)$. Note that
$	|\cC\cap  h^+ | \leq |\cC' \cap h^+|$, 
or else we could interchange $\cC,\cC'$ in $\psi^+$ and get a contradiction to it being a minimizer of $G_{h^+,\varphi}$ (maximizing its overlap with $h^+$). 
For $\hat h^+$, the analogous inequality is strict:
\begin{equation}\label{eq:C-C'-hat-h}
	|\cC' \cap \hat h^+| < |\cC\cap \hat h^+|\,,
\end{equation}
since, by construction of $\cC$, interchanging $\cC,\cC'$ in $\hat\psi^+$ would give a minimizer $\psi'$ satisfying $\psi' \not\leq \hat\psi^+$, a contradiction to the fact that $\hat\psi^+$ is a maximal element in the corresponding set of minimizers. 

Next, consider $h$ vs.\ $\hat h$, let $\{\sB_i\}$ be the bubbles of the bubble group $\fB$ they differ by, and by a slight abuse of notation, use $\fB \cap \varphi$ to denote $(\bigcup_i \sB_i)\cap \varphi$ and $\fB \setminus \varphi$ to denote $(\bigcup_i \sB_i)\setminus \varphi$. Then
\begin{equation}\label{eq:C-h-hat-h}
	|\cC \cap h^+| = |\cC \cap \hat h^+| - |\cC\cap(\fB\setminus \varphi)| + |\cC \cap (\fB \cap \varphi)| \,,
\end{equation}
simply because the face set $\fB\setminus\varphi$ in $\hat h^+$ is replaced by $\fB\cap \varphi$ in $h$. The same identity holds for $\cC'$, but now it can be improved: since $\psi^+ > \hat\psi^+ \geq \varphi$ on the region corresponding to its face set $\cC'$, we see that $\cC' \cap (\fB\cap\varphi) = \emptyset$, whence 
\begin{equation*}%
	|\cC' \cap h^+| =  |\cC' \cap \hat h^+| - |\cC'\cap(\fB\setminus \varphi)|\,.
\end{equation*}
Using \cref{eq:C-C'-hat-h} on the right-hand side of the last identity, and then plugging in \cref{eq:C-h-hat-h}, we get
\begin{align}
	|\cC \cap h^+| - |\cC' \cap h^+|
 &> -|\cC\cap(\fB\setminus \varphi)|
  + |\cC'\cap(\fB\setminus \varphi)|
   +|\cC \cap (\fB \cap \varphi)| \geq -|\cC\cap(\fB\setminus \varphi)|\,.
 \label{eq:C-h+-C'-h+}\end{align}
The above applies simultaneously for all  components $\cC$ as above, so for ease of notation, henceforth let $\cC$ be their union (no longer assumed to be connected).

Let $\psi_0$ be the tiling obtained by interchanging $\cC,\cC'$ in $\psi^+$. Let $\psi_1$ be obtained from $\psi_0$ by deleting every $(\varphi,\psi_0)$-bubble $\sB$ such that $\Upsilon(\sB)$ intersects one of the $\Upsilon(\sB_i)$'s, replacing it by $\sB\cap\varphi$.
Note that a face $f \in \psi_0 \setminus \psi_1$ cannot be in $h^+$. Indeed, by construction, $\psi_0\in \mathbb{H}_{\hat\psi^+}^-$ but it also must be connected by a path $P$ in $\psi_0 \setminus \varphi$ to at least one face $f'$ with $\Upsilon(f')$ in some $\Upsilon(\sB_i)$ (it is part of a $(\varphi,\psi_0)$-bubble that was deleted, and cannot be a part of $\varphi$). As $\psi_0 \subset \psi^+ \cup \hat\psi^+ \subset\mathbb{H}_{\varphi}^+$ by \cref{clm:psi+cup-above-phi}, it follows that $\hat\psi^+ > \varphi$ along $\Upsilon(P)$, so $\Upsilon(P)\subseteq \Upsilon(\hat\psi^{+,\sqcup} \xor\hat\psi^{-,\sqcap})$ by \cref{clm:psi+cup-above-phi,clm:psi+cup-above-psicup}. By our definition of the bubble group, this means that $\Upsilon(f) \in \delta \cup \bigcup \Upsilon(\sB_i)$ for $\delta$ as per \cref{def:bubble-group} w.r.t.~$\hat h$; hence, $\Upsilon(f)$ cannot be in the projection of any other $(\hat h, \varphi)$-bubble that was not deleted, so $f\notin h^+$. Also, any face $f \in \psi_1 \setminus \psi_0$ has to project to $\Upsilon(\fB)$ (as we just established that $\Upsilon(\psi_0\setminus \psi_1)$ is a subset of $\delta \cup \bigcup \Upsilon(\sB_i)$),
where by construction $h^+ = \varphi$ so any such $f$ is in $h^+$.
Thus,
\begin{equation}\label{eq:psi0-h+}
|(\psi_1 \setminus \psi_0) \cap h^+| - |(\psi_0 \setminus \psi_1) \cap h^+|= |\psi_1 \setminus \psi_0| = |\psi_0 \setminus \psi_1| \geq   |\cC \setminus \varphi|  \geq |\cC\cap(\fB\setminus \varphi)|.
\end{equation}
Combining \cref{eq:C-h+-C'-h+,eq:psi0-h+}, we see that $|h^+ \cap \psi_1| >  |h^+ \cap \psi^+|$ which contradicts the fact that $\psi^+$ was a minimizer.
\end{proof}
We now infer \cref{it:monotone-relation}, and its importance for the local consistency of bubble groups---if $\hat h$ has two bubble groups $\fB\neq \fB'$ and we delete $\fB$, this should not alter (neither expand nor shatter)~$\fB'$:
\begin{corollary}\label{cor:mon-equiv-rel}
In the setting of \cref{clm:psi+-mon} one has $\Upsilon(\psi^{+,\sqcup} \xor \psi^{-,\sqcap}) \subseteq \Upsilon( \hat\psi^{+,\sqcup} \xor \hat\psi^{-,\sqcap})$. As a consequence of this, if $\fB'$ is a different (full) bubble group in $\hat h$, then it remains one in $h$.
\end{corollary}
\begin{proof}
 The fact $\Upsilon(\psi^{+,\sqcup} \xor \psi^{-,\sqcap}) \subseteq \Upsilon( \hat\psi^{+,\sqcup} \xor \hat\psi^{-,\sqcap})$ follows from \cref{clm:psi+-mon} in the same manner that we concluded \cref{cor:legal-equiv-rel} from its preceding claims, as $\hat\psi^{+,\sqcup},\hat\psi^{-,\sqcap}$ sandwich $\psi^{+,\sqcup},\psi^{-,\sqcap}$.
 
 This monotonicity implies, by \cref{def:bubble-group}, that $\fB'$ is either a single bubble group or is shattered into a collection of bubble groups in $h$ (deleting $\fB$ cannot cause $\fB'$ to reach out to a new bubble). We now argue that only the former can happen. If we look at the boundary of $\fB'$, there we had $ \hat\psi^{+,\sqcup}= \hat\psi^{-,\sqcap}$, and therefore also $\psi^{+,\sqcup}=\psi^{-,\sqcap}$ by the last claim. Furthermore on any face $f$ with $\Upsilon(f) \in \Upsilon(\fB_i')$ we must have $h^\pm= \hat h^\pm$. Overall, the $\hat \psi^\pm$ and $\psi^\pm$ must be solutions of the same optimization problem with the same boundary condition, which concludes the proof.
\end{proof}
We end this section by defining analogs of $H_{h}$, $G_{h,\varphi}$, $\sV$ for a bubble group $\fB=(\{\sB_i\},\{\cC_j\})$. Let
\begin{align*}
    \Upsilon(\fB):=\Upsilon\Big(\bigcup\{\sB\in\fB\} \Big)\cup \bigcup\{\cC_j\in\fB\}
\end{align*}
and let $H(\fB), G^\gr(\fB),Z_\mu^\infty(\fB), \sV(\fB)$ be the respective values of $H_h,G_{h,\varphi}^\gr, Z_\mu^\infty(h),\sV(h)$ restricted to the region $\Upsilon(\fB)$ in the triangular lattice $\T_N$. E.g.,
$H(\fB) $$= \sum_{\sB\in\fB} H(\sB)$;
$G^\gr(\fB)$ is the minimum of $G_{h,\varphi}$ over tilings $\psi$ of this region; and $Z_\mu^\infty(\fB)$ is the number of tilings $\psi$ achieving this minimum. 
With these definitions, if $\{\fB_i(h,\varphi)\}$ are the bubble groups of $h,\varphi$, then
\begin{equation}\label{eq:bubble-group-H}
\exp\Big[ -\hat\beta H_{h} + \alpha G_{h,\varphi}^\gr - \log Z_\mu^\infty(h)-\lambda\sV(h) \Big]=  \prod_{i} \exp  \Big[-\hat\beta H (\fB_i)  + \alpha G^\gr(\fB_i) - \log Z_\mu^\infty(\fB_i) -\lambda \sV(\fB_i)\Big]\,,
\end{equation}
the factorization per bubble group as described in \cref{obs:bubble-grp-relations}.

Note that for every bubble group $\fB=(\{\sB_i\},\{\cC_j\})$ we have $|\Upsilon(\fB)| \leq \sum_i |\sB_i|+\sum_j|\cC_j|$ as per \cref{def:bubble-group}, and $|\cC_j|=|\delta\cap\Upsilon^{-1}(\cC_j)|\leq|(\varphi\xor\psi^{+,\sqcup})\cap\Upsilon^{-1}(\cC_j)|+|(\varphi\xor\psi^{-,\sqcap})\cap\Upsilon^{-1}(\cC_j)|$ (above every triangle of $\cC_j$, the plaquettes of $\psi^{+,\sqcup}$ and of $\psi^{-,\sqcap}$ differ); thus, 
 \cref{clm:area-of-bubble-grp} shows that
\begin{equation}\label{eq:area-of-bubble-grp}
|\Upsilon(\fB)| \leq 5\sum_{i} |\sB_i|\,.
\end{equation}

Recall that our potential $\sV$ from \cref{eq:gen-potential} was defined as $\sum \ff(S_i)$, where the $S_i$'s are the connected components of $\psi^\sqcup \setminus h$. When $h,\varphi$ are such that $\fB$ is the only $(h,\varphi)$-bubble group (as in the context of $\sV(\fB)$), one has $\sum |S_i|=|\Upsilon(\psi^\sqcup\setminus h)\cap\Upsilon(\fB)|$. 
Take $\psi_0\equiv \psi^\sqcup$ and $h_0\equiv h$ on $\Upsilon(\fB)$ and $\psi_0\equiv h_0 \equiv \varphi$ elsewhere, to arrive at the following expression which will be useful later on:
\begin{align} 
 \sum |S_i| &=|\Upsilon(\psi^\sqcup \setminus h) \cap \Upsilon(\fB) | =
|\psi_0| - |h_0\cap \psi_0 | = (|h_0\cap\varphi|- |h_0\cap\psi_0|) + (|\varphi| - |h_0\cap\varphi|) \nonumber\\
&= G_{h_0,\varphi}^\gr + |\varphi\setminus h_0| 
=G^\gr(\fB) + \sum_{\sB\in\fB} |\varphi\cap \sB| \,. \label{eq:V(fB)-pinning-alt}
\end{align}

\begin{remark}
    \label{rem:bubble-groups-no-ce}
Let us emphasize that the compatibility condition between different bubble groups is a fairly complicated long-range interaction, as demonstrated in the delicate proof of \cref{cor:mon-equiv-rel}. (The simpler candidate for the definition of a bubble group $\fB$, taking $\delta := \psi^\sqcup\xor\psi^\sqcap$ in \cref{def:bubble-group}, would not enjoy said monotonicity, yet still would have had this complicated long-range interaction.)

This long-range interaction is part of what makes $\pi$ much more difficult to understand than $\nu$. Whereas bubbles $\sB$ were susceptible to cluster expansion analysis if the only constraint is compatibility (as in $\mu$, unlike $\nu$ which also has long-range interactions via the $\fg_r$ terms), this is no longer the case for bubble groups $\fB$. Compatibility of bubble groups is not a pairwise relation, and even if the interactions were only zero range (compatibility), we would not be able to fit it into a cluster expansion framework. Moreover, as we will see later in \cref{sec:bound_integral}, the long term interaction term $\int_\alpha^\infty\mu_{h,\hat\alpha}(\overline{G}_h)\d\hat\alpha$ present in $\pi$ is also substantially more difficult to control than its analogue for~$\nu$.
\end{remark}

\section{Deterministic 
approximation of height functions via tilings }\label{sec:alg}

In this section, we continue to study at a deterministic level the problem of optimizing $G_{h,\varphi}$, with the aim of finding bounds on $G^\gr$ on a bubble group $\fB$. This will be achieved by providing an explicit deterministic tiling for a given $h$ that will essentially act independently within each bubble.

\subsection{Algorithm for approximating a height function inside a single bubble}
Before providing the details of the construction, let us briefly discuss the overall strategy and some of the complications that make the full implementation of our approach fairly involved.

First, let us note that a simple way to make an algorithm generalize from the case of a single bubble $\sB$ to the general case is to only look for approximating tilings of the region $\sB \cap \varphi$ since by definition these regions are disjoint so this will be our a priori target. With this in mind, one can think of the problem as follows: given a region $S$ and a function $h$ whose boundary conditions are compatible with a tiling, find the tiling $\psi$ with the same boundary conditions maximizing $|\psi \cap h|$. A very naive approach to turn an arbitrary function into a monotone one is to consider the running minimum (or some variant thereof). This does not seem precise enough for our purpose since it is too easy for a small local defect to have a very large influence in such an approximation algorithm; it is also challenging to control the 3D geometry through the run of the algorithm. Instead, we will describe our functions as a sequence of level lines, arbitrary curves touching the boundary together with loops for \SOS and simple random walk paths for tilings (assume to help with the visualization that $\Z^2$ is drawn with a $45$ degree angle so that a tiling level line takes only \southeast and \northeast steps). Running the approximation greedily, level by level, then turns the problem into a 2D problem which is substantially more tractable. 

The simplest example is to approximate via a tiling a single self-avoiding level line which connects two boundary points of a simply connected domain. A natural approximation algorithm is the following straightforward greedy process starting from the \west endpoint:
\begin{itemize}
    \item if the SOS level line $\Gamma_h$ does a \southeast or \northeast step, then copy it to the output tiling level line $\Gamma_\psi$;
    \item if $\Gamma_h$ does a \southwest or \northwest step, wait until it re-enters a point accessible from your current position (i.e., a \east-facing quarter-plane) and let $\Gamma_\psi$ shortcut directly to that point;
    \item repeat this procedure until reaching the endpoint of $\Gamma_h$. 
\end{itemize}
A convenient feature of this algorithm is that, when its output is a poor approximation of $h$, the level lines of $h$ follow those of $\varphi$ except along few detours, each made of few nearly straight pieces, so that such $h$ are easy to enumerate. Also, it turns out that this algorithm will respect the ordering of level lines, so dealing with a single level will be essentially enough to treat the full (multi level) configuration. However, as presented, it is easy to create an example (see \cref{fig:alg}) where $\Gamma_\psi$ does not end at the same point as $\Gamma_h$, which would create issues when trying to fit the approximation $\psi$ associated to a given bubble into the larger tiling $\varphi$, hence one must include an extra condition to make sure that the ``correct'' endpoint always stays accessible as we draw $\Gamma_\psi$ (see the definition of $\Gamma_h^\reg$ in the proof). 

In a simply connected domain, this procedure is relatively simple to analyze. However, the presence of holes in the domain is a major difficulty (and is truly needed for our application, which bundles together disconnected bubbles, potentially distant from one another, through the bubble group criterion).
Indeed, suppose that $\Gamma_h$ goes through two holes; then $\Gamma_\psi$ will have three separate paths to process, which need to be used in a fixed order \east to \west. However, there is actually no requirement for $\Gamma_h$ to go through the holes in the same order (see \cref{fig:alg-with-holes} for examples of ``bad'' $\Gamma_h$ configurations). In such situations, the behavior of the greedy algorithm becomes much more complicated, and a large portion of the proof is dedicated to the control of these cases.

Finally, we will need to take into account the possible presence of loops. Those that do not affect the boundary conditions can actually be safely ignored; however, it is possible for a loop to surround a hole, thus modifying the boundary conditions on that hole. Unfortunately, it turns out that there are examples where even very small loops can create boundary conditions so constrained that they can only be satisfied by a single tiling. To deal with such scenarios, one must be willing to sacrifice a part of the boundary condition constraint (which does create compatibility issues between the approximation provided for different bubbles---a point that we had mentioned above as something we would want to avoid---but in a manageable way), choosing which of the loops to preserve and which to ignore.

\begin{proposition}%
\label{prop:alg}
Let $\varphi$ be a tiling of the full plane. Associate to $\varphi$ a height function by pinning an arbitrary face to $0$. Let $\cS$ be a connected set of $s$ faces of $\varphi$ and consider the set $\cH_\cS$ of \SOS height functions $h$ with a single bubble with respect to $\varphi$ with floor $\cS$, i.e., which contain every tile of $\varphi \cap \cS^c$ but have no intersection with $\varphi \cap \cS$. For $\epsilon>0$, 
let $\cH^{\mathrm{ex}}_\cS$ be the set of $h\in\cH_\cS$ such that 
\[ |h|-|\varphi|<\epsilon s \quad\mbox{and}\quad \nexists \mbox{ a tiling $\psi$ of $\cS'\supset \cS$ with }\quad\left\{\begin{array}{l}\psi = \varphi\mbox{ on }\partial \cS'\\ |\psi\cap h| - |\varphi \cap h| \geq \epsilon^{1/3} s\\ |\cS' \setminus \cS| \leq \epsilon^{2/3}s\end{array}\right.\,.\] There exists an absolute constant $C>0$ such that, for every $\epsilon\in(0,\frac1C)$ and every face $o$ of $\varphi$, taking the union over all connected sets $\cS\ni o$ of $s$ faces of $\varphi$, we have
\[\Big|\bigcup_{\cS\ni o}\cH^{\mathrm{ex}}_\cS\Big|\leq\exp\big[C\epsilon^{1/3}\log(1/\epsilon)\,s\big]\,.\]
\end{proposition}
Note that in the above proposition, since $h, \varphi, \psi$ only differ on finitely many tiles, the size differences $|h| - |\varphi|$ and $|\psi\cap h| - |\varphi \cap h| $ are well defined.
\begin{remark}
    The powers of $\epsilon$ in the proposition are non-optimal, but apart from that the statement is almost optimal. It is necessary to consider a slight enlargement $\cS'$ of the domain since one can construct domains $\cS$ which admit only a single tiling $\varphi$ but where there is a positive entropy for \SOS configurations $h$ with $\varphi \cap h = \emptyset$ and $|h| - |\varphi| = o(s)$ (consider a hexagon with a slit corresponding a completely filled lowest column). It is also possible to create sets $\cS$ with families of at least $e^{c \epsilon |\log \epsilon|s}$ many bubbles with $|h| - |\varphi| \leq \epsilon s$ for which the best approximations still satisfy $|h \cap \psi| \leq C\epsilon s$.
\end{remark}

\begin{proof}

We first describe an algorithm which constructs, for every $h\in\cH_\cS$, a tiling approximating~$h$.

For this proof, we write $S = \Upsilon_{001}(\cS)$, $S' = \Upsilon_{001}( \cS')$ and we consider the following coordinates and directions on $\cP_{001}$. We say that the direction $x_1$ is South-West (\southwest) while $x_2$ is South-East (\southeast) and the other directions (\northeast,\northwest) are derived from them. Recall that for every $k\in \Z$, the height-$k$ level set is the collection of dual edges $xy$ such that $h(x)<k$ and $h(y)\geq k$. One can use the \northeast splitting rule to turn it into a collection of self-avoiding loops and lines (starting and ending at the boundary, where the boundary points can be either external or internal). 

We first assume that there are no loops in this decomposition, in which case $S'=S$, and treat loops afterwards.

Note that a level line in a tiling is described as a walk on $\Z^2$ using only \northeast and \southeast moves while a level line of an \SOS function is a self-avoiding path allowed to use all 4 directions. Note however that since the boundary conditions are compatible with $\varphi$ (including for the holes), the ending point of any level line in $h$ is accessible using \northeast and \southeast moves from its starting point.
Let $D$ be the simply connected closure of $S$ (adding said holes). The height-$k$ level set of $h$ agrees with that of $\varphi$ outside $D$, so its decomposition contains at most one bi-infinite path; we call \emph{height-$k$ level line} of $h$ the part of this path between its first and its last visit to $D$, and we define in the same way the level lines of the other height functions below, all of which agree with $\varphi$ outside $S$. In particular, all the height-$k$ level lines have the same endpoints.
For each level line $\Gamma_h$ of $h$, parameterize its edges as $\Gamma_h(t)$.

\begin{figure}
\vspace{-0.1in}
    \begin{tikzpicture}
    \begin{scope}[scale=0.6]

    \coordinate (T1) at (19,0.5);

    \filldraw[fill=red!5,draw=none] (13,3.5)--(T1)--(13,-2.5)--(13,3.5);
    \draw[red!50,dashed] (13,3.5)--(T1)--(13,-2.5);
    \node[red!80!black] at (18,-0.5) {$\cQ_\west$};
    
    \draw [line width=3pt,green] (0,0)
            \foreach \y in {1,-1,-1,-1,1,1,1,1,-1,-1,-1,1,1,1,1,1,-1,-1,-1} {
        -- ++(1,\y*0.5)
        };
    \draw [thick,blue] (0,0)
            \foreach \x/\y in {1/1,1/-1,-1/-1,1/-1,1/-1,1/1,1/1,1/1,-1/1,1/1,1/-1,1/1,1/-1,1/1,1/-1,-1/-1,1/-1,1/1,1/1,1/1,-1/1,1/1,1/1,1/1,1/-1,1/-1,1/-1,-1/-1,1/-1} {
        -- ++(\x,\y*0.5)
        };

    \draw [very thick,cyan,decorate, decoration=bumps] (15,2.5)--(18,1);

    \node[circle,scale=0.4,fill=gray] (o) at (0,0) {};
    \node[circle,scale=0.4,fill=gray] at (T1) {};

    \node[circle,scale=0.4,fill=gray,label={[label distance=0pt]above:{\small$\mathsf a_1$}}] (a1) at (2,0) {};
    \node[circle,scale=0.4,fill=gray,label={[label distance=0pt]above:{\small$\mathsf b_1$}}] (b1) at (4,-1) {};
    \node[circle,scale=0.4,fill=gray,label={[label distance=0pt]below:{\small$\mathsf a_2$}}] (a2) at (6,0) {};
    \node[circle,scale=0.4,fill=gray,label={[label distance=0pt]above:{\small$\mathsf b_2$}}] (b2) at (7,0.5) {};
    \node[circle,scale=0.4,fill=gray,label={[label distance=0pt]above:{\small$\mathsf a_3$}}] (a3) at (9,0.5) {};
    \node[circle,scale=0.4,fill=gray,label={[label distance=-1pt]below:{\small$\mathsf b_3$}}] (b3) at (10,0) {};
    \node[circle,scale=0.4,fill=gray,label={[label distance=0pt]below:{\small$\mathsf a_4$}}] (a4) at (14,1) {};
    \node[circle,scale=0.4,fill=gray,label={[label distance=0pt]below:{\small$\mathsf b_4$}}] (b4) at (16,2) {};
   
    \node[green!80!black,font=\Large] at (11.25,-1.25) {$\Gamma_\psi$};
    \node[blue!80!black,font=\Large] at (10.75,1.5) {$\Gamma_h$};
    
    \end{scope}
\end{tikzpicture}
\vspace{-0.15in}\caption{Illustration of the algorithm to approximate the height function $h$ by a tiling~$\psi$. There are no holes in this example, so the quarter-plane $\cQ_{\west}$ is the right-hand side of the domain between the level lines $\Gamma_{\min}$ and $\Gamma_{\max}$ that bound $\Gamma_h$. The wavy line marks $\Gamma_h^\reg \setminus \Gamma_h$.
}
\label{fig:alg}
\vspace{-0.15in}
\end{figure}
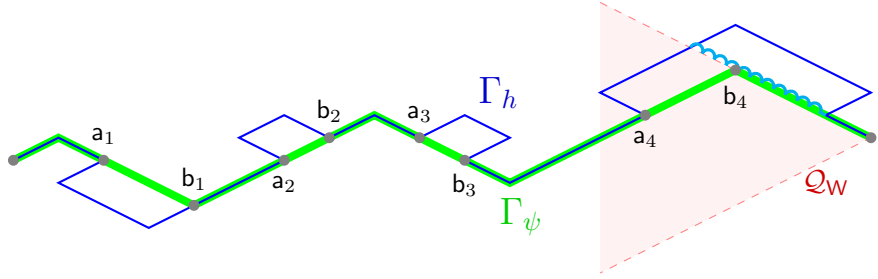

Consider the following approximation algorithm. Our goal will be to control the number of initial height functions $h$ where it behaves ``badly'' and only results in a $\psi$ with $|h \cap \psi| <\epsilon^{1/3} s $. %
 Let $\psi_{\max}$ and $\psi_{\min}$ be the maximal and minimal monotone surfaces compatible with the boundary condition of $S$ (i.e., the maximal and minimal tilings whose heights agree with $\varphi$ outside of $\cS$), which are well-defined since (even in a non-simply connected setting) monotonicity is closed under taking a maximum or minimum. 
 We perform two ``preprocessing'' steps: first, we replace $h$ by $(h \wedge \psi_{\max}) \vee \psi_{\min}$; second, we turn all level sets into level lines and loops using the \northeast splitting rule and then erase all the loops. Denote the resulting ``regularized'' \SOS function by $h^\reg$ and its height-$k$ level lines by $\Gamma^\reg_{h,k}$ ($k\in\Z$). The algorithm will process the $\Gamma^\reg_{h,k}$ one by one, sorted by heights from the highest to the lowest. Define the ``West facing'' and ``East facing'' quarter-planes \[\cQ_\west(x) = x+\{u_1 \geq 0,\,u_2\leq 0 \}\quad,\quad 
\cQ_\east(x) = x+\{u_1 \leq 0,\,u_2 \geq 0\}\,,
\]
and do as follows (we describe the procedure for a single $k$, suppressing its index for brevity):
\begin{enumerate}
[label=\textbf{Step~\arabic*}:, ref=\arabic*, wide=0pt, itemsep=1ex]
\addtocounter{enumi}{-1}
\item\label[step]{it:alg-init} Set $t_0 = 0$ and $i=1$.
    \item \label[step]{it:alg-iter} Look along $\Gamma_h^\reg$ until the start of the first \southwest or \northwest step. Call that time $t'_i$ and let $\mathsf a_i = \Gamma_h^\reg(t'_i)$. We let $\Gamma_\psi[0, t'_i] = \Gamma_h^\reg[0, t'_i]$.%
    \item \label[step]{it:alg-rentry} Let $t_{i}$ be the first time after $t'_i$ at which $\Gamma_h^\reg$ lies in $\cQ_\east(\mathsf a_i)$ (it exists, as $\Gamma_h^\reg$ lies between two monotone paths with its endpoints by \cref{claim:def_gammahr} below, so that its endpoint lies in $\cQ_\east(\mathsf a)$ for every point $\mathsf a$ of $\Gamma_h^\reg$), and $\mathsf b_i = \Gamma_h^\reg(t_{i})$, which lies on one of the two half-lines bounding $\cQ_\east(\mathsf a_i)$ since $\Gamma_h^\reg(t_i-1)\notin\cQ_\east(\mathsf a_i)$. %
    Add to $\Gamma_\psi$ the line $[\mathsf a_i, \mathsf b_i]$. 
    \item \label[step]{it:alg-repeat} Increment $i$ by $1$ and return to \cref{it:alg-iter}, until exhausting $\Gamma_h^\reg$.
    \end{enumerate}
(See \cref{fig:alg,fig:alg-with-holes} illustrating the algorithm.) 

We emphasize that while running \cref{it:alg-iter,it:alg-rentry,it:alg-repeat}, we ignore the underlying domain. We have three elements to prove: (1) the algorithm above produces a tiling; (2) this tiling respects the boundary condition; and (3) we can appropriately bound the number of functions $h$ where the algorithm fails to construct a tiling that has a large overlap with~$h$.

\begin{figure}
    \begin{tikzpicture}
    \pgfmathsetmacro{\arwlen}{0.33}

    \begin{scope}[scale=0.9]

    \coordinate(o) at (0,0);
    \coordinate (T) at (9.5,1);
    \coordinate (a) at (2.2,1.8);    
    \coordinate (b) at (3,1);
    \coordinate (c) at (6.25,1.25);
    \coordinate (d) at (7,2);

    \begin{scope}[fill=gray!10]
    \fill[clip]plot [smooth cycle] coordinates {(o) (0.75,2.5) (2,2.5) (5,3)  (T) 
    (7,-1.) (5,-1.5) (3,-1) (2,-1.25)} ;
    \fill[green!15] (1,3)--(4,0.)--(7,3)--cycle;
    \fill[green!15] (3,3)--(5.5,0.5)--(8,3)--cycle;
    \end{scope}
    \draw plot [smooth cycle] coordinates {(o) (0.75,2.5) (2,2.5) (5,3)  (T) 
    (7,-1.) (5,-1.5) (3,-1) (2,-1.25)} ;
    
    \filldraw[red!10] (a)--($(a)+(\arwlen,\arwlen)$) to[bend left=25] ($(d)+(-\arwlen,\arwlen)$)--(d) %
    --(c)--($(c)+(-\arwlen,\arwlen)$) to[bend right=45]($(b)+(\arwlen,\arwlen)$) --(b)--cycle;
        \node[circle,scale=1.75,fill=green!50!black] (h1) at (4,0.5) {};
    \draw[dashed,green!50!black] ($(h1.south)+(0,-0.15)$) --  +(-2.5,2.5);
    \draw[dashed,green!50!black] ($(h1.south)+(0,-0.15)$)--  +(2.5,2.5);
    \node[circle,scale=1.75,fill=green!50!black] (h2) at (5.5,1) {};
    \draw[dashed,green!50!black] ($(h2.south)+(0,-0.15)$) --  +(-2.3,2.3);
    \draw[dashed,green!50!black] ($(h2.south)+(0,-0.15)$)--  +(1.8,1.8);

    \draw[thick,-stealth] (o)->($(o)+(\arwlen,0)$);

    \draw[thick,-stealth] ($(T)-(\arwlen,0)$)->(T);

    \draw[thick, blue] ($(o)+(\arwlen,0)$) to[bend left=10] (a);
    
    \draw[thick, red] ($(a)+(\arwlen,\arwlen)$) to[bend left=25] ($(d)+(-\arwlen,\arwlen)$);

    \draw[thick, red] (d) to[bend left=100] (c);
    
    \draw[thick, red] ($(c)+(-\arwlen,\arwlen)$) to[bend right=45]($(b)+(\arwlen,\arwlen)$);

    \draw[thick, blue] (b) to[bend right=55] (3.5,-0.5) to [bend right=35] ($(T)+(-\arwlen,0)$);

    \draw [thick,cyan,decorate, decoration=bumps] (a)--(b);
    \draw [thick, cyan,decorate, decoration={bumps,amplitude=0.5mm,raise=1pt}] (c)--(d) to[bend left=100] (c);

        \draw[thick,-stealth] (a)->($(a)+(\arwlen,\arwlen)$);

    \draw[thick,-stealth] ($(d)+(-\arwlen,\arwlen)$)->(d);

    \draw[thick,-stealth] (c)->($(c)+(-\arwlen,\arwlen)$);

    \draw[thick,-stealth] ($(b)+(\arwlen,\arwlen)$)->(b);

    \node[circle,scale=0.4,fill=gray] at (o) {};
   
    \node[circle,scale=0.4,fill=gray] at (T)  {}; 

    \node[circle,scale=0.4,fill=gray,label={[label distance=0pt]above:{\small$\mathsf a$}}] at (a){};
    
    \node[circle,scale=0.4,fill=gray,label={[label distance=0pt]above:{\small$\mathsf b$}}] at (b) {};
    
    \node[circle,scale=0.4,fill=gray,label={[label distance=0pt]below:{\small$\mathsf c$}}] at (c) {};
    
    \node[circle,scale=0.4,fill=gray,label={[label distance=0pt]right:{\small$\mathsf d$}}] at (d) {};

    \end{scope}
\end{tikzpicture}
    \caption{Illustration of the effect of holes (in dark green) on the level lines. Since both holes are North of the level line, no tiling level line satisfying the boundary condition can enter the light green region. The preprocessing replaces the contour from $\mathsf a$ to $\mathsf b$ by a line, plus a loop around $\mathsf c,\mathsf d$ that is deleted from~$\Gamma_h^\reg$.}
    \label{fig:alg-with-holes}
    \vspace{-0.1in}
\end{figure}
Let $\succeq$ be the partial order on level lines defined as follows. A level line $\Gamma$ at height $k$, completed by the two infinite parts of the height-$k$ level line of $\varphi$ beyond its endpoints, is a bi-infinite path which divides the plane into a part to its left and a part to its right; we orient it from West to East and call these parts North and South. We identify $\Gamma$ with the indicator function of the North component and we let $\succeq$ be inherited from the usual partial order on functions. Notice that, thanks to this identification, one can also define the maximum and minimum between level lines by applying the operation on the associated functions but in general doing so may create loops. Finally, if $\Gamma$, $\Gamma'$ are the level lines of two tilings with the same boundary points, then they can be viewed as functions of $x_2-x_1$, whence the relation $\Gamma \succeq \Gamma'$ is equivalent to $\Gamma\leq \Gamma'$ for the usual partial order on functions.

\begin{claim}\label{claim:level_line_order}
In each level line $\Gamma$ of $\psi_{\max}$, no connected component of $\Gamma \cap \mathring{S}$ (where $\mathring S=S\setminus \partial S$) has a \northeast step immediately followed by a \southeast one, and similarly, no level line $\Gamma $ of $\psi_{\min}$ admits a \southeast step immediately followed by a \northeast step. Finally, if $\Gamma_{\max}$ and $\Gamma_{\min}$ are level lines bounding the same level, then the set of faces of $S$ between $\Gamma_{\max}$ and $\Gamma_{\min}$ is simply connected.
\end{claim}
\begin{proof}
Fix $\Gamma$ a level line of $\psi_{\max}$, assume by contradiction that there is a \northeast step followed by a \southeast step above a face $f \in S$, and let $\Gamma'$ be the lowest level line using such steps above $f$. Flipping these two steps in $\Gamma'$ (to \southeast followed by \northeast) changes only the height at $f$---increasing it by $1$---thus still satisfies the boundary condition but contradicts the maximality of $\psi_{\max}$. Finally, the last statement follows from the fact that the set of monotone functions with given boundary condition is connected by flips as above. Indeed, fix $\psi \geq \psi'$ two tilings with the same (possibly non-simply connected) boundary condition. Viewing them both as tilings, let $v \in \T$ such that $\psi_{111}(v) > \psi'_{111}(v)$, move along the directions of the projection of the standard basis without crossing a tile of $\psi$ until we are stuck at a point $v'$. Along the path, it is easy to check that $\psi - \psi'$ is non-decreasing so $\psi_{111}(v') > \psi'_{111}(v')$ and $v'$ is not in a hole while by construction a flip is possible at the point $v'$.
\end{proof}

\begin{claim}\label{claim:def_gammahr}
    Let $h$ be an \SOS height function and $h^\reg := (h \wedge \psi_{\max}) \vee \psi_{\min}$. Let $\Gamma_h, \Gamma_h^\reg, \Gamma_{\max}, \Gamma_{\min}$ be the level lines associated to some level. Then
    $\Gamma_h^\reg = (\Gamma_h \wedge \Gamma_{\max}) \vee \Gamma_{\min}$. In particular, $\Gamma_h^\reg$ stays in the closed domain between $\Gamma_{\max}$ and $\Gamma_{\min}$.
\end{claim}
\begin{proof}
    For this proof, we associate a level line with a $\{0,1\}$ value function as above. Fix a level $\ell$ and denote by $\Gamma_h$ the level line separating levels $\ell$ and $\ell+1$. Note that for any face $v$, we have $\Gamma_h(v) = \one_{\{ h(v) > \ell\}} $. With this notation, we see that
    \begin{align*}
    \one_{\{(h \wedge \psi_{\max}) \vee \psi_{\min} > \ell\}} & = \one_{\{(h \wedge \psi_{\max}) > \ell\}} \vee \one_{\{\psi_{\min} > \ell\}}\\
    & = (\one_{\{h > \ell\}} \wedge \one_{\{\psi_{\max} > \ell\}}) \vee \one_{\{\psi_{\min} > \ell\}}\, ,
    \end{align*}
    as desired.
\end{proof}

\begin{claim}\label{clm:se-followed-by-ne}
Let $\mathsf a$ be a point of $\Gamma_{\min}$ preceded by a \southeast step and followed by a \northeast step. The \southeast and \southwest oriented half-lines starting from $\mathsf a + \epsilon e_{\uparrow}$ are crossed an odd number of times (in particular at least once) by $\Gamma_h$. The symmetric statement holds for $\Gamma_{\max}$.
\end{claim}
\begin{proof}
 By \cref{claim:level_line_order}, $\mathsf a$ must be a boundary point and, since $\mathsf a$ is on $\Gamma_{\min}$, we have $\Gamma_{\min}(\mathsf a + \epsilon e_\uparrow) = 1$ for all $\epsilon \in (0,1)$. By definition, $\Gamma_h$ and $\Gamma_{\min}$ have the same boundary condition when they are both viewed as functions from $D$ to $\{0,1\}$ so this must also be the case for $\Gamma_h$. On the other hand, since $\mathsf a$ is followed by a $\northeast$ step, the half line starting from $\mathsf a + \epsilon e_\uparrow$ oriented \southwest crosses $\Gamma_{\min}$ once so the boundary condition at the first point where it intersects $\partial D$ (call it $\mathsf b^\epsilon$) must be a $0$. Since $\Gamma_h(\mathsf a + \epsilon e_\uparrow) = 1$ but $\Gamma_h(\mathsf b^\epsilon) = 0$, $\Gamma_h$ must cross the line from $\mathsf a + \epsilon e_\uparrow$ to $\mathsf b^\epsilon$ an odd number of times.
\end{proof}

We say that $t$ is an agreement time for $\Gamma_h^\reg$ if at that time the algorithm is in \cref{it:alg-iter} or if it is at one of the endpoints $t_i', t_i$. The other times form the \emph{excursions} $(t'_i,t_i)$.
\begin{claim}\label{clm:QE-non-intersecting}
    If $s\leq t$ are two agreement times for $\Gamma_h^\reg$, then $\cQ_\east(\Gamma^\reg_h(t)) \subseteq \cQ_\east(\Gamma_h^\reg(s))$ and $\Gamma_h^\reg[0, t)$ does not intersect $\cQ_\east(\Gamma_h^\reg(t))$.
\end{claim}
\begin{proof}
    By definition of agreement times there exists $\tilde s, \tilde t$ such that $\Gamma_h^\reg(s) = \Gamma_\psi(\tilde s)$ and $\Gamma_h^\reg(t) =\Gamma_\psi(\tilde t)$ and by construction one must have $\tilde s \leq \tilde t$. Since $\Gamma_\psi$ only moves in the \southeast and \northeast directions, $\Gamma_h^\reg(t)\in\cQ_\east(\Gamma_h^\reg(s))$, whence $\cQ_\east(\Gamma_h^\reg(t)) \subseteq \cQ_\east(\Gamma_h^\reg(s))$. If $s<t$, then $\tilde s<\tilde t$, so $\Gamma_h^\reg(s)$ differs from $\Gamma_h^\reg(t)$ and lies in $\cQ_\west(\Gamma_h^\reg(t))$, whose only common point with $\cQ_\east(\Gamma_h^\reg(t))$ is $\Gamma_h^\reg(t)$. Any other time $s<t$ belongs to an excursion $(t'_i, t_i)$ with $t_i\leq t$, and then $\Gamma_h^\reg(s) \notin \cQ_\east(\mathsf a_i) \supseteq \cQ_\east(\Gamma_h^\reg(t))$ by the definition of $t_i$ and the first assertion.  
\end{proof}

\begin{claim}\label{clm:alg-image-monotone}
    The greedy map on level lines preserves $\succeq$ and fixes monotone paths. In particular, the image of $h$ under the algorithm is a monotone function.
\end{claim}
\begin{proof}
    Monotone paths are copied unchanged. Fix two ordered level lines $\overline\Gamma_h^\reg\succeq\underline\Gamma_h^\reg$ and let $\underline\Gamma_\psi,\overline\Gamma_\psi$ be their images. We show that $\overline\Gamma_\psi\succeq\underline\Gamma_\psi$.

 We assume by contradiction that $\overline\Gamma_\psi\nsucceq\underline\Gamma_\psi$ and let $u$ be the first time where they agree but $\overline\Gamma_\psi$ does a \southeast step while $\underline\Gamma_\psi$ does a \northeast step. Since $\overline\Gamma_h^\reg\succeq\underline\Gamma_h^\reg$, it cannot be the case that both $\overline\Gamma_h^\reg$ and $\underline\Gamma_h^\reg$ agree with $\overline\Gamma_\psi, \underline \Gamma_\psi$ immediately after $u$.

If $u$ is in an excursion in both, writing their indexes $i,j$ respectively, by the claim $\overline\Gamma_h^\reg[0, t'_i]$ does not enter $\cQ_\east(\overline \Gamma_h^\reg(t'_i))$ and  $\underline\Gamma_h^\reg[0, t'_j]$ does not enter $\cQ_\east(\underline \Gamma_h^\reg(t'_j))$ and therefore by construction $\overline\Gamma_h^\reg[0, t_i]$ connects $\partial D$ to the \southeast oriented part of $\cQ_\east(u)$ while $\underline\Gamma_h^\reg$ connects $\partial D$ to the \northeast oriented part, both staying in $D \setminus \cQ_\east(u)$ (in fact here, we can even replace $D$ by a large enough ball around $u$, extending the level lines periodically if necessary). Combining this observation with the fact that $\overline\Gamma_h^\reg\succeq\underline\Gamma_h^\reg$, going along $\partial( D \setminus \cQ_\east(u))$ in positive order, one must encounter $\underline\Gamma_h^\reg(0), \overline\Gamma_h^\reg(0), \underline\Gamma_h^\reg(t_j), \overline\Gamma_h^\reg(t_i)$ in that order but this contradicts the assumption that they do not intersect.
If $u$ is in an excursion only say in $\overline\Gamma_h^\reg$, we apply the same argument using $\overline\Gamma_h^\reg[0, t_i]$ and $\underline\Gamma_h^\reg$ up to the end of the \southeast step after $u$.
\end{proof}

\begin{claim}\label{clm:alg-image-sat-bc}
    The image of $h$ by the algorithm satisfies the boundary condition.
\end{claim}
\begin{proof}
By \cref{claim:def_gammahr}, $\Gamma_{\max}\succeq\Gamma_h^\reg\succeq\Gamma_{\min}$ at each level. Applying \cref{clm:alg-image-monotone}, which fixes $\Gamma_{\max}$ and $\Gamma_{\min}$, gives $\Gamma_{\max}\succeq\Gamma_\psi\succeq\Gamma_{\min}$. Thus $\psi_{\max}\geq\psi\geq\psi_{\min}$, and all three agree with $\varphi$ on the boundary.
\end{proof}

We now turn to the case where $h$ has loops. The algorithm ignores the loops which neither touch the boundary nor surround a hole, as they do not affect the boundary condition.

\begin{figure}
    \centering
    \includegraphics[width=0.3\linewidth]{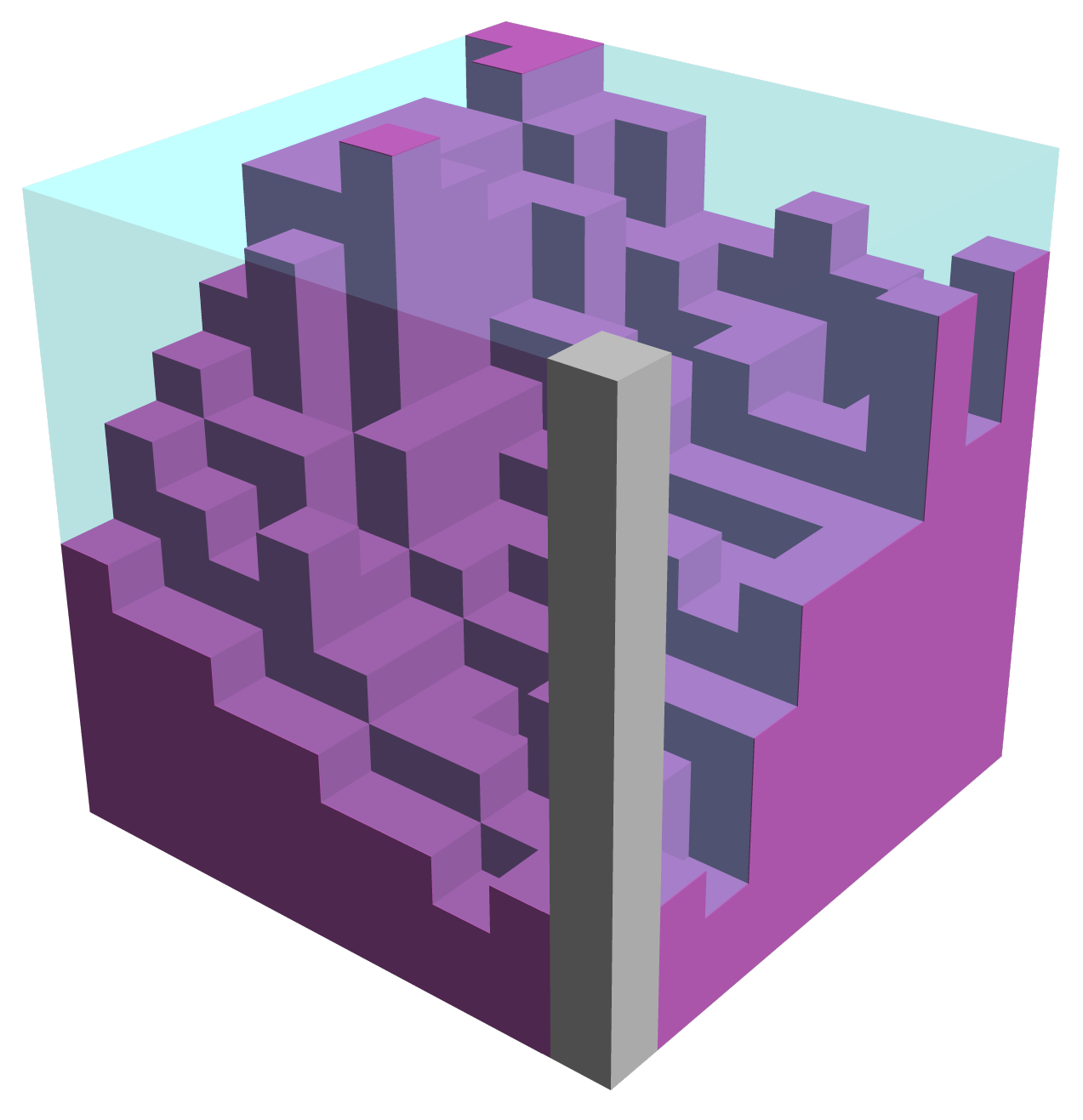}
    \caption{A configuration where loops affect the boundary condition: $\varphi$ is the completely filled cube represented with transparent light blue and $h$ is the solid configuration. Note that the presence of the gray ``tower'' in $h$ means that $\varphi$ is the only tiling satisfying the boundary condition so in order to have a nontrivial output in the algorithm we need to discard that tower. To the left of the gray tower, another one reaches the height of $\varphi$ creating a hole in the domain $S$ which would also be removed in this example but could be preserved in $h_1$ if its area was large enough compared to its height.}
    \label{fig:loop_example}
\end{figure}

When a loop touches the boundary or surrounds a hole however, ignoring it in the algorithm means that the resulting $\psi$ will not satisfy the correct boundary condition, while taking it into account may leave very few tilings satisfying the boundary condition (see the example in \cref{fig:loop_example}). One therefore needs to carefully select which loops to take into account and which loops to ignore, creating the superset $S'$ from the proposition. 

More precisely, we now consider simultaneously all level lines as oriented (say the lower side of the level line is always to its left) paths and loops, each of which is marked by a level. Note that several level lines can share an edge or even be identical except for their level (see, for example, the gray tower in \cref{fig:loop_example}). We will use the nesting partial order on loops and we note that it extends canonically even to identical loops as follows: for a set of identical positively oriented loops, we say they are ordered from the lowest level to the highest and vice versa for negatively oriented loops (a given edge is never occupied by loops with opposite orientation). Let us note that, while it is useful for the ordering to think of loops as marked with a level, the levels are determined by the rest of the configuration (including the orientations) so we will not need to enumerate over them.

From now on, we only consider the loops which touch the boundary or surround a hole. For each loop $\varrho$, we let $a(\varrho)$ be the area of $\varrho$ and $H(\varrho)$ be the total length of all loops inside $\varrho$, including $\varrho$ itself. To select which loop to keep, we proceed inductively starting from the innermost ones and erasing $\varrho$ together with all the loops inside of it if $H(\varrho) \geq \epsilon^{1/3} a(\varrho) $ (updating the values of $H$ as we erase loops). 

Recall that $S = \Upsilon_{001}(\cS)$.
We then define $\cS'$ as follows, the idea being to add to $\cS$ the interior and support of all erased loops. Any hole of $\cS$ whose projection was surrounded by an erased loop is added to $\cS'$. For any erased loop touching the boundary of $S$, note that each edge of that loop outside of $S$ must be also part of a level line of $\varphi$ and therefore has a well defined pre-image in $\varphi$ (it corresponds to a vertical tile of $\varphi$ outside of $\cS$). In fact, the pre-image of every excursion of a loop away from $S$ must be a path of vertical tiles in $\varphi$ starting and ending on $\partial \cS$ and we add to $\cS'$ all these paths together with the regions between them and $S$. Clearly, $\cS'$ is still a connected set of tiles of $\varphi$ and by construction $|\cS' \setminus \cS| \leq \epsilon^{2/3} s$. As before, we write $S' = \Upsilon_{001}( \cS')$.

We first apply the above construction ignoring all the (remaining) loops, let $h_0$ be the corresponding \SOS height function and $\psi_0$ be the tiling constructed in this way. First, note that any loop in $h$ must be either completely above or completely below $h_0$, since the level lines of $h_0$ do not cross the loops. 
We will first treat all the loops above $h_0$, and then do the same for the ones below, but from now on we focus on the first step. Let $h_1$ be the configuration obtained by adding all these loops to $h_0$ and let $\psi_1$ be the lowest tiling satisfying the boundary condition imposed by $h_1$ and with $\psi_1 \geq \psi_0$.

We sort the innermost loops $\varrho_i$ by the value of $h_1$ inside of them, breaking ties arbitrarily. Next, let $\psi_1^{(i)}$ be defined inductively by enforcing the constraint $\psi_1^{(i)} = h_1$ on the intersection of $\partial S'$ and the interior of $\varrho_1, \ldots, \varrho_i$, i.e., the lowest tiling larger than $\psi_0$ and satisfying these constraints. The $\psi_1^{(i)}$'s form an increasing sequence of tilings and we let $A_i$ be the number of horizontal tiles which differ between $\psi_1^{(i-1)}$ and $\psi_1^{(i)}$.

\begin{claim}\label{claim:area_move}
    We have $\sum_i A_i \leq s$.
\end{claim}
\begin{proof}
    By construction, $\psi^{(i)}_1$ is obtained from $\psi^{(i-1)}_1$ by setting $\psi^{(i)}_1 = h(\varrho_i)$ for some $x \in \cP_{001}$ where we had $\psi^{(i)}_1 < h(\varrho_i)$. Since the $h(\varrho_i)$ are non-increasing, each $x$ can appear in at most one $A_i$. 
\end{proof}

We then enforce the constraints associated to all the other loops to get $\psi_1$, and we let $\psi_2$ be the tiling obtained in the same way from $\psi_1$ and the loops below $h_0$; it is a tiling of $\cS'$ with $\psi_2=\varphi$ outside $\cS'$. We now describe how the level lines move, level by level. In each level, for each loop $\zeta$ we choose an innermost loop $\varrho_{i(\zeta)}$ contained in it (or $i(\zeta) = i$ if $\zeta = \varrho_i$ is already an innermost loop) and we extend the order between the $\varrho_i$ to that level. We take the loops into account one at a time in that order, and we let $\Gamma^{(0)}$ be the level line of $\psi_0$ and $\Gamma^{(j)}$ the level line after taking into account the constraints of the first $j$ loops; these lines are non-decreasing, and the last one is the level line of $\psi_1$.

\begin{figure}
    \centering
  \begin{tikzpicture}
    \pgfmathsetmacro{\arwlen}{0.33}

    \coordinate (BSW1) at (-0.36,0.65);
    \coordinate (BSE1) at (1.98,0.25);
    \coordinate (RSW1) at (0.45,.51);
    \coordinate (RSE1) at (1.45,0.67);

    \coordinate (BSW2) at ($(BSW1)+(-3,3)$);
    \coordinate (BSE2) at ($(BSE1)+(3.4,3.4)$);
    \coordinate (RSW2) at ($(RSW1)+(-3.14,3.14)$);
    \coordinate (RSE2) at ($(RSE1)+(2.98,2.98)$);

    \coordinate (RW2) at (0.45,4);
    \coordinate (RE2) at (1.45,4.1);

    \draw[line width=2.5pt, purple!20] (BSW2)--(BSW1);
    \draw[line width=2.5pt, purple!20] plot [smooth] coordinates { (-0.38,0.66) (0,0.5) (0.5,0.3) (1,-0.02) (1.9,0.22) (2,0.34) };
    \draw[line width=2.5pt, purple!20] (BSE2)--(BSE1);

    \draw[blue] (BSW1)--(BSW2);
    \draw[blue] (BSE1)--(BSE2);

    \draw[blue,thick,fill=blue!10] plot [smooth cycle] coordinates {(0,0.5) (-0.4,0.75) (0,1.5) (1.2,1.2) (2,1.5) (2,0.3) (1,0) (0.5,0.3) };

    \draw[dashed,blue] (BSW2) -- +(-1,1);
    \draw[dashed,blue] (BSE2) -- +(1.,1);

    \draw[dashed, red] (RSW1)--(RSW2) -- +(-1.,1.);
    \draw[dashed, red] (RSE1)--(RSE2) -- +(1.,1.);
    \draw[dashed, red] (RSW1)-- +(0,4.14);
    \draw[dashed, red] (RSE1)-- +(0,3.98);

    \draw[green!90!black] (BSW2) to[bend right=20](RSW2);
    
    \draw[black] plot [smooth] coordinates {(RSW2) ($(RSW2)+(0.6,0.4)$) 
    ($(RSW2)+(1.3,0.4)$) 
    ($(RSW2)+(2.2,0)$) (RW2)};

    \draw[cyan] (RW2) to[bend left=30](RE2);

    \draw[black] plot [smooth] coordinates {(RE2) ($(RE2)+(0.6,-0.2)$) 
    ($(RE2)+(1.3,-0)$) 
    ($(RE2)+(2.2,0)$) (RSE2)};

    \draw[green!90!black] (RSE2) to[bend right=30](BSE2);

    \path[fill=orange, fill opacity=0.1] plot [smooth] coordinates {(RSW2) ($(RSW2)+(0.6,0.4)$) 
    ($(RSW2)+(1.3,0.4)$) 
    ($(RSW2)+(2.2,0)$) (RW2)} -- (RSW1);

    \path[fill=orange, fill opacity=0.1] plot [smooth] coordinates {(RE2) ($(RE2)+(0.6,-0.2)$) 
    ($(RE2)+(1.3,-0)$) 
    ($(RE2)+(2.2,0)$) (RSE2)} -- (RSE1);

    \draw[red,thick,fill=red!10] plot [smooth cycle] coordinates {(0.5,0.5) (0.3,.9) (0.6,0.8) (0.8,1.2) (1.5,0.8) (1.25,0.6) };

    \node[color=blue] at (1.85,0.5) {$\zeta$};
    \node[color=red] at (1,0.75) {$\varrho$};

    \node[circle,scale=0.4,fill=gray,label={[label distance=0pt]above:{\small$\mathsf a$}}] at (BSW2){};
    \node[circle,scale=0.4,fill=gray,label={[label distance=0pt]above:{\small$\mathsf b$}}] at (RSW2){};
    \node[circle,scale=0.4,fill=gray,label={[label distance=-2pt]above left:{\small$\mathsf c$}}] at (RW2){};
    \node[circle,scale=0.4,fill=gray,label={[label distance=-2pt]above right:{\small$\mathsf d$}}] at (RE2){};
    \node[circle,scale=0.4,fill=gray,label={[label distance=0pt]above:{\small$\mathsf e$}}] at (RSE2){};
    \node[circle,scale=0.4,fill=gray,label={[label distance=0pt]above:{\small$\mathsf f$}}] at (BSE2){};
    
\end{tikzpicture}
\vspace{-0.1in}
    \caption{A schematic representation of the proof of \cref{claim:entropy_loop}. The original path is represented with green/black/cyan colors depending on the division in the proof and the final path is represented in purple below both the blue loop $\zeta$ and the red loop $\varrho_{i(\zeta)}$. The area in orange is bounded by $A_i$, the area between the loop $\varrho$ and its corresponding level line (not represented in the picture) which must stay north of the green/black/cyan path.}
    \label{fig:entropy_loop}
    \vspace{-0.1in}
\end{figure}
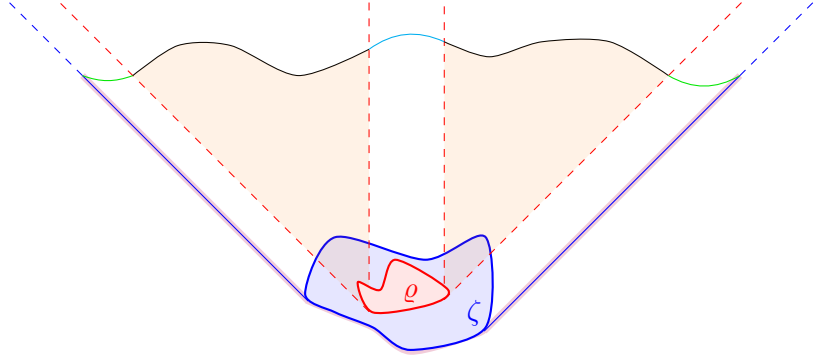

\begin{claim}\label{claim:entropy_loop}
If $\zeta$ is the $j$-th loop taken into account at its level, then the steps of $\Gamma^{(j-1)}$ which are not steps of $\Gamma^{(j)}$ lie on a subpath of $\Gamma^{(j-1)}$ which is the concatenation of five paths: two paths all of whose steps but at most $|\zeta|$ are in a single direction, a path of length at most $\diam(\varrho_{i(\zeta)})\leq|\zeta|$, and two paths each of which delimits, together with two half-lines, an area at most $A_{i(\zeta)}$.
\end{claim}
\begin{proof} By construction, $\Gamma^{(j)}$ is the lowest level line above $\Gamma^{(j-1)}$ satisfying the boundary condition forced by $\zeta$. Draw a \southeast to \northwest half line starting at the \southwest-most point of the loop $\zeta$; note that $\Gamma^{(j)}$ cannot cross this line but we can assume $\Gamma^{(j-1)}$ does (otherwise no step of $\Gamma^{(j-1)}$ moves on the \west side). Also draw a parallel line from the \southwest-most point of $\varrho_{i(\zeta)}$ and a straight \north one and then copy with symmetric directions on the \east side. Denote by $\mathsf{a}, \mathsf{b}, \mathsf{c}, \mathsf{d}, \mathsf{e}, \mathsf{f}$ the intersections of the path $\Gamma^{(j-1)}$ with these rays, marked from left to right in \cref{fig:entropy_loop}, so that the steps of $\Gamma^{(j-1)}$ which move lie between $\mathsf a$ and $\mathsf f$. The path from $\mathsf{c}$ to $\mathsf{d}$ between the two vertical lines has length at most $\diam(\varrho_{i(\zeta)})$. The path from $\mathsf{b}$ to $\mathsf{c}$ stays by definition in a ``$\frac{1}{8}$-plane'' and the area it delimits in this ``$\frac{1}{8}$-plane'' is smaller than $A_{i(\zeta)}$, and similarly for the path from $\mathsf d$ to $\mathsf e$. Finally, the path from $\mathsf{a}$ to~$\mathsf{b}$ has at most $|\zeta|$ \northeast steps, and symmetrically the path from $\mathsf e$ to $\mathsf f$ has at most $|\zeta|$ \southeast steps.
\end{proof}

The preprocessing step also bounds the number of retained loops. Fix $i$, order the loops assigned to $i$ as $\zeta_1,\ldots,\zeta_{n_i}$ from innermost to outermost, and put $L_0=0$ and $L_j:=\sum_{r=1}^j|\zeta_r|$. Since these loops survive preprocessing,
\[ L_j\leq H(\zeta_j)<\epsilon^{1/3}a(\zeta_j)\leq\epsilon^{1/3}|\zeta_j|^2\,,\]
so, using $(L_j-L_{j-1})/\sqrt{L_j}\leq2(\sqrt{L_j}-\sqrt{L_{j-1}})$,
\[ n_i\leq\epsilon^{1/6}\sum_{j=1}^{n_i}\frac{L_j-L_{j-1}}{\sqrt{L_j}}
\leq2\epsilon^{1/6}\sqrt{L_{n_i}}\,.
\]
Thus $\sum_i n_i^2\leq4\epsilon^{1/3}\sum_\zeta|\zeta|\leq4\epsilon^{4/3}s$. By \cref{claim:area_move} and Cauchy--Schwarz,
\begin{equation}\label{eq:loop-sizes}
\sum_i\sqrt{A_i}\,n_i\leq\Big(\sum_i A_i\Big)^{1/2}\Big(\sum_i n_i^2\Big)^{1/2}\leq2\epsilon^{2/3}s\,,
\end{equation}
and the same bound holds for the loops below $h_0$.

It remains to prove the bound in the statement. Below, $C$ denotes an absolute constant whose value may change from line to line, and $\epsilon$ is taken small enough.

Fix a face $o$ of $\varphi$ and $h\in\bigcup_{\cS\ni o}\cH^{\mathrm{ex}}_\cS$, the union being over the connected sets $\cS\ni o$ of $s$ faces of $\varphi$. Then $h\in\cH^{\mathrm{ex}}_\cS$ for its floor $\cS:=\varphi\setminus h$, and we let $\sB:=\varphi\xor h$ and $\cR:=h\setminus\varphi$ be its bubble and its roof, so that $|h|-|\varphi|=H(\sB)=|\cR|-s$.
\begin{enumerate}[label=\textbf{Part~\arabic*.}, ref=\arabic*, wide=0pt, itemsep=1ex, listparindent=\parindent, parsep=0pt]
\item \label[part]{pt:count-detours} \emph{(Detours and marked points.)}
For $k\in\Z$, let $\gamma_k$ be the height-$k$ level line of $\varphi$, a bi-infinite path of \northeast and \southeast steps. As $h=\varphi$ outside a finite set, the height-$k$ level set of $h$ agrees with $\gamma_k$ outside a finite set, and the \northeast splitting rule decomposes it into a bi-infinite path $\Gamma_k$, which agrees with $\gamma_k$ outside a finite set, and finitely many loops (if $\gamma_k$ is empty, as happens for instance when $\varphi$ is flat, then so is $\Gamma_k$). We call a step of $\Gamma_k$ a \emph{$\varphi$-step} if it is a step of $\gamma_k$ traversed in the same direction, and a \emph{detour step} otherwise, and we call a maximal run of consecutive detour steps of $\Gamma_k$ a \emph{detour}. A detour step is either a \emph{defect}, i.e., a \southwest or \northwest step, or a step whose plaquette belongs to $\cR$.

Let $d$ be the total number of defects of the paths $(\Gamma_k)_{k\in\Z}$, and $\ell_\circ$ the total length of the loops. In a large enough box, $\Gamma_k$ and $\gamma_k$ have the same endpoints, so that $|\Gamma_k|-|\gamma_k|$ is twice the number of defects of $\Gamma_k$, while $h$ and $\varphi$ have the same number of horizontal faces. Hence
\begin{equation}\label{eq:count-d-loops}
|h|-|\varphi|=2d+\ell_\circ\,,\qquad\mbox{whence}\qquad d<\tfrac12\epsilon s\quad\mbox{and}\quad\ell_\circ<\epsilon s\,.
\end{equation}
Moreover $H(\sB)\geq1$: otherwise $d=\ell_\circ=0$, so that $h$ would be a tiling, and $\psi=h$ would have $|\psi\cap h|-|\varphi\cap h|=|\cR|=s$, contradicting $h\in\cH^{\mathrm{ex}}_\cS$. Hence $\epsilon s>1$ and $|\cR|<2s$, and the total number of detour steps of the paths $\Gamma_k$ is at most $|\cR|+d\leq3s$.

For every $k$, let $M_k$ be the set of vertices of $\gamma_k$ at which a detour of $\Gamma_k$ starts or ends. We claim that $\Gamma_k$ is determined by $\gamma_k$, by $M_k$ and by the list of the detours of $\Gamma_k$, each given as a path from its starting point, in the order in which $\Gamma_k$ traverses them. Indeed, every plaquette of the height-$k$ level set of $h$ is traversed at most once, so a maximal run of $\varphi$-steps of $\Gamma_k$ cannot pass through a point $m\in M_k$: it would use both steps of $\gamma_k$ at $m$, whereas a detour starting (resp.\ ending) at $m$ is preceded (resp.\ followed) by one of them. Therefore, $\Gamma_k$ follows $\gamma_k$ up to the first point of $M_k$, then its first detour, then $\gamma_k$ from the endpoint of that detour up to the next point of $M_k$, then its second detour, and so on, until it leaves the last point of $M_k$ along $\gamma_k$.

\item \label[part]{pt:count-cover} \emph{(Localization.)}
Let $U$ be the union of $\Upsilon(\sB)$ and of the projections of the plaquettes of $\varphi\cap h$ on which $h$ and $\varphi$ induce opposite orientations. Such a plaquette lies in the wall of $\varphi$ above an edge $xy$ of $\Z^2$, at a level $k$ with (up to exchanging $x$ and $y$) $h(x)<k\leq\varphi(x)$ and $\varphi(y)<k\leq h(y)$. Hence the horizontal face of $\varphi$ above $x$ belongs to $\cS$, and the plaquettes of $\varphi\cap h$ in this wall form a vertical strip whose top plaquette shares an edge with a face of $\cS$, namely, either with the horizontal face of $\varphi$ above $x$ or with the plaquette of the wall just above the strip, which is not in $h$. Each plaquette of the strip is traversed by a defect or by a loop, so $U$ is a connected set of at most $2|\sB|+2(d+\ell_\circ)\leq5s$ faces of $\T$.
{\color{black}\begin{fact}\label{fact:epsilon-T}
For any $\epsilon>0$ and connected set $Y$ of faces of the triangular lattice $\T$, the number of faces of $\lceil\epsilon^{-1}\rceil \T$ intersecting $Y$ is at most $12 \lceil \epsilon|Y|\rceil $.
\end{fact}
\begin{proof}
Let $\T'$ be the sub-lattice of $\T$ made of tiles with side-length $\lceil1/\epsilon\rceil $, and let $T$ be the number of tiles of $\T'$ that intersect $Y$. The faces of the triangular lattice $\T$ can be colored using 6 colors such that two faces with the same color share neither an edge nor a vertex. Select the single-color sub-lattice containing the most faces of $\T'$ which intersect $Y$ and let $T'$ be the number of tiles in this set. If $T' > 1$, since $Y$ is connected, we can associate to each of these tiles a path in $Y$ of length $1/
(2 \epsilon)$ connecting $Y \cap t$ up to distance $1/(2 \epsilon)$ of $t$. Also of course $T'\geq T/6$. Overall, we see that if $T > 6$, then $|Y| \geq T/(12\epsilon)$ and in particular $T \leq 12 \lceil \epsilon |Y| \rceil$.
\end{proof}}
By \cref{fact:epsilon-T} and as $\epsilon s>1$, the set $U$ meets at most $72\epsilon s$ faces of $\lceil\epsilon^{-1}\rceil\T$. These form a set which is connected for the adjacency of sharing a vertex and contains the face of $\lceil\epsilon^{-1}\rceil\T$ that contains a fixed triangle of $\Upsilon(o)$, so there are at most $e^{C\epsilon s}$ choices for this set. Given it, let $S^\ast$ be the union of its faces and of the faces of $\lceil\epsilon^{-1}\rceil\T$ that share a vertex with them, so that $|S^\ast|\leq Cs/\epsilon$, and let $V^\ast$ be the set of vertices of the paths $(\gamma_k)_{k\in\Z}$, each recorded together with its level $k$, that are endpoints of a plaquette of $\varphi$ whose projection is contained in $S^\ast$. Distinct plaquettes of $\varphi$ project onto disjoint lozenges, so that $|V^\ast|\leq|S^\ast|\leq Cs/\epsilon$. Finally, a point of $M_k$ is a common endpoint of a detour step, whose plaquette projects into $U$, and of a step of $\gamma_k$, whose plaquette shares an edge with that plaquette and hence projects into $S^\ast$; therefore $\bigcup_kM_k\subseteq V^\ast$.

\item \label[part]{pt:count-legs} \emph{(Straight pieces, when $h$ has no loops.)}
A \emph{leg} is a path with a \emph{principal direction}, \northeast or \southeast, and a set of \emph{exceptional} steps containing all steps in other directions. We will first decompose the paths into legs, then cut the detours into \emph{straight pieces}, each using a single one of the four directions. For an absolute constant $C_1$ and $N:=C_1\epsilon^{1/3}s$, we claim that
\begin{equation}\label{eq:count-legs}
\#\{\mbox{detours}\}\leq N\qquad\mbox{and}\qquad\#\{\mbox{straight pieces}\}\leq N\,.
\end{equation}
In this part we prove \cref{eq:count-legs} when $h$ has no loops, using the run of the algorithm on $h$, for which $\cS'=\cS$. The height-$k$ level line processed by the algorithm is the part of $\Gamma_k$ between its first and its last visit to $D$, which contains every detour of $\Gamma_k$. By \cref{clm:alg-image-monotone,clm:alg-image-sat-bc}, the output $\psi$ is a tiling of $\cS$ with $\psi=\varphi$ outside $\cS$; as $\varphi\cap h=\varphi\setminus\cS$ and $h\in\cH^{\mathrm{ex}}_\cS$, we get
\begin{equation}\label{eq:count-gain}
|\psi\cap\cR|=|\psi\cap h|-|\varphi\cap h|<\epsilon^{1/3}s\,.
\end{equation}
Fix a level line $\Gamma_h$ processed by the algorithm, with $d_h$ defects, and let $\Gamma_h^\reg$, $\Gamma_{\min}$, $\Gamma_{\max}$ and $\Gamma_\psi$ be the level lines of $h^\reg$, $\psi_{\min}$, $\psi_{\max}$ and $\psi$ at the same height. We write $x$ and $y$ for the coordinates in which a \northeast step increases $x$ by one and a \southeast step increases $y$ by one, i.e., $x=-u_1$ and $y=u_2$ up to additive constants, so that north means larger $x-y$ and a step is forward or a defect according to whether it increases or decreases $x+y$. The \emph{line} of a \northeast step of $\Gamma_{\min}$ or $\Gamma_{\max}$ is the line $\{x=c\}$, $c\in\Z+\frac12$, that it crosses, and the line of a \southeast step is the line $\{y=c\}$ that it crosses. As $\Gamma_h$ has the same endpoints as $\Gamma_{\min}$ and $\Gamma_{\max}$, it crosses each such line once more forward than backward; we call the step \emph{covered} if $\Gamma_h$ crosses its line at least twice, i.e., if this line carries a defect of $\Gamma_h$. Finally, for a height function $g$, let $U_g$ be the set of cells at which $g$ is at least the height of $\Gamma_h$, so that the steps of the level set of $g$ at that height have the cells of $U_g$ on their left; with $h^\reg$ as in \cref{claim:def_gammahr}, $U_{h^\reg}=(U_h\cap U_{\psi_{\max}})\cup U_{\psi_{\min}}$ and $U_{\psi_{\min}}\subseteq U_{\psi_{\max}}$. Since $\psi_{\min}$ and $\psi_{\max}$ are monotone, $U_{\psi_{\min}}$ contains the \northeast and \northwest neighbours of each of its cells, and the complement of $U_{\psi_{\max}}$ contains the \southwest and \southeast neighbours of each of its cells.

\begin{claim}\label{clm:reg-steps}
Every step of the level set of $h^\reg$ is either a step of $\Gamma_h$ traversed in the same direction, or a step of $\Gamma_{\min}$ or of $\Gamma_{\max}$. Consequently, the steps of $\Gamma_h$ lying on the loops of this level set number at most $2d_h$, and every step of $\Gamma_h$ which is neither a step of $\Gamma_h^\reg$ nor on such a loop has one of its two cells in $U_{\psi_{\min}}\setminus U_h$ or in $U_h\setminus U_{\psi_{\max}}$.
\end{claim}
\begin{proof}
Let $e$ be a step of the level set of $h^\reg$, with the cell $u\in U_{h^\reg}$ on its left and the cell $q\notin U_{h^\reg}$ on its right. If $u\in U_h$ and $q\notin U_h$, then $e$ is a step of $\Gamma_h$ traversed in the same direction. If $u,q\in U_h$, then $q\notin U_{\psi_{\max}}$ while $u\in U_{\psi_{\max}}$, so $e$ is a step of $\Gamma_{\max}$. If $u\notin U_h$, then $u\in U_{\psi_{\min}}$ while $q\notin U_{\psi_{\min}}$, so $e$ is a step of $\Gamma_{\min}$. The steps of $\Gamma_{\min}$ and $\Gamma_{\max}$ all increase $x+y$, so along each loop of the level set of $h^\reg$, whose total displacement vanishes, the steps of $\Gamma_h$ have a nonpositive total displacement in $x+y$; that is, at least half of them are defects, which gives the second assertion. Finally, consider a step of $\Gamma_h$ with the cell $u\in U_h$ on its left and $q\notin U_h$ on its right. If $u\in U_{h^\reg}$ and $q\notin U_{h^\reg}$, it is a step of the level set of $h^\reg$ traversed in the same direction, hence a step of $\Gamma_h^\reg$ or of one of its loops. Otherwise $u\notin U_{h^\reg}$, so $u\in U_h\setminus U_{\psi_{\max}}$, or $q\in U_{h^\reg}$, so $q\in U_{\psi_{\min}}\setminus U_h$.
\end{proof}

\begin{claim}\label{clm:two-legs}
Let $E$ be a subpath of $\Gamma_h$ which lies north of $\Gamma_{\min}$ and meets it only at its endpoints $\mathsf A$ and $\mathsf D$, where $\mathsf D$ lies east of $\mathsf A$. Let $\delta$ be the number of defects of $E$ and $L$ the number of covered steps of $\Gamma_{\min}$ between $\mathsf A$ and $\mathsf D$. Then $E$ is the concatenation of two paths $E_1$ and $E_2$ such that the steps of $E_1$ which are not \northeast steps and the steps of $E_2$ which are not \southeast steps number at most $2\delta+L$ in total. The same holds, with the roles of \northeast and \southeast exchanged, for a subpath of $\Gamma_h$ lying south of $\Gamma_{\max}$, meeting it only at its endpoints and ending east of its starting point.
\end{claim}
\begin{proof}
Cut $E$ at one of its vertices $Q$. On $E[\mathsf A,Q]$, the \southeast steps outnumber the \northwest steps by $y_Q-y_{\mathsf A}$, so this path has $(y_Q-y_{\mathsf A})+2\#\{\northwest\}+\#\{\southwest\}$ steps which are not \northeast steps; similarly, $E[Q,\mathsf D]$ has $(x_{\mathsf D}-x_Q)+2\#\{\southwest\}+\#\{\northwest\}$ steps which are not \southeast steps. Hence these two numbers add up to at most
\begin{equation}\label{eq:two-legs-cut}
2\delta+(y_Q-y_{\mathsf A})+(x_{\mathsf D}-x_Q)\,.
\end{equation}

We next show that no uncovered \southeast step $s_1$ of $\Gamma_{\min}[\mathsf A,\mathsf D]$ precedes an uncovered \northeast step $s_2$ of $\Gamma_{\min}[\mathsf A,\mathsf D]$. Suppose otherwise, and let $\{y=y^*\}$ and $\{x=x^*\}$ be their lines. Then $\Gamma_h$ crosses each of these lines exactly once, and $E$ crosses both, so the parts of $\Gamma_h$ before and after $E$ do not cross them. As the endpoints of $\Gamma_h$ are those of $\Gamma_{\min}$, these parts stay in the quarter-planes $\{x<x^*,y<y^*\}$ and $\{x>x^*,y>y^*\}$ respectively. Between $s_1$ and $s_2$, the path $\Gamma_{\min}$ makes a \southeast step followed by a \northeast step at a vertex $P$, which lies in the quarter-plane $Q_{\mathrm S}:=\{x<x^*,y>y^*\}$. By \cref{clm:se-followed-by-ne} and its proof, the cell $F$ just north of $P$ is a boundary cell on which $\Gamma_h$ takes the value $1$; as $E$ meets $\Gamma_{\min}$ only at its endpoints, it does not pass through $P$. If $E$ does not enter $Q_{\mathrm S}$, then no point of $\Gamma_h$, or of the level line of $\varphi$ beyond its endpoints, lies in $Q_{\mathrm S}$, so all the cells meeting $Q_{\mathrm S}$ take the same value for $\Gamma_h$. This value is $1$, as $F$ meets $Q_{\mathrm S}$; but the half-line going south from $P$ stays in $Q_{\mathrm S}$ and south of $\Gamma_{\min}$, so the cells it meets far from $D$ take the value $0$. If $E$ enters $Q_{\mathrm S}$, it crosses $Q_{\mathrm S}$ from the line $\{y=y^*\}$ to the line $\{x=x^*\}$, north of $\Gamma_{\min}$, and the cells meeting $Q_{\mathrm S}$ between $\Gamma_{\min}$ and $E$, among which $F$, are not separated from $E$ by any other point of $\Gamma_h$; those adjacent to $E$ lie on its south side, where the value is $0$. Both cases are contradictory.

Consequently, if $V$ is a vertex of $\Gamma_{\min}[\mathsf A,\mathsf D]$ lying after its last uncovered \northeast step and before its first uncovered \southeast step, then the \southeast steps of $\Gamma_{\min}[\mathsf A,V]$ and the \northeast steps of $\Gamma_{\min}[V,\mathsf D]$ are covered, whence $(y_V-y_{\mathsf A})+(x_{\mathsf D}-x_V)\leq L$. If $V\in\{\mathsf A,\mathsf D\}$ we take $Q=V$. Otherwise the segment of the half-line going north from $V$ up to $\partial D$ separates $\mathsf A$ from $\mathsf D$ in the part of $D$ north of $\Gamma_{\min}$, so $E$ has a vertex on it, of the form $Q=V+m(1,-1)$ with $m\geq0$, since every step changes $x+y$ by one. In both cases $(y_Q-y_{\mathsf A})+(x_{\mathsf D}-x_Q)\leq L$, and \cref{eq:two-legs-cut} concludes the proof. The case of $\Gamma_{\max}$ is symmetric, using the symmetric statement in \cref{clm:se-followed-by-ne}.
\end{proof}

\begin{claim}\label{clm:mismatch}
Let $f$ and $f'$ be consecutive steps of $\Gamma_h$ which are both steps of $\Gamma_h^\reg$, traversed in the same direction. If $f'$ does not follow $f$ on $\Gamma_h^\reg$, then $f$ and $f'$ are defects, one \northwest and one \southwest step. In particular, this happens for at most $d_h$ pairs $(f,f')$.
\end{claim}
\begin{proof}
Let $v$ be the common vertex of $f$ and $f'$, and $N,E,S,W$ the four cells around it. Each cell adjacent to $f$ or $f'$ lies on the same side of $\Gamma_h$ and of $\Gamma_h^\reg$, i.e., has the same value for $h$ and for $h^\reg$. If $f$ and $f'$ were opposite edges at $v$, the four cells would have the same values for $h$ and $h^\reg$, and the splitting rule would pair the edges at $v$ in the same way for both. Hence $f$ and $f'$ are adjacent edges at $v$, three of the four cells have the same values, and the fourth cell $X$ has different values, so that one of the two level sets has degree $4$ at $v$ and the other has degree $2$. It is the level set of $h^\reg$ which has degree $4$, as otherwise it would consist of $f$ and $f'$ near $v$ and $f'$ would follow $f$ on $\Gamma_h^\reg$. If $X\in U_{h^\reg}\setminus U_h$, then $X\in U_{\psi_{\min}}$ while its two neighbours among $N,E,S,W$ are not in $U_{h^\reg}\supseteq U_{\psi_{\min}}$, which forces $X=N$; then $S$ is the only one of the four cells in $U_h$, and $\Gamma_h$ passes $v$ by a \northwest step followed by a \southwest step. If $X\in U_h\setminus U_{h^\reg}$, then $X\notin U_{\psi_{\max}}$ while its two neighbours are in $U_{h^\reg}\subseteq U_{\psi_{\max}}$, forcing $X=S$; then $N$ is the only cell outside $U_h$, and $\Gamma_h$ passes $v$ by a \southwest step followed by a \northwest step. The last assertion follows, as $f$ is a defect, determining the pair.
\end{proof}

We cut $\Gamma_h$ into the following \emph{segments}, each of which is a leg.
\begin{enumerate}[(a)]
\item \label{it:seg-a} Consider the steps of $\Gamma_h$ which are not steps of $\Gamma_h^\reg$. Each of those lying on a loop of the level set of $h^\reg$, of which there are at most $2d_h$ by \cref{clm:reg-steps}, and each of those traversing a step of $\Gamma_{\min}$ or $\Gamma_{\max}$ in the opposite direction, which are defects, forms a segment by itself, with its step regarded as exceptional. By \cref{clm:reg-steps}, every other such step has a cell in $U_{\psi_{\min}}\setminus U_h$ or in $U_h\setminus U_{\psi_{\max}}$, and then its other cell is also in $U_{\psi_{\min}}$, respectively outside $U_{\psi_{\max}}$, as otherwise the step would traverse a step of $\Gamma_{\min}$, respectively $\Gamma_{\max}$, in the opposite direction; thus both of its cells lie north of $\Gamma_{\min}$ or both south of $\Gamma_{\max}$. A maximal run of consecutive steps of $\Gamma_h$ with both cells north of $\Gamma_{\min}$ starts and ends on $\Gamma_{\min}$, as the steps of $\Gamma_h$ next to it have a cell outside $U_{\psi_{\min}}$ and $\Gamma_h$ starts and ends on $\Gamma_{\min}$; similarly, a maximal run of steps with both cells south of $\Gamma_{\max}$ starts and ends on $\Gamma_{\max}$. Cutting these runs at the vertices of $\Gamma_{\min}$, respectively $\Gamma_{\max}$, we obtain \emph{arcs}, which lie north of $\Gamma_{\min}$, respectively south of $\Gamma_{\max}$, and meet it only at their endpoints; the \emph{range} of an arc is the part of $\Gamma_{\min}$, respectively $\Gamma_{\max}$, between its endpoints. An arc is \emph{forward} if it ends east of its starting point, and \emph{outermost} if it is not contained in the region between $\Gamma_{\min}$, respectively $\Gamma_{\max}$, and another arc. The outermost arcs on the same side have disjoint ranges; as distinct steps of $\Gamma_{\min}$ have distinct lines and each covered line carries a defect of $\Gamma_h$, the sum of $L$ over the outermost forward arcs on one side is at most $d_h$. By \cref{clm:two-legs}, each outermost forward arc $E$ is the concatenation of two segments with at most $2\delta_E+L_E$ exceptional steps in total, whose principal directions are \northeast and then \southeast for an arc north of $\Gamma_{\min}$, and \southeast and then \northeast for an arc south of $\Gamma_{\max}$. Moreover $\delta_E+L_E\geq1$: otherwise, by the proof of \cref{clm:two-legs}, $E$ and the part of $\Gamma_{\min}$ between its endpoints would both consist of \northeast steps followed by \southeast steps, and they would coincide. Hence there are at most $2d_h$ outermost forward arcs on each side. Every step of the other arcs forms a segment by itself, regarded as exceptional. To bound their number, note that the arcs whose range contains a given step of $\Gamma_{\min}$ all cross its line, which is crossed $2b+1$ times by $\Gamma_h$ if $b$ is the number of backward crossings, and at most one of these arcs is outermost; summing over the steps of $\Gamma_{\min}$, the ranges of the arcs which are not outermost have total length at most $2d_h$. As a forward arc is longer than its range by twice its number of defects, and a backward arc has at least as many defects as other steps, the arcs on one side other than the outermost forward ones have total length at most $4d_h$. Altogether, the steps of $\Gamma_h$ which are not steps of $\Gamma_h^\reg$ are cut into at most $19d_h$ segments with at most $17d_h$ exceptional steps.
\item \label{it:seg-b} Fix $i$, let $\mathsf a_i=\Gamma_h^\reg(t'_i)$, and let $\delta_i$ be the number of defects of the path $\Gamma_h^\reg[t'_i,t_i]$. Write the points of this path as $\mathsf a_i+x\,\vec u+y\,\vec v$, with $\vec u$ and $\vec v$ the \northeast and \southeast unit steps. All of its points but the last one have $x\leq0$ or $y\leq0$, and the last one has $x,y\geq0$; as a step changes one coordinate by one, the last point has $x\in\{0,1\}$ or $y\in\{0,1\}$, say $x\in\{0,1\}$. The path then has at most one more \northeast step than \southwest steps, so all of its steps but at most $2\delta_i+1$ are \southeast steps, and we set $\mathsf e_i:=\southeast$ (in the symmetric case, $\mathsf e_i:=\northeast$). Every maximal run of steps of $\Gamma_h$ which are steps of $\Gamma_h^\reg$ traversed by $\Gamma_h^\reg$ between the times $t'_i$ and $t_i$ is a segment with principal direction $\mathsf e_i$. Such a run starts at the first step of $\Gamma_h$, at the step of $\Gamma_h^\reg$ at time $t'_i$, right after a step of $\Gamma_h$ which is not a step of $\Gamma_h^\reg$, or at a pair as in \cref{clm:mismatch}: indeed, if the step preceding the run on $\Gamma_h$ is a step of $\Gamma_h^\reg$ which $\Gamma_h^\reg$ follows by the first step of the run, then the run starts at time $t'_i$. As there are at most $d_h$ indices $i$, and none if $d_h=0$, this gives at most $22d_h$ segments with at most $3d_h$ exceptional steps.
\item \label{it:seg-c} Every remaining step of $\Gamma_h$ is a step of $\Gamma_h^\reg$ which is copied into $\Gamma_\psi$ in \cref{it:alg-iter}, and forms a segment by itself, with its step regarded as exceptional. Such a step is not a defect, as \cref{it:alg-iter} stops at the first defect; hence, if it is a detour step, it is a step of $\Gamma_\psi$ whose plaquette belongs to $\cR$, and by \cref{eq:count-gain} there are fewer than $\epsilon^{1/3}s$ such detour steps in total.
\end{enumerate}
Retain the segments in \cref{it:seg-a,it:seg-b} and the detour singletons in \cref{it:seg-c}. These cover all detour steps and, over all levels, number at most $m:=41d+\epsilon^{1/3}s$, with at most $e:=20d+\epsilon^{1/3}s$ exceptional steps. Cutting immediately before and after each exceptional step, keeping it as a singleton, gives at most $m+2e$ straight pieces. A forward straight piece meets the monotone reference path $\gamma_k$ in a single interval (possibly empty), while a backward piece has no $\varphi$-steps. Deleting the $\varphi$-steps therefore leaves at most $2(m+2e)$ straight pieces partitioning the detours. The number of detours is at most the number of these pieces, so \cref{eq:count-d-loops} gives \cref{eq:count-legs}.

\item \label[part]{pt:count-loops} \emph{(Loops.)}
Suppose now that $h$ has loops, and let $h_0$, $\cS'$, $\psi_0$ and $\psi_2$ be as in the treatment of loops above. As $\psi_2$ is a tiling of $\cS'$ with $\psi_2=\varphi$ outside $\cS'$ and $|\cS'\setminus\cS|\leq\epsilon^{2/3}s$, the definition of $\cH^{\mathrm{ex}}_\cS$ gives $|\psi_2\cap h|-|\varphi\cap h|<\epsilon^{1/3}s$, whence $|\psi_2\cap\cR|\leq|\psi_2\cap h|-|\varphi\cap h|+|\cS'\setminus\cS|<2\epsilon^{1/3}s$. Removing the loops does not change the bi-infinite components of the level sets, so the paths $\Gamma_k$ are level lines of $h_0$, and we argue as in \cref{pt:count-legs} with $h_0$ and the run of the algorithm on $\cS'$ that produces $\psi_0$. The only difference concerns the segments described in \cref{it:seg-c}, i.e., the steps of $\Gamma_h$ copied into $\Gamma_{\psi_0}$: those whose plaquette belongs to $\psi_2$ are $\varphi$-steps or have their plaquette in $\psi_2\cap\cR$, so that fewer than $2\epsilon^{1/3}s$ of them are detour steps, while the others lie on $\Gamma_{\psi_0}\setminus\Gamma_{\psi_2}$, which we now describe.

Fix a level, and let $\Gamma^{(0)}=\Gamma_{\psi_0},\Gamma^{(1)},\ldots,\Gamma^{(K)}=\Gamma_{\psi_1}$ be the successive level lines at that level defined above, which are non-decreasing. By \cref{claim:entropy_loop}, the part $\Gamma^{(j-1)}\setminus\Gamma^{(j)}$ moved by the $j$-th loop $\zeta$ lies on a subpath of $\Gamma^{(j-1)}$ made of five paths: two paths all of whose steps but at most $|\zeta|$ are in a single direction, a path of length at most $|\zeta|$, and two paths each of which delimits, together with two half-lines, an area at most $A_{i(\zeta)}$, and therefore consists of at most $C(1+\sqrt{A_{i(\zeta)}})$ straight runs (a partition of an integer $n$ has at most $\sqrt{2n}$ distinct parts). The region between $\Gamma^{(j-1)}$ and $\Gamma^{(j)}$ consists of $r_j$ \emph{lenses}, i.e., regions bounded by an arc of each of the two lines with common endpoints. At the left end of a lens, $\Gamma^{(j-1)}$ makes a \northeast step and $\Gamma^{(j)}$ a \southeast step, and at its right end it is the reverse; hence, between two consecutive lenses, $\Gamma^{(j-1)}$ makes a \southeast step followed by a \northeast step, and the five-path description gives $r_j\leq C(1+|\zeta|+\sqrt{A_{i(\zeta)}})$. Moreover, $\Gamma^{(0)}\setminus\Gamma^{(K)}$ is the union over $j$ of $(\Gamma^{(j-1)}\setminus\Gamma^{(j)})\cap\Gamma^{(0)}$. Let $c_j$ be the number of connected components of the region between $\Gamma^{(0)}$ and $\Gamma^{(j)}$, and let $m_j$ be the number of connected components of the region between $\Gamma^{(0)}$ and $\Gamma^{(j-1)}$ which are adjacent to $\Gamma^{(j-1)}\setminus\Gamma^{(j)}$. Each of the latter merges with an adjacent lens, so that $c_j\leq c_{j-1}-m_j+r_j$, whence $\sum_jm_j\leq\sum_jr_j$ as $c_0=0$, while $(\Gamma^{(j-1)}\setminus\Gamma^{(j)})\cap\Gamma^{(0)}$ has at most $2r_j+m_j$ connected components. The loops below $h_0$ are treated in the same way, starting from $\Gamma_{\psi_1}$. Cutting all these components along the five paths and their straight runs, and regarding as exceptional the steps of the first two paths which are not in their dominant direction and all the steps of the third one, we find that $\Gamma_{\psi_0}\setminus\Gamma_{\psi_2}$ is covered by at most $C\sum_\zeta(1+|\zeta|+\sqrt{A_{i(\zeta)}})$ legs with at most $3\sum_\zeta|\zeta|$ exceptional steps, the sums being over the loops taken into account. By \cref{eq:loop-sizes}, applied to the loops above and below $h_0$, $\sum_\zeta\sqrt{A_{i(\zeta)}}=\sum_in_i\sqrt{A_i}\leq4\epsilon^{2/3}s$, and there are at most $\ell_\circ/4$ loops, of total length $\ell_\circ<\epsilon s$; hence these legs number at most $C\epsilon^{2/3}s$, with at most $3\epsilon s$ exceptional steps. We then replace the segments described in \cref{it:seg-c} that lie on $\Gamma_{\psi_0}\setminus\Gamma_{\psi_2}$ by the maximal runs of steps of $\Gamma_h$ copied into $\Gamma_{\psi_0}$ that lie on a single one of these legs, with the principal direction and the exceptional steps of that leg. By \cref{clm:mismatch}, and as the algorithm copies consecutive steps of $\Gamma_h^\reg$ to consecutive steps of $\Gamma_{\psi_0}$, consecutive copied steps are consecutive steps of $\Gamma_{\psi_0}$, except at most $d$ times, so such a run ends at the end of its leg, right before the first step of another leg, at the start of a segment as in \cref{it:seg-a} or \cref{it:seg-b}, or at a pair as in \cref{clm:mismatch}; hence there are at most $C\epsilon^{2/3}s+42d$ of them. With these segments, the straight-piece decomposition of \cref{pt:count-legs} gives \cref{eq:count-legs}, with a larger constant $C_1$.

The loops themselves are recorded as follows. We add to $M_k$ the vertices of $\gamma_k$ at which a loop at level $k$ passes from a step of $\gamma_k$ traversed in the same direction to another step or vice versa; these points still belong to $V^\ast$, and they number at most $2\ell_\circ$. At each of them, the loop uses one of the two steps of $\gamma_k$, so, as in \cref{pt:count-detours}, no maximal run of $\varphi$-steps of a path $\Gamma_k$ or of a loop passes through a point of $M_k$. Hence the reconstruction of $\Gamma_k$ in \cref{pt:count-detours} is unchanged, and a loop with such a step is determined by the sets $M_k$ and by its other steps, which we record, loop after loop in a fixed order of their first points of departure from $\gamma_k$, as one-step straight pieces, of total length at most $\ell_\circ$. A loop without such steps projects into $U$; we record it by its steps from a starting vertex, chosen among the at most $C|S^\ast|$ vertices of the dual lattice of $\cP_{001}$ adjacent to a column whose horizontal face of $\varphi$ projects into $S^\ast$, without recording its level. Indeed, once the paths $\Gamma_k$ and the other loops are known, the levels of these loops are determined, from the outermost to the innermost, by the values of $h$ just outside them and by their orientations. As there are at most $\ell_\circ/4<\epsilon s$ loops, of total length at most $\ell_\circ$, the loops without such steps, each given by its starting vertex and its sequence of steps followed by an end marker, can be chosen in at most $\binom{C|S^\ast|+\epsilon s}{\epsilon s}5^{2\epsilon s}\leq e^{C\epsilon\log(1/\epsilon)s}$ ways.

\item \label[part]{pt:count-count} \emph{(Counting.)}
By \cref{pt:count-detours,pt:count-loops}, $h$ is determined by $\varphi$ and the following data: the coarse faces chosen in \cref{pt:count-cover}; the set $M:=\bigcup_kM_k\subseteq V^\ast$, of size at most $2N+2\ell_\circ\leq3N$, with one bit per point indicating whether a loop leaves $\gamma_k$ there; the ordered list of straight pieces of the detours (for increasing $k$, each in traversal order), followed by those of the loops with a step of $\gamma_k$, each piece given by its direction, length and a bit marking the end of its detour; and the loops without such steps. The marked points determine where each detour starts, and its pieces are reconstructed successively from their directions and lengths. By \cref{eq:count-d-loops,eq:count-legs}, there are at most $N+\ell_\circ\leq2N$ pieces, of total length at most $3s+\ell_\circ\leq4s$. Neither the reconstruction nor the possible data depend on the floor $\cS$. Thus, using $|V^\ast|\leq Cs/\epsilon$ and $\binom{n+m}{m}\leq(e(n+m)/m)^m$,
\begin{align*}
\Big|\bigcup_{\cS\ni o}\cH^{\mathrm{ex}}_\cS\Big|&\leq e^{C\epsilon s}\binom{Cs/\epsilon+3N}{3N}2^{3N}\cdot8^{2N}\binom{4s+2N}{2N}\cdot e^{C\epsilon\log(1/\epsilon)s}\\
&\leq\exp\big[C\epsilon^{1/3}\log(1/\epsilon)\,s\big]\,,
\end{align*}
as $N=C_1\epsilon^{1/3}s$, which concludes the proof.
\end{enumerate}

\end{proof}

\subsection{Application for bubble groups}
\Cref{prop:alg} will provide us with a crucial estimate on $\beta H_h(\fB) - \alpha G_{h,\varphi}^\gr(\fB)$ for a ``typical'' bubble group $\fB$.
Recall from \cref{eq:area-of-bubble-grp} that for every bubble group $\fB=(\{\sB_i\},\{\cC_j\})$ we have $|\Upsilon(\fB)| \leq 5\sum_{i} |\sB_i|$.

\begin{lemma}\label{lem:approx-bubble-grp}
There exists an absolute constant $C>0$ such that, for every $\varphi$, $\epsilon>0$ and $s>C$, there are at most $\exp(C \epsilon^{1/4} s)$
bubble groups $\fB$ of size $|\Upsilon(\fB)|=s$, containing a fixed point $o$,
with
\begin{equation}\label{eq:good-bubble-grps}
\beta H_{h}(\fB) - \alpha G^\gr_{h, \varphi}(\fB) \leq  \epsilon^{1/3}  (\alpha \epsilon^{1/3} \wedge \beta\epsilon ) s\,.
\end{equation}
In addition, if $\epsilon\leq \frac1{150}$ then for all such $\fB$ we have
\[\beta H_h(\fB)+\lambda \sV(\fB) \geq \beta/2
+ \lambda s/16\,.\]
Furthermore, under the stronger condition in \cref{eq:gen-potential} on $\sV$, for any $C^\star >0$ the lower bound in the last inequality may be changed to $\beta/2 + C^\star s$ by taking $\beta$ large enough.
\end{lemma}
\begin{proof}
First observe that $H_h$ is additive on the bubbles $\sB_i$ by its definition in \cref{eq:bubble-group-H}. As for~$G_{h,\varphi}^\gr$, since the set $S'$ in \cref{prop:alg} is larger than just the bubble, we can unfortunately not directly sum the bound. Let $\cS_i:=\varphi\cap\sB_i$ be the floor of $\sB_i$. If the \SOS height function $h$ corresponding to $\sB_i$ does not belong to $\cH^{\mathrm{ex}}_{\cS_i}$ and $H(\sB_i)<\epsilon|\cS_i|$, let $\psi_i$ be a tiling of a superset $S'_i$ of $\cS_i$ as in the definition of $\cH^{\mathrm{ex}}_{\cS_i}$ in \cref{prop:alg}, which exists as $h\notin\cH^{\mathrm{ex}}_{\cS_i}$, with $|\psi_i\cap h|-|\varphi\cap h|\geq\epsilon^{1/3}|\cS_i|$, and otherwise set $\psi_i=\varphi\restriction_{S_i}$. Each tile in some $\psi_i \cap h$ can have one of three possible orientations (i.e., it can belong to a translate of the $\cP_{100}, \cP_{010}$ or $\cP_{001}$ direction) so one type must account for at least $1/3$ of the total. Suppose by symmetry (in this proof, we will not consider $h$ as anything but an arbitrary set of faces) that it is the $\cP_{100}$ direction. We can view each tiling as a monotone function over a $\Z^2$ tiling of the $\cP_{100}$ plane exactly as we often do over the $\cP_{001}$ plane and, similarly to the above, we can assume that at least $1/2$ of these faces satisfy $\psi_i(u) \geq \varphi(u)$. Extend each $\psi_i$ by $\varphi$ outside $S'_i$ (with $S'_i=S_i$ for unchanged bubbles), and denote the resulting full-plane tiling by $\widehat\psi_i$. Then
\[
\psi=\max\big(\{\widehat\psi_i\}\cup\{\varphi\}\big)
\]
in this representation is a tiling, since the maximum of monotone functions is monotone. We also note that, since for each $i$, $|S'_i \setminus S_i| \leq \epsilon^{2/3}|S_i|$, going from the $\psi_i$ to $\psi$ misses at most $\epsilon^{2/3}\sum |\Upsilon_{111}(\fB_i)|$ faces. Overall, we get
\[
\sum_i | \psi \cap h \cap \Upsilon_{111}^{-1}(S_i) | \geq \frac{1}{6} \sum_i | \psi_i \cap h  | - 2\epsilon^{2/3}\sum_i |S_i|.
\]
Since the $S_i' \setminus S_i$ also contain all the faces where $\varphi$ and $h$ agree but not $\psi$, we then get
\begin{equation}\label{eq:G^g-upper-bnd}
    G^\gr_{h,\varphi}(\fB) \leq G^\gr_{h, \varphi}(\psi) \leq  -\frac{1}{6}\sum_{i} |\psi_i \cap h| + 2\epsilon^{2/3}\sum_i |S_i| \,.
\end{equation}
Here both sums may be restricted to the set $I$ of indices with $\psi_i\neq\varphi\restriction_{S_i}$, since $S'_i=S_i$ and $\psi_i\cap h=\emptyset$ for the other indices. As $|\psi_i\cap h|\geq\epsilon^{1/3}|\cS_i|$ for $i\in I$, we get, for $\epsilon\leq24^{-3}$ (which we may assume, the bound of the lemma being trivial otherwise once $C$ is large enough),
\[G^\gr_{h,\varphi}(\fB)\leq-\frac1{12}\sum_{i\in I}|\psi_i\cap h|\,.\]

Consider $\psi$ constructed in this manner, and call a $\sB_i$ {\tt bad} if the corresponding \SOS height function $h$ belongs to $\cH^{\mathrm{ex}}_{\cS_i}$, with the same choice of $\epsilon$ as we have here, and {\tt good} otherwise. 

We will enumerate bubble groups with $\beta H_{h}(\fB) - \alpha G^\gr_{h, \varphi}(\fB) \leq  \epsilon^{1/3}  (\alpha \epsilon^{1/3} \wedge \beta\epsilon )s $ using the above bound on $G^\gr$.
Observe that we may assume 
\begin{equation}\label{eq:good-bubble-size}
\sum_i |\sB_i|\one_{\{\sB_i\text{ is {\tt good}}\}} \leq 36 \epsilon^{1/3} s\,,
\end{equation} since every {\tt good} bubble $\sB_i$ has either $H_h(\sB_i)\geq \epsilon |\cS_i|$ or $i\in I$ with $|\psi_i \cap h| \geq \epsilon^{1/3}|\cS_i|$ (recall $h\cap \varphi=\emptyset$ on an $(h,\varphi)$-bubble); as $|\sB_i|=2|\cS_i|+H_h(\sB_i)$, in the first case $\beta H_h(\sB_i)\geq\frac{\beta\epsilon}{3}|\sB_i|$, and in the second case, where $H_h(\sB_i)<\epsilon|\cS_i|$, $\frac{\alpha}{12}|\psi_i\cap h|\geq\frac{\alpha\epsilon^{1/3}}{36}|\sB_i|$, so by the last display every {\tt good} $\sB_i$ contributes at least $\frac1{36}(\alpha \epsilon^{1/3} \wedge\beta\epsilon) |\sB_i|$ to $\beta H_{h}(\fB) - \alpha G^\gr_{h, \varphi}(\fB)$.

To control the contribution of {\tt bad} bubbles, first observe the following.
\begin{fact}\label{fact:bad-bubble-size}
Let $\sB_i$ be an $(h,\varphi)$-bubble. If $\sB_i$ is {\tt bad} as defined above then $|\sB_i|> 1/\epsilon$.
\end{fact}
\begin{proof}
For the height function corresponding to $\sB_i$ to belong to $\cH^{\mathrm{ex}}_{\cS_i}$, necessarily $|h \cap \sB_i|>|\varphi \cap \sB_i|$, as otherwise $h\cap \sB_i$ is itself a tiling whence $|h \cap \psi_i| =|\sB_i|/2  > \epsilon^{1/3} |\sB_i|$ (by taking $\psi_i=h \cap \sB_i$), so $\sB_i$ is {\tt good}. Since we must also have $H< \epsilon |\sB_i|$ from the definition of $\cH^{\mathrm{ex}}_{\cS_i}$, we deduce that $|\sB_i| > 1/\epsilon$.
\end{proof}

We can now start to enumerate our bubble groups. First, using \cref{fact:epsilon-T}, we choose a connected set of $12\epsilon s$ faces of $\lceil \epsilon^{-1}\rceil \T$ and there are $C_0^{12\epsilon s}$ ways to do it, with $C_0$ the growth constant for animals in the triangular lattice $\T$. The total area of this set in $\T$ is thus at most $12 s / \epsilon$. 
Then we choose the centers of all the {\tt bad} bubbles: there are $\kappa \leq \epsilon s$ of them because of \cref{fact:bad-bubble-size} so there are fewer than $\sum_\kappa\binom{12 s/\epsilon}{\kappa} \le \exp ( C \epsilon s \log (1/\epsilon) )$ possible choices at this step. Similarly, using \cref{eq:good-bubble-size}, we can choose the area covered by {\tt good} bubbles and an \SOS configuration for each of them with at most $\binom{12 s/\epsilon}{\epsilon^{1/3} s} C^{\epsilon^{1/3} s}$ choices in total. Finally, we choose the {\tt bad} bubbles given their centers. Here the center of $\sB_i$ is a triangle of $\Upsilon(\cS_i)$ chosen by a fixed rule, and we let $o_i$ be the face of $\varphi$ projecting onto it. A {\tt bad} bubble $\sB_i$ is determined by the corresponding height function, which belongs to $\bigcup_{\cS\ni o_i}\cH^{\mathrm{ex}}_\cS$, the union being over the connected sets $\cS\ni o_i$ of $a_i:=|\cS_i|$ faces; note that the floor $\cS_i$ is not determined by the center. Choosing first the sizes $a_i$, with $\sum_ia_i\leq s$, in at most $\binom{s+\epsilon s}{\epsilon s}$ ways, \cref{prop:alg} shows that there are at most $\exp(C\epsilon^{1/3}\log(1/\epsilon)s)$ possible choices for the {\tt bad} bubbles in total.
Overall, the number of bubble groups we enumerated over is at most
\[
C_0^{12\epsilon s} \exp ( 2 C \epsilon^{1/3} s \log (1/\epsilon)  ) C^{\epsilon^{1/3} s} \tbinom{s+\epsilon s}{\epsilon s}\exp( C \epsilon^{1/3}\log(1/\epsilon) s) \leq \exp( C'\epsilon^{1/4} s)\,;
\]
 thus, at most $\exp[c\epsilon^{1/4} s]$ bubble groups $\fB$ have 
$\beta H_{h}(\fB) - \alpha G^\gr_{h, \varphi}(\fB) \leq  \epsilon^{1/3}  (\alpha \epsilon^{1/3} \wedge \beta\epsilon ) s$ (our criterion as per \cref{eq:good-bubble-grps} for an exceptional bubble group $\fB$), as required.

It now remains to address the exceptional bubble groups $\fB$ via the potential $\sV(\fB)$. 
Recall that $\sV(\fB)=\sum_{i=1}^m \ff(S_i)$ where $(S_i)_{i=1}^m$ (for some $m\geq 1$) are the connected components of $\psi^\sqcup\setminus h$ that intersect $\Upsilon(\fB)$. By \cref{eq:V(fB)-pinning-alt},
 \begin{align*}\sum_{i=1}^m |S_i|&= G^\gr_{h,\varphi}(\fB) + \sum_{\sB\in\fB}|\varphi\cap\sB| = G^\gr_{h,\varphi}(\fB) + \sum_{\sB\in\fB}(|\sB| - H_h(\sB))/2 \\ 
&\geq -\frac{\beta H_h(\fB)-\alpha G^\gr_{h,\varphi}(\fB)}{\alpha\wedge\beta} + \frac1{10}|\Upsilon(\fB)|\,,\end{align*}
where we used \cref{eq:bubble-H} in the first line and \cref{eq:area-of-bubble-grp} in the second one.  The numerator in the first term is (by definition) at most $\epsilon^{2/3}(\alpha\wedge \beta)|\Upsilon(\fB)|$ for every bubble group $\fB$ satisfying \cref{eq:good-bubble-grps}. For every $\epsilon\leq \frac1{150}$, we thus have 
\[ \sum_{i=1}^m |S_i| \geq \big(-\frac1{28}+\frac1{10}\big)|\Upsilon(\fB)| > \frac1{16}|\Upsilon(\fB)|\,.\]
Since $\psi^\sqcup$ maximizes $|h\cap\psi|$, the restriction of $h$ to $\Upsilon(S_i)$ cannot be a tiling: replacing $\psi^\sqcup$ there by $h$ would increase the overlap. Thus each component contributes at least one excess face, and $H_h(\fB)\geq m\geq1$. For $a>0$ with $\ff(S)\geq a(|S|-\sM_0)$, the preceding area bound gives
\[
\beta H_h(\fB)+\lambda\sV(\fB)
\geq(\beta-\lambda a\sM_0)m+\frac{\lambda a}{16}|\Upsilon(\fB)|
\geq\frac\beta2+\frac{\lambda a}{16}|\Upsilon(\fB)|
\]
for $\beta\geq2\lambda a\sM_0$. Take $a=1$ under \cref{eq:gen-potential-weaker}. Under \cref{eq:gen-potential}, any $a>0$ is allowed by choosing $\sM_0$ large enough; taking $a=16C^\star/\lambda$ proves the stronger bound.
\end{proof}

Towards an application for the dynamics that will appear in the proof of \cref{thm:pi}, we need an analogous statement for pairs $(S, \{\fB_i\})$ where $S$ is a connected set of faces of $\lceil \frac{1}{\gamma}\rceil \T$ that intersects each of the bubble groups in $\{\fB_i\}$.
\begin{lemma}\label{lem:approx-update}
 There exists an absolute constant $C>0$ such that the following holds for every $\varphi$, $\epsilon, \gamma>0$ and $s>C$. There are at most $\exp[C(\epsilon^{1/4} + \gamma^2) s]$
 pairs $(S, \{\fB_i\})$, where $S$ is a connected set obtained as the union of triangles of $\lceil \frac1\gamma\rceil\T$ and $\{\fB_i\}$ is a collection of disjoint bubble groups such that $\Upsilon( \fB_i ) \cap S \neq \emptyset$ for all $i$,
 with $s = |S|\;\vee\; \sum_i |\Upsilon(\fB_i)|$ and containing a fixed point $o$,
 so that
 \[
 \sum_i \left(\beta H_{h}(\fB_i) - \alpha G^\gr_{h, \varphi}(\fB_i)\right) \leq  \epsilon^{2/3} (\alpha \epsilon^{1/3} \wedge \beta\epsilon ) \sum_i |\Upsilon(\fB_i)|\,.
 \]
 In addition, if $\epsilon \leq \frac1{512}$ then for all such exceptional $(S,\{\fB_i\})$ we have
 \[
 \sum_i \Big(\beta H_h(\fB_i)+\lambda \sV(\fB_i) \Big)\geq \frac{\beta}2 + \frac{\lambda}{20} \sum_i |\Upsilon(\fB_i)| \,.
 \]
 Furthermore, under the stronger condition in \cref{eq:gen-potential} on $\sV$, for any $C^\star >0$ the lower bound in the last inequality may be changed to $\beta/2 + C^\star\sum_i |\Upsilon(\fB_i)|$ by taking $\beta$ large enough.
 \end{lemma}
 \begin{proof}
 The proof is completely analogous to \cref{lem:approx-bubble-grp} so we only give a sketch. The total area of bubble groups that are {\tt good} in the sense of \cref{lem:approx-bubble-grp} is at most $C\epsilon^{1/3}\sum_i|\Upsilon(\fB_i)|$ while {\tt bad} bubble groups in the sense of \cref{lem:approx-bubble-grp} provide very little entropy. The proof is then identical to \cref{lem:approx-bubble-grp} simply replacing bubbles by bubble groups, with the additional step of a final enumeration over $S$. It is there (using \cref{fact:epsilon-T}) where the $\exp[\gamma^2 s]$ term appears. 

 For the second assertion, denote $\beta H_h(\fB)-\alpha G^\gr_{h,\varphi}(\fB)$ by $F(\fB)$ in this proof for brevity. Our assumption  $\sum_i F(\fB_i) \leq \epsilon^{2/3}(\alpha\epsilon^{1/3}\wedge\beta\epsilon)M$, where $M:=\sum_i|\Upsilon(\fB_i)|$, implies that 
 \[ \sum_i |\Upsilon(\fB_i)|\one_{\{F(\fB_i)< 8\epsilon^{2/3}(\alpha\epsilon^{1/3}\wedge\beta\epsilon)|\Upsilon(\fB_i)|\}} > \tfrac78 M\,.\]
 By \cref{lem:approx-bubble-grp}, if $\epsilon\leq \frac1{150}$ (and $\beta>2\lambda\sM_0$) then $\beta H_h(\fB)+\lambda \sV(\fB) \geq \beta/2 + \lambda|\Upsilon(\fB_i)|/16$ holds for every $\fB_i$ with $F(\fB_i) \leq \epsilon^{1/3}(\alpha\epsilon^{1/3}\wedge \beta\epsilon)|\Upsilon(\fB_i)|$; in particular, this holds for every $\fB_i$ that contributes to the above sum (as $\epsilon\leq\frac1{512}$).
 So, the restriction of $\sum_i (\beta H_h(\fB_i)+\lambda\sV(\fB_i))$ to these $\fB_i$'s already suffices for the sought lower bound on $\sum_i\left(\beta H_h(\fB_i)+\lambda\sV(\fB_i)\right)$, as $\frac1{16}\cdot \frac78M>M/20$.

Finally, for the statement pertaining to the stronger condition in \cref{eq:gen-potential}, fix $C^\star>0$. Appealing to the stronger bound provided by \cref{lem:approx-bubble-grp} for this potential, the prefactor $\lambda/16$ may be replaced by $\frac87 C^\star$ when $\beta$ is large enough. This leads to a final lower bound of $\beta /2 + C^\star s$, as required.
 \end{proof}

\section{Local decomposition of \texorpdfstring{$\pi_{\varphi, \hat \beta}$}{\textpi}}\label{sec:pi}
Our goal in this section is to decompose $\int \pi_{\varphi,\hat\beta}(\cdot)\d\hat \beta$ as follows.

\begin{theorem}\label{thm:pi-str}
Consider potentials $\sV$ as in  \cref{def:pinning-V}. For any $C^\star > 0$ there exists $C>0$ such that the following holds. If $0<\lambda < \frac{1}{C}$ and $\alpha\wedge \beta> C \lambda^{-21}$, then there exist functions $\fg^\pi_r$ for $r=2^k$ with $k=0,1,\ldots$, such that for every $N$,
\begin{align}\label{eq:orig-prop3}
\sup_\varphi \bigg| \int_{\beta}^\infty \pi_{\varphi,\hat\beta}(H_h)\d \hat\beta - \sum_{x\in\T_N} \sum_{\substack{0 \leq r < N/2\\ r=2^k}} \fg^\pi_r(\varphi\restriction_{ B(x,r)})
\bigg|&\leq C e^{-\beta/2- C^\star N}
\,,
\end{align}
and for every integer $r=2^k$ ($k\geq 0$) one has $\|\fg_r^\pi\|_\infty \leq C e^{-\beta/2-C^\star r}$.
\end{theorem}

\begin{theorem}\label{thm:pi}
For potentials $\sV$ satisfying the weaker condition in  \cref{eq:gen-potential-weaker}, there exists an absolute constant $C>0$ such that, if $0<\lambda<\frac1{C}$ and $\alpha\wedge \beta\geq C\lambda^{-30}+\sM_0$ then
the conclusion of \cref{thm:pi-str} holds true with $C^\star$ replaced by $\lambda^3/(C\sM_0)$.
\end{theorem}

Recalling the definition of $\pi_{\varphi,\hat\beta}$ from \cref{eq:pi-def}, thanks to \cref{eq:bubble-group-H} we have that
\begin{equation}\label{eq:pi-simplified}
\pi_{\varphi,\hat\beta}(h) \propto \exp\bigg[ \sum_{\substack{\fB\,:\; \fB\text{ is an}\\ \text{$(h,\varphi)$-bubble group}}} \!\!\!\!\!\Big[-\hat\beta H(\fB) + \alpha G^\gr(\fB)-\log Z_\mu^\infty(\fB)-\lambda\sV(\fB)\Big]-\int_\alpha^\infty \mu_{h,\hat\alpha}(\overline G_h)\d\hat\alpha\bigg]\,.
\end{equation}
As was the case for the measures $\mu$ and $\nu$, we analyze the measure $\pi$ by defining an appropriate Gibbs sampler, and establishing that (a) it contracts w.r.t.\ an appropriate metric on \SOS configurations, and therefore mixes rapidly; and (b) its speed of propagating information is uniformly bounded. As mentioned in  \cref{sec:sketch-step-3}, each of these tasks is significantly more challenging than the analogs for $\mu$ and $\nu$ because the effects of the term $\int \mu_{h,\hat\alpha}(\overline G_h)\d\hat\alpha$ are hard to control.

More precisely, we will see (see, e.g., the proof of \cref{prop:bound_muh2}) that the \SOS configuration at two different positions may interact strongly through this term when there exists a dense set of bubble groups connecting them. This introduces two difficulties: First, a large dense set comprised of small bubble groups $\{ \fB_i\}$  will be a bottleneck for any ``local'' dynamics since in such a case the interaction through the integral might dominate the other terms for each individual $\fB_i$, stopping us from adding or removing them one by one. That is why the dynamics in \cref{sec:glauber-pi} will be defined with possibly large single updates: it is designed to allow such a bottleneck to be removed in a single step.
Second, at the technical level, we will need to measure the local density in our bounds on the integral, hence the notion of enclosure and fenced set defined below.

\subsection{Bounds on the integral interaction}\label{sec:bound_integral}

Let $\rho$ denote the \emph{density} of a set $S \subset \T$ with respect to an \SOS configuration $h$:
\begin{equation}
    \label{eq:def-bubble-density}
\rho(S, h) = \frac{1}{|S|} \sum_{\fB \in h} |\Upsilon(\fB)| \one_{\{\Upsilon(\fB) \cap S \neq \emptyset\}} \, .
\end{equation}

\begin{definition}[Enclosure]\label{def:bubble-grp-enclose}
    Let $h$ be an \SOS configuration and let $\fB_0$ be an $(h,\varphi)$-bubble group. An \emph{enclosure} $R_0$ of $\fB_0$ with density $\rho_0$, which unless specified otherwise we take as $\rho_0=\frac1{10}$, is a connected set $R_0$ satisfying:
    \begin{enumerate}[label=(\arabic*)]
    \item \label{it-R-1} $R_0 \supseteq \Upsilon(\fB_0)$.
    \item \label{it-R-2} No $(h,\varphi)$-bubble group $\fB$ has $\Upsilon( \fB) \cap \partial R_0 \neq \emptyset$.
    \item \label{it-R-3} It has density (as per \cref{eq:def-bubble-density}) at least $\rho_0$. %
        \item \label{it-R-4} There is no set $R$ with $|R| >  |R_0|$ that satisfies \cref{it-R-1,it-R-2,it-R-3}.    
    \end{enumerate}
    Similarly, an enclosure of a set of bubble groups $\{ \fB^0_i \}$ is a set satisfying the above condition and such that each of its connected components contains at least one of the $\Upsilon(\fB^0_i)$.
\end{definition}
\begin{remark}\label{rem:enclosure-exists}
 Every bubble group $\fB_0$ has an enclosure $R_0$ as above: Indeed, $\Upsilon(\fB_0)$ trivially satisfies \cref{it-R-1,it-R-2} while \cref{it-R-3} is given by \cref{eq:area-of-bubble-grp}; thus, the collection of sets $R$ satisfying \cref{it-R-1,it-R-2,it-R-3} is nonempty and any element of maximal size can be chosen as $R_0$. We emphasize that, typically, a bubble group will admit many different enclosures.
\end{remark}

\begin{definition}[Fenced set]\label{def:fenced-set}
    Let $h$ be an \SOS configuration. A set $R$ is said to be \emph{fenced} with density $\rho_0$, which unless specified otherwise we take as $\rho_0=\frac1{10}$, if no $(h,\varphi)$-bubble group $\fB$ has $\Upsilon( \fB) \cap \partial R \neq \emptyset$ and if every connected set $S\subseteq \T_N\setminus R$ but with $\partial S \cap \partial R \neq \emptyset$ has $\rho(S, h) < \rho_0$.
\end{definition}
An enclosure of density $\rho_0<1$ is fenced with density $\rho_0/(1-\rho_0)$. Indeed, if a connected exterior set $S$ touching it meets groups of total projected area $M\geq\rho_0|S|/(1-\rho_0)$, adjoining $S$ and closing under those groups adds at most $|S|+M$ triangles and at least $M$ occupied triangles. This preserves density at least $\rho_0$, contradicting maximality. Thus enclosures of densities $1/10$ and $3/20$ are fenced at densities $1/9$ and $3/17$, respectively; both are below $1/5$. Being a fenced set is a decreasing condition on the set of bubble groups of $h$.

Our first step is to control the effect of adding or removing a set of bubble groups on $\int \mu_{h,\hat\alpha}(\overline G_h)\d\hat\alpha$.
\begin{proposition}\label{prop:bound_muh}
    For all $\rho_0 < \frac{1}{5}$, there exists $C$ such that the following holds. Let $\hat h$ be an \SOS configuration, let $\{\fB_i\}$ be a collection of bubble groups of $\hat h$ and let $h$ be the configuration obtained by removing the $\fB_i$ from $\hat h$, i.e., $h = \hat h \xor (\bigcup_i \bigcup_{\sB \in \fB_i} \sB)$. If $R$ is a fenced set with density $\rho_0$ with respect to $\hat h$ (as per \cref{def:fenced-set}) such that $\bigcup_i \Upsilon(\fB_i) \subset R$, then
    \begin{equation}\label{eq:integral_bound_muh}
    \Big| \int_{\alpha}^\infty \big(\mu_{\hat h,\hat\alpha}(\overline G_{\hat h}) - \mu_{h,\hat\alpha}(\overline G_h)\big)\d \hat\alpha \Big| \leq C |R|\,.        
    \end{equation}
\end{proposition}
\begin{proof}
    Note that for any $h$, denoting by $\{\sB_i\}$ a consistent collection of $\varphi \xor \psi$ bubbles,
    \begin{align*}
    \log Z_\mu^\alpha &\geq \log \Big(\!\sum_{\{\sB_i \} \subset R} e^{- \alpha \sum_i \overline{G}_h(\sB_i)} \Big)+ \log \Big(\!\sum_{\{\sB_j \} \subset R^c} e^{- \alpha\sum_j \overline{G}_h(\sB_j)} \Big) \\
    \log Z_\mu^\alpha &\leq 
     \log \Big(\!\sum_{\{\sB_i \} \subset R} \! e^{- \alpha \sum_i \overline{G}_h(\sB_i)} \Big) + \log\Big(\!\sum_{\{\sB_j \} : \sB_j \cap \partial R \neq \emptyset} \! e^{- \alpha \sum_j \overline{G}_h(\sB_j)} \Big) + \log\Big(\!\sum_{\{\sB_k \} \subset R^c} \! e^{- \alpha \sum_k \overline{G}_h(\sB_k)} \Big).
    \end{align*}
    Indeed, since the compatibility condition between bubbles is just that they cannot intersect, the first inequality amounts to restricting the partition function to configuration $\psi$ with $\psi = \varphi$ on $\partial R$ and the second is over-counting by ignoring the compatibility condition between the three sets of bubbles. Applying \cref{lem:grimmett} to all terms, we get
    \begin{align*}
\int_\alpha^\infty  \!\!\! \mu_{h,\hat\alpha} (\overline G_h( R) \mid \psi = \varphi \text{ on } \partial R) \d \hat{\alpha} &+ \log Z_{\mu}^\infty( R) + \int_\alpha^\infty \!\!\! \mu_{h,\hat\alpha} ( \overline{G}_h( R^c) \mid \psi = \varphi \text{ on } \partial R)\d \hat{\alpha} + \log Z_{\mu}^\infty( R^c) \\
&\leq \int_\alpha^\infty \mu_{h,\hat\alpha} (\overline{G}_h)\d \hat{\alpha} + \log Z_{\mu}^\infty \leq \\
\int_\alpha^\infty \!\!\!
\mu_{h,\hat\alpha} (\overline G_h( R) \mid \psi = \varphi \text{ on } \partial R)\d \hat{\alpha} &+ \log Z_{\mu}^\infty( R) + \int_\alpha^\infty  \!\!\! \mu_{h,\hat\alpha} ( \overline{G}_h( R^c) \mid \psi = \varphi \text{ on } \partial R)\d \hat{\alpha} + \log Z_{\mu}^\infty( R^c) \\ &\qquad + \int_\alpha^\infty  \!\!\!\mu_{h,\partial R, \hat\alpha}( \overline G_h( \partial R))\d \hat{\alpha} + \log Z_\partial^\infty\, ,
    \end{align*}
    where the measure $\mu_{h,\partial R, \hat\alpha}$ corresponds to the middle term in the upper bound, i.e., it is supported on collections of disjoint bubble groups all of which intersect $\partial R$ and has a Hamiltonian $\overline{G}_h$.
    Indeed, thanks to the fact that the energy $\overline{G}$ is zero-range, the measure obtained when applying \cref{lem:grimmett} just to the partition function of configurations inside $R$ coincides with the conditional measure $\mu$ given $\psi = \varphi \text{ on } \partial R$.

    In the expression above, note further that since $\partial R$ does not intersect any bubble group, we get $ \log Z_{\mu}^\infty = \log Z_{\mu}^\infty( R) +  \log Z_{\mu}^\infty( R^c)$ and $ \log Z_\partial^\infty = 0$. Also, the terms outside of $R$ are equal for $h$ and $\hat{h}$. Overall, we therefore have
    \begin{multline*}
     \int_{\alpha}^\infty \mu_{\hat{h},\hat\alpha} ( \overline{G}_{\hat{h}}( R) )\d \hat{\alpha} -  \int_{\alpha}^\infty \mu_{h,\hat\alpha} ( \overline{G}_{h}( R) )\d \hat{\alpha} - \int_{\alpha}^\infty \mu_{h, \partial R,\hat\alpha} (\overline{G}_{h, \hat{\alpha}}( \partial R) )\d \hat{\alpha}  \\
        \leq  \int_{\alpha}^\infty \mu_{\hat{h}, \hat \alpha} ( \overline{G}_{\hat h} )\d \hat{\alpha}- \int_{\alpha}^\infty \mu_{h, \hat \alpha}( \overline{G}_{ h} )\d \hat{\alpha} \leq    \\
        \int_{\alpha}^\infty \mu_{\hat{h}, \hat{\alpha}} ( \overline{G}_{\hat{h}}( R) )\d \hat{\alpha} -  \int_{\alpha}^\infty \mu_{h, \hat{\alpha}} ( \overline{G}_{h}( R) )\d \hat{\alpha} + \int_{\alpha}^\infty \mu_{\hat{h}, \partial R, \hat{\alpha}} (\overline{G}_{h}( \partial R) )\d \hat{\alpha}\,,
    \end{multline*}
    and applying \cref{lem:grimmett} in the opposite direction
    \begin{multline*}
       \log Z_{\mu_{\hat{h}}}^\alpha( R)-  \log Z_{\mu_{\hat{h}}}^\infty( R) -  \log Z_{\mu_{h}}^\alpha( R) +  \log Z_{\mu_{h}}^\infty( R) -  \log Z_{\mu_{h, \partial}}^\alpha(\partial R) \\
        \leq  \int_{\alpha}^\infty\mu_{\hat{h}, \hat \alpha} ( \overline{G}_{\hat h} )\d \hat{\alpha}- \int_{\alpha}^\infty \mu_{h, \hat \alpha}( \overline{G}_{ h} )\d \hat{\alpha} \leq  \\
               \log Z_{\mu_{\hat{h}}}^\alpha( R)-  \log Z_{\mu_{\hat{h}}}^\infty( R) -  \log Z_{\mu_{h}}^\alpha( R) +  \log Z_{\mu_{h}}^\infty( R) +  \log Z_{\mu_{\hat{h}, \partial}}^\alpha(\partial R)\,. 
    \end{multline*}    
    All the $\log Z (R)$ terms are bounded by $C_0 |R|$ because they count tilings of $R$.
    For the $\log Z_{\mu_{h,\partial}}^\alpha( \partial R)$ and $\log Z_{\mu_{\hat h,\partial}}^\alpha( \partial R)$ terms, first note that by definition of a fenced set there is no dense set of bubble groups whose size is of order $R$ except for $R$ itself. Hence, in $\mu_{h, \partial R}$ every collection of $(\psi, \varphi)$-bubbles with size larger than $R$ has a $\overline{G}$-cost comparable to its size, so their collective contribution to $Z( \partial R)$ is at most  $O(|\partial R|)$. Therefore, one can restrict the sum to collections with size $O(R)$, which again shows that $\log Z (\partial R)$ is at most by $C |R|$ for some $C$. 
    This yields \cref{eq:integral_bound_muh}.
\end{proof}

The above proposition will be useful to understand the dynamics close to a region of interest where we do not expect the true value of the interactions to be particularly small. 
The following bound will refine that bound in terms of a single enclosure $R_0$ and the sizes of the bubble groups as opposed to the more complicated fenced sets.
\begin{lemma}\label{lem:short_range_integral}
    There exists a constant $C$ such that the following holds. Fix $h$ an \SOS configuration, let $\fB^0$ be a bubble group that can be added to $h$ and let $h'$ be the corresponding configuration. 
Let~$R_0$ be an enclosure of $\fB^0$ in $h'$.    
    Further fix a set $\{ \fB^1_i\}$ of bubble groups which can be added in both $h$ and $h'$, denoting the resulting configurations by $\hat h, \hat h'$.
    Let $S$ be a connected set intersecting each~$\Upsilon( \fB^1_i)$. If $\dist( \bigcup_i \Upsilon(\fB^1_i),R_0) \leq 40(|R_0| + \sum |\Upsilon( \fB^1_i)|)$, then
        \[
    \Big| \int_{\alpha}^\infty \big(\mu_{\hat h',\hat\alpha}(\overline G_{\hat h'}) - \mu_{h',\hat\alpha}(\overline G_{h'})\big)\d \hat\alpha \Big| \leq C( |S|+ \sum_i |\Upsilon( \fB_i^1)| + |R_0|)\,,
    \]
    and similarly for the pair $(h, \hat h)$.
\end{lemma}
\begin{proof}
    Let $R_1$ be an enclosure for the $\fB^1_i$ in configuration $\hat h'$ with density $\frac{3}{20}$. It is a fenced set with density $3/17$ in any of $h, h', \hat h ,\hat h'$ and therefore by \cref{prop:bound_muh} it is enough to bound its size.
    
    Suppose by contradiction that $|R_1| \geq 1000( |S| + \sum_i |\Upsilon( \fB_i)| +|R_0|)$. Let $R'$ be obtained as $(R_1 \cup S) \setminus R_0$ to which we add a minimal path to $\partial R_0$ and a path on the outer boundary of $R_0$. Since each connected component of $R_1$ contains one of the $\Upsilon(\fB_i^1)$ and $S$ intersects all of them, $R_1 \cup S$ is connected and therefore $R'$ must also be connected since any connection in $R_1 \cup S$ going through $R_0$ can now go through the outer boundary path.
    By definition $\rho(R_1, \hat h') \geq \frac{3}{20}$ so 
    \[
    |R'| \rho(R', h) \geq |R_1| \frac{3}{20} - \sum_i |\Upsilon( \fB^1_i)| - |R_0| \, ,
    \]
    while 
    \begin{align*}
    |R'|& \leq |R_1| + |S|+ \dist\Big( \bigcup \Upsilon(\fB_i^1),R_0\Big) + 11|R_0| \\
    &\leq |R_1|+ |S| +51 |R_0| + 40 \sum | \Upsilon( \fB_i^1) | \, .
    \end{align*}
    Combining the two bound, we see that 
    \begin{align*}
        \rho(R', h) & \geq \Big(\frac{3}{20} - \frac{\sum_i |\Upsilon( \fB^1_i)| + |R_0|}{|R_1|} \Big) \Big( 1 - \frac{|S|}{|R_1|} - 51 \frac{\sum_i |\Upsilon( \fB^1_i)| + |R_0|}{|R_1|}\Big) \\
        & \geq \frac{3}{20} - 10 \frac{|S|+\sum_i |\Upsilon( \fB^1_i)| + |R_0|}{|R_1|} > \frac{1}{9} \, ,
    \end{align*}
where the last inequality used that $|R_1|\geq1000(|S|+\sum_i |\Upsilon(\fB_i^1)|+|R_0|)$.
    This contradicts the fact that $R_0$ is fenced with density $1/9$ in $h$, as it is an enclosure of density $1/10$ in $h'$.
\end{proof}

We also need a statement controlling the interaction between far away bubble groups. The first step is to show a form of exponential decay in the measure $\mu_h$:
\begin{proposition}\label{prop:coupling_muh}
    For all $\rho < \frac{1}{5}$, there exists $C(\rho)$ so that the following holds for all $\hat \alpha$ large enough. Fix $h'$ and an $(h',\varphi)$-bubble group $\fB_0$. Let $h=h'\xor\bigcup_{\sB\in \fB_0}\sB$. Let $ R^\ddagger_0$ be a fenced set with density $\rho$ containing~$\fB_0$.
    Then there is a coupling $\mu_{h, h', \hat \alpha}$ of
    $\mu_{h,\hat\alpha}, \mu_{h',\hat\alpha}$ 
    so that for 
   all $f$,
    \[
    \mu_{h, h', \hat \alpha} (f \in \Upsilon( \psi \xor \psi') ) \leq C e^{- \frac{1}{2}\hat \alpha \dist(f, R_0^\ddagger)/C }\,.
    \]
\end{proposition}
\begin{proof}
We define the coupling as follows. Sample the surfaces
\[ \psi \sim \mu_{h,\hat\alpha}\qquad\mbox{and}\qquad
\hat \psi \sim \mu_{h',\hat\alpha} 
\]
independently, and let $A = \Upsilon( \psi \xor \hat\psi)$. Define the surface $\psi'$ by \begin{enumerate}[(i)]
    \item 
$ \psi '= \hat\psi$ outside $A$ and on any connected component of $A$ intersecting $R_0^\ddagger$;
\item $ \psi' = \psi$ on any other connected component of $ A$. 
\end{enumerate}
By \cref{lem:zero-range-G-bar}, we see that $\hat\psi$ and $\psi'$ have the same law, and consequently  $\mu_{h, h', \hat \alpha} (f \in \Upsilon( \psi' \xor \psi) ) $ is bounded by the probability that there exists a connected component of $A$ connecting $f$ to $R_0^\ddagger$. (Note that the zero-range interactions of $\mu$ are crucial to making this switching legal.) 
Consider the part of this component outside of $R_0^\ddagger$; it cannot be denser than $\rho < \frac{1}{5}$ by the fenced-set hypothesis. By \cref{eq:V(fB)-pinning-alt}, the probability that at least $\frac12$ of it is covered by bubbles in either $ \varphi\xor\psi$ or $ \varphi\xor \hat \psi$ is therefore bounded by
\[ C_0e^{- \frac{1}{2}(\hat \alpha (\frac15-\rho)  - C_0)\dist(f, R_0^\ddagger) }\,,\] with the absolute constant $C_0$ accounting for the enumeration over all possible sets of bubbles (via a Peierls argument, done for $\psi$ under $\mu_{h,\hat\alpha}$ and separately for $\hat\psi$ under $\mu_{h',\hat\alpha}$).
\end{proof}

We can now control the interaction between two far away bubble groups via the following.
\begin{lemma}\label{lem:pre-bound-muh-d}
For all $\rho < \frac{1}{6}$, there exists $C(\rho)$ so that the following holds for every $ \alpha$ large enough.
 Let $h$ be an \SOS configuration and let $\fB_0, \fB_1$ be two bubble groups that do not intersect and which can both be added to $h$. 
 As before, let \[ h'=h\xor\bigcup_{\sB\in \fB_0}\sB\quad,\quad \hat h' = h \xor \bigcup_{\sB\in \fB_0\cup \fB_1}\sB \quad,\quad\hat h = h \xor \bigcup_{\sB\in\fB_1}\sB\,.\]
 Let $\hat R_0, \hat R_1$ be enclosures with density $\rho$ of $\fB_0$ and $\fB_1$  w.r.t.\ $\hat h'$ as per \cref{def:bubble-grp-enclose}. Then 
 \[
     \bigg|\int_{\alpha}^\infty \Big(\mu_{\hat h',\hat\alpha}( \overline{G}_{\hat h'}) -\mu_{h',\hat\alpha}( \overline{G}_{h'}) - \mu_{\hat h, \hat \alpha}( \overline{G}_{\hat h}) +\mu_{h,\hat\alpha}( \overline{G}_{h})\Big) \d \hat\alpha\bigg| \leq C ( |\hat R_0| \wedge |\hat R_1| )e^{- \alpha \dist( \hat R_0, \hat R_1)/C} \,.
\]
   Moreover, if $( \fB^0_i)_{1 \leq i \leq m_0}$ and $( \fB^1_i)_{1 \leq i \leq m_1}$ are two families of bubble groups 
   that can be simultaneously added (together) to~$h$, then
    \[
     \bigg|\int_{\alpha}^\infty \Big(\mu_{\hat h',\hat\alpha}( \overline{G}_{\hat h'}) -\mu_{h',\hat\alpha}( \overline{G}_{h'}) - \mu_{\hat h, \hat \alpha}( \overline{G}_{\hat h}) +\mu_{h,\hat\alpha}( \overline{G}_{h})\Big) \d \hat\alpha\bigg| \leq C \sum_{i, j} ( |\hat R^0_i|\wedge |\hat R^1_j|)   e^{- \alpha \dist(\hat R^0_i, \hat R^1_j)/C} \,,
    \]
    where the $\hat R^{0}_i$, $\hat R^1_j$ are enclosures of the $\fB^0_i$ and $\fB^1_j$ respectively.
\end{lemma}
\begin{proof}
We start with the first case involving only two bubble groups. The general idea is to prove this result via an application of \cref{prop:coupling_muh}, creating a separate coupling for each face $f\in\T_N$ that would show that $f$ has an exponentially small contribution, as it cannot be close simultaneously to $\hat R_0$ and $\hat R_1$. 

The enclosures are fenced with density $\rho/(1-\rho)<1/5$. Since this condition in \cref{prop:coupling_muh} is monotone on the set of bubble groups, we can apply \cref{prop:coupling_muh} to both the pair $(h, h')$ and $(\hat h , \hat h')$ using the same set $\hat R_0$. For the same reason, if we add $\fB_1$, we can still use \cref{prop:coupling_muh} using $\hat R_1$. Overall, we see that we can define couplings such that
    \begin{align*}
    \mu_{h, h', \hat \alpha}( f \in \Upsilon (\psi \xor  \psi') )  &\leq C\, e^{- \hat \alpha \,\dist(f, \hat R_0) /C } \,,\\  
    \mu_{ \hat h, \hat h', \hat \alpha}( f \in \Upsilon (\hat \psi \xor \hat \psi') ) &\leq C \,e^{- \hat \alpha \dist(f, \hat R_0) /C } \,,\\
       \mu_{h,  \hat h, \hat \alpha}( f \in \Upsilon ( \psi \xor \hat \psi) ) &\leq C\, e^{- \hat \alpha  \dist(f, \hat R_1)/C} \,,\\
        \mu_{ h',  \hat h', \hat \alpha}( f \in \Upsilon ( \psi' \xor \hat \psi') ) &\leq C \, e^{-\hat \alpha \dist(f, \hat R_1 )/C } \,.
\end{align*}

Now, we can decompose the integral over the contribution of each face $x$ of $\T_N$. Let 
\[
\overline{G}_h^x (\psi) = \sum_{f \in \Upsilon^{-1} (x) } \overline{g}_h(f) \one_{\{f \in \psi\}} \, 
\] using the decomposition of $\overline G_h$ from \cref{lem:zero-range-G-bar} (as usual, for brevity, in what follows we omit the parameter $\psi$ from $\overline G_h^x$, as well as the corresponding $\psi',\hat\psi,\hat\psi'$ from $\overline G^x_{h'},\overline G^x_{\hat h},\overline G^x_{\hat h'}$, respectively). Then
\begin{align*}
     \bigg|\int_{\alpha}^\infty & \Big(\mu_{\hat h',\hat\alpha}( \overline{G}_{\hat h'}) -\mu_{h',\hat\alpha}( \overline{G}_{h'}) - \mu_{\hat h, \hat \alpha}( \overline{G}_{\hat h}) +\mu_{h,\hat\alpha}( \overline{G}_{h})\Big) \d \hat\alpha\bigg|  \\
     & \leq \sum_{x \in \T_N} \bigg| \int_{\alpha}^\infty \Big(\mu_{\hat h',\hat\alpha}( \overline{G}^x_{\hat h'}) -\mu_{h',\hat\alpha}( \overline{G}^x_{h'}) - \mu_{\hat h, \hat \alpha}( \overline{G}^x_{\hat h}) +\mu_{h,\hat\alpha}( \overline{G}^x_{h})\Big) \d \hat\alpha\bigg| \\
     & \leq 2 \sum_{x \in \T_N} \int_{\alpha}^\infty \min \Big( \mu_{h, h', \hat \alpha}( x \in \Upsilon (\psi' \xor \psi) ) + \mu_{\hat h, \hat h', \hat \alpha}( x \in \Upsilon (\hat \psi' \xor \hat \psi) ) + 2 \one_{\{x \in 
     \Upsilon(\fB_0)\}} , \,  \\
     & \qquad \qquad \qquad  \quad \quad\;\;\;\mu_{h, \hat h, \hat \alpha}( x \in \Upsilon (\psi \xor \hat \psi) ) + \mu_{h', \hat h', \hat \alpha}( x \in \Upsilon (\hat \psi' \xor  \psi') )+  2\one_{\{x \in \Upsilon( \fB_1)\}} \Big) \, \d \hat\alpha.
\end{align*}
Indeed, for the last line, note that we can bound the contribution of a single $x$ using two possible pairings of the four terms and we are free to use the better one depending on $x$. Furthermore, when $x \notin \Upsilon( \fB_0)$ then $\overline{G}_h^x = \overline{G}_{h'}^x $, and we bound the difference in the expectation of this variable under the two measures $\mu_{h,\hat\alpha}$ and $\mu_{h',\hat\alpha}$,
recalling that $|\overline G^x_h| \leq 1$ always holds, via  \[\int_\alpha^\infty | \mu_{h', \hat \alpha}( \overline G_{h'}^x ) - \mu_{h, \hat \alpha}( \overline G_{h}^x ) |\d\hat\alpha \leq 2 \int_\alpha^\infty \mu_{h, h', \hat \alpha}( x \in \Upsilon (\psi \xor \psi')) \d \hat \alpha\,.\] 
For $x\in\Upsilon(\fB_0)$ we used the trivial upper bound $|\overline G_h^x| + |\overline G_{h'}^x|\leq 2$.
The bounds for the other terms are similar (replacing $\fB_0$ by $\fB_1$).

Plugging in the bounds on $\mu_{h, h', \hat \alpha}$ and using the fact that the indicators $\one_{\{x\in\Upsilon(\fB_0)\}}$ and $\one_{\{x\in\Upsilon(\fB_1)\}}$ cannot hold simultaneously, we see that overall
 \begin{align*}
     \bigg|\int_{\alpha}^\infty \big(\mu_{\hat h',\hat\alpha}( \overline{G}_{\hat h'}) -\mu_{h',\hat\alpha}( \overline{G}_{h'}) - \mu_{\hat h, \hat \alpha}( \overline{G}_{\hat h}) +\mu_{h,\hat\alpha}( \overline{G}_{h})\big) \d \hat\alpha\bigg| & \leq C \sum_{f \in \T_N} e^{- \alpha (\dist(f, \hat R_0) \vee \dist(f, \hat R_1))/C} \\
     & \leq C ( |\hat R_0| \wedge |\hat R_1|) e^{-  \alpha \dist(\hat R_0, \hat R_1)/C}\,,
\end{align*}
concluding the proof of the first part. (To see the last inequality, suppose without loss of generality that $|\hat R_0| \leq |\hat R_1|$, and take $d = \dist(\hat R_0,\hat R_1)$. The set of all $f$ at distance at least $d/2$ from $\hat R_0$ contributes at most $C|\hat R_0| e^{-(\alpha/C) d/2}$ by summability of the exponent; on the other hand, every other face $f$ must have distance at least $d/2$ from $\hat R_1$ and there are at most $C |\hat R_0| d^2$ such faces.)

For the case with multiple bubble groups, the desired result follows directly from the monotonicity of fenced sets in \cref{prop:coupling_muh} and the observation that
\begin{multline*}
\mu_{\hat h',\hat\alpha}( \overline{G}_{\hat h'}) -\mu_{h',\hat\alpha}( \overline{G}_{h'}) - \mu_{\hat h, \hat \alpha}( \overline{G}_{\hat h}) +\mu_{h,\hat\alpha}( \overline{G}_{h}) \\
= \sum_{i,j} \mu_{h_{i+1, j+1},\hat\alpha}( \overline{G}_{h_{i+1, j+1}}) -\mu_{h_{i+1, j},\hat\alpha}( \overline{G}_{h_{i+1, j}}) - \mu_{h_{i, j+1},\hat\alpha}( \overline{G}_{h_{i, j+1}}) + \mu_{h_{i, j},\hat\alpha}( \overline{G}_{h_{i, j}}) \, ,
\end{multline*}
with $h_{i,j}$ the configuration obtained by adding the bubble groups $\fB^0_1, \ldots, \fB^0_i$ and $\fB^1_1, \ldots, \fB^1_j$.
\end{proof}

As for the bound involving only a single difference, it will be useful to have a version involving simpler quantities than the $\hat R_i$ above.
\begin{lemma}\label{lem:bound_muh-d}
Let $h$, $\fB_0$, $\fB_1$ be an \SOS configuration and two bubble groups that can be added to it as above. Let $h', \hat h, \hat h'$ be obtained by adding $\fB_0$, $\fB_1$ or both to $h$.
If an enclosure $R_0$ of $\fB_0$ in the configuration $h'$ satisfies 
 \[ \fd := \dist\left( \Upsilon (\fB_1),  R_0\right) > 40 (|R_0| + |\Upsilon(\fB_1)|)\,,\]
    then, for an absolute constant $C_0$, 
     \[
     \bigg|\int_{\alpha}^\infty \Big(\mu_{\hat h',\hat\alpha}( \overline{G}_{\hat h'}) -\mu_{h',\hat\alpha}( \overline{G}_{h'}) - \mu_{\hat h, \hat \alpha}( \overline{G}_{\hat h}) +\mu_{h,\hat\alpha}( \overline{G}_{h})\Big) \d \hat\alpha\bigg| \leq C_0 e^{- \alpha \fd /C_0} \,.
    \] 
 Moreover, for two families of bubble groups $( \fB^0_i)_{1 \leq i \leq m_0}$ and $( \fB^1_i)_{1 \leq i \leq m_1}$, if there exist enclosures $R_i^0$ of the $\fB_i^0$ in $h'$ such that for all $i$
 \[ \dist\Big( \bigcup_j \Upsilon (\fB^1_j), R^0_i\Big) > 40 \Big(|R^0_i| + \sum_j |\Upsilon(\fB^1_j)|\Big)\,,\]
 then
 \[
     \bigg|\int_{\alpha}^\infty \Big(\mu_{\hat h',\hat\alpha}( \overline{G}_{\hat h'}) -\mu_{h',\hat\alpha}( \overline{G}_{h'}) - \mu_{\hat h, \hat \alpha}( \overline{G}_{\hat h}) +\mu_{h,\hat\alpha}( \overline{G}_{h})\Big) \d \hat\alpha\bigg| \leq C_0 \sum_i   e^{- \alpha \dist(R^0_i, \bigcup_j \Upsilon(\fB^1_j))/C_0} \,.
 \]
\end{lemma}
\begin{proof}
Again we start with the case involving only two bubble groups. 
Assume we have such an $R_0$. Let $\hat R_0, \hat R_1$ be enclosures of $\fB_0, \fB_1$ in $\hat h'$ but with density $3/20$, \cref{lem:pre-bound-muh-d} applies so what remains to do is to bound $\dist(\hat R_0, \hat R_1)$ and their sizes.

We first argue that
\[\dist(R_0, \hat R_1) \geq \fd/10\,.\]
Suppose by contradiction that $\dist(R_0, \hat R_1) \leq \fd/10$. Since $\dist( \Upsilon(\fB_1), R_0)
\geq \fd$, then
necessarily $|\hat R_1| \geq \frac{9}{10} \fd \geq 36 (| R_0| + |\Upsilon (\fB_1)|)$.
Let $U$ be the union of $R_0$, $\hat R_1$ and a shortest path between them, closed under the groups of $h'$ it meets. Before closing, count only the groups already contained in $R_0$ or $\hat R_1$. Writing $J:=|R_0\cap\hat R_1|$ and $M_1:=|\Upsilon(\fB_1)|$, their mass minus one tenth of the union's area is at least
\[
\frac1{20}|\hat R_1|-M_1-\frac9{10}J-\frac{\fd}{100}
\geq\left(\frac1{20}-\frac1{36}-\frac1{90}\right)|\hat R_1|>0.
\]
Closing adds occupied triangles and preserves this inequality. Since $U$ strictly contains $R_0$, this contradicts maximality of $R_0$.

The same argument rules out any connected set of density at least $3/20$ whose boundary crosses no group in $\hat h'$ containing both $\Upsilon(\fB_1)$ and a point of $R_0$. Thus $\hat R_0$ does not meet $\Upsilon(\fB_1)$, and has the same density in $h'$ and $\hat h'$. Maximality of $R_0$ gives $|\hat R_0|\leq|R_0|$. Since $\hat R_0$ is connected and meets $R_0$, it follows that
\[
\dist(\hat R_0,\hat R_1)\geq\fd/10-|\hat R_0|
\geq\fd/10-|R_0|\geq\fd/20.
\]
Now apply \cref{lem:pre-bound-muh-d} and absorb $|R_0|<\fd/40$ into the exponential.

For multiple groups, use their total projected mass in place of $M_1$ in the same union argument. For every $i$, this gives $|\hat R_i^0|\leq|R_i^0|$ and separation at least one twentieth of the original distance from $R_i^0$ to the added family. Apply the family version of \cref{lem:pre-bound-muh-d}; the number of groups in the added family is bounded by its total projected area and is absorbed into the exponential by the distance hypothesis.
\end{proof}

In the proof of \cref{thm:pi} we will compare the dynamics on two nested domains, whose interaction terms are computed with the measures $\mu_{D,h,\hat\alpha}$ and $\mu_{D',h,\hat\alpha}$ on tilings of $\Lambda^\varphi_D$ and of $\Lambda^\varphi_{D'}$ respectively (for a domain $D$, which is either $\T_N$ or a ball, $\mu_{D,h,\hat\alpha}$ is the measure $\propto e^{-\hat\alpha\overline G_h(\psi)}$ on tilings $\psi$ of $\Lambda^\varphi_D$). This requires the following analog of \cref{lem:bound_muh-d}, in which the role of the far-away bubble groups $\{\fB^1_j\}$ is played by the boundary $\partial D$. Unlike for the measure $\nu$ (cf.\ \cref{lem:nu-domain-change}), the estimate requires the region between the bubble groups and $\partial D$ to have low density, as the interaction propagates through dense regions.
\begin{lemma}\label{lem:bound_muh-domain}
For every $\rho<\frac15$ there exists $C_0=C_0(\rho)$ such that the following holds for every $\alpha$ large enough. Let $D\subseteq D'$ be two domains as above, let $h$ be an \SOS configuration, let $\{\fB_i\}$ be bubble groups with $\bigcup_i\Upsilon(\fB_i)\subseteq D$ which can be added to $h$, and let $\hat h$ be the resulting configuration. Let $R^\ddagger\subseteq D$ be a fenced set with density $\rho$ w.r.t.\ $\hat h$ (as per \cref{def:fenced-set}) containing $\bigcup_i\Upsilon(\fB_i)$, and let $0<\fd\leq\dist(R^\ddagger,\partial D)$ be such that every connected set $S$ with $|S|\geq\fd/4$ and $\dist(S,R^\ddagger)\leq\fd$ has $\rho(S,\hat h)<\rho$. Then
\[
\bigg|\int_\alpha^\infty\Big(\mu_{D',\hat h,\hat\alpha}(\overline G_{\hat h})-\mu_{D',h,\hat\alpha}(\overline G_h)-\mu_{D,\hat h,\hat\alpha}(\overline G_{\hat h})+\mu_{D,h,\hat\alpha}(\overline G_h)\Big)\d\hat\alpha\bigg|\leq C_0|R^\ddagger|e^{-\alpha\fd/C_0}\,.
\]
\end{lemma}
\begin{proof}
We follow the proof of \cref{lem:pre-bound-muh-d}, with the pairing of the four terms according to the domain replacing the pairing according to $\fB_1$. First, the proof of \cref{prop:coupling_muh} applies verbatim on the domain $D$ (or $D'$), as it only uses the zero-range property of $\overline G_h$ from \cref{lem:zero-range-G-bar} and a Peierls estimate outside the fenced set; since $R^\ddagger$ is fenced w.r.t.\ $\hat h$, hence also w.r.t.\ $h$, we obtain couplings $\mu_{D,h,\hat h,\hat\alpha}$ and $\mu_{D',h,\hat h,\hat\alpha}$ of the respective pairs of measures with
\[\mu_{D,h,\hat h,\hat\alpha}\big(f\in\Upsilon(\psi\xor\hat\psi)\big)\leq Ce^{-\hat\alpha\dist(f,R^\ddagger)/C}\,,\qquad \mu_{D',h,\hat h,\hat\alpha}\big(f\in\Upsilon(\psi'\xor\hat\psi')\big)\leq Ce^{-\hat\alpha\dist(f,R^\ddagger)/C}\,.\]
Second, for a fixed reference configuration $h$ (and likewise for $\hat h$) we couple $\mu_{D,h,\hat\alpha}$ and $\mu_{D',h,\hat\alpha}$ by a switching argument across $\partial D$: sample $\psi\sim\mu_{D,h,\hat\alpha}$ and $\psi'\sim\mu_{D',h,\hat\alpha}$ independently, let $A=\Upsilon(\psi\xor\psi')$, and let $\psi''$ be equal to $\psi'$ on every connected component of $A$ intersecting $\partial D$ or not contained in $D$, and equal to $\psi$ elsewhere. By \cref{lem:zero-range-G-bar}, $\psi''$ has the law $\mu_{D',h,\hat\alpha}$, and $\psi''\xor\psi$ is contained in the union of the components of $A$ connecting a point to $\partial D$. For $f$ with $\dist(f,R^\ddagger)\leq\fd/2$, a component of $A$ connecting $f$ to $\partial D$ contains a connected set $S$ of size at least $\fd/2$ with $\dist(S,R^\ddagger)\leq\fd$ (the portion of the component between distances $\fd/2$ and $\fd$ from $R^\ddagger$), on which $\rho(S,\hat h)<\rho$ by hypothesis. Arguing as in the proof of \cref{prop:coupling_muh} (a Peierls estimate for the bubbles of $\varphi\xor\psi$ and of $\varphi\xor\psi'$ covering at least half of $S$, whose energy is at least $(\frac15-\rho)|S|$ by \cref{eq:V(fB)-pinning-alt} as the bubble groups of $h$, resp.\ $\hat h$, have density less than $\rho$ on $S$), we get
\[ \mu\big(f\in\Upsilon(\psi''\xor\psi)\big)\leq Ce^{-\frac12(\hat\alpha(\frac15-\rho)-C)\fd/2}\qquad\mbox{whenever $\dist(f,R^\ddagger)\leq\fd/2$}\,.\]
Now decompose $\overline G_h=\sum_{x\in\T_N}\overline G^x_h$ as in the proof of \cref{lem:pre-bound-muh-d}, and bound the contribution of each $x$ by the better of the two pairings: pairing the terms with the same domain, the contribution of $x$ is at most $2\int_\alpha^\infty[\mu_{D,h,\hat h,\hat\alpha}(x\in\Upsilon(\psi\xor\hat\psi))+\mu_{D',h,\hat h,\hat\alpha}(x\in\Upsilon(\psi'\xor\hat\psi'))+2\one_{\{x\in\bigcup_i\Upsilon(\fB_i)\}}]\d\hat\alpha$ (recall $|\overline G^x_h|\leq1$, and $\overline G^x_h=\overline G^x_{\hat h}$ for $x\notin\bigcup_i\Upsilon(\fB_i)$); pairing instead the terms with the same reference configuration, the functions being integrated coincide, and the contribution of $x$ is at most $2\int_\alpha^\infty[\mu(x\in\Upsilon(\psi''\xor\psi))+\hat\mu(x\in\Upsilon(\hat\psi''\xor\hat\psi))]\d\hat\alpha$ with $\hat\mu$ the switching coupling for the reference $\hat h$. Using the first pairing for the faces with $\dist(x,R^\ddagger)>\fd/2$, which contribute at most $C|R^\ddagger|e^{-\alpha\fd/C}$ by summability of the exponential, and the second pairing for the at most $C|R^\ddagger|\fd^2$ faces with $\dist(x,R^\ddagger)\leq\fd/2$, we obtain the claim.
\end{proof}

Finally, since the term $\int_\alpha^\infty \mu_{h,\hat\alpha}(\overline{G}_h)\d\hat\alpha$  is always positive and since only dense regions of 
bubble groups contribute significantly to it, intuitively one could hope that it could be monotone under the deletion of  bubble groups. This does not seem to be correct but one can still in a sense ignore its effect, at the cost of losing the $G^\gr$ and $\log( Z_\mu^\infty)$ terms.
\begin{proposition}\label{prop:bound_muh2}
    Let $\varphi$ be a tiling of $\T_N$, and suppose $h,\hat h$ are two \SOS configurations where the $(\hat h,\varphi)$-bubbles consist of all the $(h,\varphi)$-bubbles in addition to $(\hat h,\varphi)$-bubble groups $\{\fB_i\}_{i=1}^m$. Then
    \[
    \frac{\pi_{\varphi, \hat \beta}(\hat h)}{\pi_{\varphi, \hat \beta}(h)} \leq \exp\bigg(\sum_{i=1}^m  \Big[- \hat \beta H (\fB_i) - \lambda \sV( \fB_i ) \Big]\bigg)\,.
    \]
\end{proposition}
\begin{proof}
It will be useful to use the expression for $G_{h,\varphi}(\psi)$ from \cref{eq:G-def-2}.
Note that for every tiling~$\psi$,
\[ (\varphi\xor\psi) \cap (\varphi\xor \hat h) = (\varphi\xor\psi)\cap\Big((\varphi\xor h)\uplus \bigcup_{i}\fB_i\Big)\,,\]
where $\bigcup_i \fB_i$ is short for $\bigcup_i \bigcup_{\sB\in\fB_i}\sB $, the union of bubbles in the bubble group $\fB_i$, which are disjoint to $\varphi\xor h$ by definition. In particular,
\[G_{\hat h,\varphi}(\psi) = G_{h,\varphi}(\psi) - \Big|(\varphi\xor\psi)\cap \big(\bigcup_{i}\fB_i\big)\Big|\,.\]
Since \cref{cor:psi-min-max-construction} implies that the minimizers of $G_{h,\varphi}$ are constructed independently between $(h,\varphi)$-bubble groups (which would \emph{not} hold at the level of bubbles!), it follows that
\[ G_{\hat h}^\gr = G_{h}^\gr + \sum_{i} G^\gr(\fB_i)\,,\]
and combining the last two displays we see that
\begin{align}
\overline G_{\hat h,\varphi}(\psi) &=
 \overline G_{ h,\varphi}(\psi)
 - \Big|(\varphi\xor\psi)\cap \big(\bigcup_{i}\fB_i\big)\Big| - \sum_i G^\gr(\fB_i)
\nonumber \\
& \leq
\overline G_{ h,\varphi}(\psi)
 - \sum_i G^\gr(\fB_i) \label{eq:G-hat-h-ub}\,.
\end{align}
In particular,
\[ Z_{\mu}^\alpha(\hat h) =\sum_\psi e^{-\alpha \overline G_{\hat h}(\psi)} \geq
Z_{\mu}^\alpha(h) e^{\alpha \sum_i G^\gr(\fB_i)}\,,
\]
and so, using \cref{lem:grimmett} to recover the log partition functions from the integrals, we have
\begin{align*}
\int_{\alpha}^\infty \mu_{\hat h,\hat\alpha}(\overline G_{\hat h})\d\hat\alpha - \int_{\alpha}^\infty \mu_{ h,\hat\alpha}(\overline G_{ h})\d\hat\alpha &=\Big(\log Z_\mu^\alpha(\hat h) - \log Z_\mu^\infty(\hat h)\Big)-\Big( \log Z_\mu^\alpha(h) - \log Z_\mu^\infty(h)\Big) \\
 &\geq \log Z_\mu^\infty(h) -\log Z_\mu^\infty(\hat h) + \alpha\sum_i G^\gr(\fB_i) \,.
\end{align*}
Finally, the aforementioned independence of the minimizers between bubble groups implies that
\[ Z_\mu^\infty(\hat h) = Z_\mu^\infty(h) \prod_i Z_\mu^\infty(\fB_i)\]
(again via \cref{cor:psi-min-max-construction} and the fact that we are looking at bubble groups rather than bubbles), which translates the last inequality into
\begin{align*}
\int_{\alpha}^\infty \mu_{\hat h,\hat\alpha}(\overline G_{\hat h})\d\hat\alpha - \int_{\alpha}^\infty \mu_{ h,\hat\alpha}(\overline G_{ h})\d\hat\alpha \geq \sum_i\Big[ \alpha G^\gr(\fB_i)-\log Z_\mu^\infty(\fB_i) \Big] \,.
\end{align*}
Consequently, we see from \cref{eq:pi-simplified} that
\begin{align*}
    \frac{\pi_{\varphi,\hat\beta}(\hat h)}
    {\pi_{\varphi,\hat\beta}(h)} &= 
    \exp\bigg[ \sum_{i} \Big[-\hat\beta H(\fB_i) + \alpha G^\gr(\fB_i)-\log Z_\mu^\infty(\fB_i)-\lambda\sV(\fB_i)\Big]
    \\ &\qquad\qquad -\bigg(\int_\alpha^\infty \mu_{\hat h,\hat\alpha}(\overline G_{\hat h})\d\hat\alpha-\int_\alpha^\infty \mu_{h,\hat\alpha}(\overline G_h)\d\hat\alpha\bigg)\bigg] \\
    &\leq 
\exp\bigg[ \sum_{i} \Big[-\hat\beta H(\fB_i)-\lambda\sV(\fB_i)\Big]
\bigg]\,, 
\end{align*}
as required.
\end{proof}

\subsection{Glauber dynamics on bubble groups for \texorpdfstring{$h$}{h}}\label{sec:glauber-pi}
Given a fixed reference tiling $\varphi$ of $\T_N$, define the following dynamics $(h_t)$ on SOS configurations, using a scale parameter 
\begin{equation}
\gamma := \lambda / C
\end{equation}
where $C$ is a large enough absolute constant (it will suffice to take it, e.g., as $10^5 C_*$ where $C_*$ is the enumeration constant on animals in the lattice $\T$).
\begin{enumerate}
\item \label{it:glauber-clocks-pi}
\begin{enumerate}[label=(\alph*), ref=(\theenumi\alph*)]\item \label{it:glauber-clocks-pi-S} Assign an independent rate-1 Poisson clock to every pair $(S, \{\fB_i \})$ , where $S$ is a finite connected set of triangles in the lifted lattice $\lceil \frac{1}{\gamma}\rceil\T$, taken modulo translations by periods of $\T_N$, and $\{\fB_i\}$ is a set of pairwise-disjoint bubble groups such that $\Upsilon(\fB_i)\cap S\neq\emptyset$ for every $i$, and $\sum_{i} |\Upsilon(\fB_i)| \geq \gamma|S|/2000$. \item
\label{it:glauber-clocks-pi-S-dagger} 
For $S^\dagger$ a single triangle in $\lceil \frac{1}{\gamma}\rceil\T$, assign a rate-1 clock to every nonempty finite family $(S^\dagger,\{\fB_i\})$ as above, without a density restriction. A label already covered by (a) receives only one clock.
\end{enumerate}
In intersection and distance conditions, $S$ means its projection to $\T_N$; $|S|$ always counts triangles in the lift. Distinct lifted sets index distinct clocks, even if their projections coincide; two labels with the same projection and the same family act identically, so the lifting merely multiplies the rate at which the projected update is attempted by the number of its lifts (the total rate of eligible updates remains finite, as the density condition bounds $|S|$ by $2000\sum_i|\Upsilon(\fB_i)|/\gamma$). For a restricted domain, only the bubble groups must be supported in that domain. This convention permits the animal extensions used in healing even at the largest scale, where they do not fit in $\T_N$. It preserves reversibility, and the enumeration over $S$ still costs $e^{C\gamma^2|S|}$ as in \cref{lem:approx-update}.

\item \label{it:glauber-rand-pi}%
If either (i) no current bubble projects onto $S$ in $h_t$ and adding every $\sB\in\bigcup\fB_i$ will result in a configuration $\hat h$ where $\{\fB_i\}$ are precisely the bubble groups with $\Upsilon(\fB_i)\cap S\neq \emptyset$, or (ii) the current $h_t$ is such an $\hat h$,
then let $\{h, \hat h\}$ denote the configurations $\{ h_t, h_t \xor (\bigcup_i \bigcup\{\sB\in\fB_i\})\}$ with $\hat h$ the configuration with the bubbles. The dynamics moves to $h$ with probability $w_h/(w_h+w_{\hat h})$ and otherwise it moves to $\hat h$, where
\begin{align*}
w_h &= \exp\bigg[ -\int_\alpha^\infty \mu_{h,\hat\alpha}(\overline G_h)\d\hat\alpha\bigg]\,,\\
w_{\hat h} &= \exp\bigg[ \sum_{i} \Big[-\hat\beta H(\fB_{i}) + \alpha G^\gr(\fB_{i})-\log Z_\mu^\infty(\fB_{i})-\lambda\sV(\fB_{i})\Big]-\int_\alpha^\infty \mu_{\hat h,\hat\alpha}(\overline G_{\hat h})\d\hat\alpha\bigg]\,.
\end{align*}
\end{enumerate}
Compared to the Glauber dynamics we considered for $\nu$, this is similar but instead of changing only a single bubble at a time we can now add or remove a dense set of bubble groups. This is still reversible for $\pi_{\varphi,\hat\beta}$ as per \cref{eq:pi-simplified}. For irreducibility, choose any occupied coarse triangle $S^\dagger$ and its entire family of intersecting groups. Its clock can remove that family with positive probability, irrespective of its density. Repeating this decreases the number of groups until $h=\varphi$; reversibility supplies the reverse paths.

Define $\dist_{\fB}(h,h')$ to be the minimal weight of a path in the graph where two configurations are adjacent if they differ on a single bubble group $\fB$ of the larger configuration, such an edge having weight $|\Upsilon(\fB)|$. (The weights, rather than the plain graph distance, are needed in the healing step of the proof of \cref{prop:pi-contraction} below, where an update may erase $\fB_0$ together with a sparse family of bubble groups in one copy only.)

We will refer to a pair $(S,\{\fB_i\})$ as in \cref{it:glauber-clocks-pi} as an \emph{update}, and define its \emph{size} as
\begin{equation}\label{eq:update-size}
\|(S,\{\fB_i\})\| := |S|\;\vee\;\sum_i|\Upsilon(\fB_i)|\;\vee\;\sum_i H(\fB_i)\,.
\end{equation}
Note that the number of updates with $|S|=s$ is infinite (even for a single triangle $S$, the bubble groups meeting it may be arbitrarily large), so every enumeration of updates below is in terms of their size.
\begin{claim}\label{clm:update-count}
There exists an absolute constant $C$ such that, for every $\varphi$, every $x\in\T_N$ and every $s\geq1$, the number of updates $(S,\{\fB_i\})$ of size at most $s$ with $x\in S$, each of whose bubbles contains a face of $\varphi$, is at most $e^{Cs}$. Moreover, every such update consists of at most $s/2$ bubble groups.
\end{claim}
(Every bubble of every \SOS configuration $h$ contains a face of $\varphi$, except in the degenerate case where $h$ shares no edge with $\varphi$ and $h$ itself is the unique bubble, whose energy exceeds $N^2$; thus the restriction in the claim is harmless.)
\begin{proof}
Let $(S,\{\fB_i\})$ be such an update and let $\sB_1,\ldots,\sB_m$ be all the bubbles of all the $\fB_i$. The faces of $\varphi$ in different bubbles are distinct, so $m\leq\frac12\sum_i|\Upsilon(\fB_i)|\leq s/2$ (each face of $\varphi$ in $\fB_i$ projects to two triangles of $\Upsilon(\fB_i)$), which gives the second assertion, and by \cref{eq:bubble-H}
\[ \sum_{j=1}^m|\sB_j| = \sum_{j=1}^m\big(2|\varphi\cap\sB_j|+H(\sB_j)\big)\leq \sum_i\Big(|\Upsilon(\fB_i)|+H(\fB_i)\Big)\leq 2s\,.\]
The set $U:=S\cup\bigcup_i\Upsilon(\fB_i)$ is a connected set of at most $2s$ triangles containing $x$ (each $\Upsilon(\fB_i)$ is connected and intersects $S$), and the $\fB_i$ are determined by the bubbles $\sB_1,\ldots,\sB_m$ (as the bubble groups of the configuration $\varphi\xor\bigcup_j\sB_j$). We can thus enumerate: $U$, as a lattice animal in $\T$ rooted at $x$ ($C_*^{2s}$ choices); the lifted coarse animal $S$ rooted at a fixed lift of $x$ ($e^{Cs}$ choices, allowing a polynomial factor for the choice of root); the set of faces of $\varphi$ projecting into $U$ that serve as roots, one per bubble ($2^{2s}$ choices, as there are at most $|U|$ such faces); the sizes $|\sB_1|,\ldots,|\sB_m|$ ($2^{2s}$ choices); and finally each bubble $\sB_j$ as a connected set of $|\sB_j|$ faces of $\Z^3$ containing its root ($C^{|\sB_j|}$ choices). Altogether this gives at most $e^{Cs}$ updates.
\end{proof}

\begin{proposition}\label{prop:pi-contraction}
The dynamics $(h_t)$ is contracting w.r.t.\ $\dist_\fB$; that is, if $\alpha\wedge  \hat\beta$ is large enough (independently of $\varphi$), then for every pair of initial states $(h_0,h'_0)$ for two instances of the chain, there exists a coupling of $(h_t,h'_t)$ such that,
\[\E [\dist_\fB(h_t,h'_t)]\leq e^{- t/4} \dist_\fB(h_0,h'_0)\,.\]
\end{proposition}

\begin{remark}\label{rk:dependency_beta_lambda}
    The proof of the proposition is where the strictest constraint on the relation between $\lambda$ and $\beta$ appears for the potentials of \cref{def:pinning-V}. As written, we need $\beta \lambda^{21}$ to be large enough: the prefactor in \cref{eq:S-infection-size} is $\lambda^{20}$, and the enumeration of the small updates in \cref{it:short-small} costs one more power of $\lambda$; but we made no effort to make this dependence optimal. Our techniques would not improve this  beyond a polynomial because this is the relation between energy and entropy for bad bubble groups.
\end{remark}

\begin{proof}[\emph{\textbf{Proof of \cref{prop:pi-contraction}}}]
    As in the two previous cases, it suffices to consider $h_0, h'_0$ that differ on a single bubble group $\fB_0$ and again we assume by symmetry that $h'_0$ contains $\fB_0$; the path coupling argument applies verbatim to the weighted metric $\dist_\fB$, so it suffices to exhibit a coupling under which $\frac{\d}{\d t}\E[\dist_\fB(h_t,h'_t)]\restriction_{t=0}\leq-\frac14|\Upsilon(\fB_0)|$. The coupling assigns to every update label the same clock and the same uniform variable in both copies, with a single exception, in the healing scenario (\cref{it:healing}) below, where a label that is eligible only in $h'$ is paired with a label that is eligible only in $h$ (both have rate $1$ and the pairing is injective, so this is a legitimate coupling). 

Let $R_0$ be an enclosure of $\fB_0$ (recall \cref{rem:enclosure-exists}), and let us first consider the scenario wherein
 \[ |R_0| > 12/\gamma \,.\]
We also collect for future reference the following conditions on an update $(S, \{ \fB_i \})$:
\begin{align}
&\mbox{[\emph{Healing size}]} \qquad\qquad\;\; |S| \geq \frac{8\cdot 10^8 C_1}{(\alpha \wedge \hat \beta) \gamma} |R_0|\, , \label{eq:S-healing-size}\\
&\mbox{[\emph{Infection size}]} \qquad\qquad \|(S,\{\fB_i\})\| \geq \frac{3 C_1}{(\alpha \wedge \hat \beta) \gamma^{20}} |R_0|=:\Sigma_0\,,
\label{eq:S-infection-size}\\
&\mbox{[\emph{Typicality}]}\; \sum_i \Big[\hat\beta H(\fB_{i}) - \alpha G^\gr(\fB_{i})+\log Z_\mu^\infty(\fB_{i})+\lambda\sV(\fB_{i})\Big] \geq 2 C_1 \Big( |S| + |R_0| + \sum_i |\Upsilon(\fB_i) |\Big)\,,
\label{eq:healing-bubble-grp-terms} \\
&\mbox{[\emph{Distance}]} \qquad\qquad\qquad \dist\left( \cup \Upsilon (\fB_i),  R_0\right) > 40 \Big(|R_0| + \sum_i |\Upsilon(\fB_i)|\Big)\,,
\label{eq:S-distance}
\end{align}
with $C_1$ chosen to be the maximum between all the constants from \cref{lem:short_range_integral,lem:bound_muh-d} and eight times the constant $C$ of \cref{clm:update-count} (so that there are at most $e^{C_1 s/8}$ updates of size at most $s$ containing a given point).

We will say that a bubble group is \emph{typical} if it satisfies the typicality condition and \emph{atypical} otherwise and that it is at long-range or short-range if it satisfies or not the distance condition.

Note that the distance condition is exactly the same as in \cref{lem:short_range_integral,lem:bound_muh-d}. Let us also argue that if an update $(S, \{\fB_i\})$ satisfies the distance condition, then it can either be added in both $h_0, h_0'$ or in neither of them. If the $\fB_i$ can be added to $h_0'$, then by monotonicity (\cref{cor:mon-equiv-rel}) it can also be added to $h_0$, even without using the distance condition. Now assume that it can be added to $h_0$, by the distance condition the bubbles of the $\fB_i$ cannot intersect the ones of $\fB_0$ so we can still define configurations $\hat h_0, \hat h_0'$ by adding these bubbles to $h_0$ and $h_0'$ respectively. If there is a path in $\T$ separating $\Upsilon(\fB_0)$ from $\bigcup \Upsilon(\fB_i)$ staying outside of the projection of the bubble groups of both $\hat h_0$ and $\hat h_0'$, then by the local consistency of bubble groups (as in the proof of \cref{cor:mon-equiv-rel}), we find that the $\fB_i$ are bubble groups of $\hat h_0'$. If there is no such path, then in particular there is a set connecting $\bigcup \Upsilon(\fB_i)$ to $R_0$ with bubble group density more than $1/2$ in one of $\hat h_0$ or $\hat h_0'$ but this contradicts the distance assumption and the fact that $R_0$ is an enclosure.

Also, %
an atypical update $(S, \{ \fB_i \})$ that satisfies the healing size condition (or the infection size of course) must satisfy
\begin{align}
    \sum \Big(\hat\beta H(\fB_{i}) - \alpha G^\gr(\fB_{i})\Big) & \leq 2 C_1 \Big( |S| + |R_0| + \sum_i |\Upsilon(\fB_i) |\Big) \nonumber \\
    & \leq 2C_1 \Big(  \frac{2000}{\gamma} + \frac{2000}{\gamma} \frac{\gamma (\alpha \wedge \hat \beta) }{8\cdot 10^8 C_1} + 1\Big) \sum_i | \Upsilon( \fB_i)| \, ,
\end{align}
and therefore, if $\alpha \wedge \hat \beta$ is large enough, any atypical large update falls into the condition from \cref{lem:approx-update} for, say, $\epsilon = \frac1{512}$ (the dominant term is the middle one, involving $4\cdot 10^5 C_1$, which amounts to at most $\epsilon=\frac1{570}$ in the condition from \cref{lem:approx-update}) and thus satisfies
  \begin{equation}\label{eq:S-energy-atypical}
\sum_i \Big(\hat\beta H(\fB_{i}) +\lambda\sV(\fB_{i})\Big)\geq \hat\beta/2  +\frac{\lambda}{20}\sum_i |\Upsilon(\fB_i)|\,.        \end{equation}
Finally, an atypical update satisfying the infection size condition has, again assuming $\alpha \wedge \hat \beta$ is large enough, and now also assuming $\gamma$ is small enough,
\begin{align}
    \sum \Big(\hat\beta H(\fB_{i}) - \alpha G^\gr(\fB_{i})\Big) & \leq 2 C_1 \Big( \frac{2000}{\gamma} + \frac{2000}{\gamma} \frac{(\alpha \wedge \hat \beta) \gamma^{20} }{3 C_1} + 1\Big)  \sum |\Upsilon( \fB_i)| \nonumber \\
    & \leq \epsilon^{5/3} ( \alpha \wedge \hat \beta) \sum |\Upsilon( \fB_i)| 
\end{align}
for $\epsilon := \gamma^9$ qualifying for an application of \cref{lem:approx-update} (with room to spare, as the dominant term in the right hand of the first line featured $\gamma^{19}$). We will apply said lemma later, with the enumeration over such cases. Note that in the last display we used that, for an atypical update satisfying the infection size condition, the size \cref{eq:update-size} is $|S|\vee\sum_i|\Upsilon(\fB_i)|$, whence $|R_0|\leq \frac{(\alpha\wedge\hat\beta)\gamma^{20}}{3C_1}\big(|S|\vee\sum_i|\Upsilon(\fB_i)|\big)$. Indeed, atypicality gives $\hat\beta\sum_iH(\fB_i)\leq2C_1(|S|+|R_0|+\sum_i|\Upsilon(\fB_i)|)$, so if the maximum in \cref{eq:update-size} were attained by $\eta:=\sum_iH(\fB_i)$ we would get from \cref{eq:S-infection-size} that $\hat\beta\eta\leq2C_1(2+\frac{(\alpha\wedge\hat\beta)\gamma^{20}}{3C_1})\eta<\hat\beta\eta$, a contradiction (as $\alpha\wedge\hat\beta\leq\hat\beta$, $\gamma$ is small and $\hat\beta>8C_1$).

    Again, we can identify four scenarios:
    \begin{enumerate}[(1)]
        \item {}[\emph{blocked move}] We select $(S,\{\fB_i\})$ where $\{\fB_i\}$ can neither be added nor removed in both configurations. In that case nothing happens and the distance is unchanged. Note that this happens with a very high rate since the number of pairs $(S, \{\fB_i \})$ grows exponentially with the size of $S$ but almost none of these will be compatible with the current configuration.
        
        \item {}[\emph{healing}] \label{it:healing} We select $S$ satisfying the healing size condition (\cref{eq:S-healing-size}) and $S \cap \Upsilon( \fB_0) \neq \emptyset$, together with bubble groups $\{\fB_i\}, \{\fB_i'\}$ which are exactly the one whose projection intersect $S$ in $h$ and $h'$ respectively. In particular, $\fB_0$ appears in the set $\{ \fB'\}$ but not in~$\{ \fB\}$. The label $(S,\{\fB'_i\})$ is a clock (its density is verified below), eligible in $h'$ and not in $h$. The label $(S,\{\fB_i\})$ is eligible in $h$ (the $\fB_i$ are bubble groups of $h$ by \cref{cor:mon-equiv-rel}) and not in $h'$, provided that it is a clock, i.e., that $\{\fB_i\}$ is nonempty and $\sum_i|\Upsilon(\fB_i)|\geq\gamma|S|/2000$; in that case we pair the two labels, and otherwise $(S,\{\fB'_i\})$ is run in $h'$ with no move in $h$. Write $M:=\sum_i|\Upsilon(\fB_i)|$ and $M':=\sum_i|\Upsilon(\fB'_i)|=M+|\Upsilon(\fB_0)|$, and let $p$ (resp.\ $p'$) be the probability that $h$ (resp.\ $h'$) keeps its family when its label rings. If $(S,\{\fB'_i\})$ is typical then $p'\leq e^{-C_1(|S|+|R_0|+M')}$, since the contribution of the integral to the weights is at most $C_1(|R_0|+|S|+M')$ by \cref{lem:short_range_integral} ($C_1 \geq C_0$ by construction and the distance condition is automatically verified since $S$ intersects $\fB_0$), and if it is atypical then $p'\leq e^{-\hat\beta/2-\frac{\lambda}{20}M'}$ by \cref{eq:S-energy-atypical} and \cref{prop:bound_muh2}. The same bounds hold for $p$ with $M$ in place of $M'$ whenever $(S,\{\fB_i\})$ is a clock. Consequently, once $\hat\beta$ is large enough, $p'M'\leq\frac1{40}$ and $pM\leq\frac1{40}$.
\begin{itemize}
\item If the two labels are paired, then with probability at least $1-p-p'$ both copies remove their families and the distance drops to $0$; otherwise the distance is at most $M'$ if $h'$ kept its family and at most $M$ if $h$ did. The contribution to $\frac{\d}{\d t}\E[\dist_\fB]$ is thus at most $-|\Upsilon(\fB_0)|+p'M'+pM\leq-|\Upsilon(\fB_0)|+\frac1{20}\leq-\frac25|\Upsilon(\fB_0)|$.
\item If $\{\fB_i\}=\emptyset$ then $\{\fB'_i\}=\{\fB_0\}$, and $h'$ removes $\fB_0$ with probability $1-p'\geq\frac12$, decreasing the distance to $0$.
\item If $\{\fB_i\}\neq\emptyset$ is sparse, i.e., $0<M<\gamma|S|/2000$, then with probability $1-p'\geq\frac12$ the copy $h'$ removes all of $\{\fB'_i\}$ whereas $h$ is unchanged, after which the two copies differ exactly by the family $\{\fB_i\}$ and their distance is at most $M$. For the sets $S$ considered below we will have $M\leq\frac1{12}|\Upsilon(\fB_0)|$, so this case contributes at most $-(1-p')|\Upsilon(\fB_0)|+\frac1{12}|\Upsilon(\fB_0)|\leq-\frac25|\Upsilon(\fB_0)|$.
\end{itemize}
Overall, any set $S$ as above contributes at least $\frac25|\Upsilon(\fB_0)|$ to the rate of decrease of the expected distance. It remains to show that there are enough such sets. Take $R_0$ as above. 
        By \cref{fact:epsilon-T}, the number of triangles of $\lceil \frac{1}{\gamma}\rceil \T$ it intersects is at most $12 \lceil \gamma |R_0|\rceil $, in particular any set $S$ containing all of these along with $\gamma |R_0|$ other triangles taken arbitrarily has $|S|\leq(13\gamma|R_0|+12)\lceil\frac1\gamma\rceil^2\leq15|R_0|/\gamma$ (recall that $\gamma|R_0|>12$), and will satisfy $\sum_i |\Upsilon(\fB'_i)| \geq |R_0|/10 \geq \gamma |S|/2000$ (every bubble group of $h'$ intersecting $R_0$ is contained in $R_0\subseteq S$ by \cref{it-R-2}, hence belongs to $\{\fB'_i\}$, and the first inequality is the density of the enclosure $R_0$). The same two facts give the bound $M\leq\frac1{12}|\Upsilon(\fB_0)|$ used in the sparse case above: there $|\Upsilon(\fB_0)|+M\geq|R_0|/10$ and $M<\gamma|S|/2000\leq\frac{3}{400}|R_0|$, whence $|\Upsilon(\fB_0)|\geq\frac{37}{400}|R_0|$ and $M<\frac{3}{37}|\Upsilon(\fB_0)|$. There are at least $e^{C_* \gamma |R_0| - o(\gamma|R_0|)}$ many such choices, where $C_*$ is the enumeration constant for animals in $\T$ (choose an animal in $\lceil\frac{1}{\gamma}\rceil\T$ and adjoin it to $S$, say by gluing its leftmost (then bottom most) point to the rightmost (then top most) point of the original $S$).
        
        \item {}[\emph{long-range infection}] \label{it:long-range-infect}  We select  $(S,\{\fB_i\}_{i\geq 1})$ or $(S^\dagger,\{\fB_i\})$ satisfying the distance condition (\cref{eq:S-distance}). As noted above, if it can be added to one configuration it can be added to the other one too.
        Write
        \[ p = \frac{w_{\hat h}}{w_{h}+w_{\hat h}} \quad,\quad 
        p' = \frac{w_{\hat h' }}{w_{h'}+w_{\hat h'}} \,.
        \]
        By the exact same argument leading to \cref{eq:p-p'} for the dynamics $(\eta_t)$ considered there, one has
        \[
        |p - p'| \leq  \bigg| \int_{\alpha}^\infty \mu_h + \mu_{\hat h'} - \mu_{h'} - \mu_{\hat h}  \bigg|\,.
        \]
        We can apply \cref{lem:bound_muh-d} to bound the integral term by $C \exp \big( - \alpha \dist( R_0, \bigcup_i \Upsilon( \fB_i) )/C \big)$. Given the distance condition (which also bounds the change in the distance, $\sum_i|\Upsilon(\fB_i)|$, by $\frac1{40}\dist(R_0,\bigcup_i\Upsilon(\fB_i))$), if $\alpha$ is large enough the total contribution of that case is at most a constant.
        
        \item {}[\emph{short-range infection}] 
        \label{it:short-range-infect}
        We select $(S ,\{ \fB_i \})$ or $(S^\dagger, \fB)$ which does not satisfy the distance condition (\cref{eq:S-distance}) and does not fit in the healing case (\cref{it:healing}). Consider three cases.  
        \begin{enumerate}[(a), ref=(\theenumi\alph*)]
            \item {}[\emph{small}] \label{it:short-small} The update does not satisfy \cref{eq:S-infection-size}, i.e., $\|(S,\{\fB_i\})\|<\Sigma_0$ (in particular this case does not exist if $R_0$ is too small). In that case, we will not try to control the probability and we bound the distance assuming it increases by the total area $\sum_i|\Upsilon(\fB_i)|$ of the bubble groups involved, which is at most $\Sigma_0$ by \cref{eq:update-size}. Since the distance condition fails, $S\cup\bigcup_i\Upsilon(\fB_i)$ (a connected set of size at most $2\Sigma_0$) is within distance $40(|R_0|+\Sigma_0)$ of $R_0$, so by \cref{clm:update-count} the number of such updates is at most $C|R_0|(|R_0|+\Sigma_0)^2e^{C\Sigma_0}$. Its total contribution is at most
\[\cI_0:=\one_{\{\Sigma_0>\lceil\gamma^{-1}\rceil^2\}}\,C|R_0|(|R_0|+\Sigma_0)^2\Sigma_0e^{C\Sigma_0}\,;\]
the indicator follows from $|S|\geq\lceil\gamma^{-1}\rceil^2$. If $\cI_0>0$ and $\alpha\wedge\hat\beta\geq C'\gamma^{-21}$, then $|R_0|\geq c\gamma^{-3}$, so the logarithm of the polynomial prefactor is $o(\gamma|R_0|)$ uniformly. Taking $C'$ large makes $C\Sigma_0\leq\frac14 C_*\gamma|R_0|$, and hence $\cI_0$ is dominated by the healing rate. This gives the hypothesis $\alpha\wedge\beta\geq C\lambda^{-21}$ in \cref{thm:pi-str}.
            
            \item{} [\emph{large typical}] \label{it:short-large-typ} 
            The update satisfies \cref{eq:S-infection-size,eq:healing-bubble-grp-terms}. Note that $S$ does not intersect $\fB_0$ otherwise it would instead appear in the ``healing'' term and therefore the same set $\{\fB_i\}$ rings for both $h$ and $h'$. Write $\sigma:=\|(S,\{\fB_i\})\|\geq\Sigma_0$ and $\eta:=\sum_iH(\fB_i)$. Suppose first that the $\fB_i$ can be removed from both; as in the healing case, by \cref{lem:short_range_integral} (which bounds the contribution of the integral to the ratio of the weights by $C_1(|S|+\sum_i|\Upsilon(\fB_i)|+|R_0|)$), the probability to keep them in either of the two configurations is at most
\begin{multline*}
\exp\Big[-\sum_i\Big(\hat\beta H(\fB_i)-\alpha G^\gr(\fB_i)+\log Z^\infty_\mu(\fB_i)+\lambda\sV(\fB_i)\Big)
+C_1\Big(|S|+\sum_i|\Upsilon(\fB_i)|+|R_0|\Big)\Big]\\
\leq e^{-\frac{C_1}{2}(|S|+|R_0|+\sum_i|\Upsilon(\fB_i)|)-\frac{\hat\beta}{4}\eta}\,.
\end{multline*}
where we used that the first sum is at least $\frac34\cdot2C_1(|S|+|R_0|+\sum_i|\Upsilon(\fB_i)|)+\frac14\hat\beta\eta$ by \cref{eq:healing-bubble-grp-terms} and the nonnegativity of $-G^\gr$, $\log Z^\infty_\mu$ and $\sV$. Similarly, if the $\fB_i$ can be added, this happens in either configuration with probability at most the same quantity. In both cases the right-hand side is at most $e^{-\frac{C_1}{4}(\sigma+|R_0|)}$ (as $\hat\beta\geq C_1$), and the distance increases by at most $\sum_i|\Upsilon(\fB_i)|\leq\sigma$. Overall, since removing the same set synchronously in both $h$ and $h'$ does not increase the distance\footnote{It is possible to find pairs $(S, \{ \fB_i \} )$ where the $\fB_i$ can be added to $h$ and not $h'$ but with $S \cap \Upsilon(\fB_0) = \emptyset$ if one of the $\fB_i$ intersects $\fB_0$ and $S$ or if adding bubbles make some bubble group merge. Even in that case, bubble groups need to be added to one configuration to increase the distance.}, any update in this case contributes at most $2\sigma e^{- \frac{C_1}{4}(\sigma+|R_0|)}$ to the rate of increase of the distance. By \cref{clm:update-count} and the choice of $C_1$, the number of updates of size $\sigma$ within distance $40(|R_0|+\sigma)$ of $R_0$ is at most $C|R_0|(|R_0|+\sigma)^2e^{C_1\sigma/8}$, and the total contribution is bounded by
\[
C|R_0|e^{-C_1|R_0|/4}\sum_{\sigma\geq\max(\Sigma_0,\lceil\gamma^{-1}\rceil^2)}\sigma(|R_0|+\sigma)^2e^{-C_1\sigma/8}\leq c|R_0|\,.
\]
Here the retained enclosure decay absorbs the polynomial in $|R_0|$, and $c$ is as small as needed by our choices of constants.
            
            \item {} [\emph{large atypical}]
            \label{it:short-large-atyp}
            The update satisfies \cref{eq:S-infection-size} but not \cref{eq:healing-bubble-grp-terms}. As in the healing case, when this happens the pair $(S, \{ \fB_i \})$ must be in the exceptional set of \cref{lem:approx-update} with now, as noted after \cref{eq:S-infection-size}, $\epsilon = \gamma^9$ (recall also that the size of such an update is $|S|\vee\sum_i|\Upsilon(\fB_i)|$, as noted there).
            By \cref{prop:bound_muh2} the probability to add or keep the $\fB_i$ is bounded by $e^{- \hat \beta/2 - \frac{\lambda}{20} \sum_i |\Upsilon(\fB_i)|}$ and it is the only case where the distance can increase, by at most $\sum_i|\Upsilon(\fB_i)|$. Recall that $\sum |\Upsilon(\fB_i)| \geq \gamma |S|/2000$.
            The number of such pairs with $\max( |S|, \sum |\Upsilon( \fB_i)|) = s$ is bounded by $e^{C( \gamma^{9/4} + \gamma^2) s}$ by \cref{lem:approx-update}. Altogether, the contribution of this case is at most  
            \begin{equation}\label{eq:large-atypical-contribution}
            s\exp\Big[ C \gamma^2s - \frac{\lambda \gamma}{40000} s - \hat \beta/2\Big]
            \end{equation}
            which is summable if $\lambda/\gamma$ is large enough so again that case contributes at most $c |R_0|$ with a $c$ which can be made arbitrarily small by taking $\hat \beta$ large enough.
        \end{enumerate}
    \end{enumerate}

Overall, we see that, starting $(h_t,h_t')$ at $(h,h')$ which differ by a single bubble group $\fB_0$ and recalling we assumed first that the associated enclosure $R_0$ has size at least $12/\gamma$, $\frac{\d}{\d t} \E [ \dist_\fB(h_t,h_t')] \restriction_{t=0} $ is at most \[ -\frac25|\Upsilon(\fB_0)|\, e^{C_* \gamma |R_0| - o(\gamma|R_0|)} + \cI_0 + c |R_0|   \leq  -\frac14|\Upsilon(\fB_0)| \,,\]
where the middle term is the contribution of \cref{it:short-small}, which is dominated by the healing term as $\alpha\wedge\hat\beta\geq C'\gamma^{-21}$.

\begin{remark}
The potential $\sV$ was utilized in \cref{it:healing,it:short-large-atyp} (\emph{healing} and \emph{large atypical short-range infection}). In the former, it plays a weak role: we just need to kill the total area of the bubble groups intersecting $S$ (which is the potential damage when we heal in one copy and not in the other), i.e., in that case any $\sV\gtrsim\log\sum_i|\Upsilon(\fB_i)|$ would work. It is however important to have $\sV \gtrsim |S|$ in the latter (see \cref{eq:large-atypical-contribution}).
\end{remark}

If $|R_0| < 12/\gamma$, then assuming $\alpha$ and $\hat\beta$ are large enough, $C_0 |R_0| \leq \frac12 (\alpha \wedge \hat \beta)$ with $C_0$ from \cref{prop:bound_muh}. Since any bubble group $\fB$ has either $H(\fB) \geq 1$ or $h\restriction_{\Upsilon(\fB)}$ is a tiling whence $G^\gr(\fB) = -|\Upsilon(h\setminus\varphi)\cap\Upsilon(\fB)| \leq - 1$, the effect of the integral cannot dominate the other terms so the proof is quite straightforward. We keep the nomenclature of the previous case.

If $S^\dagger$ is a coarse triangle intersecting $\Upsilon(\fB_0)$, let $\{\fB'_i\}$ and $\{\fB_i\}$ be the families of bubble groups intersecting $S^\dagger$ in $h'$ and in $h$, respectively, and let $M:=\sum_i|\Upsilon(\fB'_i)|$. Closing $R_0\cup S^\dagger$ under the bubble groups intersecting it or its boundary yields a connected set $R\supseteq\Upsilon(\fB_0)$ with no bubble group intersecting $\partial R$, in which every triangle outside $R_0\cup S^\dagger$ is occupied. If $R\neq R_0$ then $\rho(R,h')<\frac1{10}$ by the maximality \cref{it-R-4} of $R_0$, so the number of occupied triangles of $R$, which is at least $M$, is less than $\frac19(|R_0|+|S^\dagger|)$; otherwise $M\leq|R_0|$. In either case $M\leq3\gamma^{-2}$. Both labels $(S^\dagger,\{\fB'_i\})$ and $(S^\dagger,\{\fB_i\})$ are clocks (the latter if $\{\fB_i\}\neq\emptyset$), eligible in $h'$ and in $h$ respectively, and we pair them as in \cref{it:healing}. The energy of each family is at least $\alpha\wedge\hat\beta$, while the contribution of the integral to the weights is at most $C_1(|S^\dagger|+|R_0|+M)\leq7C_1\gamma^{-2}\leq\frac12(\alpha\wedge\hat\beta)$ by \cref{lem:short_range_integral}, so each copy keeps its family with probability at most $e^{-\frac12(\alpha\wedge\hat\beta)}$. Arguing as in \cref{it:healing}, this label contributes at most $-|\Upsilon(\fB_0)|+2Me^{-\frac12(\alpha\wedge\hat\beta)}\leq-\frac12|\Upsilon(\fB_0)|$ to the rate of change of the expected distance, and so does the label $(S^\dagger,\{\fB_0\})$ when $\{\fB_i\}=\emptyset$. There exists at least one such $S^\dagger$, so overall the healing rate is at least $\frac12|\Upsilon(\fB_0)|$.
The long-range infection case is completely analogous. Finally, in the short-range infection, if $\alpha$ and $\hat \beta$ are large enough there are no small sets (indeed the condition on the size of small sets was designed to only include the cases where the integral might dominate the other terms). The case with large sets has to be slightly modified to include the updates $(S^\dagger,\{\fB_i\})$ which fail the distance condition: for a family of total area $M$ there are at most $C\gamma^2(|R_0|+M)^2$ such coarse triangles $S^\dagger$, the update increases the distance by at most $M$, and it does so with probability at most $2e^{-C_1(|S^\dagger|+|R_0|+M)}$ if it is typical and at most $2e^{-\hat\beta/2-\frac{\lambda}{20}M}$ if it is atypical (an atypical such update is exceptional in the sense of \cref{lem:approx-update} with $\epsilon=\gamma^9$, as $|S^\dagger|+|R_0|\leq3\gamma^{-2}$ and $\alpha\wedge\hat\beta\geq C\gamma^{-21}$), whereas the number of such updates of size $s$ is at most $e^{C_1s/8}$ in the former case and at most $e^{C(\gamma^{9/4}+\gamma^2)s}$ in the latter, by \cref{clm:update-count,lem:approx-update}. Summing over $M$ as in \cref{it:short-large-typ,it:short-large-atyp}, these updates contribute at most $c|R_0|$. The total rate of change of the distance is hence
\[
-\frac{1}{2}|\Upsilon(\fB_0)| + c|R_0|
\]
and, together with $|R_0| \leq 12/\gamma$, we see that if $\alpha, \hat \beta$ are large enough this is still less than $-\frac14|\Upsilon(\fB_0)|$, which concludes the proof.
\end{proof}

\subsection{Propagation of information} \label{sec:propagation-of-inf} To prove \cref{thm:pi} we must control the speed at which information propagates through the dynamics. This step will be considerably more delicate compared to the analysis of the dynamics for $\mu$ and $\nu$ studied in \cref{sec:2-out-of-3}, due to the potential emergence of large dense regions, encouraged by the long-range interaction of $\int \mu(\overline G_h)\d\hat\alpha$.

\begin{proof}[Proof of \cref{thm:pi}]
As in the proof of \cref{eq:orig-prop1,eq:orig-prop2} for $\mu$ and $\nu$, let $r=2^k$ with $k\geq 1$, and recall that $\Lambda_r^\varphi$ denotes $\Upsilon(\varphi\restriction_{B(o,r)})$, i.e., the tiles of $\varphi$ whose projection to $\T_N$ intersects $B(o,r)$. Denote by $\pi_r$ the measure defined as $\pi_{\varphi,\hat\beta}$ but on SOS configurations of $\Lambda_r^\varphi$, that is,
\[
 \pi_r(h)\propto \exp\bigg[-\hat\beta H_h + \alpha G_{h,\varphi}^\gr-\log Z_{\mu_{r,h}}^\infty-\lambda\sV(h)- \int_{\alpha}^\infty \mu_{r, h,\hat\alpha}(\overline{G}_h) \d \hat \alpha\bigg] \,,
\]
where $\mu_{r, h,\hat\alpha}$ is again defined only over tilings of $B(o,r)$.
Again, in case $r \geq N/2$, we replace $B(o,r)$ in the definition of $\pi_r$ by the full torus $\T_N$, that is, we take $\pi_r = \pi_{\varphi,\hat\beta}$. (With this definition, every $B(o,r)$ that is strictly contained in $\T_N$ is also simply connected.)

\subsubsection*{Constructing the local function}
With the above definition, denote by $\{\sB\in h\}$ for some bubble $\sB$ the event that $\sB$ appears in $h$ as a (complete) $(h,\varphi)$-bubble, and, for every bubble $\sB$ with $H(\sB)\geq1$, let
\[ 
f^\pi_{2r, \sB}(\varphi\restriction_{B(o,2r)}) := \int_{\beta}^\infty \big[\pi_{2r}(\sB\in h) - \pi_r(\sB\in h)\big]\d\hat\beta\,. \]
(Bubbles with $H(\sB)=0$ do not contribute to $H_h=\sum_\sB H(\sB)$, and for them the above integral may diverge: when $h$ coincides with a tiling on the bubble group of $\sB$, the probability $\pi_{r}(\sB\in h)$ is of order $e^{-\alpha|\sB|/2}$ uniformly in $\hat\beta$. Such bubbles are therefore excluded from the decomposition below.) Similarly to the case of $\mu,\nu$ in \cref{eq:mu-integral-with-f,eq:nu-integral-with-f}, we have that
\begin{equation}\label{eq:pi-integral-with-f} 
\int_\beta^\infty \pi_{\varphi,\hat\beta}(H_h)\d \hat\beta = \sum_{x\in\T_N} \sum_{\substack{\text{$(h,\varphi)$-bubble $\sB$}  \\ \Upsilon(\sB)\ni x,\;H(\sB)\geq1}} \frac{H(\sB)}{|\Upsilon(\sB)|} \sum_{\substack{r=2^k \\ \text{\emph{for }$k\geq 0$}}} f_{r,\sB}^\pi (\varphi\restriction_{B(x,r)}) \,,
\end{equation}
with the exceptional endpoint function defined as $f_{1,\sB}^\pi(\varphi\restriction_{B(o,1)}) := \int_\beta^\infty \pi_1(\sB\in h) \d \hat\beta$ (again for $H(\sB)\geq1$). Note that $f_{1,\sB}^\pi = 0$ by construction if $\sB$ has diameter more than $2$. For the finitely many types of bubbles small enough to fit in a ball of radius $1$, a standard Peierls argument using \cref{prop:bound_muh2} shows that $\| f_{1, \sB}^\pi \|_\infty \leq e^{-\beta/2}$ provided that $\beta$ is large enough.

We now wish to argue that 
\begin{equation}\label{eq:f-pi-bound}
\| f_{2r,\sB}^\pi\|_\infty\leq C \exp[-\lambda^3(r+|\sB|)/C]\,.
\end{equation}
We need a Peierls bound that already sums over bubbles. For $m\geq1$, set
\[
F_{x,m}(h):=\sum_{\substack{\sB\in h:\;\Upsilon(\sB)\ni x\\ H(\sB)>0,\;|\sB|\geq m}}\frac{H(\sB)}{|\Upsilon(\sB)|}\,.
\]
Uniformly over the domains $D$ under consideration,
\begin{equation}\label{eq:pi-weighted-peierls}
\pi_D\big(F_{x,m}^{5/4}\big)\leq C e^{-2\hat\beta/3-c\lambda\gamma m}\,.
\end{equation}
Indeed, as $F_{x,m}^{5/4}\leq F_{x,m}+F_{x,m}^2$, it suffices to bound $\pi_D(F_{x,m})$ and $\pi_D(F_{x,m}^2)$: expand the square and mark the one or two bubbles. Delete the groups meeting a coarse cover $S$ of an enclosure of their groups (a connected set, as both groups project onto $x$). Given the deleted collection and $S$, the map is injective; the choices of marks and their weights cost only a polynomial in the update size $\sigma$. Apply the typical/exceptional estimates of \cref{it:short-large-typ,it:short-large-atyp} at $\hat\beta/4$, retaining the factor $e^{-3\hat\beta\eta/4}$, where $\eta\geq1$ is the deleted excess area. The typical labels have summable weight $e^{-c\sigma}$ after enumeration. For exceptional labels, $\eta\leq\sum_i|\Upsilon(\fB_i)|$ and the entropy $e^{C\gamma^2\sigma}$ is beaten by $e^{-\lambda\gamma\sigma/40000}$. Enclosures smaller than $12/\gamma$ are handled directly by \cref{prop:bound_muh,clm:update-count}, retaining their energy cost as in the small-enclosure contraction argument. Also $\sigma\geq m/2$, as $|\Upsilon(\sB)|+H(\sB)=|\sB|$. Summing the labels, including their polynomial multiplicities, gives \cref{eq:pi-weighted-peierls}: the resulting factors $e^{C\gamma^{-2}}$ are absorbed by the spare $e^{-\hat\beta/12}$ (already under $\alpha\wedge\hat\beta\geq C\gamma^{-21}$). All estimates used at $\hat\beta/4$ remain valid after increasing $C$; in particular $\hat\beta/4>2\lambda\sM_0$.

In particular, for a fixed bubble,
\begin{equation}\label{eq:peierls_pi}
\pi_r(\sB\in h)\leq Ce^{-\hat\beta/2-c\lambda\gamma|\Upsilon(\sB)|}\qquad\mbox{if }H(\sB)\geq1\,.
\end{equation}
The same enumeration without the positive-excess mark gives an aggregate exponential tail for the sizes of all bubbles above $x$, including those with $H=0$. If their deleted collection has zero excess, use its $-\alpha G^{\mathrm{gr}}$ cost; otherwise use the positive-excess estimate above.

Once again, the decay in $r$ will come from a bound on the speed of propagation of information under the dynamics $(h_t)$. 
This time we will run (two coupled instances of) the dynamics for time $
T = c \gamma^2 \lambda r$ as we will specify later. Note that, now that the interactions have a small rate of exponential decay $\lambda$, we are even more limited in the time for which we will analyze the dynamics than in the $\nu$ case. Compared to that case, there is also a further difficulty whenever a large enclosure appears which we first explain heuristically. Indeed, suppose that at some time $t$, a bubble group $\fB$ differs between $h$ and $h'$ and is associated to a large $R_0$ with $|R_0 | \gg \alpha \wedge \hat \beta$ in one of them. As in the small short-range case of the proof of \cref{prop:pi-contraction}, there are an order $e^{c|R_0|/(\alpha \wedge \hat\beta)}$ many clocks which can each contribute to the propagation of the information from the boundary so the speed of information appears extremely large. This is however misleading because if this happens, there is an even larger rate for updates removing all bubbles from $R_0$ so overall it is still very unlikely for information to actually move using the above mechanism. 

The proof will now follow a multi-scale induction, aiming to show that a connected set of size $s$ (the scale parameter) which might get populated by a dense collection of bubble groups, is likely to have its bubbles vanish before it gets a chance to support a long-range infection step in the dynamics. 
To this end, define the local density near $x$ of sets of size of order $s$: 
\begin{equation}
    \label{eq:rho-def}
    \rho_x(s,h) := \max \left\{  \rho( S, h)\,:\; \mbox{$S\ni x$ connected of size $s\leq |S|\leq 2s$}\right\}
\end{equation}
(and $\rho_x(s,h)=-\infty$ if no such set $S$ exists), as well as the local time spent in configurations where $\rho$ exceeds a given threshold:
\begin{equation}\label{eq:L-def}
    \sL_x(s, r) = \int_0^1 \one_{\{\rho_x(s, h_t) \geq r\}}\d t\,.
\end{equation}
We count all dense witnesses through $x$, whether or not $x$ itself is in a bubble. The simultaneous clearing argument below proves this stronger bound.

We aim to relate $\sL_x$ to the number of times that a large (dense) set was created. To that end,
consider the dynamics $(h_t)$ and let $(S^{(k)},\{\fB_i^{(k)}\})_{k\geq 1}$ be the random set of updates where
\begin{itemize}
    \item the corresponding clock rang as per \cref{it:glauber-clocks-pi} in its definition,
along $t\in(0,1)$;
    \item \label{it:analysis-glauber-req-2} $ \sum_i \hat\beta H_h( \fB_i^{(k)}) - \alpha G^\gr_{h,\varphi}( \fB_i^{(k)} ) \geq \epsilon (\alpha\wedge  \hat\beta) \sum_i |\Upsilon(\fB_i^{(k)})|$;
\item the update resulted in the addition of $\{\fB_i^{(k)}\}$ to $h_t$ as per \cref{it:glauber-rand-pi}.
\end{itemize}

As we will later see, the idea behind the second requirement is to distinguish between updates with a typical $H_h - G^\gr_{h,\varphi}$ cost and atypical counterexamples. The counterexamples (violating said requirement) are somewhat simpler to handle, as we employ \cref{prop:bound_muh2} independently of the rest of the current configuration. For typical updates however (which do satisfy this requirement), one needs to use \cref{lem:bound_muh-d} which depends on the local density.

To simplify notations, throughout this proof we will write $|\bigcup_i \fB_i| := \sum_i |\Upsilon(\fB_i)|$. Let
\begin{equation}\label{eq:Q-def}
\sQ_x(s) := \#\bigg\{k\,:\; x \in S^{(k)}\,,\big|\bigcup_i \fB_i^{(k)} \big| \geq s  \bigg\}\,.
\end{equation}

The induction will be based on the following two statements, proved in tandem. 

\begin{lemma}\label{lem:local-time-Q}
Take $\epsilon=\gamma^{14}$ in the definition \cref{eq:Q-def} of $\sQ_x$. There exists an absolute constant $C$ such that, if $\alpha \wedge \hat\beta\geq C\gamma^{-30}$, then for every configuration $h_0$ and every scale $s\geq 1/\gamma^{2}$, one has that
\begin{align}
\label{eq:local-time-1/10}
    \E[ \sL_x(s,\tfrac1{10}) ] \leq e^{-\frac14 \gamma s}&\qquad\mbox{for all $x$}\,;\\
\label{eq:Q-bound}
\E [ \sQ_x(  s/(\epsilon^2(\alpha\wedge\hat\beta))) ] \leq e^{- \frac1{5}\gamma s}& \qquad\mbox{for all $x$}\,.
\end{align}
\end{lemma}

\begin{proof}
We will prove the following implications inductively on $s$, with a base case of $s>|\T_N|$:
\begin{enumerate}[label=\textbf{Step~\arabic*}:, ref=\arabic*, wide=0pt, itemsep=1ex]
    \item \label[step]{st:1/10-to-Q}
    \cref{eq:local-time-1/10} for a given $s$ implies \cref{eq:Q-bound} for that $s$;
    \item \label[step]{st:Q-to-1/12}
    \cref{eq:local-time-1/10,eq:Q-bound} for a given $s$ together imply \cref{eq:local-time-1/10} for $s/2$.
\end{enumerate}
(To carry this out, we will take advantage of the fact that these estimates hold for all $x$. The value $\epsilon=\gamma^{14}$ enters in two ways: the smallness of $\epsilon$ is used in \cref{st:Q-to-1/12} in order for the exceptional updates to have small entropy as per \cref{lem:approx-update}, and in \cref{st:1/10-to-Q} it dictates the requirement $\alpha\wedge\hat\beta\geq C\gamma^{-30}$, see \cref{rem:eps-choice}.)

\smallskip\noindent\emph{Base case.} The statement of \cref{eq:local-time-1/10} for $s>|\T_N|$ is trivial (as we then have $\rho_x(s,h)=-\infty$).

\smallskip\noindent \emph{Proof of \cref{st:1/10-to-Q}.}
Write $A:=\alpha\wedge\hat\beta$ and $q:=s/(\epsilon^2A)$, and first count updates with $q\leq M:=\sum_i|\Upsilon(\fB_i)|<2q$; the bin is empty if $2q<1$. Put $\eta:=\sum_iH(\fB_i)$. For non-singleton clocks, $|S|\leq4000q/\gamma$, so \cref{clm:update-count} gives at most $e^{C(q/\gamma+\eta)}$ updates with $x\in S$ and a given $\eta$. For singleton clocks, $S$ is fixed by $x$. Root each of their at most $M/2$ groups in $S$; the number of families is at most $\exp[C(M\log(2/\gamma)+M+\eta)]$, again at most $e^{C(q/\gamma+\eta)}$. By \cref{prop:bound_muh2}, the acceptance probability is at most $e^{-\hat\beta\eta}$. Thus updates with $\eta>q/\gamma$ have total accepted rate at most $e^{-\hat\beta q/(C\gamma)}$, while the remaining clocks have total rate at most $e^{Cq/\gamma}$.

For an update counted by $\sQ_x$, its energy cost is at least $s/\epsilon$. Use enclosures of density $3/20$. Join the enclosure components through $S$ if necessary. Here $|S|\leq4000q/\gamma$ for non-singleton clocks, while $|S|=\lceil\gamma^{-1}\rceil^2\leq4s$ for singleton clocks. If the total enclosure size is at least $200s$, joining through $S$ and deleting the proposed groups lowers its density by at most $(4s+4000q/\gamma+2q)/(200s)<1/20$; hence the joined set witnesses density at least $1/10$ in the current configuration. Summing \cref{eq:local-time-1/10} over dyadic sizes at least $s$ and witnesses within distance $O(s+q/\gamma)$ (and the corresponding larger balls at larger scales) bounds the expected duration of these times by $Cs^C e^{-\gamma s/4}$. At all other times the enclosure has total size less than $200s$, and \cref{prop:bound_muh} bounds the interaction cost by $Cs$, giving acceptance probability at most $e^{-s/\epsilon+Cs}$. Consequently the expected number of accepted updates in this bin is at most
\[
e^{-\hat\beta q/(C\gamma)}+e^{Cq/\gamma}\left(e^{-s/\epsilon+C(s+q/\gamma)}+Cs^Ce^{-\gamma s/4}\right)
\leq e^{-9\gamma s/40}\,.
\]
Here $A\geq C\gamma^{-2}\epsilon^{-2}=C\gamma^{-30}$, $\epsilon=\gamma^{14}$ and $s\geq\gamma^{-2}$. Summing the bins $[2^jq,2^{j+1}q)$, using the induction bounds at the corresponding larger scales, gives \cref{eq:Q-bound}.

\smallskip\noindent \emph{Proof of \cref{st:Q-to-1/12}.}
Fix $s$ and assume \cref{eq:local-time-1/10,eq:Q-bound} for $2 s$.

Let $\tau_1$ be the first time that $\rho_x(s,h_t)\leq1/13$. Before $\tau_1$, choose a connected witness $W\ni x$ of size in $[s,2s]$ and density greater than $1/13$. Adjoin every group meeting $W$, and close under groups meeting the boundary. If their total projected area is $M$, the resulting set has size at most $|W|+M$ and density at least $1/14$. Choose a connected superset $R$ of maximal size with density at least $1/15$ and no group crossing its boundary, so $|R|\geq s$. Every connected $V\subset R^c$ touching $R$ has density less than $1/14$: otherwise adjoining $V$ and its groups would enlarge $R$ while preserving density $1/15$. After all groups meeting $R$ are removed, decomposing any connected set through $x$ into its parts outside $R$ therefore gives density at most $1/14$. This single update ends the excursion.

The healing enumeration of \cref{it:healing} applies with density $1/15$: a lifted coarse cover of $R$, extended by $\gamma|R|$ coarse triangles, has size at most $(13+o(1))|R|/\gamma<2000|R|/(15\gamma)$. These clocks are eligible, and each clears $R$ with probability at least $\frac12$. Indeed, as in \cref{it:healing} (with $R$ in the role of $R_0$ in \cref{eq:healing-bubble-grp-terms,lem:short_range_integral}), the probability to keep the family is at most $e^{-C_1(|S|+|R|+M)}$ if the label is typical, while an atypical label is exceptional in the sense of \cref{lem:approx-update} with $\epsilon=\gamma^9$ (as $|S|\leq15|R|/\gamma$ and $M\geq|R|/15$), whence \cref{prop:bound_muh2} together with \cref{eq:S-energy-atypical} bounds that probability by $e^{-\hat\beta/2}$. Their total clearing rate is at least $e^{\gamma|R|}\geq e^{\gamma s}$. The lifting convention ensures this also when $R$ fills the domain. Consequently, conditionally at every excursion's start, $\tau_1$ is dominated by an exponential variable of mean $e^{-\gamma s}$.

Write again $A:=\alpha\wedge\hat\beta$ and $q:=s/(\epsilon^2A)$. After each return to density at most $1/13$, an excursion above $1/10$ must start with an accepted addition. Updates with projected area $M<4q$ have total accepted rate at most $Cs^Ce^{Cq/\gamma}$ near $B(x,2s)$, by the excess-area summation in \cref{st:1/10-to-Q}. Updates failing the energy condition defining $\sQ$ are exceptional in \cref{lem:approx-update} with parameter $\epsilon^{3/5}=\gamma^{42/5}$; their count $e^{2C\gamma^2s'}$ is summable against their acceptance bound $e^{-\hat\beta/2-\lambda M/20}$, as in \cref{it:short-large-atyp}, contributing at most $Cs^C$ to this rate. Every other addition is counted in some $\sQ_y(2q)$. For projected area $M$, its root $y\in S$ is within distance $O(s+M/\gamma)$ of $x$; summing over dyadic $M\geq2q$ and these roots, the induction bound at scales $2s,4s,\ldots$ gives at most $Cs^Ce^{-2\gamma s/5}$ such additions in expectation.

Thus the expected number $J$ of excursions is at most $1+Cs^Ce^{Cq/\gamma}+Cs^Ce^{-2\gamma s/5}$. The healing estimate above holds conditionally at the start of each excursion, so the strong Markov property gives
\[
\E[\sL_x(s,\tfrac1{10})]\leq
\left(1+Cs^Ce^{Cs/(\gamma\epsilon^2A)}+Cs^Ce^{-2\gamma s/5}\right)e^{-\gamma s}
\leq e^{-\gamma s/4}\,.
\]
This proves the second step.

Having established the induction, the proof is complete.
\end{proof}

\begin{remark}\label{rem:Q-bound2}

It will be useful to regard the following special case of \cref{eq:Q-bound}, obtained by substituting $s=\gamma^{28}(\alpha\wedge\hat\beta)s'$ (recall that $\epsilon=\gamma^{14}$ in the definition of $\sQ_x$):
    \begin{equation}\label{eq:Q-bound2}
    \E[ \sQ_x( s' ) ] \leq \exp\Big[- \frac{\gamma^{29}(\alpha \wedge \hat\beta)}{5} s'\Big] \quad \mbox{for all }\quad  s' \geq \frac{1}{\gamma^{30} (\alpha \wedge \hat\beta)}\,.
    \end{equation}
The rate $\frac15\gamma^{29}(\alpha\wedge\hat\beta)$ is smaller than $\alpha\wedge\hat\beta$ by a large power of $\gamma$; this is the price of the local-time machinery, which replaces an energy estimate by one that holds uniformly over the dense regions of $h_t$. Note that, under the hypothesis $\alpha\wedge\hat\beta\geq C\gamma^{-30}$ of \cref{lem:local-time-Q}, the range of validity of \cref{eq:Q-bound2} contains all $s'\geq1$.

\end{remark}

\begin{remark}\label{rem:eps-choice}

The exponents in \cref{lem:local-time-Q} are dictated by the following two requirements, and we made no attempt to optimize them. On the one hand, an update which is not counted in $\sQ_x$ has $\sum_i(\hat\beta H(\fB_i)-\alpha G^\gr(\fB_i))<\epsilon(\alpha\wedge\hat\beta)\sum_i|\Upsilon(\fB_i)|$, so it is exceptional in the sense of \cref{lem:approx-update} with parameter $\epsilon^{3/5}$, and the entropy bound $e^{C(\epsilon^{3/20}+\gamma^2)s}$ provided by that lemma must be beaten by the factor $e^{-\lambda\gamma s/4000}$ coming from the potential (cf.\ \cref{eq:large-atypical-contribution}); this forces $\epsilon^{3/20}\ll\gamma^2$, which is why we take $\epsilon=\gamma^{14}$. On the other hand, after summing the excess-area tail with its acceptance probability, a projected-area bin of size $s/(\epsilon^2(\alpha\wedge\hat\beta))$ has rate at most $e^{Cs/(\gamma\epsilon^2(\alpha\wedge\hat\beta))}$, since $|S|$ can be larger by $2000/\gamma$. Comparing this with $e^{-\gamma s/4}$ requires $\alpha\wedge\hat\beta\gtrsim\gamma^{-2}\epsilon^{-2}=\gamma^{-30}$, giving the hypothesis of \cref{thm:pi}. For potentials as in \cref{def:pinning-V} the first requirement is void (see \cref{sec:better-rates}), and the hypothesis $\alpha\wedge\beta\geq C\lambda^{-21}$ of \cref{thm:pi-str} comes from the contraction estimate alone, see \cref{rk:dependency_beta_lambda}.

\end{remark}

Let $(h_t,h'_t)$ be two coupled instances of the dynamics on respective domains $B(o,r)$ and $B(o,2r)$, from an initial configuration which agrees on $B(o,r)$, where every update of $h'_t$ that is confined to $B(o,r)$ uses the joint law as per \cref{prop:pi-contraction} (and updates of $h'_t$ in the annulus $B(o,2r)\setminus B(o,r)$ sampled via the product measure of $h_t$ and $h'_t$ on this event). Run the dynamics for time
    \[T = c \gamma^2 \lambda r \,,\]
with $c$ to be chosen small enough later and write
\[\left|\pi_{2r}(\sB\in h) - \pi_r(\sB\in h)\right| \leq \Xi_1 + \Xi_2 + \Xi_3\,,
\]
where
\begin{align*} \Xi_1 &:= \left| \pi_{r}(\sB\in h) - \P(\sB\in h_{T})\right|\,,\\
\Xi_2 &:= \left| \pi_{2r}(\sB\in h) - \P(\sB\in h'_{T})\right|\,,\\
\Xi_3 &:= \left| \P(\sB\in h_{T}) - \P(\sB\in h'_{T})\right|\,.
\end{align*}
As in the two previous cases, using \cref{prop:pi-contraction} for $t=T$, we get
\[ \Xi_1 \leq e^{-c\gamma^2 \lambda r/4 } |B(o,r)| \leq e^{-(c \gamma^2 \lambda/4 - o(1) )r}\,, \quad \quad  \Xi_2 \leq e^{-c\gamma^2 \lambda r/4} |B(o,2r)| \leq e^{-(c \gamma^2 \lambda/4 -o(1) ) r }\,,\]
and so we can focus on the bound on $\Xi_3$.

As in the treatment of the measure $\nu$, on the event where, at time $T$, the bubble $\sB$ only appears in one configuration, we can find a sequence of updates $(S^{(1)}, \{\fB^{(1)}_j\}), \ldots, (S^{(m)}, \{\fB^{(m)}_j\})$ ($m\geq 1$) with $\Upsilon(\sB) \cap (\bigcup_j \Upsilon(\fB_j^{(m)} )) \neq \emptyset $ and such that all updates occur successively, each in only one of $h$ or $h'$ before time $T$. We can also assume that whenever $(S^{(k)}, \{ \fB^{(k)}_j \})$ occurs, the minimal distance from  $S^{(k)} \cup \bigcup_j \Upsilon(\fB^{(k)}_j)$ to any disagreement is reached at a point of $\bigcup_i \Upsilon(\fB_i^{(k-1)})$. We let $j_k, r_k, s_k, t_k$ denote respectively this distance, the diameter of $S^{(k)} \cup \bigcup_j \Upsilon(\fB^{(k)}_j)$, the size $\max(|S^{(k)}| , \sum_j |\Upsilon (\fB_j^{(k)})|,\sum_j H(\fB_j^{(k)}))$ of the update as per \cref{clm:update-count} and the time of the update. We also let $\cF_k$ denote the $\sigma$-algebra of events measurable with respect to the dynamics up to time~$t_k$ and we let $n_k := \lfloor t_k - t_{k-1} \rfloor$. As in the $\nu$ case, we bound $\Xi_3$ by a union bound over all sequences $(S^{(1)}, \{\fB^{(1)}_j\}), \ldots, (S^{(m)}, \{\fB^{(m)}_j\})$.

The first step is to enumerate over sequences of $(j_k,r_k,s_k, n_k)_{k \geq 0}$ and to condition on all the $\cF_k$ to obtain that $\Xi_3$ is at most
\begin{equation}\label{eq:prod_propagation}
 \sum_{\substack{(n_k) \\ \sum n_k < T}} \sum_{\substack{(j_k,r_k) \\ \sum j_k + r_k \geq r}}  \!\!\!\!\!\prod_k \Big[\sum_{s_k \geq r_k/2}\!\P\Big( \mbox{compatible $(S^{(k)}, \{ \fB_i^{(k)}\})$ in $[t_{k-1}+ n_k, t_{k-1}+ n_k+1]$} \mid \cF_{k-1}\Big)  \Big] \,.
\end{equation}

The conditional probability in the right-hand side is bounded according to three regimes. Throughout, we write $A:=\alpha\wedge\hat\beta$, we let $C$ be the constant of \cref{clm:update-count} (so that at most $e^{Cs}$ updates of size at most $s$ contain a given point of $\T_N$), and we let $C_0$ be the larger of the constants of \cref{lem:bound_muh-d,lem:bound_muh-domain}. Recall that the size $s_k$ of the $k$-th update bounds $|S^{(k)}|$, $\sum_j|\Upsilon(\fB^{(k)}_j)|$ and $\sum_jH(\fB^{(k)}_j)$, that $\sum_j|\Upsilon(\fB_j^{(k)})|\geq\gamma|S^{(k)}|/2000$ for every non-singleton clock, and that $r_k\leq 2s_k$ (as $S^{(k)}\cup\bigcup_j\Upsilon(\fB^{(k)}_j)$ is connected). An update is \emph{tall} if $\sum_jH(\fB^{(k)}_j)\geq s_k/3$, and otherwise it is \emph{wide}; for a wide update one has $s_k=|S^{(k)}|\vee\sum_j|\Upsilon(\fB^{(k)}_j)|$, hence $\sum_j|\Upsilon(\fB^{(k)}_j)|\geq\gamma s_k/2000$, except possibly for singleton clocks with $s_k=\lceil\gamma^{-1}\rceil^2$; these are included among the small updates in Regime~III (and cannot occur in Regime~II). By \cref{prop:bound_muh2}, a tall update is accepted, in either of the two chains, with probability at most $e^{-\hat\beta s_k/3}$; since there are at most $Cs_{k-1}(j_k+r_k)^2e^{Cs_k}$ updates of size $s_k$ at distance $j_k$ from the previous disagreement, tall updates contribute at most $s_{k-1}e^{-\hat\beta s_k/4}$ to the probability in \cref{eq:prod_propagation} whenever $s_k\geq c_0\gamma j_k$ for an absolute constant $c_0$ (which we fix below), and they are treated together with the wide ones when $s_k<c_0\gamma j_k$. Our goal in each regime is the bound
\begin{equation}\label{eq:per-step-bound}
\sum_{s_k} s_k\, \P\Big( \mbox{compatible $(S^{(k)}, \{ \fB_i^{(k)}\})$ in $[t_{k-1}+ n_k, t_{k-1}+ n_k+1]$} \;\big|\; \cF_{k-1}  \Big)  \leq s_{k-1}\, e^{- \frac{\gamma^2 \lambda}{C'} (j_k + r_k) - \frac{\gamma^{3} A}{C'}  }
\end{equation}
for an absolute constant $C'$, where the factor $s_{k}$ on the left-hand side compensates for the factor $s_{k-1}$ on the right-hand side, so that after plugging \cref{eq:per-step-bound} into \cref{eq:prod_propagation} only the exponential factors remain. Recall also that $\sQ_x$ is defined with $\epsilon=\gamma^{14}$ as per \cref{lem:local-time-Q}, so that an update not counted in any $\sQ_x$ (a \texttt{bad} update) has $\sum_j(\hat\beta H(\fB^{(k)}_j)-\alpha G^\gr(\fB^{(k)}_j))<\gamma^{14}A\sum_j|\Upsilon(\fB^{(k)}_j)|$, whence it is exceptional in the sense of \cref{lem:approx-update} with $\epsilon=\gamma^{42/5}$ (as $\epsilon^{2/3}(\alpha\epsilon^{1/3}\wedge\hat\beta\epsilon)\geq\epsilon^{5/3}A$). In particular, by \cref{lem:approx-update} (and \cref{clm:update-count}), the number of \texttt{bad} updates of size $s$ containing a given point is at most $e^{C(\gamma^{21/10}+\gamma^2)s}\leq e^{2C\gamma^2 s}$, and each of them is accepted with probability at most $e^{-\hat\beta/2-\frac{\lambda}{20}\sum_j|\Upsilon(\fB^{(k)}_j)|}$ by \cref{prop:bound_muh2}.

\smallskip\noindent\emph{Regime I: $j_k \geq 160 A$ and $s_k \leq c_0\gamma j_k$.} Consider the set of times in $[t_{k-1} + n_k, t_{k-1} + n_k + 1] $ at which
\begin{equation}\label{eq:good-times}
\rho_x(s, h_t) < \tfrac{1}{10}\qquad\mbox{for all $x$ with $\dist\Big( x, \bigcup_i \Upsilon(\fB_i^{(k)})\Big) \leq 2 j_k$ and all $s \geq j_k/640$}
\end{equation}
(and the same for $h'_t$). By \cref{eq:local-time-1/10} (applied to every such $x$, of which there are at most $Cs_kj_k^2$, and to the dyadic scales $s\geq j_k/640$), the expected Lebesgue measure of the complement of this set is at most 
$ C s_k j_k^2\gamma^{-1} e^{- \gamma j_k/2560} \leq  e^{- \gamma j_k/C}$, so by Markov's inequality the probability that it exceeds $e^{- \gamma j_k/(2C)}$ is at most $e^{-\gamma j_k/(2C)}$. Enumerating over all updates of size $s_k$ at distance $j_k$ from the previous disagreement, the probability that one of them rings during the complement is at most 
\[
C s_{k-1} j_k^2 e^{- \gamma j_k/(2C)} e^{C s_k} \leq  s_{k-1}e^{- \gamma j_k/(4C)} \leq s_{k-1}e^{- \gamma( j_k + r_k)/(8C)} \,,
\]
where we used $s_k\leq c_0\gamma j_k$ with $c_0:=1/(8C^2)$, $\gamma j_k\geq160\gamma A$ (which is large by our hypothesis on $A$) and $r_k\leq 2s_k\leq j_k$. On the other hand, at any instant when \cref{eq:good-times} holds, every enclosure $R$ (with density $\frac3{20}$) of the groups $\{\fB^{(k)}_i\}$ in the configuration obtained by adding them satisfies $|R|<j_k/640$, since otherwise $R$ would be a connected set of size at least $j_k/640$ containing a point at distance at most $2j_k$ from $\bigcup_i\Upsilon(\fB^{(k)}_i)$ and with density at least $\frac3{20}-640c_0\gamma>\frac1{10}$ before the addition, contradicting \cref{eq:good-times}; hence $R$ is at distance at least $j_k/2$ both from every current disagreement and from $\partial B(o,r)$, and every connected set of size at least $j_k/8$ within distance $j_k$ of $R$ has density less than $\frac1{10}+8c_0\gamma<\frac{3}{20}$ after the addition. We can thus apply \cref{lem:bound_muh-d} (with $\{\fB^0_i\}$ the groups of the update, $R^0_i$ their enclosures, and $\{\fB^1_j\}$ the differences between $h_t$ and $h'_t$) to the difference between the two chains of the interaction term with the \emph{same} domain, and \cref{lem:bound_muh-domain} (with $R^\ddagger=R$, $\rho=3/17$, $\fd=j_k/2$, $D=B(o,r)$ and $D'=B(o,2r)$; its hypothesis on the density of the sets near $R$ was verified above) to the difference, in the chain $h'_t$, between the interaction terms on the two domains. Reasoning as in the long-range infection part of the proof of \cref{prop:pi-contraction}, the probability that the update is accepted in exactly one of the two chains is at most $C_0j_ke^{-\alpha j_k/(2C_0)}$, and summing over the at most $Cs_{k-1}j_k^2e^{Cs_k}$ updates (which includes the tall ones) we see that the updates occurring at good times contribute at most
\[
C s_{k-1}j_k^2e^{Cs_k}  e^{- \alpha j_k/(2C_0)} \leq C s_{k-1}  e^{- A (j_k + r_k)/(8C_0)} \,.
\] 
Both bounds are at most $s_{k-1}e^{-\frac{\gamma^2\lambda}{C'}(j_k+r_k)-\frac{\gamma^{3}A}{C'}}$, as $j_k\geq160A$ and $\gamma$ is small.

\smallskip\noindent\emph{Regime II: $j_k \geq 160 A$ and $s_k > c_0\gamma j_k$.} Tall updates were handled above. For a wide update counted in some $\sQ_x$, necessarily $x\in S^{(k)}$ and $\sum_j|\Upsilon(\fB_j^{(k)})|\geq\gamma s_k/2000$, so the expected number of such updates is at most $Cs_{k-1}(j_k+r_k)^2e^{-\frac{\gamma^{30}A}{10000}s_k}$ by \cref{eq:Q-bound2}, which is at most $s_{k-1}e^{-\frac{\gamma^{31}c_0A}{20000}j_k-\frac{\gamma^{30}A}{40000}r_k}$ (using $s_k>c_0\gamma j_k$, $r_k\leq 2s_k$ and the largeness of $\gamma^{30}A\,c_0\gamma j_k$). The remaining wide updates are \texttt{bad}; by the discussion preceding Regime~I they contribute at most
\[
s_{k-1} e^{2C\gamma^2 s_k} e^{-\hat\beta/2-\frac{\lambda\gamma}{40000}s_k} \leq s_{k-1}e^{-\frac{\lambda\gamma}{80000}s_k}\leq s_{k-1}e^{-\frac{c_0\lambda\gamma^2}{160000}j_k-\frac{\lambda\gamma}{320000}r_k}
\]
(recall that $\gamma=\lambda/C$ with $C$ large, so $2C\gamma^2\leq\lambda\gamma/80000$). All three bounds are of the form \cref{eq:per-step-bound}, since $\gamma^{31}A\gg\gamma^2\lambda$ and $j_k\geq160A$; retaining a fixed fraction of each exponent also supplies $\gamma^3A/C'$.

\smallskip\noindent\emph{Regime III: $j_k<160A$.} Set $s^\star:=\gamma^{-28}$. Tall updates contribute at most $Cs_{k-1}(A+s_k)^2e^{-\hat\beta s_k/4}$. For wide updates counted in $\sQ_x$ with $s_k\geq s^\star$, \cref{eq:Q-bound2} gives $Cs_{k-1}(A+s_k)^2e^{-\gamma^{30}As_k/10000}$. Wide \texttt{bad} updates of any size contribute at most
\[
Cs_{k-1}(A+s_k)^2e^{-\hat\beta/2-\lambda\gamma s_k/80000}\,.
\]
(The singleton exception above cannot be \texttt{bad}: its energy is at least $A$, whereas $\gamma^{14}AM<A$.) Summing these bounds with the factor $s_k$ gives \cref{eq:per-step-bound}, since $r_k\leq2s_k$, $j_k<160A$ and $\gamma^{30}As^\star=\gamma^2A$.

Only updates with $s_k<s^\star$ remain. Enclose their proposed groups at density $3/20$, and let $L:=A/(4C_1)$, where $C_1$ also dominates the constant of \cref{prop:bound_muh} at fenced density $3/17$. If the enclosures have total size at least $L$, join their components through $S$ (at most $s^\star$ additional sites); the joined set has density before addition at least $3/20-2s^\star/L\geq1/10$, because $A\geq C\gamma^{-30}$. Thus \cref{eq:local-time-1/10}, summed over dyadic sizes $q\geq L$ and witnesses within distance $O(A+q)$ of the previous disagreement, bounds the expected duration of these times by $Cs_{k-1}A^C e^{-\gamma A/(16C_1)}$. At other times the enclosure is smaller than $L$, and \cref{prop:bound_muh} bounds the interaction cost by $C_1L=A/4$; the energy is at least $A$, so acceptance has probability at most $e^{-A/2}$.

For a sufficiently large absolute $C_2$, Markov's inequality bounds the probability that the exceptional duration exceeds $e^{-\gamma A/C_2}$ by $Cs_{k-1}A^Ce^{-\gamma A/C_2}$. There are at most $Cs_{k-1}A^2e^{Cs^\star}$ candidate clocks. Stopping their compensator when this duration reaches $e^{-\gamma A/C_2}$ therefore bounds their total contribution, including the factor $s_k\leq s^\star$, by
\[
Cs_{k-1}A^Cs^\star e^{Cs^\star}\left(e^{-\gamma A/C_2}+e^{-A/2}\right)
\leq s_{k-1}e^{-\gamma A/(2C_2)}\,.
\]
Here $Cs^\star=o(\gamma A)$; the resulting bound implies \cref{eq:per-step-bound}, retaining the factor $e^{-\gamma^3A/C'}$.

\smallskip
To sum \cref{eq:prod_propagation}, put $a:=\gamma^2\lambda/C'$ and $b:=\gamma^3A/C'$. Reserve half the spatial decay in \cref{eq:per-step-bound}; the sum over each pair $(j_k,r_k)$ costs at most $\sum_{j,r\geq0}e^{-a(j+r)/2}\leq Ca^{-2}$. Since $Ca^{-2}e^{-b/2}\leq1/4$ under our hypothesis on $A$, and there are at most $2^{K+T}$ time sequences of length $K$,
\[
\Xi_3\leq e^{-ar/2}\sum_{K\geq1}2^{K+T}(Ca^{-2}e^{-b})^K
\leq 2^T e^{-ar/2-b/2}\,.
\]
Taking $T=c\gamma^2\lambda r$ with $c$ small gives $\Xi_3\leq e^{-ar/4-b/2}$. Together with the bounds on $\Xi_1,\Xi_2$ and \cref{eq:peierls_pi}, this gives \cref{eq:f-pi-bound}.

To sum the bubbles, write $F_x:=F_{x,1}$ and, for $r\geq8$, let $F_x^{\sharp}$ denote the sum defining $F_x$ restricted to bubbles of size less than $r/4$. By \cref{eq:pi-weighted-peierls} and Jensen's inequality, $\pi_D(F_x-F_x^\sharp)\leq\pi_D(F_{x,r/4}^{5/4})^{4/5}\leq Ce^{-8\hat\beta/15-c\lambda\gamma r}$ for $D\in\{B(o,r),B(o,2r)\}$. For the truncated sums, couple $\pi_r$ and $\pi_{2r}$ as follows: start from $h'_0\sim\pi_{2r}$, let $h_0$ be its restriction to $B(o,r)$, run $(h_t,h'_t)$ under the propagation coupling above, and attach to $h_t$ a stationary copy $\tilde h_t\sim\pi_r$ via the coupling of \cref{prop:pi-contraction}, so that $\P(\tilde h_T\neq h_T)\leq e^{-T/4}\E[\dist_\fB(\tilde h_0,h_0)]\leq 2|\Lambda^\varphi_r|e^{-T/4}$. Then $(\tilde h_T,h'_T)$ is a coupling of $\pi_r$ and $\pi_{2r}$, and the argument bounding $\Xi_3$, with the terminal set $B(o,r/4)$, shows that the collections of bubbles of size less than $r/4$ above $o$ in $\tilde h_T$ and in $h'_T$ differ with probability at most $Cr^Ce^{-c\gamma^2\lambda r}$. Off this event every such bubble contributes equally to $F_o^\sharp(\tilde h_T)$ and to $F^\sharp_o(h'_T)$, so by H\"older's inequality (with exponents $\frac54$ and $5$) and \cref{eq:pi-weighted-peierls},
\begin{multline*}
\sum_{\substack{\sB:\;\Upsilon(\sB)\ni o\\H(\sB)\geq1}}\frac{H(\sB)}{|\Upsilon(\sB)|}\big|\pi_{2r}(\sB\in h)-\pi_r(\sB\in h)\big|
\leq Ce^{-8\hat\beta/15}\big(Cr^Ce^{-c\gamma^2\lambda r}\big)^{1/5}+Ce^{-8\hat\beta/15-c\lambda\gamma r}\\
\leq Ce^{-\hat\beta/2-c\lambda^3r}\,,
\end{multline*}
where the polynomial prefactor was absorbed by halving $c$, at the cost of a power of $\lambda^{-1}$ which the spare factor $e^{-\hat\beta/30}$ absorbs. Consequently we may integrate over $\hat\beta\geq\beta$ and sum absolutely to define
\[
\fg_r^\pi:=\sum_{\substack{\sB:\;\Upsilon(\sB)\ni o\\H(\sB)\geq1}}\frac{H(\sB)}{|\Upsilon(\sB)|}f_{r,\sB}^\pi\,,
\qquad \|\fg_r^\pi\|_\infty\leq Ce^{-\beta/2-c\lambda^3r}\,,
\]
and the sum over $x$ in \cref{eq:pi-integral-with-f}, with the omitted largest scales, gives \cref{eq:orig-prop3}, retaining half of the spare factor to absorb the polynomial prefactors there as well.\qedhere

\end{proof}

\subsection{Better exponential rates for the stronger potential}\label{sec:better-rates}
In this section we describe the modifications that boost the exponential decay of the form $e^{-c\lambda^3 r}$ on $\|\fg_r\|_\infty $ as per \cref{thm:pi}, to the form $e^{-C^\star r}$ for an arbitrarily large constant $C^\star>0$ as per \cref{thm:pi-str}.

\begin{proof}[Proof of \cref{thm:pi-str}]
Fix $C^\star>0$ arbitrarily large, our target constant for the bound on the rate of exponential decay of $\|\fg_r\|_\infty$.

The only modification we make in the dynamics introduced in \cref{sec:glauber-pi} is in the update rates: all clocks in \cref{it:glauber-clocks-pi}, including the coarse-singleton clocks, now have rate $\exp(C^\star |S|)$. The two labels paired in the healing scenario (\cref{it:healing}) share the same set $S$, hence the same rate, and the unpaired label of the sparse case is treated as before, so the coupling argument carries over.

To analyze the effect of this modification on the contraction of the dynamics (where we still aim to establish \cref{prop:pi-contraction} as it was), as before we look at the $3$ scenarios of non-blocked moves:
\begin{enumerate}[(i)]
\item \emph{Healing} (\cref{it:healing}): Any move decreasing the distance still occurs, only now with a larger rate. We can therefore use the same lower bound as in \cref{it:healing} for the rate of decrease of the distance, namely $e^{C_* \gamma |R_0| - o(\gamma|R_0|)}$.
\item \emph{Long-range infection} (\cref{it:long-range-infect}): Again, the bound on the probability that a single update creates a defect still holds and gives (using the distance and density conditions) 
\[
C \exp\bigg( - 20 \alpha \Big( |R_0| + \sum_i | \Upsilon(\fB_i)|\Big) + \frac{\alpha}{2} \dist(S, R_0)\bigg)\,.\] For $\alpha$ large enough, this is still summable uniformly over $|R_0|$ even taking into account the higher rate of updates.

\item \emph{Short-range infection} (\cref{it:short-range-infect}):
\begin{itemize}
\item \emph{Small} (\cref{it:short-small}): Here we are simply enumerating over all updates up to a maximal size $\Sigma_0=\frac{3 C_1}{(\alpha \wedge \hat \beta) \gamma^{20}} |R_0|$. With the new rates, these contribute a rate $\cI_0e^{C^\star\Sigma_0}$ to the increase in the distance, which is still much smaller than the healing rate, provided that $\alpha$ and $\hat \beta$ are large enough.

\item \emph{Large typical} (\cref{it:short-large-typ}): For that case, the condition \cref{eq:healing-bubble-grp-terms} was chosen so that each update contributes at most $|S|e^{-C_1 |S|}$ to the increase of the distance. If $C_1$ is chosen large enough (which we are free to do, see the beginning of the proof of \cref{prop:pi-contraction}), we can enumerate over all moves weighted by their rates and see that these updates have a small total contribution to the distance.

\item \emph{Large atypical} (\cref{it:short-large-atyp}): This is the only case where the stronger potential is needed. Compared to the original one, the improved bound in \cref{lem:approx-update} with a constant $20(C^\star + C)/\gamma$ where $C$ is the entropy of updates shows that each update contributes at most  $\exp(-\hat\beta/2 - (C^\star + C) |S|)$ which is again summable over all updates rates. In fact, this case becomes much easier than in the original analysis, because we do not need to use the fact that atypical updates have small entropy.
\end{itemize}
\end{enumerate}

We stress that we did not look to strengthen the contraction bound given in \cref{prop:pi-contraction} (and above only argued that it stays valid for the modified dynamics), and instead the purpose of the modified rule was to obtain better bounds on the propagation of information in \cref{sec:propagation-of-inf}.

To this end, one revises \cref{lem:local-time-Q} to ask that $\epsilon< 1/C^\star$ (replacing the upper and lower bounds that depended on $\gamma$) and 
the bounds on $\E[\sL_x]$ and $\E[\sQ_x]$ would feature $C^\star$ instead of $\gamma$:
\begin{align}
\label{eq:local-time-1/10-str}
    \E[ \sL_x(s,\tfrac1{10}) ] \leq e^{-\frac14 C^\star s}&\qquad\mbox{for all $x$}\,;\\
\label{eq:Q-bound-str}
\E [ \sQ^\epsilon_x(  s/(\epsilon^2(\alpha\wedge\hat\beta))) ] \leq e^{- \frac1{5}C^\star s}& \qquad\mbox{for all $x$}\,.
\end{align}
To show that this modified version of the lemma holds true, we look at the two proof steps:

\smallskip\noindent \emph{Proof of \cref{st:1/10-to-Q}.}
 The first bound is about the probability that a certain type of update is accepted so is unchanged at $\exp( - \frac{s}{\epsilon} + C( 1 + \frac{4000}{\gamma \epsilon^2 (\alpha \wedge \hat \beta)} + \frac{2}{\epsilon^2 (\alpha \wedge \hat \beta)})s)$. Using the revised condition $\epsilon < 1/C^\star$, for $(\alpha \wedge\hat \beta)$ large enough the contribution of these updates to $\E[ Q_x]$ is smaller than $e^{- C^\star s}$. The second bound counts updates that occur during some $\sL_y(s,\frac1{10})$ for $y$ in a suitable ball around $x$, so it becomes $e^{(C + C^\star) \frac{s}{\gamma\epsilon^2 (\alpha\wedge\hat\beta)} - C^\star s/4}$ using the stronger induction assumption.

\smallskip\noindent \emph{Proof of \cref{st:Q-to-1/12}.}
This is where the higher rates are needed. The maximal set $R$ in \cref{st:Q-to-1/12} has size at least $s$, and its clearing clocks now have rate at least $e^{C^\star s}$. Thus the same simultaneous clearing argument gives an exponential upper bound of mean $e^{-C^\star s}$ for $\tau_1$.
The ``recovery time'' it takes to move from density $\frac1{13}$ to $\frac1{10}$ is bounded as before: Updates below a size $s/(\epsilon^2 (\alpha \wedge \hat\beta))$ still cannot occur at a comparatively slow rate because of their size. The good updates counted in $\sQ$ are controlled using the stronger induction hypothesis. Finally, the contribution of bad updates is controlled by applying \cref{lem:approx-update} with a large enough constant (instead of only getting a $\lambda s$ term as in the original argument).

With this revised lemma in hand, one can derive the sought bound on $\|\fg_r\|_\infty$ as follows.
First, let the time horizon $T$ (which was formerly taken to be $c\gamma\lambda r$) be
\[ T = C^\star r/10\,, \]
so that the total variation distance to equilibrium measured by $\Xi_1$ and $\Xi_2$ would be smaller: replacing the previous bounds of $e^{-(c\lambda - o(1))r}$ on these quantities, 
we would get 
$ \Xi_i \leq e^{-c C^\star r}$ for $i=1,2$, and it remains to bound $\Xi_3$. To this end, using the stronger condition on the potential $\sV$, one can replace the bound in \cref{eq:peierls_pi} on the probability to find a defect by the stronger estimate
\begin{equation}\label{eq:peierls_pi-str}
\pi_r ( \sB \in h ) \leq C e^{- \hat\beta/2  -C^\star |\Upsilon(\sB)|}\,.
\end{equation}
As such, in the propagation of information, long-range, small size jumps that get controlled by the size of $\sL_x$ are now assigned a factor of $e^{- C^\star j_k}$ (replacing $e^{-c \gamma j_k}$ in the previous analysis), whereas the remaining long-range small size jumps receive a factor of  $e^{-c(\alpha\wedge\hat\beta) j_k}$. 
Finally, long-range large size jumps are bounded by the $\sQ_y$, which again yields a factor of $e^{-c C^\star j_k}$. Overall, $\Xi_3 \leq e^{-c C^\star r}$. The weighted Peierls argument of \cref{eq:pi-weighted-peierls}, now with spatial rate $C^\star$, and the same truncation and H\"older argument give the asserted bound on $\|\fg_r\|_\infty$ (rescaling the arbitrary constant $C^\star$).
\end{proof}

\subsection{Proof of \cref{thm:phi-weakly-interacting}}
Fix $C^\star>0$, let $C_0$ be the absolute constant of \cref{thm:mu-nu},
and let $C_1=C_1(C^\star)$ be the constant supplied by \cref{thm:pi-str}
for the parameter $C^\star$. In what follows, every requirement that
$\alpha\wedge\beta$ should be large (as a function of $C^\star$) is guaranteed
by the hypothesis $\alpha\wedge\beta\ge C\lambda^{-21}\ge C^{22}$ upon
enlarging the constant $C$ of the theorem as a function of~$C^\star$.

For a tiling $\varphi$ of $\T_N$ write
\[
I^\mu(\varphi):=\int_\alpha^\infty
\mu_{\varphi,\hat\alpha}(\tfrac12|\varphi\xor\psi|)\d\hat\alpha\,,\quad
I^\nu(\varphi):=\int_\alpha^\infty
\nu_{\varphi,\bar\alpha}(\tfrac12|\eta\xor\varphi|)\d\bar\alpha\,,\quad
I^\pi(\varphi):=\int_\beta^\infty
\pi_{\varphi,\hat\beta}(H_h)\d\hat\beta\,,
\]
so that \cref{prop:P(phi)-via-3-measures} (and
\cref{eq:phi-expression-gr2,eq:Z-phi-gr''} specifying the normalizer) is the exact identity 
\begin{equation}\label{eq:exact-identity}
\P_{\alpha,\beta,\lambda}(\varphi)=\frac1{\Zsosbar_{N,\beta,\lambda}}\,
e^{E(\varphi)}\,,\qquad E:=-I^\mu+I^\nu+I^\pi\,,
\end{equation}
where the normalizer $\Zsosbar_{N,\beta,\lambda}$ (which is $e^{\beta|\varphi|}\Zsos_{N,\beta,\lambda}$ for 
$\Zsos_{N,\beta,\lambda}$ from \cref{eq:tilted-sos})
does not depend on $\varphi$.

For every $r=2^k$ with $0\le r<N/2$, set
\[
\fg_r:=-\fg^\mu_r+\fg^\nu_r+\fg^\pi_r\,,
\]
where $\fg^\mu_r,\fg^\nu_r$ are given by \cref{thm:mu-nu} and $\fg^\pi_r$ is
given by \cref{thm:pi-str}, and as such,
\[ \|\fg_r\|_\infty \leq (2C_0+C_1)e^{-(\alpha\wedge\beta)/2-C^\star r}\]
provided that $\alpha \geq C_0 C^\star$.
Similarly, letting
\[
\tilde\P_{\alpha,\beta,\lambda}(\varphi) = \frac1{Z_N} e^{G(\varphi)}\quad\mbox{for}\quad G(\varphi):=\sum_{x\in\T_N}\ \sum_{\substack{0\le r<N/2\\ r=2^k}}
\fg_r(\varphi\restriction_{B(x,r)})
\]
and a normalizer $Z_N = \sum_\varphi e^{G(\varphi)}$,
we have by \cref{eq:orig-prop1,eq:orig-prop2,eq:orig-prop3} that
\begin{equation}\label{eq:eps-N}
\sup_\varphi\left|E(\varphi)-G(\varphi)\right|\leq (2C_0+C_1)e^{-(\alpha\wedge\beta)/2 -C^\star N} =: \epsilon_N\,.
\end{equation}
Consequently, comparing $Z_N$ to $\Zsosbar_{N,\beta,\lambda}=\sum_\varphi e^{E(\varphi)}$ we see that
\[
 \left|\log \Zsosbar_{N,\beta,\lambda} - \log Z_N\right| \leq \epsilon_N\,,
\]
and can conclude that
\[
\sup_\varphi \bigg|\log \frac{\P_{\alpha,\beta,\lambda}(\varphi)}{\tilde\P_{\alpha,\beta,\lambda}(\varphi)}
\bigg| \leq 2\epsilon_N\,.
\]

For potentials satisfying
\cref{eq:gen-potential-weaker} the argument is verbatim the same with
\cref{thm:pi} taking the place of \cref{thm:pi-str}: the inputs
\cref{eq:orig-prop1,eq:orig-prop2} are unchanged (the measures $\mu,\nu$
do not involve $\sV$), the exponent $C^\star$ is replaced throughout by
$\lambda^3/(C\sM_0)\le1$---so that the requirement
$\alpha/C_0\ge C^\star$ of Step~2 becomes
$\alpha/C_0\ge\lambda^3/(C\sM_0)$, which holds automatically---and the
hypothesis on the parameters becomes
$\alpha\wedge\beta\ge C\lambda^{-30}+\sM_0$.
\qed

\section{Concluding Theorem~\ref{thm:GFF-convergence} and Corollary~\ref{cor:variance} modulo Theorem~\ref{thm:gmt-refinement}}\label{sec:concluding-thm-1}

 Recalling our global strategy, at this point, we have defined an approximation of the SOS height function $h$ by a tiling $\varphi$, then established that the marginal law of $\varphi$ is that of a weakly interacting tiling in \cref{thm:phi-weakly-interacting}. This was done through the decomposition of $\varphi$ into $\mu,\nu,\pi$ (characterized in \cref{thm:mu-nu,thm:pi}), proving along the way that the conditional law of $h$ given $\varphi$ is given by small perturbations with exponentially decaying interactions. 

The next natural step would be to prove \cref{thm:gmt-refinement} to obtain the convergence of $\varphi$. That would require tools that are quite different from those used up to this point, so instead we first focus on the proof of \cref{thm:GFF-convergence} modulo \cref{thm:gmt-refinement}, which is closer in spirit to the previous sections. There are three points in this argument:
 naturally we need to prove that $h$ itself has a limit in infinite volume, then we have to check that $h - \varphi$ does not contribute to the scaling limit of $h$. However, we also need to change the underlying coordinate system. Indeed, while $h$ is defined as a function from $\Z^2$ to $\Z$, our current results on $\varphi$ use the natural setting of lozenge tilings and for example the height function associated to $\varphi$ is defined from the vertices of the triangular lattice to $\Z$. We see that we thus need to show how to transport the convergence from one convention to the other. 

\begin{proposition}
    For every $\lambda > 0$, and slope $\theta=(\theta_1,\theta_2)$ with $\theta_1,\theta_2> 0$, there exists $\beta_0$ such that for all $\beta \geq \beta_0$ the following holds. Let $(h_N, \varphi_N)$ be sampled according to \cref{eq:big-measure} on the torus $\Lambda_N$ and recall that by convention $h$ is pinned to $0$ at some point $o$. As $N$ goes to infinity, the pair $(h_N, \varphi_N)$ converges locally (seen as subset of plaquettes in $\Z^3$) to some $(h, \varphi)$.

    Furthermore, the law of $(\nabla h, h - \varphi)$ is translation invariant (under $\cP_{001}$ translation) and ergodic.
\end{proposition}
\begin{proof}
    The existence of a local limit for $\nabla \varphi_N$ is one of the outputs of \cref{thm:gmt-refinement}. The measure $\nabla \varphi$ must be invariant under $\cP_{111}$ translations because $\nabla \varphi_N$ was for all $N$ (this is one of the points where the setting of the torus is particularly convenient).

    As an output of \cref{thm:gmt-refinement}, we also have that the cumulants of the edge occupations variables decay polynomially with the distance between the edges (see \cref{eq:cumulants}). In particular, we must have for any two local events $E_1, E_2$ and writing $\tau_u$ for the translation by $u$, $\Cov( \one_{E_1}, \one_{\tau_u E_2}) \to 0$ as $\|u\| \to \infty$. It is easy to deduce that the measure must be ergodic (still for $\cP_{111}$ translations): 
    consider a translation invariant event $E$; there must exist a local event $E_1$ with $\P(E \xor E_1) \leq \epsilon$ for all $\epsilon$. By translation invariance for any translation $\tau$, $\P(E \xor \tau E_1) \leq \epsilon$ and since $E_1$ and $\tau E_1$ are asymptotically independent, we conclude.

    We turn to the joint law of $h$ and $\varphi$. By Skorokhod's Theorem, we can assume without loss of generality that $\varphi_N \to \varphi$ almost surely.
    By \cref{eq:f-pi-bound}, we know that for any $r$ and any $N$ large enough, the conditional law of $h_N \xor \varphi_N \restriction_{B(o, r)}$ given $\varphi_N$ is close to the law of (the restriction of) $h\xor\varphi_N$ from the simply connected domain $\varphi_N \restriction_{B(o, 2r)} $, with an error term exponentially small in $r$. This automatically proves tightness of the law of $h_N\xor \varphi_N$ and shows that any subsequential limit must also be close to the law in the finite domains $\varphi \restriction_{B(o, 2r)} $. Since $r$ was arbitrary and since the laws of $h$ in finite domains are automatically continuous functions of $\varphi$, we have proved convergence for the joint measure.

    Again the translation invariance holds automatically for the joint measure because it was true already in finite volume and the decay of correlation in $h$ given $\varphi$ and the ergodicity of $\varphi$ show that the joint law is ergodic.
    \end{proof}

\begin{remark}
        We know from \cref{eq:peierls_pi} (and the bound following it for bubbles of zero excess area) that the size of bubbles in $h \xor \varphi$ has an exponential tail. It is not hard to see similarly that $h \xor \varphi$ has almost surely no infinite component. This can already be interpreted as saying that the large-scale geometry of $h$ is identical to that of $\varphi$.
\end{remark}

\begin{lemma}\label{lem:GGF_SOS_convention}
    Fix $\fa, \fb, \fc$, $\lambda > 0$ and $\beta$ large enough. Let $\varphi$ be sampled according to the measure $\mu_\infty = \mu_{\infty}^{\beta, \lambda}$ given by \cref{thm:gmt-refinement} and let $\sigma$ and $\fL$ be given by \cref{it:phi_GFF_limit} of that theorem. There exists $\bar \fL$ such that in the \SOS convention,
    \[
    \varphi_{001}( n \cdot) - \E \varphi_{001} (n \cdot) \to \sigma \sqrt{3} |\det(\bar \fL)| \GFF \circ \fL \circ \bar \fL \,.
    \]
    Furthermore $\bar \fL$ only depends on $\fa, \fb, \fc$ but not $\lambda$ or $\beta$.
\end{lemma}
\begin{proof}
    The first step is to define $\bar \fL$. For concreteness, suppose $\varphi$ is pinned so that $o = (0,0,0) \in \varphi$. By translation invariance, both $\E[\varphi_{111}]$ and $\E[ \varphi_{001}]$ are linear maps so they can each be associated to a plane of $\R^3$. However, by \cref{thm:gmt-refinement}, we know that $(\varphi_{111}(n v) - \E \varphi_{111}(nv))/n$ goes to $0$ for every non-root $v \in \cP_{111}$. In particular, the planes $\E[\varphi_{111}]$ and $\E[\varphi_{001}]$ must coincide. Since the slope of $\varphi$ is $\theta$ as per \cref{eq:h-periodic-bc}, this is the plane of equation $\theta_1x_1+\theta_2x_2+x_3=0$, whose normal $(\theta_1,\theta_2,1)$ is proportional to the vector of densities of the three types of lozenges; we call it $\cP_{\fa\fb\fc}$. We let $\bar \fL$ be the map from $\cP_{001}$ to $\cP_{111}$ such that $\bar \fL(u) = v$ if and only if there exists $ x\in \cP_{\fa\fb\fc}$ such that $\Upsilon_{111}(x) = v$ and $\Upsilon_{001}(x) = u$.

    Seeing $U$ as a union of faces of $\Z^2$, it is clear that $\sum_{u \in U} \left(\varphi_{001}(u) - \E[\varphi_{001}(u)]\right)$ is the signed volume of the set \[\Big\{ x= (x_1, x_2, x_3) : \E[ \varphi_{001}(x_1,x_2) ] \leq x_3 \leq \varphi_{001}(x_1, x_2) \text{ or } \varphi_{001}(x_1, x_2) \leq x_3 \leq \E[ \varphi_{001}(x_1, x_2) ] \Big\}\]  (up to an $O(|\partial U|)$ error coming from the  approximation of $\cP_{\fa\fb\fc}$ by a step function). For $V$ a set of vertices of the triangular lattice, $\sum_{v \in V} \left(\varphi_{111}(v) - \E[\varphi_{111}(v)]\right)$ actually has a similar interpretation: Indeed, note that adding a single cube to $\varphi$ (at a point where it can be done without violating the monotonicity condition) leaves $\varphi_{111}$ invariant everywhere except at the projection of the diagonal of the cube where $\varphi_{111}$ increases by $1$. Therefore $\sum_{v \in V} \left(\varphi_{111}(v) - \E[\varphi_{111}(v)]\right)$ is also the signed volume of a set bounded by $\varphi$ and $\cP_{\fa\fb\fc}$, with the only difference that now the ``sides'' must be in the direction $(111)$ instead of $(001)$ (see \cref{fig:projection}).

    \begin{figure}
        \centering
        \includegraphics[width=.4\textwidth]{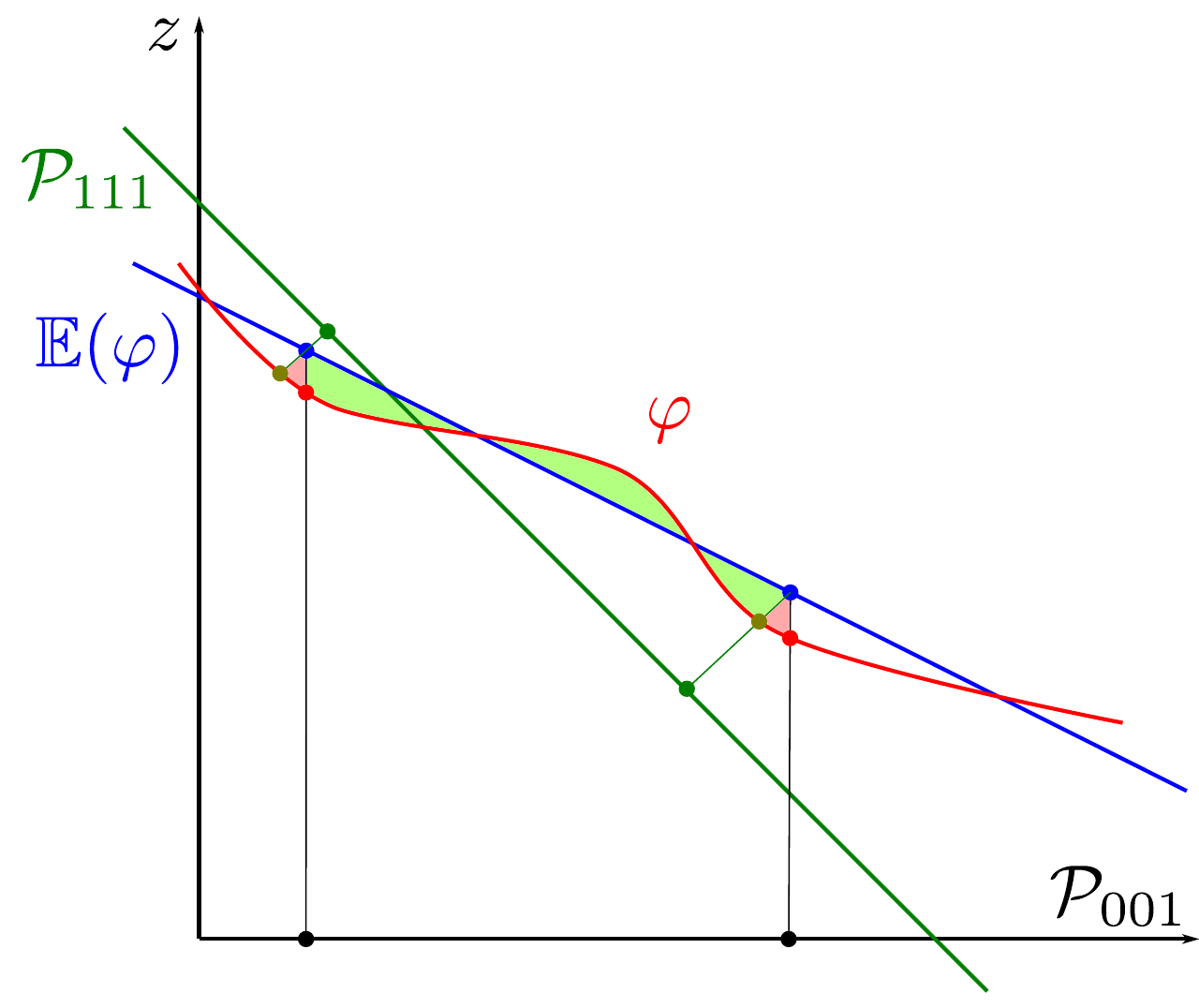}
        \caption{A schematic 2\Dim representation of the proof of \cref{claim:projection}. The set $U$ corresponds to the interval between the black dots on $\cP_{001}$, $V = \overline{\fL}(U)$ is the interval between the green dots with the construction of the map $\bar \fL$ shown using the thin lines and the blue dots. In this example, the sum over $u \in U$ counts the green area together with the right red component while the sum over $v\in V$ counts the green and left red component.}
        \label{fig:projection}
    \end{figure}

    Consider now $U$ fixed and let $V = \bar \fL(U)$, with an arbitrary convention on the boundary which will not have any impact. 
    \begin{claim}\label{claim:projection}
        There exists $C$ (depending on $\fa, \fb, \fc$) such that, if there exists $V' \supset V$ and $M$ such that $\sup_{v' \in V'} | \varphi_{111}(v') - \E[\varphi_{111}(v')]| \leq M$ and $\inf_{v \in \partial V, v' \in \partial V'} |v- v'| \geq CM $, then
    \[
    \bigg| \sum_{u \in U} \big(\varphi_{001}(u) - \E[\varphi_{001}(u)] \big)- \sum_{v \in V}\big( \varphi_{111}(v) - \E[\varphi_{111}(v)]\big) \bigg| \leq |\partial V| C M^3\,.
    \]
    \end{claim}
    \begin{proof}
  
    Fix a point $p_1 \in \varphi$ such that $v := \Upsilon_{111}(p_1) \in V$ but $\Upsilon_{001}(p_1) \notin U$. Let $p_2$ be the unique point in $\cP_{\fa\fb\fc}$ such that $\Upsilon_{111}(p_2) = v$, noting that by assumption $\| p_1 - p_2\|_1 \leq 3M$ and consequently $\| \Upsilon_{001}(p_1) - \Upsilon_{001}(p_2)\| \leq 2M$. On the other hand, by definition $\Upsilon_{001}(p_2) = \bar \fL^{-1}( v ) $ so it must be in~$U$. Overall, we see that $v$ must be at distance at most $2M$ of $\partial V$.
    
    Conversely, take $p_1 \in \varphi$ such that $u:= \Upsilon_{001}(p_1) \in U$ and $\Upsilon_{111}(p_1) \notin V$ and suppose by contradiction that $|\varphi_{001}(u) - \E[ \varphi_{001}(u) ] | \geq CM$ for a constant $C$ to be chosen large enough later.  As before, let $p_2$ be the point of $\cP_{\fa\fb\fc}$ such that $\Upsilon_{001}(p_2) = u$ and let $v_2 = \Upsilon_{111}(p_2) = \bar \fL(u)$. Let $p_2'$ be the point of $\varphi$ such that $\Upsilon_{111}(p_2') = v_2$, by assumption, since $v_2 \in V$, we must have $\|p_2' - p_2\|_1 \leq 3M$. Since $\|p_1 - p_2\| \geq CM$ and they only differ in the $z$ coordinate, there must exist a path in $\varphi$ starting at $p_2'$, (going to $p_1$) a doing at most $2M$ steps in the $x$ or $y$ directions while it does at least $CM$ steps in the $z$ direction, either all increasing or all decreasing depending on the sign of $\varphi_{001}(u) - \E[\varphi_{001}(u)]$. If $C$ is large enough, this contradicts the fact that $\varphi$ stays close to $\cP_{\fa\fb\fc}$ in a large neighborhood of $V$. We can then argue as in the previous paragraph that $v_2$ must be within distance $CM$ of $\partial V$. 

    Overall, the contributions to $
    \left| \sum_{u \in U} \big(\varphi_{001}(u) - \E[\varphi_{001}(u)] \big)- \sum_{v \in V}\big( \varphi_{111}(v) - \E[\varphi_{111}(v)]\big) \right|$
    can only come from cubes within distance $CM$ of the boundary of $V$, which concludes the proof of the claim. 
\end{proof}

    Recall from \cref{subsec:setup} that we interpret $\varphi_{001}(n \cdot) - \E[ \varphi_{001}( n\cdot)]$ as a piecewise constant function defined from $\R^2$ to $\R$ and that we write integrals as the $L^2$ scalar product. Recall also that, since the full plane Gaussian free field is only defined as a distribution up to a global shift, we only need to prove convergence of 
    \[
    \langle \varphi_{001}(n \cdot) - \E[ \varphi_{001}( n\cdot)] , f \rangle
    \]
    for test functions $f$ with $0$ mean. 
    For clarity let us first focus on the convergence of
 \[
    \langle \varphi_{001}(n \cdot) - \E[\varphi_{001}( n\cdot)] , \one_{U^+} - \one_{U^-} \rangle
    \]
    for two disjoint bounded open sets $U^{\pm}$ with smooth boundaries and the same areas. 
    As an output of \cref{thm:gmt-refinement} and in particular from \cref{eq:cumulants}, we see that all cumulants of $\varphi_{111}(v) - \varphi_{111}(o)$ of order more than $2$ are bounded uniformly in $v$ and that $\Var(\varphi_{111}(v) - \varphi_{111}(o))$ grows logarithmically in $|v-o|$. In particular, applying Markov to a high enough moment and Borel--Cantelli, we see that almost surely, for all $n$ large enough
    \[
    \sup_{v \in B(o, n)} |\varphi_{111}(v) - \E[\varphi_{111}(v)]| \leq n^{1/4}\,.
    \]
    We can then apply \cref{claim:projection} to both $\sum_{u \in  n U^+}$ and $\sum_{u \in n U^-}$. The error term coming from the claim is $O(n^{{7/4}})$ since $|\partial n U^{\pm}| = O(n)$ and the other errors coming from approximating the sets $U^{\pm}$ by a union of squares contribute $O(n^{5/4})$ so overall
    \begin{align*}
    \sum_{u} \big(\varphi_{001}(u) - \E[\varphi_{001}(u)]\big)\big(\one_{nU^+}(u) - \one_{nU^-}(u)\big) &= \sum_{v} \big(\varphi_{111}(v) - \E[\varphi_{111}(v)]\big)\big(\one_{nV^+}(v) - \one_{nV^-}(v)\big) \\&+ O(n^{7/4} )\,.
    \end{align*}
    Moving from a discrete sum to an $L^2$ inner product, \[\sum_{u\in nU} \big(\varphi_{001}(u) - \E[\varphi_{001}(u)]\big)\one_{n U}(u) = \langle \varphi_{001}(\cdot) - \E[\varphi_{001}( \cdot)] , \one_{nU}\rangle = n^2 \langle \varphi_{001}(n \cdot) - \E[\varphi_{001}( n\cdot)] , \one_{U}\rangle\]
    whereas, considering $\varphi$ as a piecewise constant function over hexagons in the dual of the lattice $\T$,
    \[\operatorname{Area}(\mbox{hexagon}) \cdot \sum_{v\in nV} \big(\varphi_{111}(v) - \E[\varphi_{111}(v)]\big)\one_{n V}(v) = 
    n^2 \langle \varphi_{111}(n \cdot) - \E[\varphi_{111}( n\cdot)] , \one_{V}\rangle\,.\]
    The area of a hexagon in the dual lattice to $\T$ is $1/\sqrt3$, and so
    \[
    \langle \varphi_{001}(n \cdot) - \E[\varphi_{001}( n\cdot)] , \one_{U^+} - \one_{U^-} \rangle = \sqrt{3}\left\langle \varphi_{111}(n \cdot) - \E[\varphi_{111}( n\cdot)] , \one_{\bar \fL U^+} - \one_{\bar \fL U^-} \right\rangle  + O( n^{-1/4} )
    \]
    and hence, since $\varphi_{111} - \E[ \varphi_{111}]$ converges to $\sigma \GFF \circ \fL$,
    \begin{align*}
    \langle \varphi_{001}(n \cdot) - \E[\varphi_{001}( n\cdot)] , \one_{U^+} - \one_{U^-} \rangle &\longrightarrow \sqrt{3} \left\langle \GFF \circ \fL , \one_{\bar \fL U^+} - \one_{\bar \fL U^-} \right\rangle \\ &= \sqrt{3} \left|\det( \bar \fL)\right| \left\langle \GFF \circ \fL \circ \bar \fL , \one_{U^+} - \one_{U^-} \right\rangle\,. 
    \end{align*}
    For a general test function $f$, we simply apply the previous arguments with the integral formula $f = \int_0^\infty \one_{ \{ f > t \}} \d t- \int_0^\infty \one_{ \{ f < -t\}}\d t$. The superlevel sets in that formula are smooth when $f$ is by the implicit function theorem. We can exchange the integral and the limit $n \to \infty$ because the error in the above proof clearly only depends on $|\partial U^\pm|$.
\end{proof}

The final step, as mentioned above, is to prove that $h - \varphi$ does not contribute to the scaling limit of $h$. Note that for $h$, only the \SOS convention makes sense so we stick to this convention and drop the indices.
\begin{lemma}\label{lem:h_minus_cond_exp}
Uniformly over $\varphi$ and $u$, $h( u) - \E[ h(u) \mid \varphi]$ has an exponential tail, $h( \cdot) - \E[h(\cdot) \mid \varphi]$ has exponentially decaying covariance and for all $U$, 
\[
\langle h( n \cdot) - \E[h(n \cdot) \mid \varphi] , \one_{U} \rangle \xrightarrow[]{\;L^2\;} 0\,.
\]
\end{lemma}
\begin{proof}
    This is essentially a corollary of the analysis of $\pi_{\varphi, \hat\beta}$ in \cref{sec:pi}. Indeed, as already mentioned in \cref{rk:h_conditionnal}, the conditional law of $h$ given $\varphi$ is $\pi_{\varphi, \beta}$. The exponential tail of $h(o)- \E[h(o)\mid\varphi]$ is an immediate consequence of the exponential bound on the size of bubbles given in \cref{eq:peierls_pi} and in the bound following it for bubbles of zero excess area (note that it is given there for $\pi_r$ but since the bound is uniform over $r$ it applies for a measure on the torus too). In particular, $h(\cdot) - \E[h(\cdot) \mid \varphi]$ is square integrable and its variance is uniformly bounded in $\varphi$.

    For the covariance, fix $x$ and $y$ and let $r = |x-y|$. First note that, again since the size of bubbles has an exponential tail, it is enough to look only on the event where the bubbles at both $x$ and $y$ are smaller than $r/10$. We can then replace the actual law of $h$ given $\varphi$ by the law in $B(o, 2r)$ with boundary condition $\varphi$ again with an exponentially small error in $r$. Then, running the dynamics of \cref{sec:pi} for a time $c r$ with $c$ small enough, except on an event of exponentially small probability, information does not have time to propagate between $B(x, r/2)$ and $B(y, r/2)$ and so on that event $h \restriction_{B(x, r/2)}$ and $h \restriction_{B(y, r/2)}$ are independent. Overall, we see that the covariance decays exponentially with $r$ as desired.

    Once we control the pointwise variance and covariances, we deduce immediately that, uniformly over $\varphi$, $\Var( \langle h( n \cdot) - \E[h(n \cdot) \mid \varphi] , \one_{U} \rangle ) = O(n^{-2} )$  which concludes.
\end{proof}

\begin{lemma}\label{lem:cond_exp_minus_phi}
    For the law on full plane tilings $\varphi$ given by \cref{thm:gmt-refinement}, $\E[h(o) \mid \varphi] - \varphi(o)$ is bounded and has mean $m$, where $m:=\E[h(u)]-\E[\varphi(u)]$ does not depend on $u$, $\E[h(\cdot) \mid \varphi] - \varphi(\cdot)$ has a decaying covariance and for all $U$, 
    \[
    \langle \E[h(n \cdot) \mid \varphi] - \varphi(n \cdot) - m, \one_{U} \rangle \xrightarrow[]{\;L^2\;} 0\,.
    \]
\end{lemma}
\begin{proof}
    $\E[h(o) \mid \varphi] - \varphi(o)$ is bounded again because of the arguments in \cref{sec:pi}. More precisely, since we proved that the size of bubbles has an exponential tail uniformly over $\varphi$, the expected size is bounded as a function of $\varphi$.
    
    By translation invariance, $\E[h(u)] - \E[ \varphi(u)]$ does not depend on $u$, and we call it $m$.
 
    For the covariance, we first note that
    \begin{align*}
    \E[h(o) \mid \varphi] = \varphi(o) &+ \sum_{\sB:o\in\Upsilon(\sB)} (h_\sB(o) - \varphi(o)) \pi_1(\sB\in h) \\ &+ \sum_{r\geq 1} \sum_{\sB: o\in\Upsilon(\sB) } (h_\sB(o) - \varphi(o)) (\pi_{2r}(\sB\in h) - \pi_r(\sB\in h) )\,,
    \end{align*}
    where $h_\sB(o)$ is the height of $o$ in any configuration containing $\sB$ as a bubble. This is almost the same as the decomposition in \cref{sec:pi} except that there is no integral over $\hat\beta$ and we have a slightly different weight per bubble. The arguments of \cref{sec:pi} apply directly and prove that
    \[
   \Big| \sum_{\sB: o\in\Upsilon(\sB)} (h_\sB(o) - \varphi(o)) (\pi_{2r}(\sB\in h) - \pi_r(\sB\in h) )\Big| \leq e^{- c r}
    \]
    for some constant $c$, i.e., $ \E[h(o) \mid \varphi] - \varphi$ is almost a local function in the same sense as $\fg$. Combined with the polynomial decay of the cumulants of the edge occupation variables in \cref{eq:cumulants}, we see that the covariance between $\E[h(o) \mid \varphi]$ and $\E[h(u) \mid \varphi]$ must decay to $0$ as $\|u\|$ goes to infinity (at a rate at least $e^{- \epsilon \sqrt{\log\|u\|}}$). It then immediately follows that $\Var(\langle  \E[h(n \cdot) \mid \varphi]-\varphi( n \cdot) , \one_{U} \rangle) \to 0$ which concludes the proof.
\end{proof}

Concluding the proof of \cref{thm:GFF-convergence} is then simply a matter of combining the previous results.
\begin{proof}[Proof of \cref{thm:GFF-convergence}]
    We write 
    \[
    h - \E[h] = \big( \varphi - \E[\varphi] \big) + \big( h - \E[h \mid \varphi] \big) + \big(\E[h \mid \varphi] - \varphi - m\big)
    \]
    with $m=\E[h]-\E[\varphi]$ as in \cref{lem:cond_exp_minus_phi}. The first parenthesis in the right-hand side converges to the Gaussian free field by \cref{lem:GGF_SOS_convention} while the other two converge to 0 by \cref{lem:h_minus_cond_exp,lem:cond_exp_minus_phi}.
\end{proof}

Let us note that in a sense \cref{thm:GFF-convergence} gives a much more precise statement about the limit than a simple control of the variance but the fact that the $\GFF$ is a distribution means it does not formally imply anything for pointwise moments. For the sake of completeness, we still derive the variance of height increments in the SOS convention.
\begin{proposition}
    As $\|u\|\to\infty$, 
    \[
    \Var( h(u) - h(o) ) = \frac{\sigma^2}{\pi p_\fa^2}\log\|u\| + o( \log\|u\| )\,,
    \]
    where $\sigma$ is given by \cref{thm:gmt-refinement}.
\end{proposition}
\begin{proof}
    Note that, by \cref{lem:h_minus_cond_exp,lem:cond_exp_minus_phi}, it is enough to prove the result for $\varphi_{001}(u) - \varphi_{001}(o)$ instead. Fix $u$, which, according to the convention in \cref{subsec:setup} and \cref{fig:projection_lozenge}, we label using half-integers. Let us denote by $e_{\uparrow}$ the projection on $\cP_{111}$ of the unit vector $e_3$. The first step is to describe concretely how to read the $\varphi_{001}$ height in terms of the $\varphi_{111}$ one.

     By definition $\varphi_{001}(u) = z$ if and only if the plaquette $(u_1 \pm \tfrac12, u_2 \pm \tfrac12, z)$ is in $\varphi$, which is then equivalent to the fact that
    \[
    \varphi_{111}\big( \Upsilon_{111} (u_1 + \tfrac12, u_2 + \tfrac12, z)\big) = \varphi_{111}\big( \Upsilon_{111} (u_1 - \tfrac12, u_2 - \tfrac12, z)\big)= z\,.
    \]
    Also, since both $\E[ \varphi_{111}(\cdot) ]$ and $\E[ \varphi_{001}(\cdot) ]$ describe the same plane $\cP_{\fa\fb\fc}$ and by definition of $\overline{\fL}$ we have $\E[ \varphi_{001}(u) ] = \E[\varphi_{111}( \overline{\fL}(u)) ]$. Using the linearity of the projection we therefore obtain that
    \[
    \Upsilon_{111} (u_1 + \tfrac12, u_2 + \tfrac12, z) = \overline{\fL}(u) + \Upsilon_{111}\big( \tfrac12, \tfrac12, z - \E[ \varphi_{001}(u) ]\big) = \overline{\fL}(u) + (z - \E[ \varphi_{001}(u) ] - \tfrac12) e_\uparrow\,.
    \]
    It follows that we would have
    \begin{equation}\label{eq:phi-001-u-exceed-mean-by-k}
    \varphi_{001}(u) = k + \E[ \varphi_{001} (u) ] 
    \end{equation}
    if and only if 
    \[ \varphi_{111}( \overline{\fL}(u) + (k+ \tfrac12) e_\uparrow) = \varphi_{111}( \overline{\fL}(u) + (k- \tfrac12) e_\uparrow) = k + \E[ \varphi_{111} (\overline{\fL}(u) )] \,.
    \]
    Since $k \to\varphi_{111}( \overline{\fL}(u) + k e_\uparrow) - k$ is non-increasing in $k$ we can equivalently reformulate the above as saying that \cref{eq:phi-001-u-exceed-mean-by-k} occurs if and only if
    \[
    \varphi_{111}( \overline{\fL}(u) + (k+ \tfrac12) e_\uparrow) - (k+ \tfrac12)  < \E[ \varphi_{111} (\overline{\fL}(u) )] \text{ and } \varphi_{111}( \overline{\fL}(u) + (k- \tfrac12) e_\uparrow) - (k- \tfrac12) > \E[ \varphi_{111} (\overline{\fL}(u) )] \, .
    \]
    
    Note that for any $v \in \cP_{111}$, $\E[\varphi_{111}( v + e_{\uparrow}) - \varphi_{111}(v) ] = 1- p_\fa$ . Fix $v_0 = \overline{\fL}(u) + \tfrac12 e_\uparrow$.
    First let us consider the good event $G$ where both $|\varphi_{111}(v_0) -\E[\varphi_{111}(v_0)]| \leq \log\|v_0\|$ and 
    \[
    \sup_{v \in B( v_0, \log^2\|v_0\|) } |\varphi_{111}(v) -\E[\varphi_{111}(v)] - \varphi_{111}(v_0) + \E[\varphi_{111}(v_0)]| \leq \log^{1/4}(\|v_0\|)\,.
    \]
    The probability of $G$ goes to $1$ as $\|u\| \to \infty$ because the variance of $\varphi_{111}(v_0)$ is of order $\log\|v_0\|$ and all cumulants of $\varphi_{111}(v) - \varphi_{111}(v_0)$ of order larger than $2$ are bounded (moments of order $20$ suffice for the union bound over the $O(\log^4\|v_0\|)$ points of the ball). Let $v_\pm$ be the points 
    \[
    v_0 + \frac{\varphi_{111}(v_0) - \E[\varphi_{111}(v_0)]}{p_\fa} e_{\uparrow} \pm \frac{2\log^{1/4}(\|v_0\|)}{p_\fa} e_{\uparrow}\,.\]
    On $G$ they are both in $B(v_0, \log^2\|v_0\|)$ and therefore
    \begin{align*}
    \varphi_{111}(v_+)  &\leq  \varphi_{111}(v_0)+ \frac{1}{p_\fa} \E[ \varphi_{111}(v_+) - \varphi_{111}(v_0)  ] + \log^{1/4}(\|v_0\|) \\
    & \leq\varphi_{111}(v_0) + \frac{1-p_\fa}{p_\fa}  \big(\varphi_{111}(v_0) - \E[\varphi_{111}(v_0)] \big) +  \frac{2- 2p_\fa}{p_\fa} \log^{1/4}(\|v_0\|) +  \log^{1/4} \|v_0\| \\
    & < \E[\varphi_{111}(v_0)] + \frac{1}{p_\fa} \big(\varphi_{111}(v_0) - \E[\varphi_{111}(v_0)] \big) + \frac{2}{p_\fa}\log^{1/4}(\|v_0\|)\,,
    \end{align*}
    and similarly
    \[
    \varphi_{111}(v_-)  - \Big( \frac{\varphi_{111}(v_0) - \E[\varphi_{111}(v_0)]}{p_\fa} - \frac{\log^{1/4}(\|v_0\|)}{p_\fa}\Big)> \E[\varphi_{111}(v_0)]\,.
    \]
    Therefore, still on the event $G$,
    \[
    \varphi_{001}(u) - \E[\varphi_{001}(u) ] = \frac{\varphi_{111}(v_0) - \E[\varphi_{111}(v_0) ]}{p_\fa}+ O( \log^{1/4}(\|u\| ))\,.
    \]
    Since $\Var(\varphi_{111}(v_0))=\frac{\sigma^2}{\pi}\log\|v_0\|+o(\log\|v_0\|)$ (this pointwise statement is not implied by the convergence in \cref{thm:gmt-refinement}; it follows from the bounds on the two-point function of \cite{GMT20}, with the covariance $-\frac1{2\pi}\log|u-v|$ of our normalization of the \GFF), we see that $\E\big[ \big(\varphi_{001}(u) - \E[ \varphi_{001}(u) ]\big)^2\one_G \big] = \frac{\sigma^2}{\pi p_\fa^2} \log\|u\| + o(\log\|u\|)$ as desired.

    We now turn to the analysis of the bad event $G^c$. As for the good event (or \cref{lem:GGF_SOS_convention}), we will use some uniform control on $\varphi_{111} - \E[ \varphi_{111}]$ to move between the two conventions, but this is a bit more delicate because the domain has to end on the deviation at the reference point $v_0$.
    For $k \geq 1, \ell \geq 0$ and $(k,\ell) \neq (1,0)$, let $B_{k, \ell}$ be the event where both $(k -1)\log \|v_0\| \leq |\varphi_{111}(v_0) -\E[\varphi_{111}(v_0)]| < k \log \|v_0\| $ and
    \[
    \sup_{ v \in B(v_0, k \ell \log^2(\|v_0\|))} |\varphi_{111}(v) -\E[\varphi_{111}(v)] - \varphi_{111}(v_0) + \E[\varphi_{111}(v_0)]| \geq k \ell \log^{1/4}(\|v_0\|)\,,
    \]
    but where the analogous statement for $\ell +1$ does not hold. 
        Applying a moment bound with powers $10$ we see that (if $\|v_0\|$ is large enough)
    \[
    \P( B_{k, 0} ) \leq \frac{C}{k^{10} \log^5(\|v_0\|)}\,,
    \]
    while, with moments of order $40$ and a union bound over the $O(k^2\ell^2\log^4\|v_0\|)$ points of the ball, we find that for $\ell \geq 1$
    \[
     \P( B_{k, \ell } ) \leq \frac{C}{k^{37} \ell^{37} \log^5(\|v_0\|)} \,.
    \]
    On the event $B_{k, \ell}$, reasoning as in the case of $G$ we have $|\varphi_{001}(u) - \E[\varphi_{001}(u) ]| \leq 2 k (\ell + 1) \log^2(\|v_0\|)$ and therefore the total contribution of all the events $B_{k \ell}$ to the variance is $O(1/\log\|v_0\|)$. On the other hand, by Borel--Cantelli and the above bound on the probabilities, we see that (up to a negligible event) $G^c \subset \bigcup B_{k \ell}$, which concludes the proof.
\end{proof}

\section{Renormalization group analysis: proof of Theorem~\ref{thm:gmt-refinement}}\label{sec:gmt-refine}

In this section, we show how to improve the results of \cite{GMT17,GMT20} to be able to apply the output of \cref{thm:phi-weakly-interacting}. 
We start with a fairly lengthy setup: \cref{sec:Kasteleyn} gives a very short introduction to Kasteleyn theory in non-interacting dimers, it can be safely skipped if one already knows this theory. \cref{sec:micro1} presents our strategy to deal with the microcanonical setting. \cref{sec:grassman_review} introduces the formalism of Grassmann integrals and lists the properties needed later, again it can be safely skipped if the reader is already familiar with the formalism. The main arguments start in \cref{sec:grassmann_formulation} where we rewrite our problem using Grassmann integrals; this is the core of the section with the improved bounds on the determinants appearing in the truncated expectations (see \cref{def:truncated-exp}).
\Cref{sec:sketch-renormalize} gives an overview of the induction of \cite{GMT20}, which we apply essentially as a black box. Finally, we provide the adaptation to the microcanonical setting in \cref{sec:micro2}.

When discussing the GMT framework for proving \cref{thm:gmt} and the obstacles standing in the way of proving the refined \cref{thm:gmt-refinement}, the reader should have the following two examples in mind:

\begin{example}[Prototypical application of \cref{thm:gmt}]\label{ex:gmt-appl}
For admissible $\fa,\fb,\fc>0$ and a small enough fixed $\delta>0$, the distribution $\mu_N$ is on tilings $\varphi$ of the torus $\T_N$ given by
\[ \mu_N(\varphi) \propto \fa^{n_\fa(\varphi)}\fb^{n_\fb(\varphi)}\fc^{n_\fc(\varphi)} \exp\Big[\delta\sum_x \sum_{y:\,\dist(x,y)=1} \one_{\{\operatorname{type}(\varphi,x)=\operatorname{type}(\varphi,y)\}}\Big]
\]
where $n_\fs(\cdot)$ is the number of lozenges of type $\fs$, and $\operatorname{type}(\cdot,x)$ is the lozenge type at the face~$x$.
\end{example}

\begin{example}[Prototypical application of \cref{thm:gmt-refinement}]\label{ex:gmt-appl-refine}
For admissible $p_\fa,p_\fb,p_\fc>0$ and a large enough fixed $C>0$, the distribution $\mu_N$ is on tilings $\varphi$ of the torus $\T_N$ whose lozenge types are $n_\fb(\varphi)=\lfloor p_\fb N^2 \rfloor$ and $n_\fc(\varphi)=\lfloor p_\fc N^2 \rfloor$ (with $n_\fa=N^2-n_\fb-n_\fc$) that is given by
\[ \mu_N(\varphi) \propto \exp\Big[\sum_x \sum_{r}  e^{- C r} \one_{\{x \text{ connected to }\partial B(x,r)\mbox{ in $\varphi$ via lozenges of the same type}\}}\Big]\,.\]
\end{example}

\subsection{Crash course on Kasteleyn theory}\label{sec:Kasteleyn}

In the following, we will sometimes need coordinates on the hexagonal lattice. We use the following convention. First, the two bipartite classes will be called black and white and we will use different variables $b$ and $w$ for black and white vertices. We assume without loss of generality that the lattice is composed of regular hexagons with a horizontal side and that the black vertex is on the right of that edge. We fix a horizontal edge and say that \emph{both} the white and black vertices adjacent to that edge have coordinate $(0,0)$. For other points, we use non-orthogonal coordinates as indicated in \cref{fig:coordinates}. We say that horizontal (resp.\ \northwest to \southeast and \northeast to \southwest) edges have types $\fa$, $\fb$ and $\fc$ respectively.

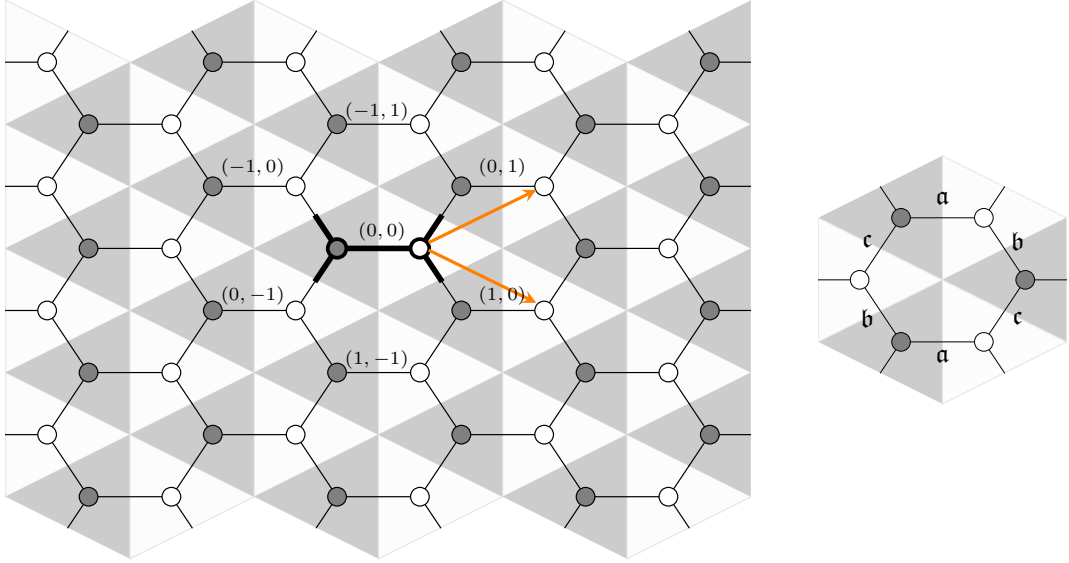
\begin{figure}
\vspace{-0.1in}
  \begin{tikzpicture}[font=\tiny]
   \node (fig1) at (0,0) {
    	\includegraphics[width=0.6\textwidth]{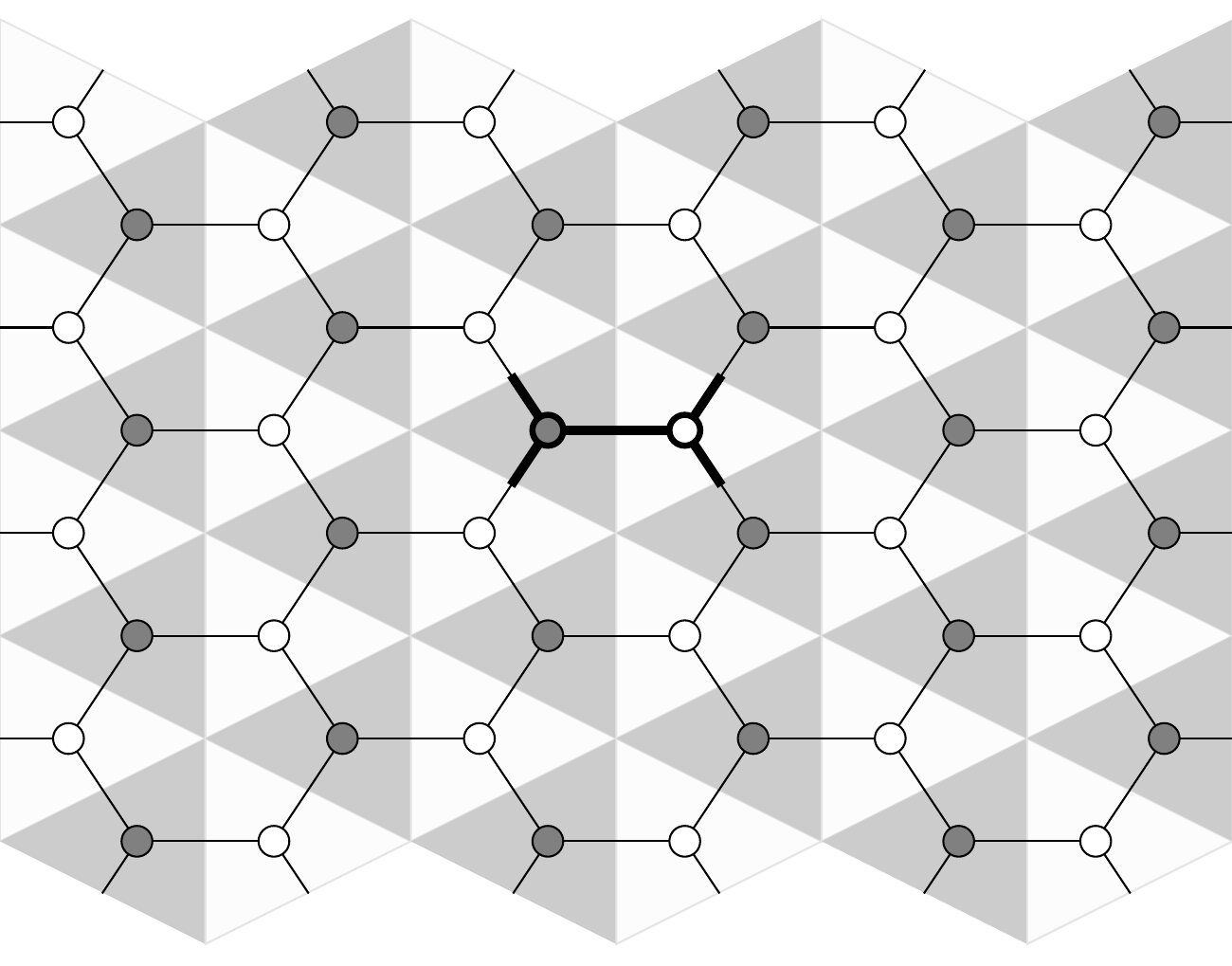}};
   \node (fig2) at (7.5,0) {  
    \includegraphics[width=.2\textwidth]{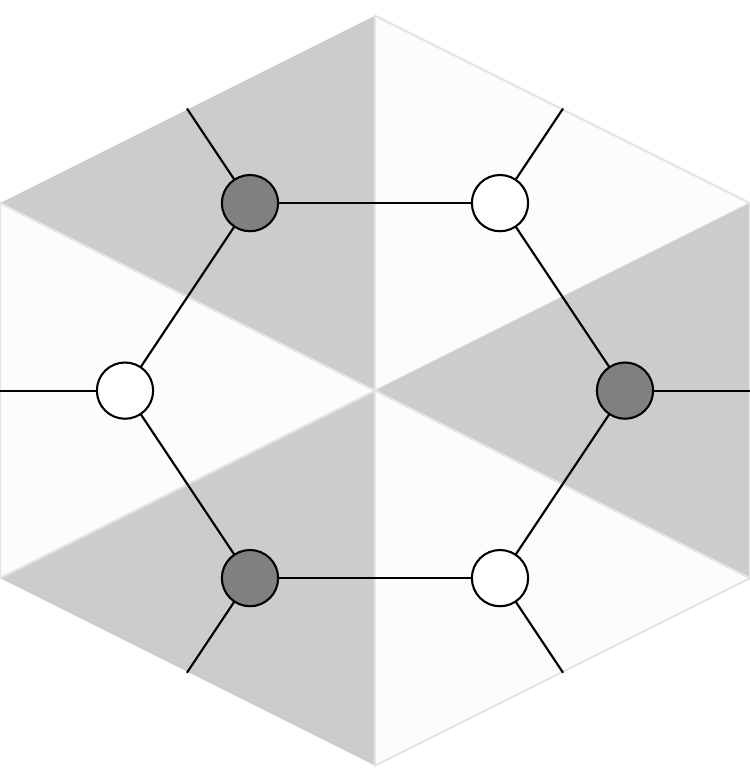}};
    \draw[->,orange,very thick,-stealth] (0.67,0.5)--(2.1,1.2);
    \draw[->,orange,very thick,-stealth] (0.67,0.4)--(2.1,-.3);
    \node at (0.05,0.65) {$(0,0)$};
    \node at (1.65,1.5) {$(0,1)$};
    \node at (1.65,-.2) {$(1,0)$};
    \node at (-1.67,1.5) {$(-1,0)$};
    \node at (-1.67,-.2) {$(0,-1)$};
    \node at (-0.02,2.25) {$(-1,1)$};
    \node at (-0.02,-1.05) {$(1,-1)$};
    \node[font=\small] at (7.5,1.1) {$\fa$};
    \node[font=\small] at (8.5,0.52) {$\fb$};
    \node[font=\small] at (6.5,0.52) {$\fc$};
    \node[font=\small] at (7.5,-1) {$\fa$};
    \node[font=\small] at (6.5,-0.5) {$\fb$};
    \node[font=\small] at (8.5,-0.5) {$\fc$};
    \end{tikzpicture}
    \vspace{-0.1in}
    \caption{Left: The coordinate system. Right: The types of edges.}
    \label{fig:coordinates}
    \vspace{-0.05in}
\end{figure}

Consider a simply connected piece of the hexagonal lattice $D$ and construct a matrix $K$ indexed by black vertices on the rows and white vertices on the columns with $K(b,w) = \one_{\{b \sim w\}}$. In the formula for $\det(K)$ in terms of permutation, we see that the only non-zero terms come from permutations which only associate adjacent white and black vertices, i.e., the non-zero permutations are exactly describing lozenge tilings and $\det(K) = \sum_{\varphi} \epsilon(\varphi)$ where $\epsilon(\varphi)$ is the signature of the permutation described by $\varphi$. Now it is also true that the set of tilings of a simply connected domain is closed under rotation of hexagons in the tiling, which in terms of permutation are simply multiplication by a cycle of $3$ elements. Since a cycle of length $3$ has signature $+1$,  $\epsilon(\varphi)$ does not actually depend on $\varphi$ and $|\det(K) | = \# \{ \text{tilings of $D$}\}$. Applying the above argument to both $D$ and $D\setminus \{b_0, w_0\}$ for some pair of adjacent vertices $(b_0w_0)$, the comatrix formula for the determinant says that
\[
|K^{-1}(w_0, b_0) |= \frac{\# \{ \text{tilings of $D\setminus \{b_0,w_0\}$}\}}{\# \{ \text{tilings of $D$}\}} = \P( (b_0w_0) \text{ occupied})
\]
with the $\P$ under a uniformly chosen tiling. In fact, the comatrix formula has a generalization to minors which gives for any finite set of edges $(b_iw_i)$
\begin{equation}\label{eq:pattern_proba}
\P( \forall i, \,  (b_iw_i) \text{ occupied}) = |\det( K^{-1}(b_i, w_j) )|\,.
\end{equation}
We thus see that a very natural (and fruitful) approach to the dimer model is to try to identify~$K^{-1}$.

From this argument, it is also easy to see that if now we put weights (possibly even complex) $\fa$, $\fb$, $\fc$ depending on the types of edges, the associated determinant becomes a partition function of tilings depending on their weights. However, again since hexagon rotations connect all tilings and preserve the number of tiles of each type, this is a trivial operation at the combinatorial level. We will not really use it here but it can however be a relevant change to do in order to get nice asymptotic formulas for $K^{-1}$.

On the torus, it is not quite true that rotations connect all tilings: they can connect all tilings with a given number of tiles of each type but this is not fixed anymore. In terms of height, a tiling can be seen as describing a discrete vector field which is still closed (i.e., its integral along a path is invariant under a deformation of the path) but in a non-simply connected domain this is not enough for it to admit a true primitive. If we still want to have a primitive, we obtain in general an ``additively multi-valued function,'' which can be seen simply as a function in the full plane whose derivatives are periodic and described by the tiling of the torus. It can be seen that the information about the number of tiles of each type can be easily read in terms of the global slope of the primitive, or equivalently, the height gaps between two copies or the integral of the derivative along two non-contractible cycles. We call that information the \emph{instanton component} of the configuration. As a consequence, introducing (positive) weights $\fa, \fb, \fc$ on the edges depending on their types has an interesting effect on the associated dimer measure: it affects the law of the instanton component while keeping the configuration uniform given it. 
Algebraically, this manifests as follows. Let $K_{\fa, \fb, \fc}$ denote the weighted matrix with weights depending on the types of edges for the $N \times N$ torus. Then in all tilings the number of tiles of each type must be a multiple of $N$ and
\begin{equation}\label{eq:detK}
   \det(K_{\fa, \fb, \fc}) = \sum_{k,\ell} \pm(k,\ell) \fa^{N^2 - N k - N \ell} \fb^{N k} \fc^{N\ell} Z(k,\ell) \,,
\end{equation}
where $\pm(k,\ell)$ is a sign which in fact only depends on the parity of $k$ and $\ell$ and $Z(k, \ell)$ is the number of tilings $\varphi$ with lozenge types $n_\fa(\varphi)=N^2-Nk-N\ell$, $n_\fb(\varphi)=N k$ and $n_\fc(\varphi)= N\ell $.

It is actually possible to identify completely the function $\pm(k,\ell)$ (since this depends on further conventions on the matrix we will not do it here) and to then obtain the unsigned partition function as a combination of four determinants. Note that the required sign changes cannot be achieved by changing the signs of the weights $\fb$ and $\fc$, as for even $N$ the four matrices $K_{\fa,\pm\fb,\pm\fc}$ have the same determinant; instead, replacing $\fb$ by $\fb e^{i\pi/N}$ multiplies the contribution of the sector $n_\fb(\varphi)=Nk$ by $(-1)^k$, as in \cref{eq:detK-twist} below, and similarly for $\fc$. 
This is the classical route to study the dimer model on the torus (see for example \cite{KOS06}) and in particular this is used in \cite{GMT17,GMT20}. 

In what follows, a \emph{pattern} $P$ is an arbitrary set of lozenges in $\T$. For now, it can be disconnected. (Later, in the proof of our main estimate in \cref{prop:grassmann_formulation}, we will work with connected patterns.)
The analog of \cref{eq:pattern_proba} in this context, given a pattern $P$ containing $n_\fa(P), n_\fb(P), n_\fc(P)$ tiles of each type, becomes
\begin{align}\label{eq:pattern_torus}
    \sum_{\varphi} \pm(k,\ell)  \fa^{N^2 - Nk - N \ell} \fb^{N k} \fc^{N \ell} \one_{\{P \in \varphi\}} &= \pm(P) \fa^{n_\fa(P)} \fb^{n_\fb(P)} \fc^{n_\fc(P)} \det ( K_{\fa, \fb, \fc} \restriction_{P^c} ) \nonumber \\
    &=\pm(P)\det ( K_{\fa, \fb, \fc} )  \prod_{e\in P} K_{\fa, \fb, \fc}(e) \det ( K^{-1}_{\fa, \fb, \fc} \restriction_{P} )\,,
\end{align}
where the second line only holds if $K_{\fa, \fb, \fc}$ is invertible. Let us emphasize that the above formula holds even if the pattern $P$ wraps around the torus and indeed even if it determines completely $k$ and $\ell$. Indeed, the signs $\pm(k, \ell)$ in \cref{eq:detK} are determined by looking at the sign of the permutation associated to a non-contractible loop as a function of its homotopy class. This is not affected when removing the vertices of $P$ from the graph and therefore the same function $\pm(k, \ell)$ has to appear when expanding $\det(K\restriction_{P^c})$ and $\det(K)$. Only a global sign can appear to account for the change of convention when removing some rows and columns of the matrix $K$.

By taking the correct $\fa, \fb, \fc$ it is possible to choose which pairs $(k, \ell)$ will have the main contribution to the partition function, or, in probabilistic terms, to choose the typical instanton component and the choice is actually explicit: Fix $p_\fa, p_\fb, p_\fc$ positive and summing to $1$, let $\Delta$ be a triangle with angles $\pi p_\fa, \pi p_\fb, \pi p_\fc$ and let $\fa, \fb, \fc$ be the length of the edges opposite to the associated angle. Then (at least asymptotically) the measure with weights $\fa,\fb,\fc$ gives probabilities $p_\fa, p_\fb, p_\fc$ to edges of the corresponding types. It is further known that, for this measure, the instanton component converges to a discrete Gaussian distribution without any normalization \cite{BdT09,KSW16}; this will be used in \cref{sec:micro2}.

Because of the signs, \cref{eq:pattern_torus} does not directly give a formula for the probability of a pattern $P$. 
However, it turns out that the $4$ variants of $K$ are close enough that one can safely
ignore this issue (more-or-less, see \cref{sec:grassman_review,sec:micro2} for more on the subtleties around this point) and still write probabilities as determinants of (any version of) $K^{-1}$ with a very small error. This is particularly useful since the $K$ matrices are \emph{diagonal in Fourier space}: Indeed, fix $\fa,\fb,\fc$ (or fix $p_\fa, p_\fb, p_\fc$ and take $\fa, \fb, \fc$ according to the above construction), $k = (k_1, k_2)$ and let $\hat w_k = \sum_w e^{i \left< k, w \right>} e_w$ and $\hat b_k = \sum_b e^{-i \left< k, b \right>} e_b$ where $\left<k,w\right> = k_1 n_1 + k_2 n_2$ with $(n_1, n_2)$ the coordinates of $w$. We have
\begin{align*}
    K \hat w_k & = \sum_w e^{i \left< k, w \right>} ( \fa e_b + \fb e_{b + (1,0)}  + \fc e_{b + (0,1)}) %
                 = \sum_b e^{i \left<k,b\right>} ( \fa + \fb e^{-i k_1} + \fc e^{-i k_2} ) e_b \\
                & = ( \fa + \fb e^{-i k_1} + \fc e^{-i k_2} ) \hat b_{-k}\,,
\end{align*}
where in the computation, we used $b$ to denote the black vertex with the same coordinates as $w$. We see that it is relevant to introduce the Newton polynomial $P(z_1, z_2) = \fa + \fb z_1 + \fc z_2$ and we may note that its zeros are exactly $z_1 = e^{i\pi(p_\fa + p_\fb)}$ and $z_2 = e^{-i \pi (p_\fa + p_\fc)}$.

\subsection{Moving to a microcanonical setting}\label{sec:micro1}

In the context of this paper, we have to consider the dimer model on a torus with \emph{fixed} instanton component instead of the more classical ``canonical'' setting where all possible configurations are possible but have varying weights. In this section we show how to adapt the Kasteleyn theory presented above to this setting.
Going back to \cref{eq:detK}, the idea is that instead of summing $4$ copies to remove the sign, we will introduce complex phases to $\fb,\fc$ and then use a Fourier inversion formula to extract a single sector.

Throughout \cref{sec:gmt-refine}, we write $n(\varphi):=(n_\fb(\varphi),n_\fc(\varphi))$ for the counts of a tiling $\varphi$ of $\T_N$ (so that $n_\fa(\varphi)=N^2-n_\fb(\varphi)-n_\fc(\varphi)$), recalling that both coordinates of $n(\varphi)$ are multiples of $N$, and we fix the target counts $n_N=(n_{\fb,N},n_{\fc,N})$ from \cref{thm:gmt-refinement}, setting
\begin{equation}\label{eq:pN-def}
p^{(N)}:=\big(p^{(N)}_\fa,p^{(N)}_\fb,p^{(N)}_\fc\big):=\Big(1-\frac{n_{\fb,N}+n_{\fc,N}}{N^2},\frac{n_{\fb,N}}{N^2},\frac{n_{\fc,N}}{N^2}\Big)\,,
\end{equation}
so that $Np^{(N)}\in\Z^3$ and $p^{(N)}\to p=(p_\fa,p_\fb,p_\fc)$ as $N\to\infty$. Since $p$ lies in the interior of the simplex, we may and will assume that $N$ is large enough so that $N\leq n_{\fb,N}$, $N\leq n_{\fc,N}$ and $n_{\fb,N}+n_{\fc,N}\leq N^2-N$. We write $n_{\fa,N}:=N^2-n_{\fb,N}-n_{\fc,N}$ and refer to $\{\varphi: n(\varphi)=n_N\}$ as the \emph{target sector}. We denote by $\epsilon(\varphi)\in\{-1,1\}$ the sign of the permutation associated to $\varphi$ in the expansion of \cref{eq:detK}; as explained there, $\epsilon(\varphi)$ only depends on the parities of $n_\fb(\varphi)/N$ and $n_\fc(\varphi)/N$, so it is constant on every sector, and we denote by $\epsilon_N$ its value on the target sector.

Fix positive weights $w=(\fa,\fb,\fc)$. For $\vartheta=(\vartheta_1,\vartheta_2)\in\R^2$ we write $K(\vartheta)=K_{\fa,\fb e^{i\vartheta_1},\fc e^{i\vartheta_2}}$, and, for a tiling $\varphi$,
\[ w^{n(\varphi)}:=\fa^{n_\fa(\varphi)}\fb^{n_\fb(\varphi)}\fc^{n_\fc(\varphi)}\qquad\mbox{and}\qquad \langle\vartheta,n(\varphi)\rangle:=\vartheta_1 n_\fb(\varphi)+\vartheta_2 n_\fc(\varphi)\,.\]
In these notations, \cref{eq:detK} reads
\begin{equation}\label{eq:detK-twist}
\det K(\vartheta)=\sum_{\varphi}\epsilon(\varphi)\, w^{n(\varphi)}\,e^{i\langle\vartheta,n(\varphi)\rangle}\,.
\end{equation}
Since $n(\varphi)\in(N\Z)^2$, the right-hand side is $\frac{2\pi}{N}$-periodic in each coordinate of $\vartheta$, and it suffices to consider the twists in the grid
\begin{equation}\label{eq:Theta-N-def}
\Theta_N:=\Big\{\frac{2\pi}{N^2}(m_1,m_2)\,:\; m_1,m_2\in\{0,1,\ldots,N-1\}\Big\}\,.
\end{equation}
The following elementary identity is the only form of Fourier inversion that we shall use.
\begin{lemma}\label{lem:fourier-inversion}
Let $\Phi$ be a complex-valued function on tilings of $\T_N$ and let $n=(n_\fb,n_\fc)\in(N\Z)^2$ be such that $0<n_\fb<N^2$ and $0<n_\fc<N^2$. Then
\[
\frac{1}{N^2}\sum_{\vartheta\in\Theta_N} e^{-i\langle\vartheta,n\rangle}\sum_\varphi \epsilon(\varphi)\,w^{n(\varphi)}\,e^{i\langle \vartheta,n(\varphi)\rangle}\,\Phi(\varphi)=\epsilon(n)\, w^{n}\sum_{\varphi:\,n(\varphi)=n}\Phi(\varphi)\,,
\]
where $\epsilon(n)$ denotes the common value of $\epsilon(\varphi)$ over the sector $\{n(\varphi)=n\}$.
\end{lemma}
\begin{proof}
Writing $n(\varphi)-n=N(k,\ell)$ with $k,\ell\in\Z$ and $\vartheta=\frac{2\pi}{N^2}(m_1,m_2)$, we see $e^{i\langle\vartheta,n(\varphi)-n\rangle}$ is equal to $e^{2\pi i (m_1 k+m_2\ell)/N}$, whence
$\frac1{N^2}\sum_{\vartheta\in\Theta_N}e^{i\langle\vartheta,n(\varphi)-n\rangle}=\one_{\{k\equiv 0 \bmod N\}}\one_{\{\ell\equiv 0\bmod N\}}$. 
As $0\leq n_\fb(\varphi)\leq N^2$ and $0<n_\fb<N^2$ we have $|k|<N$, so $k\equiv 0\bmod N$ forces $k=0$, and similarly $\ell=0$. %
\end{proof}

Note that the signs $\epsilon(\varphi)$, which in the canonical setting are removed via a combination of four determinants, are harmless here: \cref{lem:fourier-inversion} recovers each sector up to the common sign $\epsilon(n)$, which cancels in any ratio of two quantities attached to the same sector. This is all we will need, since the output of the renormalization group analysis will only be used through ratios of the form $\cZ_\fg(A)/\cZ_\fg(0)$ (see \cref{eq:Z(A)-def} below) within the target sector.

The weights $w$ are an auxiliary device: the left-hand side of the identity in \cref{lem:fourier-inversion} depends on them, but the right-hand side only does so through the explicit factor $w^n$. Their role is to control which sectors carry most of the mass of the signed sums in \cref{eq:detK-twist}. In the non-interacting model, choosing $w$ as the weights associated to $p^{(N)}$ via the triangle construction of \cref{sec:Kasteleyn} makes the target sector typical: it is known\footnote{See \cite[Thm.~4]{BdT09} for $\fa=\fb=\fc$ and \cite[Thms.~1--4]{KSW16} for general weights (see also \cite[Thm.~7]{Dubedat15}); the instanton component converges to a discrete Gaussian distribution without any normalization.} that the mass of $\sum_{n}w^n Z(n)$ then concentrates on sectors within $O(N)$ of $n_N$, i.e., within $O(1)$ in the variables $(k,\ell)$ of \cref{eq:detK}, and that the target sector itself carries a positive fraction of it. In the interacting model the interaction shifts the typical slope by an amount of order $\sum_r\|\fg_r\|_\infty$, which moves the typical sector by order $N$ in the variables $(k,\ell)$; to compensate, the weights will be tuned in \cref{sec:micro2} as a function of $\fg$ and of $N$ (see \cref{rem:why-tuning} for why this is unavoidable).

With \cref{lem:fourier-inversion} in mind, we will proceed with the proof by analyzing the signed, twisted generating function for fixed $\vartheta\in\Theta_N$ and only treat the sum over $\vartheta$ after the whole renormalization procedure. This is in fact analogous to \cite{GMT17, GMT20} where also most of the analysis is focused on a fixed matrix out of the $4$ needed to cancel the signs in \cref{eq:detK}; the only difference is that we need uniformity over the $N^2$ twists in $\Theta_N$ instead of over $4$ of them.

\subsection{Grassmann integrals}\label{sec:grassman_review}

We begin with a short review of the Grassmann integral notation, see Section 4 of \cite{GentileMP} for a more detailed and informed account. It will be used crucially to allow us to write complicated series of minors in a way that looks like perturbed Gaussian integrals (similar to say the partition function of the $\phi^4$ model) and therefore to import into our problem the intuition and language of Gaussian integration.
\begin{definition}
    Consider two separate sets of $n$ indices $\{ b_i\}$ and $\{w_j \}$. The Grassmann integral is a formal notation using $2n$ variables $\psi_{w_i}, \psi_{b_j}$ defined by the following rule
    \begin{align*}
    &\int \psi_{w_1} \ldots \psi_{w_n} \psi_{b_1} \ldots \psi_{b_n} \d \psi = 1 \,,\\
    &\psi_{w} \psi_{w'} = - \psi_{w'} \psi_w \, , \quad \psi_{b} \psi_{b'} = - \psi_{b'} \psi_b \, , \quad \psi_w \psi_b = \psi_b \psi_w\,,
    \end{align*}
    and any integral not containing all variables exactly once gives $0$. Analytic functions of the variables $\psi$ are defined by their series expansion, which always gets truncated to a finite number of terms by the anti-symmetry assumption.
\end{definition}

The link between determinants and Grassmann integrals is the following.
\begin{proposition}
    Let $K$ be a matrix whose rows are indexed by the $b_i$ and columns are indexed by the $w_j$, we have
    \[
    \int \exp( \vec{\psi}_b \cdot K \cdot \vec{\psi}_w )\d\psi = \det(K)\,.
    \]
    More generally, for every subsets $I = \{i_k\}$ and $J=\{j_l\}$ of indices of $w_i$'s and $b_j$'s, resp., one has
     \[
    \int \prod_{i \in I} \prod_{j \in J} \psi_{w_i} \psi_{b_j} \exp( \vec{\psi}_b \cdot K \cdot \vec{\psi}_w )\d\psi = \pm \det(K \restriction_{I^c, J^c})\,.
    \]
\end{proposition}

\begin{proof}
    This is a direct computation: Since only monomials of degree exactly $2n$ have a non-zero integral, the exponential in the first statement is not different from $\frac{1}{n !} ( \vec{\psi}_b \cdot K \cdot \vec{\psi}_w )^n$. Developing the power, we see that the only non-vanishing terms come from taking every $\psi_w$ and $\psi_b$ exactly once so we can enumerate all terms using two permutations $\sigma_w$ and $\sigma_b$. Note that pairs $\psi_w \psi_b$ commute with other pairs so we can reorder the $\psi$'s to set $\sigma_b$ to the identity, simplifying the $\frac{1}{n !}$ term. This gives
    \[
    \int \exp(  \vec{\psi}_b \cdot K \cdot \vec{\psi}_w ) =  \sum_{\sigma} \prod_i K_{i\sigma(i)} \int \psi_{b_1} \psi_{w_{\sigma(1)}} \cdots \psi_{b_n} \psi_{w_{\sigma(n)}}\d\psi\,.
    \]
    To use the definition of the integral, one must also reorder the $\psi_b$ and it can be checked that the resulting change of sign is exactly the signature $\epsilon(\sigma)$.

    The statement about the minor follows the same proof.
\end{proof}

Let us note now a few additional algebraic properties that follow from this main definition.
\begin{corollary}\label{cor:grassmann}
    Suppose $K$ is as above and is also invertible. Then
    \begin{itemize}
        \item $\int \psi_{w} \psi_b \exp( \vec{\psi}_b \cdot K \cdot \vec{\psi}_w ) \frac{\d \psi}{\det(K)} = K^{-1}_{b,w}$.
        \item $\int \prod_{i \in I} \prod_{j \in J} \psi_{w_i} \psi_{b_j} \exp(  \vec{\psi}_b \cdot K \cdot \vec{\psi}_w ) \frac{\d \psi}{\det(K)} = \det(K^{-1} \restriction_{I, J}) = \sum_\sigma \epsilon(\sigma) \prod K^{-1}_{ w_i,b_{\sigma(i)}}$. 
        \item If $K^{-1} = A^{-1} + B^{-1}$, then \[ \int f(\vec\psi) e^{\vec\psi K \vec\psi} \frac{\d\psi}{\det(K)} = \iint f(\vec\psi_1 + \vec\psi_2) e^{ \vec\psi_1 A \vec\psi_1} e^{ \vec\psi_2 B \vec\psi_2} \frac{\d \psi_1}{\det A} \frac{\d\psi_2}{\det B} \,.\]
    \end{itemize}
    The last equality of the second line is called the Wick formula for Gaussian Grassmann integral since it writes a $2|I|$-point correlation in terms of 2 points correlations. The last point is called the addition principle.
\end{corollary}
\begin{proof}
    The first two lines are classical formulas about minors. For the last line, we can check that the Wick expansion matches separately for all monomials.
\end{proof}

Because of the expression of the last bullet point, we introduce the notation
\begin{equation*}
\int f(\psi) \d P_g(\psi)  := \int f(\psi) \exp(  \psi g^{-1} \psi ) \frac{\d \psi}{\det g^{-1}}\,,
\end{equation*}
where $g$ is a matrix. We emphasize that the Wick formula means that we can compute with these notations without going back to the original Grassmann integrals. In fact, we can even define the value of any such integral using only the Wick formula even when $g$ is not invertible. A particularly degenerate case of this will be when some rows and columns of $g$ are identically $0$. Suppose $g = g_1 + g_2$ with $g_1( i,j) = 0$ whenever $i\notin I$ or $j \notin J$. We can still use the formula
\begin{equation*}
\int f(\psi) \d P_g(\psi) = \int \Big[  \int f(\psi + \phi) \d P_{g_2}(\phi) \Big]  \d P_{g_1}(\psi)\,,
\end{equation*}
using in both integrals the full set of Grassmann variables. However, we see that, if any variable $\psi_{w_i}$ with $i \notin I$ comes out of the integral over the $\phi$ variables, then the corresponding term will give a $0$ contribution after the second integral. In other words, when expanding (the $w$ part of) the monomials in $\psi + \phi$ in $f$, the only meaningful terms come from the cases where we never choose any $\psi_{w_i}$ for $i \notin I$. Therefore, we would have the same combinatorics by replacing $\phi + \psi$ by $\phi + \psi \restriction_I$. After that transformation, it makes sense to only consider the outer integral with respect to $g_1$ as an integral over variables indexed in $I$ instead of the full set. The inner integral can still be over the full set of variables (if $g_2$ does not itself have rows or columns of $0$) with the sum performed only over the common coordinates.

Later we will use linear change of variables extensively. In this context the multilinearity of the Wick formula provides a very nice description.
\begin{lemma}
    Let $X_1, \ldots, X_n$ (resp.\ $Y_1, \ldots, Y_n$) be linear expressions depending on some Grassmann variables $\psi^+_k$ (resp.\ $\psi_k^-$) and let $g$ be a matrix with rows indexed by the $\psi^+$ and columns indexed by the $\psi^-$. We have  
    \[
    \int  \prod_{i=1}^n X_i \prod_{j=1}^n Y_j\ \d P_{g}( \psi ) = \det\big( g(X_i, Y_j) \big)_{1 \leq i,j \leq n}\,,
    \]
    where $g(X_i, Y_j) := \int  X_i Y_j \d P_g$ can be computed simply by bilinearity. 
\end{lemma}
The core of the proof of \cref{thm:gmt-refinement} (and of the original version) is a repeated application of the last point of \cref{cor:grassmann}. For this we introduce a further notation to describe the $\log$ of a Grassmann integral.
\begin{definition}\label{def:truncated-exp}
    The truncated expectation (with respect to a propagator $g$) is
    \[
    \cE^T_g ( X_1, \ldots, X_k, n_1, \ldots, n_k ) = \frac{\partial^{\sum n_i}}{\partial_{\lambda_1}^{n_1} \ldots \partial_{\lambda_k}^{n_k}} \log \int e^{\sum \lambda_i X_i} \d P_g(\psi) |_{\vec \lambda = 0}\,,
    \]
    where the $X_i$ are expressions in terms of the Grassmann variables $\psi_k$. It can be checked that the $n_i$ have exactly the same effect as repeating variables and that $\cE^T_g ( X_1, \ldots, X_k, n_1, \ldots, n_k )$ has the same algebra as $\prod X_i^{n_i}$. Write  $\cE^T_g(X_1,\ldots,X_k)$ for $n_i$ that are all $1$ with possibly repeated variables.
\end{definition}
Also, since by definition the truncated expectation is extracting coefficients from the log of the integral, it satisfies
\[
\int e^{X} \d P_g(\psi) = \exp\Big( \sum_n \frac{1}{n!} \cE_g^T( \underbrace{X, \ldots, X}_{n\text{ times}} ) \Big)\,,
\]
and more generally if we have some extra variables $\phi$ on which we are not integrating,
\[
\int e^{X}(\psi + \phi) \d P_g(\psi) = \exp\Big( \sum_n \frac{1}{n!} \cE^T_g( \underbrace{X(\cdot + \phi),\ldots,X(\cdot+\phi)}_{n\text{ times}} ) \Big)\,.
\]
There is an expression of the truncated expectation in terms of the regular expectation (which is basically the formula giving the cumulants of a random variable in terms of its moments) so, as for the regular expectation, if the total number of $+$ and $-$ variables are not equal then
$\cE_g^T( X_1, \ldots, X_k ) = 0$. Also, since $\int 1 \d P_g(\psi) = 1$ by construction, the truncated expectation with no variable is $0$. In particular, in the above two  formulas we can start the sum at $n=1$.

To bound a truncated expectation like the one above, we will use the following theorem (Battle--Brydges--Federbush formula, see \cite[Lemma 3]{GMT17}\footnote{Compared to that statement, the constants $c_i$ are just prefactors in their variables $X$ and we have a determinant and not a Pfaffian because we use two types of variables while they only use one.}):
\begin{theorem}\label{thm:truncated_expectation_bound}
    Consider $I_1, \ldots I_s$ some sets of indices (with as many $+$ terms as $-$ terms) and write $\psi_I = \prod_{i \in I} \psi_i$. If $s > 1$ then
    \[
    \cE^T_g ( \psi_{I_1} , \ldots, \psi_{I_s} ) = \sum_{T} \pm(T)\big( \prod_{e \in T} g_e \big) \int \det( G(T, \vec t) ) \d\P_T( \vec t )\,,
    \]
    where 
    \begin{itemize}
        \item the index $T$ in the sum is a tree over $\{1, \ldots, s\}$ with each edge marked by two indices in the associated $I_j$ and $\pm(T)$ is a sign depending on $T$.
        \item The product over $e$ runs over the edges of $T$ and $g_e$ is the propagator evaluated at the endpoints of $e$.
        \item $\P_T$ is a probability measure on $[0, 1]^{s^2}$ whose support is on matrices which can be written $t_{j,j'} = u_j \cdot u_{j'}$ for vectors $u_j$ of unit $L^2$-norms.
        \item $G( T, \vec t) $ is a size $\frac12\sum |I_j | - s +1$ matrix with entries indexed by all pairs of labels not appearing in $T$ and given by $G( T , \vec t)_{(j,i), (j', i')} = \vec{t}_{j, j'} g( (j,i), (j', i')  )$.
    \end{itemize}
\end{theorem}
Note that we use two indices for each parameter of $g$ just above because we need to identify one of the $I_j$ and then the exact $i$ in it.

Let us expand a bit the definition of the matrix $G(T, \vec t)$. Start from the matrix of size $\frac12\sum |I_j |$ with rows and columns indexed respectively by the $\psi_k^+$ and $\psi_k^-$ in the $I_j$. We emphasize that if some $\psi_k$ appears in several of the $I_j$, we do want to repeat it also in the matrix so that we can really think of it as having blocks indexed in $\{1, \ldots, s\}$. Next, apply a Gram-type construction on the blocks, multiplying every entry in the block $j,j'$ by $t_{j,j'} = u_j \cdot u_{j'}$. Note that already at this point, having repeated entries in the $I_j$ does not make the determinant vanish. Finally, we remove the rows and columns associated with the (labeled) edges of the tree $T$ to get $G(T, \vec t)$.

In turn, $\det(G(T, \vec t))$ will be bounded by the Gram--Hadamard inequality, recalled next.
\begin{theorem}[Gram--Hadamard]\label{thm:hadamard}
    Let $v_i$, $\tilde v_i$ be two families of vectors in some Euclidean space. Let $G$ be the Gram matrix $G_{i,i'} = v_i \cdot \tilde{v}_{i'}$. Then
   $ |\det (G) | \leq \prod_i \lVert v_i \rVert  \cdot \lVert \tilde v_i \rVert $.
\end{theorem}

\subsection{Grassmann formulation of the problem}\label{sec:grassmann_formulation}

The goal of this section is to provide an expression in terms of Grassmann variables of the generating function of dimers. It is useful in order to extract probabilities from the generating function to also introduce extra edge weights which we encode following the notations from \cite{GMT20} as a function $A$, so we define, with the target sector $n_N$ of \cref{sec:micro1},
\begin{equation}\label{eq:Z(A)-def} \cZ_\fg( A) = \sum_\varphi \exp \Big(\sum_{x, r} \fg_r(\varphi \restriction_{B(x, r)} ) +\sum_e A(e) \one_{\{e \in \varphi\}}\Big) \one_{\{n(\varphi) = n_N \}}\,, \end{equation}
and a signed variant for the future Fourier inversion: for $\vartheta=(\vartheta_1,\vartheta_2)\in\Theta_N$ and writing $n(\varphi)-n_N=N(k,\ell)$,
\begin{equation}\label{eq:Z-twisted-def}
\cZ_\fg( \vartheta_1, \vartheta_2, A) :=  \sum_{\varphi} \epsilon(\varphi)\, \fa^{- Nk - N \ell} \fb_1^{N k } \fc_2^{N \ell} \exp \Big[\sum_{r, x} \fg_r( \varphi \restriction_{B(x,r)} )  + \sum_e A(e) \one_{\{e \in \varphi\}}\Big]\,,
\end{equation}
where, here and in \cref{sec:sketch-renormalize,sec:micro2},
\begin{equation}\label{eq:bar-b-bar-c}
\fb_1 := \fb\, e^{i\vartheta_1}\quad,\quad
\fc_2 := \fc\, e^{i\vartheta_2}\,.
\end{equation}
In words, $\cZ_\fg(\vartheta_1,\vartheta_2,A)$ is the signed sum \cref{eq:detK-twist}, tilted by the interaction and by the source $A$, and normalized by the weight $w^{n_N}e^{i\langle\vartheta,n_N\rangle}$ of the target sector; by \cref{lem:fourier-inversion},
\begin{equation}\label{eq:inversion-Z}
\frac1{N^2}\sum_{\vartheta\in\Theta_N}\cZ_\fg(\vartheta_1,\vartheta_2,A)=\epsilon_N\,\cZ_\fg(A)\,.
\end{equation}
To lighten the notation, in \cref{sec:grassmann_formulation,sec:sketch-renormalize} we write $p_\fs$ for $p^{(N)}_\fs$, so that $N^2p_\fs=n_{\fs,N}$ for $\fs\in\{\fa,\fb,\fc\}$.
Up to replacing $\fg_r$ by $\fg_r + \| \fg_r\|_\infty$ in the above definitions we can assume without loss of generality, henceforth in this section, that $\fg_r\geq 0$ for all $r$.
Throughout this section (and also in \cref{sec:sketch-renormalize}) it will be legitimate to have complex-valued weights $\fa,\fb,\fc$.
The key result is the following proposition, replacing \cite[Prop.~1]{GMT20}:
\begin{proposition}\label{prop:grassmann_formulation}
    There exist a set $\Gamma$ of connected subsets of $\T$ and functions $\sw_A : \Gamma \to \R$ such that for all $ \vartheta_1, \vartheta_2$,
    \[
    \cZ_\fg( \vartheta_1, \vartheta_2, A) = \fa^{- N^2 p_\fa} \fb_1^{-N^2 p_\fb} \fc_2^{- N^2 p_\fc}\int \exp \Big( \vec{\psi}_b \cdot K(\vartheta_1, \vartheta_2)\cdot \vec{\psi}_w + \sum_{\gamma \in \Gamma} e^{i \vartheta_1 n_\fb(\gamma)+ i \vartheta_2 n_\fc(\gamma)}\sw_A(\gamma) \psi_\gamma \Big) \d \psi
    \]
    where $\psi_\gamma := \prod_{v \in \gamma} \psi_v$ and $n_\fb(\gamma), n_\fc(\gamma)$ are the number of type $\fb$ and $\fc$ edges in a tiling of $\gamma$.

    Furthermore, the set $\Gamma$ contains the single edges $\{b,w\}$, for which $\sw_A(\{b,w\})=K(b,w)(e^{A(b,w)}-1)$, and for every $C^\star$, if $\beta$ and $\alpha$ are large enough and $\| A\|_\infty \leq 1$, the weights satisfy $0 \leq \sw_A(\gamma) \leq \exp( - C^\star |\gamma| )$ for all $\gamma$ with $|\gamma|>2$. 
\end{proposition}
When $K(\vartheta_1, \vartheta_2)$ is invertible, the integral from above can be rewritten as \[\det( K( \vartheta_1, \vartheta_2)) \int \exp( \sum \sw_A(\gamma) \psi_\gamma) \d P_g(\psi)\,,\] with $g$ the propagator associated to $K^{-1}$. 
We will further prove that the weights $\sw_A$ satisfy the so-called lattice Ward identities \cite[Eq.~(3.34)--(3.35)]{GMT20}. They involve another generating function that is now seen as depending on another Grassmann field $\tilde \psi$ which anti-commutes and has the same black and white type as $\psi$. Taking into account \cref{prop:grassmann_formulation} we define
\begin{align*}
\cZ_\fg( \vartheta_1, \vartheta_2, A, \tilde \psi) := \fa^{- N^2 p_\fa} \fb_1^{-N^2 p_\fb} \fc_2^{- N^2 p_\fc}\int \exp \Big( &\vec{\psi}_b \cdot K(\vartheta_1, \vartheta_2)\cdot \vec{\psi}_w \\&+\sum_{\gamma \in \Gamma}  e^{i \vartheta_1 n_\fb(\gamma)+ i \vartheta_2 n_\fc(\gamma)} \sw_A(\gamma) \psi_\gamma + \sum_v \psi_v \tilde \psi_v\Big) \d \psi\,.
\end{align*}
The lattice Ward identities are then
\begin{lemma}\label{lem:lattice_ward}
    For every black vertex $b$ and white vertices $w, w'$,
    \[
    \sum_{b' \sim w'} \partial_{A(b', w')}\partial_{\tilde \psi_b}\partial_{\tilde \psi_w} \log \cZ_\fg( \vartheta_1, \vartheta_2, A, \tilde \psi) |_{A = 0, \tilde \psi = 0} = - \delta_{w, w'} \partial_{\tilde \psi_b}\partial_{\tilde \psi_w} \log \cZ_\fg( \vartheta_1, \vartheta_2, A, \tilde \psi) |_{A = 0, \tilde \psi = 0} \,,
    \]
    and similarly for the vertices $b, b', w$,
        \[
    \sum_{w' \sim b'} \partial_{A(b', w')}\partial_{\tilde \psi_b}\partial_{\tilde \psi_w} \log \cZ_\fg( \vartheta_1, \vartheta_2, A, \tilde \psi) |_{A = 0, \tilde \psi = 0} = - \delta_{b, b'} \partial_{\tilde \psi_b}\partial_{\tilde \psi_w} \log \cZ_\fg( \vartheta_1, \vartheta_2, A, \tilde \psi) |_{A = 0, \tilde \psi = 0} \,.
    \]
\end{lemma}

To understand the proof of \cref{prop:grassmann_formulation}, it is useful to first reformulate a bit the final formula as follows:
\begin{align*}
    \int \exp \Big(\vec{\psi}_b & \cdot  K(\vartheta_1, \vartheta_2)\cdot \vec{\psi}_w + \sum_{\gamma \in \Gamma}   e^{i \vartheta_1 n_\fb(\gamma)+ i \vartheta_2 n_\fc(\gamma)} \sw_A(\gamma) \psi_\gamma \Big) \d \psi \\ 
    & =  \int \sum_{ \{ \gamma_i \}}  \prod_i  e^{i \vartheta_1 n_\fb(\gamma_i)+ i \vartheta_2 n_\fc(\gamma_i)} \sw_A(\gamma_i) \psi_{\gamma_i} e^{\vec{\psi}_b \cdot K(\vartheta_1, \vartheta_2)\cdot \vec{\psi}_w} \d \psi \\
    & = \sum_{ \{ \gamma_i \} \text{ disjoint} } \prod_i  e^{i \vartheta_1 n_\fb(\gamma_i)+ i \vartheta_2 n_\fc(\gamma_i)} \sw_A(\gamma_i) \det (K(\vartheta_1, \vartheta_2) \restriction_{\T \setminus \bigcup \gamma_i} )\,.
\end{align*}
The first line is simply an expansion of the exponential and in the second line, the condition that the $\gamma_i$ must be disjoint comes from the anti-commutativity of the variables (the ordering of the variables in each $\psi_\gamma$ can be chosen so that no extra sign appears in the second line). By Kasteleyn theory, if one can further associate to each $\gamma_i$ a pattern with support $\gamma_i$ which with a slight abuse of notation we still call $\gamma_i$, one then gets
\begin{multline}
\label{eq:midway_grassmann}
 \int \exp \Big(\vec{\psi}_b \cdot K(\vartheta_1, \vartheta_2)\cdot \vec{\psi}_w + \sum_{\gamma \in \Gamma} e^{i \vartheta_1 n_\fb(\gamma)+ i \vartheta_2 n_\fc(\gamma)}\sw_A(\gamma) \psi_\gamma \Big) \d \psi \\
 = \fa^{N^2 p_\fa} \fb_1^{N^2 p_\fb} \fc_2^{N^2 p_\fc}\sum_{\varphi} \pm(k, \ell) \fa^{ - Nk - N \ell} \fb_1^{N k } \fc_2^{N \ell}  \sum_{ \{ \gamma_i \} \text{ disjoint} } \prod_i \frac{\sw_A(\gamma_i)}{\prod_{e \in \gamma_i} K(e)}  \one_{\{ \gamma_i \in \varphi\}}\,,
\end{multline}
so our goal is now to get to an expression with the form of the right-hand side of \cref{eq:midway_grassmann}. Let us note that while a straightforward expansion of the exponential in the definition of $\cZ_\fg$ does provide such an expression, the resulting sets $\gamma$ are then unions of balls and their weights actually grow exponentially with their area because of the necessary enumeration over all possible tilings which is not compensated by the decay of the functions $\fg_r$ with the radius $r$ (see \cref{rem:boundary-trick} for more on this). Instead, we will reformulate our partition function in terms of the boundaries of connected sets of patterns, leveraging the facts that (1) our patterns are balls (and in particular, connected); and (2) our base (dimer) model has a Markov property. In the folow-up work~\cite{LaslierLubetzky26}, we further develop this proof idea in a more general context of cluster expansion.

\begin{proof}[Proof of \cref{prop:grassmann_formulation}]
    The first step is to write each $\fg_r( \varphi \restriction_{B(x, r)})$ as a sum of indicators and then to expand the exponential leading to
   \begin{align}
         \cZ_\fg( \vartheta_1, \vartheta_2, A) 
         &= \sum_{\varphi} \pm(k, \ell) \fa^{- Nk - N \ell} \fb_1^{N k} \fc_2^{N \ell} \exp \Big[\sum_x \sum_{P} \fg_{r(P)}(P) \one_{\{ \tau_x P \in \varphi\}}\Big] \nonumber\\
       & =\sum_{\varphi} \pm(k, \ell) \fa^{ - Nk - N \ell} \fb_1^{N k} \fc_2^{N \ell}  \prod_x \prod_{P} \big(1 + (\exp( \fg(P) ) -1) \one_{\{\tau_x P \in \varphi\}}\big) \nonumber \\
      &  = \sum_{\varphi} \pm(k, \ell) \fa^{ - Nk - N \ell} \fb_1^{N k} \fc_2^{N \ell}  \sum_{ \{ (P_j, x_j) \} } \prod_j (\exp( \fg(P_j) ) -1) \one_{\{ \forall j\, \tau_{x_j} P_j \in \varphi\}} \,. \nonumber
\end{align}
We stress that, since our functions $\fg_r$ are supported on balls $B(0,r)$, the patterns $\tau_x P$ considered above are nothing but tilings of $B(x,r)$'s (and in particular they are \emph{connected}). Recall that when we refer to a tiling of $B(x,r)$, we mean a collection of lozenges that covers its interior as well as every triangle that touches $\partial B(x,r)$.

Note that in the second line, the notation highlights the fact that we can read the value of $r$ from the pattern $P$ (therefore in the later line we drop the redundant index). We can also include the effect of the extra weights $A$ just as patterns consisting of a single edge and with a slight abuse of notation we still call their weights $\fg(P)$ (we will go back to an explicit notation for single edges when $A$ will vary later). It will also be convenient later to assume that before adding the single edges $\fg$ only gave weights to patterns of size larger than $4$, this can always be done by adding irrelevant edges as needed.

\begin{figure}
\vspace{-0.1in}
    \includegraphics[width=0.5\textwidth]{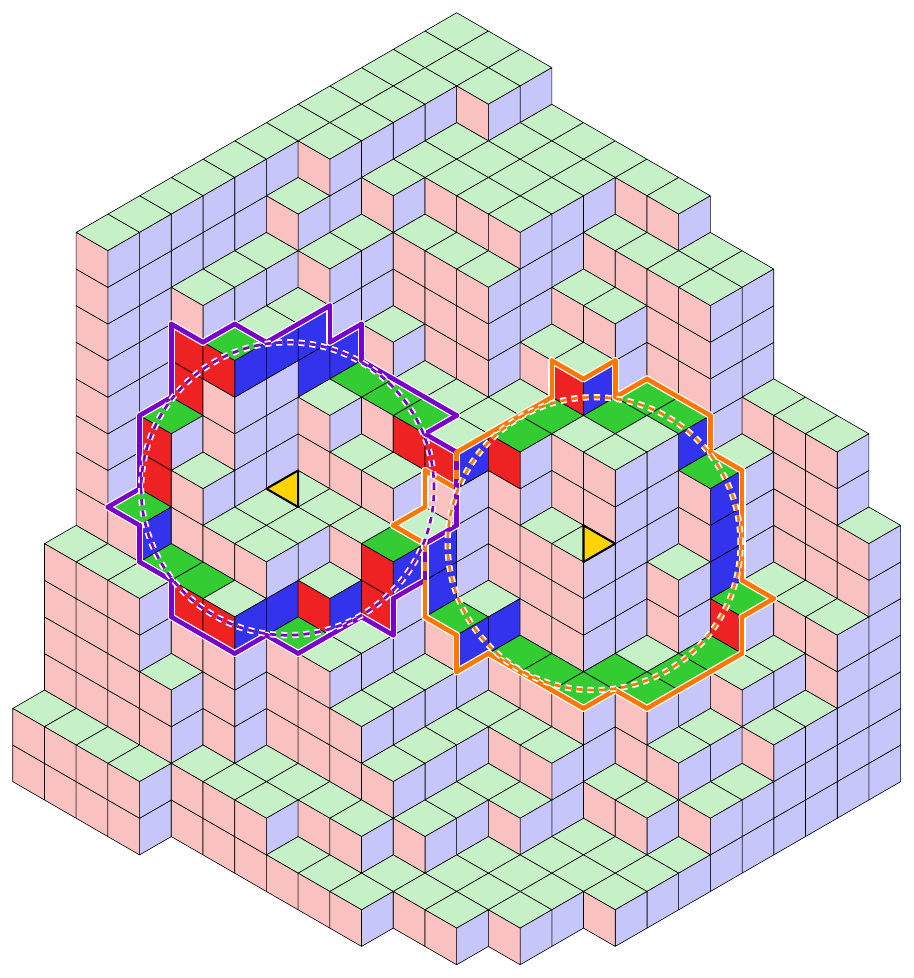}
    \vspace{-0.1in}
    \caption{An example of a boundary $\gamma$ associated with two connected patterns $P_1, P_2$. Here $P_1$ and $P_2$ are the tilings inside the purple and orange circuits and they are tilings of $B(x_i,4)$ for the two triangles $x_i$ highlighted in yellow. The set of lozenges comprising $\gamma$ is highlighted in bold colors. In this example, $P_1 \cup P_2$ is simply connected, so this is the inner boundary of $P_1 \cup P_2$; in general it would only be the \emph{outermost component} of it. Note that because of the convention on the definition of the ball $B(x, r)$ in a tiling, in this example $P_1$ and $P_2$ are connected even though the Euclidean balls around $x_1,x_2$ do not intersect.}
    \label{fig:gamma-example}
    \vspace{-0.15in}
\end{figure}

We say that two patterns $P_1, P_2$ are connected if $P_1 \cap P_2$ contains at least one full triangle. Note that in the last equation above, any set $\{ P_j\}$ containing at least two connected patterns disagreeing in their intersection gives a contribution $0$. For $\fP$ a connected collection of compatible patterns that does not wrap around the torus, let $\tilde \gamma$ be the boundary of the unbounded component of $\C \setminus \fP$ which we can see as a simple curve using edges of the triangular lattice. Let $\gamma$ then be the union of all closed lozenges of $\fP$ with a nonempty intersection with $\tilde \gamma$; see \cref{fig:gamma-example} for an illustration.

{
\newcommand{\tikzdisk}[4]{%
  \fill[#4, fill opacity=0.15] (#1,#2) circle (#3);
  \draw[#4, dashed, dash pattern=on 4pt off 3pt, line width=1pt]
        (#1,#2) circle (#3);}

\begin{figure}
\vspace{-0.1in}
\begin{tikzpicture}[scale=0.6]

  \draw[thick,fill=gray!10] (0,1) rectangle (10,9);
  \begin{scope}
    \clip (0,0) rectangle (10,10);   %

    \tikzdisk{1.30}{6.70}{1.10}{blue!70!black}
    \tikzdisk{3.00}{6.45}{1.05}{orange!90!black}
    \tikzdisk{4.70}{6.45}{1.10}{blue!70!black}
    \tikzdisk{6.40}{6.45}{1.05}{orange!90!black}
    \tikzdisk{8.00}{6.70}{1.05}{blue!70!black}
    \tikzdisk{8.80}{6.45}{1.00}{orange!90!black}

    \tikzdisk{10.00}{6.55}{1.05}{brown!80!black}
    \tikzdisk{0.00}{6.55}{1.05}{brown!80!black}

    \tikzdisk{4.70}{5.03}{0.66}{blue!55!cyan}
    \tikzdisk{4.70}{3.88}{0.66}{red!70!orange}
    \tikzdisk{3.70}{3.30}{0.62}{blue!55!cyan}
    \tikzdisk{2.70}{3.88}{0.66}{red!70!orange}
    \tikzdisk{2.70}{5.03}{0.66}{blue!55!cyan}

    \tikzdisk{8.60}{5.03}{0.66}{blue!55!cyan}
    \tikzdisk{8.60}{3.88}{0.60}{red!70!orange}
    \tikzdisk{7.60}{3.30}{0.66}{blue!55!cyan}
    \tikzdisk{6.60}{3.88}{0.62}{red!70!orange}
    \tikzdisk{6.60}{5.03}{0.66}{blue!55!cyan}

    \draw[color={rgb:red,128;green,0;blue,128}, , line width=1.7pt, line join=round, line cap=round]
      (0.76,5.85) -- (0.67,5.86) -- (0.49,5.73) -- (0.31,5.64) -- (0.16,5.60) -- (0.00,5.59);
    \draw[color={rgb:red,128;green,0;blue,128}, , line width=1.7pt, line join=round, line cap=round]
      (0.00,7.51) -- (0.13,7.50) -- (0.29,7.47) -- (0.52,7.38) -- (0.75,7.55) -- (0.91,7.63) -- (1.12,7.69) -- (1.33,7.71) -- (1.59,7.67) -- (1.84,7.56) -- (2.02,7.41) -- (2.21,7.15) -- (2.26,7.12) -- (2.32,7.13) -- (2.49,7.26) -- (2.68,7.35) -- (2.87,7.40) -- (3.07,7.41) -- (3.26,7.37) -- (3.44,7.31) -- (3.61,7.19) -- (3.77,7.04) -- (3.83,7.02) -- (3.88,7.04) -- (4.07,7.24) -- (4.26,7.36) -- (4.49,7.44) -- (4.73,7.46) -- (4.96,7.43) -- (5.16,7.35) -- (5.36,7.22) -- (5.53,7.04) -- (5.59,7.02) -- (5.64,7.04) -- (5.78,7.19) -- (5.95,7.30) -- (6.13,7.37) -- (6.32,7.41) -- (6.51,7.40) -- (6.70,7.36) -- (6.88,7.28) -- (7.06,7.16) -- (7.12,7.15) -- (7.16,7.18) -- (7.34,7.40) -- (7.54,7.54) -- (7.79,7.64) -- (8.05,7.66) -- (8.22,7.63) -- (8.39,7.58) -- (8.54,7.49) -- (8.71,7.36) -- (9.00,7.34) -- (9.28,7.23) -- (9.34,7.25) -- (9.58,7.41) -- (9.80,7.49) -- (10.00,7.50);
    \draw[color={rgb:red,128;green,0;blue,128}, , line width=1.7pt, line join=round, line cap=round]
      (10.00,5.60) -- (9.95,5.59) -- (9.71,5.63) -- (9.42,5.76) -- (9.22,5.64) -- (9.01,5.55) -- (8.99,5.50) -- (8.99,5.46) -- (9.10,5.31) -- (9.15,5.16) -- (9.17,5.04) -- (9.16,4.91) -- (9.12,4.79) -- (9.05,4.68) -- (8.94,4.58) -- (8.80,4.49) -- (8.77,4.42) -- (8.80,4.36) -- (8.98,4.22) -- (9.08,4.06) -- (9.11,3.85) -- (9.06,3.66) -- (9.00,3.56) -- (8.92,3.48) -- (8.82,3.42) -- (8.71,3.38) -- (8.52,3.38) -- (8.40,3.41) -- (8.28,3.48) -- (8.20,3.47) -- (8.17,3.41) -- (8.17,3.26) -- (8.15,3.14) -- (8.05,2.96) -- (7.97,2.87) -- (7.87,2.80) -- (7.76,2.75) -- (7.64,2.73) -- (7.52,2.74) -- (7.39,2.77) -- (7.28,2.83) -- (7.18,2.91) -- (7.07,3.10) -- (7.03,3.23) -- (7.03,3.38) -- (6.99,3.44) -- (6.92,3.45) -- (6.80,3.39) -- (6.69,3.36) -- (6.51,3.36) -- (6.31,3.44) -- (6.15,3.59) -- (6.10,3.70) -- (6.07,3.81) -- (6.07,3.92) -- (6.09,4.03) -- (6.18,4.21) -- (6.27,4.29) -- (6.38,4.37) -- (6.41,4.43) -- (6.38,4.50) -- (6.25,4.58) -- (6.15,4.68) -- (6.08,4.79) -- (6.04,4.90) -- (6.03,5.03) -- (6.04,5.16) -- (6.09,5.28) -- (6.17,5.41) -- (6.18,5.47) -- (6.14,5.52) -- (5.97,5.59) -- (5.84,5.67) -- (5.62,5.87) -- (5.57,5.88) -- (5.53,5.86) -- (5.38,5.70) -- (5.14,5.54) -- (5.11,5.48) -- (5.11,5.43) -- (5.20,5.31) -- (5.25,5.18) -- (5.27,5.06) -- (5.26,4.94) -- (5.23,4.82) -- (5.17,4.71) -- (5.08,4.61) -- (4.96,4.52) -- (4.93,4.45) -- (4.96,4.39) -- (5.07,4.31) -- (5.15,4.22) -- (5.25,4.04) -- (5.27,3.92) -- (5.26,3.80) -- (5.24,3.69) -- (5.19,3.58) -- (5.11,3.49) -- (5.02,3.41) -- (4.91,3.35) -- (4.80,3.32) -- (4.59,3.32) -- (4.48,3.36) -- (4.35,3.43) -- (4.28,3.42) -- (4.23,3.36) -- (4.23,3.23) -- (4.20,3.11) -- (4.10,2.95) -- (3.92,2.82) -- (3.71,2.77) -- (3.59,2.78) -- (3.48,2.82) -- (3.30,2.95) -- (3.20,3.11) -- (3.17,3.23) -- (3.17,3.36) -- (3.13,3.42) -- (3.05,3.43) -- (2.92,3.36) -- (2.79,3.32) -- (2.57,3.32) -- (2.45,3.37) -- (2.34,3.44) -- (2.25,3.53) -- (2.19,3.63) -- (2.13,3.84) -- (2.16,4.06) -- (2.25,4.23) -- (2.44,4.39) -- (2.47,4.46) -- (2.44,4.51) -- (2.29,4.63) -- (2.19,4.77) -- (2.14,4.94) -- (2.14,5.11) -- (2.17,5.23) -- (2.23,5.35) -- (2.32,5.45) -- (2.44,5.54) -- (2.47,5.61) -- (2.45,5.66) -- (2.28,5.81) -- (2.15,5.99) -- (2.11,6.02) -- (2.05,6.02) -- (1.89,5.88) -- (1.74,5.79) -- (1.57,5.73) -- (1.40,5.69) -- (1.23,5.69) -- (1.06,5.72) -- (0.90,5.77) -- (0.76,5.85);

   \draw[color={rgb:red,128;green,0;blue,128}, line width=1.7pt, line join=round]
      (6.89,5.55) -- (7.04,5.39) -- (7.12,5.26) -- (7.16,5.11) -- (7.17,4.96) -- (7.13,4.82) -- (7.07,4.70) -- (6.97,4.59) -- (6.82,4.50) -- (6.79,4.42) -- (6.82,4.37) -- (6.96,4.27) -- (7.06,4.14) -- (7.12,4.00) -- (7.14,3.79) -- (7.18,3.76) -- (7.23,3.75) -- (7.43,3.85) -- (7.63,3.87) -- (7.80,3.83) -- (7.98,3.74) -- (8.05,3.75) -- (8.09,3.80) -- (8.10,3.98) -- (8.16,4.13) -- (8.24,4.24) -- (8.40,4.35) -- (8.43,4.41) -- (8.41,4.49) -- (8.24,4.59) -- (8.13,4.71) -- (8.06,4.84) -- (8.03,4.99) -- (8.04,5.14) -- (8.08,5.27) -- (8.19,5.42) -- (8.34,5.55) -- (8.36,5.60) -- (8.35,5.65) -- (8.20,5.75) -- (7.93,5.74) -- (7.69,5.79) -- (7.51,5.87) -- (7.35,5.99) -- (7.29,6.00) -- (7.24,5.98) -- (7.10,5.79) -- (6.90,5.62) -- (6.89,5.55) -- cycle;
  \end{scope}   

\filldraw[color=green!50!black, pattern=north east lines, preaction={fill=green!15}, pattern color=green!50!black] (7.6,5.7) rectangle (7.67,7.7);

\filldraw[color=green!50!black, pattern=north east lines, preaction={fill=green!15}, pattern color=green!50!black] (7.6,4) rectangle (7.67,2.6);

\end{tikzpicture}
\vspace{-0.2in}
\caption{An example of a boundary $\gamma$ of a connected collection of compatible patterns~$\fP$ that wraps around the torus. Individual patterns are represented by balls, the additional cut is made of the two green lines and the resulting boundary $\gamma$ is the union of the cut and the purple curve. Note that, for the represented choice of a cut, $\fP \setminus \gamma$ has the topology of an annulus, the inner hole is on the left and everything else is considered outside of $\fP \setminus \gamma$.}
\label{fig:gamma-branch-cut}
\vspace{-0.1in}
\end{figure}
}

For $\fP$ a connected collection of compatible patterns that wraps around the torus, we first choose cuts using edges of the triangular lattice and staying inside the support of $\fP$ such that $\fP$ minus the cuts does not wrap around the torus anymore. Note further that, since we do not ask the cuts to consist of a bounded number of components, it is easy to make sure that the total length of the cuts is at most $2N$. We then let $\tilde \gamma$ be the outer boundary of $\fP$ minus the cuts, and as before we let $\gamma$ be the set of (closed) lozenges of $\fP$ with a nonempty intersection with $\tilde \gamma$. Note that most likely some lozenges of $\gamma$ will sit across the cut while in other spaces $\gamma$ will have ``width $2$.'' For any $\gamma$ obtained from the above procedure, we let $\operatorname{Int}(\gamma)$ denote the union of the bounded components of $\C \setminus \gamma$. See \cref{fig:gamma-branch-cut} for an illustration.

With a slight abuse of terminology, we will call the sets $\gamma$ obtained from the above procedure boundaries (even though they also contain lozenges, and that the cuts somewhat modify the notion of a boundary), and we let $\Gamma$ be the set of all $\gamma$ created by the above procedure.

We will need the following three functions associated to a boundary $\gamma$:
\begin{align*}
E_i(\gamma, A) & = \E\bigg( \exp\Big(\sum_x \sum_{\substack{P \text{ s.t.}\\ \tau_x P \subset \Int(\gamma)}} \fg(P) \one_{\{ \tau_x P \in \varphi\}}\Big) \bigg) 
= \E \bigg(  \sum_{ \substack{\{ (P_j, x_j) \}\text{ s.t.}\\  \tau_{x_j} P_j\subset \Int(\gamma) } }\prod_j ( e^{\fg(P_j)} -1) \one_{\{ \tau_{x_j} P_j \in \varphi\}} \bigg)\,; \\
E_c( \gamma, A) & = \E \bigg(  \sum_{ \substack{\{ (P_j,x_j)\} \text{ s.t.} \\ \partial( \bigcup P_j) = \gamma} } \prod_j ( e^{\fg(P_j)} -1) \one_{\{ \tau_{x_j} P_j \in \varphi\}} \bigg) \,;\\
\tilde \sw_A(\gamma) &= \frac{E_c(\gamma, A)}{E_i(\gamma, A)}\,.
\end{align*}
In the definition of $E_c$, the condition $\partial( \bigcup P_i) = \gamma$ is a shorthand for the fact that $\gamma$ is associated with the outermost connected component of $\bigcup P_i$ by the construction of the previous paragraph. We will also say that the set $\{(P_j, x_j)\}$ is compatible with $\gamma$ in that case. All expectations are with respect to the uniform measure on tilings of $\operatorname{Int}(\gamma)$ (with the convention for $E_c$ that $\gamma$ is appended to $\varphi$). With these notations, we come back to the sequence of equalities above, and see that
\begin{align}
    \cZ_\fg( \vartheta_1, \vartheta_2, A) & = \sum_{\varphi} \pm(k, \ell) \fa^{  - Nk - N \ell} \fb_1^{N k } \fc_2^{N \ell} \sum_{ \{ (P_j,x_j) \} } \prod_j \big(\exp( \fg(P_j) ) -1\big) \one_{\{ \tau_{x_j} P_j \in \varphi\}} \nonumber \\
    & = \sum_{\varphi} \pm(k, \ell) \fa^{  - Nk - N \ell} \fb_1^{N k } \fc_2^{N \ell} \sum^*_{ \{ \gamma_i\} } \prod_i \sum_{  \substack{\{(P_j,x_j)\}  \text{ s.t.}\\  \partial( \bigcup \tau_{x_j} P_j) = \gamma_i }} \prod_j \big(\exp( \fg(P_j) ) -1\big) \one_{\{  \tau_{x_j} P_j \in \varphi\}} \nonumber \\
    & = \sum_{\varphi} \pm(k, \ell) \fa^{  - Nk - N \ell} \fb_1^{N k } \fc_2^{N \ell} \sum^*_{\{ \gamma_i \} } \prod_i \one_{\{\gamma_i \in \varphi\}} E_c( \gamma_i )\,, \label{eq:gamma_induction}
\end{align}
where $\sum^*_{\{\gamma_i\}}$ goes over collections of $\gamma_i$ that are disjoint, and their interiors are also disjoint. From the first and second line, we simply group patterns according to the set of outermost boundaries~$\gamma_i$. From the second to the third line, we apply the Markov property of uniform tilings to the $\gamma_i$ to replace for each $\gamma_i$ the indicator $\prod_j \one_{\{\tau_{x_j} P_j \in \varphi\}}$ by $\E \big( \prod_j \one_{\{\tau_{x_j} P_j \in \varphi_i\}}\big)$ for a random tiling $\varphi_i$ of the interior of $\gamma_i$.
(Inside each $\Int(\gamma_i)$, the measure with weights $\fa, \fb, \fc$ and the uniform measure agree by gauge invariance. Equivalently, in a domain that does not wrap around the torus, the relative number of dimers of each type is always determined by the domain. Further note that the Markov property here is applied in a nonstandard form, as we want to keep the global sum over $\varphi$ that tile the whole torus. That is why we introduce an expectation term $E_c(\gamma_i)$ per $\gamma_i$ as opposed to a partition function for $\Int(\gamma_i)$.)

We now want to iterate the previous computation. As we introduced the weights $\tilde \sw_A$ to change $E_c$  into the simpler $E_i$, we can now again group the collections of patterns in each $E_i$ to get that
\begin{align*}
    \cZ_\fg( \vartheta_1, \vartheta_2, A)   & = \sum_{\varphi} \pm(k, \ell) \fa^{  - Nk - N \ell} \fb_1^{N k } \fc_2^{N \ell} \sum^*_{\{ \gamma_i \} } \prod_i \one_{\{\gamma_i \in \varphi\}} \tilde \sw_A(\gamma_i) E_i( \gamma_i ) \\ 
     &= \sum_{\varphi} \pm(k, \ell) \fa^{  - Nk - N \ell} \fb_1^{N k } \fc_2^{N \ell} \sum^*_{\{ \gamma_i \} } \prod_i \one_{\{\gamma_i \in \varphi\}} \tilde \sw_A(\gamma_i) \\ 
&\quad\qquad\qquad\qquad\qquad\qquad\qquad \cdot \E\bigg( \sum_{\{ \gamma^{(i)}_j\} }^* \prod_j \sum_{\substack{\{ ( P_\ell, x_\ell)\} \text{ s.t.} \\ \partial (\bigcup \tau_{x_\ell} P_\ell) = \gamma^{(i)}_j }} \prod_\ell (e^{\fg( P_\ell)} -1) \one_{ \{\tau_{x_\ell} P_\ell \in \varphi^{(i)}\}}
     \bigg)\,.
\end{align*}
Using the Markov property as above, we deduce  that
\begin{align*}
    \cZ_\fg( \vartheta_1, \vartheta_2, A)  
     &= \sum_{\varphi} \pm(k, \ell) \fa^{  - Nk - N \ell} \fb_1^{N k } \fc_2^{N \ell} \sum^*_{\{ \gamma_i \} } \prod_i \one_{\{\gamma_i \in \varphi\}} \tilde \sw_A(\gamma_i) \E\Big( \sum_{\{ \gamma^{(i)}_j\} }^* \prod_j E_c( \gamma_j^{(i)} )\one_{\{\gamma^{(i)}_j \in \varphi^{(i)}\}} \Big).
    \end{align*}
Applying the Markov property again but in the ``other direction'', we can then move the indicators back to $\varphi$ yielding
\begin{align*}
    \cZ_\fg( \vartheta_1, \vartheta_2, A)  
     &= \sum_{\varphi} \pm(k, \ell) \fa^{  - Nk - N \ell} \fb_1^{N k } \fc_2^{N \ell} \sum^*_{\{ \gamma_i \} } \prod_i \one_{\{\gamma_i \in \varphi\}} \tilde \sw_A(\gamma_i) \sum_{\{ \gamma^{(i)}_j\} }^* \prod_j  \one_{\{\gamma^{(i)}_j \in \varphi\}} E_c( \gamma_j^{(i)} ).
    \end{align*}

The last equation is similar to \cref{eq:gamma_induction} but with a two-level collection of boundaries instead of non-nested ones. We can iterate the procedure ad libitum which eventually creates arbitrary collections of disjoint boundaries, i.e.,
\[
\cZ_\fg( \vartheta_1, \vartheta_2, A) = \sum_{\varphi} \pm(k, \ell) \fa^{ - Nk - N \ell} \fb_1^{N k } \fc_2^{N \ell} \sum_{ \{ \gamma_i \} \text{ disjoint} } \prod_i \tilde \sw_A(\gamma_i) \one_{\{ \gamma_i \in \varphi\}}\,.
\]
This is of the form given in \cref{eq:midway_grassmann} with
\[
\sw_A(\gamma) := \prod_{e\in \gamma} K(e) \tilde \sw_A(\gamma) = \prod_{e\in \gamma} K(e) \frac{ E_c( \gamma)}{E_i(\gamma)}
\,,\]
so the proof of the first part of the proposition is complete. 

We turn to the bound on $\sw_A(\gamma)$. We first note that when we consider a connected set of compatible patterns $\fP$, the single edge patterns counted in $A$ cannot affect the support of $\fP$. Hence, in $E_c( \gamma )$, one can always choose the set of single edge patterns independently of the larger ones since they have no effect on the compatibility condition, yielding
\begin{equation}\label{eq:Ec_expression}
E_c( \gamma ) = \E\bigg[ \exp\Big( \sum_{e} A(e) \one_{\{e \in \varphi\}}\Big) \cdot \sum_{ \{ (P_j,x_j) \} \text{ compatible, }|P_j| > 2} \prod_j\big(e^{\fg(P_j)} - 1\big)\one_{\{\tau_{x_j} P_j \in \varphi\}}\bigg] \geq 0\,.
\end{equation}
For patterns of size larger than $2$, compatibility implies in particular that $\gamma \subset \cup_j P_j$ so we upper bound $E_c( \gamma )$ by only keeping this condition (this uses the positivity of $\fg$, hence  does not work for the $A$ weights). Separating the patterns in the interior and the ones intersecting $\gamma$ then yields
\begin{align*}
0 \leq E_c( \gamma ) \leq \E\bigg[ &\exp\Big( \sum_{e} A(e) \one_{\{e \in \varphi\}} +\sum_x \sum_{ \substack{P \text{ s.t.}\\ \tau_x P \in \operatorname{Int}(\gamma), |P| > 2}} \fg(P) \one_{\{\tau_x P \in \varphi\}} \Big) \\
& \cdot\sum_{ \substack{\{(P_j, x_j)\}\text{ s.t.}\\ \bigcup \tau_{x_j} P_j \supset \gamma,\\ |P_j| > 2,\, \tau_{x_j}P_j \cap \gamma \neq \emptyset}} \prod_j (e^{\fg(P_j)}-1)\one_{\{\tau_{x_j} P_j \in \varphi\}} \bigg]\,.
\end{align*}
To treat the last term in the exponent above, first we control the maximal weight \[\fm_\gamma := \max_{\substack{\{(P_j, x_j)\}\text{ s.t.}\\ \bigcup \tau_{x_j} P_j \supset \gamma,\\ |P_j| > 2,\, \tau_{x_j} P_j \cap \gamma \neq \emptyset}} \prod_j (e^{\fg(P_j)}-1)\one_{\{\tau_{x_j} P_j \in \varphi\}}\,.\] A pattern associated to a ball of radius $r$ intersects at most $6\cdot 2 \pi r$ tiles of $\gamma$ so, for any collection with boundary $\gamma$ the radii $r_j$ must satisfy $\sum_j r_j \geq \frac{|\gamma|}{12 \pi}$. The decay of $\| \fg_r \|$ thus immediately shows that $\fm_\gamma \leq e^{ -  \frac{C |\gamma|}{12 \pi}}$ where $C$ is from the statement of \cref{thm:gmt-refinement} and can be arbitrarily large. So,
\begin{align*}
\sum_{ \substack{\{(P_j, x_j)\}\text{ s.t.}\\ \bigcup \tau_{x_j} P_j \supset \gamma,\,\\ |P_j| > 2,\, \tau_{x_j} P_j \cap \gamma \neq \emptyset}} \prod_j \big(e^{\fg(P_j)}-1\big)\one_{\{\tau_{x_j} P_j \in \varphi\}} & \leq \sqrt{\fm_\gamma} \sum_{ \substack{\{(P_j, x_j)\}\text{ s.t.}\\ \bigcup \tau_{x_j} P_j \supset \gamma,\\ |P_j| > 2,\, \tau_{x_j}P_j \cap \gamma \neq \emptyset}} \bigg(  \prod_j \big(e^{\fg(P_j)}-1\big)\one_{\{\tau_{x_j} P_j \in \varphi\}} \bigg)^{1/2} \\
& \leq \sqrt{\fm_\gamma} \sum_{ \substack{\{(P_j, x_j)\}\text{ s.t.}\\ |P_j| > 2,\, \tau_{x_j} P_j \cap \gamma \neq \emptyset}}   \prod_j \big(e^{\fg(P_j)}-1\big)^{1/2} \one_{\{\tau_{x_j} P_j \in \varphi\}}\\
& = \sqrt{\fm_\gamma} \prod_{\substack{(P, x) \text{ s.t.}\\ |P| > 2 ,\, \tau_x P\cap\gamma\neq\emptyset }} \Big(1 + \big(e^{\fg(P)} -1\big)^{1/2} \one_{\{\tau_{x} P \in \varphi\}}\Big)\\
& \leq e^{ (- \frac{C}{24 \pi} + 1) |\gamma| }.
\end{align*}
The first line is the elementary bound $\sum a_i \leq (\max \sqrt{a_i})\sum \sqrt{a_i}$. Then in the second line we upper bound this by omitting the condition $\bigcup_{j} \tau_{x_j} P_j \supset \gamma$. This turns the choice of the collection of patterns $(P_j, x_j)$ into independent choices for each $P=(x,r)$ with $B(x,r(P))\cap \gamma\neq\emptyset$ as written in the third line. 
The transition to the last line used that the product in the third line is at most
\[ \exp\bigg(\sum_{y\in\gamma}\sum_r \sum_{z: B(z,r)\ni y} (e^{\fg_r(\varphi\restriction_{B(z,r)})} -1)^{1/2}\bigg) \leq 
\exp\bigg(\sum_{y\in\gamma}\sum_r (2r+1)^2 (2 \|\fg_r\|_\infty)^{1/2}\bigg) \leq 
e^{|\gamma|}\,,\]
for $C$ large enough in the statement of \cref{thm:gmt-refinement} which concludes this upper bound.

Overall, we arrive at
\[
0\leq E_c( \gamma ) \leq E_i(\gamma) \prod_{e \in \gamma} e^{A(e)} e^{(- C/24\pi + 1) |\gamma|}\,,
\]
where $C$ comes from the bound on $\fg$. As we can assume without loss of generality that $\|A\|_\infty \le 1$, we find that
\[ 0 \leq \tilde \sw_A(\gamma) \leq e^{-(C/24\pi -2)|\gamma|}\,,
\]
with a constant $C$ which can be made as large as we want as in the statement of \cref{thm:gmt-refinement}.
\end{proof}

\begin{remark}\label{rem:boundary-trick}
The more direct expansion mentioned just before the proof is as follows. One starts exactly like us by expanding $e^\fg = (e^\fg -1) + 1$ and expanding to get
\[
    \cZ_\fg( \vartheta_1, \vartheta_2, A) = \sum_{\varphi} \pm(k, \ell) \fa^{ - Nk - N \ell} \fb_1^{N k} \fc_2^{N \ell}  \sum_{ \{ (P_j, x_j) \} } \prod_j (\exp( \fg(P_j) ) -1) \one_{\{ \tau_{x_j} P_j \in \varphi\}}.
\]
Then one can directly group the patterns $(P_j, x_j)$ into connected components $\Gamma_i$ and index the sum by these connected components, giving
\[
 \cZ_\fg( \vartheta_1, \vartheta_2, A) = \sum_{\varphi} \pm(k, \ell) \fa^{ - Nk - N \ell} \fb_1^{N k} \fc_2^{N \ell}  \sum^*_{\Gamma_i} \prod_{i} \one_{\{\Gamma_i \in \varphi\}} W(\Gamma_i)
\]
with the weight $W(\Gamma_i)$ defined as $W( \Gamma_i) := \sum_{ \{ (P_j, x_j) \}\text{ s.t. } \bigcup P_j = \Gamma_i} \prod_j (e^{\fg(P_j)} - 1)$. Applying \cref{eq:midway_grassmann} then yields another Grassmann formula for $\cZ_\fg$.

The work \cite{GMT20} uses this expansion, which is the classical one. Indeed, it does not rely on any specific structure of the interaction terms or on any property of the underlying unperturbed model. By contrast, \cref{prop:grassmann_formulation} uses the Markov property of the base model to get an expression involving only boundaries. This is of particular interest in our case since our patterns are balls so their boundaries are of course much smaller than their interior. Note, however, that our approach would not really make sense in a more classical setting involving a many-body interaction. There, the most problematic terms typically involve a small number of sites far away from each other and in such a case there is nothing to gain by going from a pattern to its boundary.
\end{remark}

We turn to the proof of \cref{lem:lattice_ward}. It is fundamentally the same as \cite{GMT20} and boils down to expressing the gauge invariance but we still include a proof to show that it is not affected by the more complicated expression of the weights.
\begin{proof}[Proof of \cref{lem:lattice_ward}]
   Since the white and black vertices play symmetric roles we focus on the first identity. The first step is to make the derivatives with respect to the Grassmann variables more explicit. As any other function of the variables $\tilde \psi$, $\log\cZ_\fg(\vartheta_1, \vartheta_2, A, \tilde \psi)$ is a polynomial with degree at most one in each variable so the derivative $\partial_{\tilde \psi_b} \partial_{\tilde \psi_w}$ at $\tilde \psi = 0$ is simply extracting the prefactor associated to the monomial $\tilde \psi_{b_0} \tilde \psi_{w_0}$. Since the $\tilde \psi$ terms only come from the $e^{\sum_v \psi_v \tilde \psi_v}= \prod_{v}( 1 + \psi_v \tilde \psi_v)$ part of $\cZ_\fg(\vartheta_1,\vartheta_2, A, \tilde \psi)$, we get
    \[
    \fa^{ N^2 p_\fa} \fb_1^{N^2 p_\fb} \fc_2^{N^2 p_\fc}\partial_{\tilde \psi_{b_0}} \partial_{\tilde \psi_{w_0}} \cZ_\fg(\vartheta_1,\vartheta_2, A, \tilde \psi)\restriction_{\tilde \psi = 0} = \int \exp \Big( \vec{\psi}_b\cdot K \cdot \vec{\psi}_w + \sum_\gamma \sw_A(\gamma) \psi_\gamma \Big)\psi_{b_0} \psi_{w_0} \d \psi\,.
    \]
    Let us note that since every $\gamma$ is associated to a tiling and therefore contains as many white as black vertices, a similar expression shows that $\partial_{\tilde\psi_{v}}\cZ_\fg(\vartheta_1,\vartheta_2, A, \tilde\psi)\restriction_{\tilde \psi = 0}=0$ for any vertex $v$.

    Now we consider the value of  $\partial_{\tilde \psi_{b_0}} \partial_{\tilde \psi_{w_0}} \cZ_\fg$ when $A$ takes the special form $A(b,w) := \alpha(b) + \alpha(w)$ for some function $\alpha$.
   Since in any tiling, any triangle is always covered by exactly one lozenge, the first expression for $E_i$ shows that for such $A$
    \[
    E_i(\gamma, A) = E_i(\gamma, 0)e^{\sum_{v \in \Int(\gamma)}  \alpha(v)}\,,
    \]
    and similarly from \cref{eq:Ec_expression}, we have $E_c(\gamma, A) = E_c(\gamma, 0)e^{\sum_{v \in \Int(\gamma)\cup \gamma } \alpha(v)} $,
    so that, for contours of size at least~$4$,
    \[
    \sw_A(\gamma) = \sw_{0}(\gamma)e^{\sum_{v \in \gamma} \alpha(v)}\,.
    \]
Observe that single edge patterns do not occur for $A = 0$ since we assumed that $\fg$ only gave weights to larger patterns, but they do have to be included in $\sw_A$, and we see that they satisfy \[\sw_A(\{b,w\}) = K(b,w) (e^{\alpha(b) + \alpha(w)} -1)\,.\] Let us expand $\exp(\sum_\gamma \sw_A(\gamma) \psi_\gamma) = \prod_\gamma e^{\sw_A(\gamma) \psi_\gamma} = \prod_\gamma (1 + \sw_A(\gamma) \psi_\gamma)$ using that $e^{\psi_x} = 1+\psi_x$ for a Grassmann variable $\psi_x$.
Next, separate this into $\prod_{\gamma:|\gamma|>2} (1+\sw_A(\gamma)\psi_\gamma \prod_{\gamma':|\gamma'|=2} (1+\sw_A(\gamma')\psi_\gamma'$, and note that the $\gamma'$ are only single edge patterns, written as $\{b_j, w_j\}$. Expanding the products, we arrive at a sum over a collection $\{\gamma_i\}$ and a sum over  $\{\{b_j,w_j\}\}$. Since the $\prod_\gamma \psi_\gamma$ is only non-zero when all $\gamma$ are disjoint, we arrive at the conditions $\gamma_i \cap \gamma_{i'} = \emptyset$ and $\{ \{ b_j, w_j\} \} \subset \T \setminus \bigcup \gamma_i$ on the second sum. Finally, we move all real numbers out of the Grassmann integral.  
 Overall, we obtain
    \begin{multline*}
    \fa^{ N^2 p_\fa} \fb_1^{N^2 p_\fb} \fc_2^{N^2 p_\fc} \partial_{\tilde \psi_{b_0}} \partial_{\tilde \psi_{w_0}} \cZ_\fg(\vartheta_1,\vartheta_2, A, \tilde \psi)\restriction_{\tilde \psi = 0} \\
     = \sum_{ \substack{\{ \gamma_i\} \text{ s.t.}\\ |\gamma_i| >2 \\ \gamma_i \cap \gamma_{i'} = \emptyset}} \sum_{ \{ \{b_j, w_j\}\} \subset \T \setminus \bigcup \gamma_i} \prod_i \sw_0(\gamma_i) e^{\sum_{v \in \gamma_i} \alpha(v)} \prod_j (e^{\alpha(b_j) + \alpha(w_j)} -1)K_{\fa, \fb, \fc}(b_j,w_j)\\
    \cdot \int \psi_{b_0}\psi_{w_0} \prod_{v \in \bigcup \gamma_i \cup \bigcup \{b_j, w_j\} } \psi_v \cdot e^{\vec{\psi}_b\cdot K \cdot \vec{\psi}_w} \d\psi\,.  
    \end{multline*}
    Expanding the exponential, we see that the last Grassmann integral is a weighted sum over tilings of $\T \setminus ( \{ b_0, w_0\} \cup \bigcup \gamma_i \cup \bigcup \{ b_j, w_j\})$.  It can therefore be combined with the contribution of the single edge patterns to get a sum over tilings of $\T \setminus ( \{ b_0, w_0\} \cup \bigcup \gamma_i)$ with additional weights $e^{\alpha(b) + \alpha(w)}$ on each edge. Since every vertex is covered by exactly one tile, this factor can be combined with the sum over the union of the $\gamma_i$ as
     \[ 
     \partial_{\tilde \psi_{b_0}} \partial_{\tilde \psi_{w_0}} \cZ_\fg(\vartheta_1,\vartheta_2, A, \tilde \psi)\restriction_{\tilde \psi = 0} = \partial_{\tilde \psi_{b_0}} \partial_{\tilde \psi_{w_0}} \cZ_\fg(\vartheta_1,\vartheta_2, 0, \tilde \psi)\restriction_{\tilde \psi = 0} e^{\sum_{v \notin \{ b_0, w_0\}} \alpha(v)}\,.
    \]
    Since, via the same approach, we also have $\cZ_\fg(\vartheta_1,\vartheta_2, A)= \cZ_\fg(\vartheta_1,\vartheta_2, 0)e^{\sum_v \alpha(v)} $, we deduce that 
    \[
    \partial_{\tilde \psi_{b_0}} \partial_{\tilde \psi_{w_0}} \log( \cZ_\fg(\vartheta_1,\vartheta_2, A, e^{\alpha} \tilde \psi))\restriction_{\tilde \psi = 0} = \partial_{\tilde \psi_{b_0}} \partial_{\tilde \psi_{w_0}} \log( \cZ_\fg(\vartheta_1,\vartheta_2, 0, \tilde \psi))\restriction_{\tilde \psi = 0} \,.
    \]
    We conclude the proof by observing that the sum of derivatives over all edges containing $b_0$ is simply the derivative with respect to $\alpha(b_0)$.
\end{proof}

The above analysis reduced the model to an interacting dimer model where the activity of an interaction pattern $\gamma$ is at most $\exp(-C |\gamma|)$, for $C>0$ large enough, which now qualifies for an application of the GMT result (\cref{thm:gmt}) and gives the following.
\begin{theorem}\label{thm:gmt-improved-canonical}
The conclusions of \cref{thm:gmt} for the canonical base measure $\mu_N$ tilted by $\fg$ hold true for a tilt by a sequence of functions $\fg=(\fg_r)_{r=1}^\infty$ when relaxing the assumption $\sup_r\|\fg_r\|_\infty\leq\delta$ into that of \cref{eq:many-g}:
$\|\fg_r\|_\infty\leq \exp(-C r)$ for $C>0$ large enough and all $r$.
\end{theorem}

What remains in order to complete the proof of \cref{thm:gmt-refinement} is to go from a canonical setting to a microcanonical one. The main step accomplishing this is a Fourier inversion as was explained in \cref{sec:micro1}. To verify its validity, we must first give more details on the proof of \cref{thm:gmt} since said step will rely on several of their intermediate results that extend to our setting. Due to the complexity of those arguments, it is not feasible to reproduce those details in full here, so instead we will only sketch the renormalization argument in \cref{sec:sketch-renormalize} and then explain in more detail in \cref{sec:micro2} how the Fourier inversion can be justified.

\subsection{A sketch of the renormalization argument}\label{sec:sketch-renormalize}

The novel Grassmann representation in \cref{prop:grassmann_formulation} provided in the previous subsections would, in principle, allow us to appeal to the main theorem of \cite{GMT20} as a black box. However, there are three slight incompatibilities and a more significant one. The minor incompatibilities are: 
\begin{enumerate}
\item The framework of \cite{GMT20} is formulated on the square lattice while we work on the triangular lattice. This has no effect on the analysis and the square lattice case is in fact more general since setting one of its weights to zero recovers the hexagonal case (see \cite{GMT20} before Eq.~(2.9)). 
\item The main theorem requires finite range interactions; however, the finite range of the interaction is no longer needed after establishing \cite[Prop.~1]{GMT20} and this proposition is the analogue of our \cref{prop:grassmann_formulation}.
\item The work \cite{GMT20} only includes explicitly the convergence of the correlation between two edges or the heights at two points; however, the convergence to the Gaussian free field requires the asymptotics of correlations between an arbitrary finite number of points. This is again not a deep issue since the previous paper \cite{GMT17} does include the full Gaussian free field convergence (see the paragraph at the end of \cite[Sec. 2]{GMT20}). 
\end{enumerate}
Finally, the most delicate issue is that we have to work in the microcanonical setting while~\cite{GMT20} (and the other papers in the same series) are all using the  standard canonical setting. As discussed in \cref{sec:micro1}, this means that instead of writing correlations in terms of only 4 signed Grassmann integrals, we now need a full Fourier inversion. Ultimately, this does not create any significant difficulty as both the 4 signed integrals and all our Fourier coefficients are essentially constant and the argument establishing this fact in \cite{GMT20} extends with no significant modification to our setting. However, this is not immediately obvious, in particular because this point is handled (in \cite[Apps.~C and D]{GMT20}) at the end of a lengthy (37 pages) renormalization argument \cite[Sec.~6]{GMT20}.
Reproducing here all the details on the renormalization seems of little use compared to simply referring the interested reader to \cite[Sec.~6]{GMT20}, but at the same time, we do wish to justify the claim that such an involved proof extends beyond its explicit setting. To this end, we will give a somewhat informal description of the renormalization argument, with references to the associated equations in \cite[Sec.~6]{GMT20} in footnotes. This would hopefully help familiarize the reader with that argument sufficiently well to understand the quantities involved in the Fourier inversion.

\medskip

We start by the simplest case of the partition function $\cZ_\fg(0,0,0)$, i.e., $\vartheta=0$ and $A=0$. The general strategy is to write 
\begin{align*}
    \cZ_\fg &= \fa^{- N^2 p_\fa} \fb_1^{-N^2 p_\fb} \fc_2^{- N^2 p_\fc}  \int \exp \Big( \vec{\psi}_b\cdot K \cdot \vec{\psi}_w + \sum_{\gamma \in \Gamma} \sw_0(\gamma) \psi_\gamma \Big) \d \psi
\end{align*}
and then, using successive applications of the addition principle for Grassmann variables, get
\begin{align*}
    \cZ_\fg %
   & = \fa^{- N^2 p_\fa} \fb_1^{-N^2 p_\fb} \fc_2^{- N^2 p_\fc} \det(K) \int \exp \Big( \sum_{\gamma \in \Gamma} \sw_0(\gamma) \psi_\gamma \Big) \d P_{K^{-1}} (\psi)\\
&= \fa^{- N^2 p_\fa} \fb_1^{-N^2 p_\fb} \fc_2^{- N^2 p_\fc} \det(K) \int \d P_{g^{(n_{\max})}}(\psi^{(n)}) \ldots \int \d P_{g^{(0)}}(\psi^{(0)})  \exp\Big( \sum_{\gamma} \sw_0(\gamma) \big(\sum \psi^{(i)}\big)_{\gamma} \Big)\,,
\end{align*}
where the $g^{(n)}$ can be chosen arbitrarily as long as $\sum g^{(n)} = K^{-1}$. The renormalization aspect comes from choosing the $g^{(n)}$ to have finite range say of order $2^{n}$ so the successive integrals can be thought of as taking into account increasingly long-range correlations. It turns out that, since $K^{-1}$ is diagonal in Fourier space, it is more convenient to define the $g^{(n)}$ in Fourier space by
\[
\sum_{j \geq n} g^{(j)}\Big( \frac{1}{N^2}\sum_w e^{-i \langle k, w\rangle} \psi_w , \frac{1}{N^2}\sum_b e^{i \langle k', b\rangle} \psi_b \Big) = \Big( \frac{\chi( 2^{n} (k - p^+))}{\fa + \fb e^{i k_1} + \fc e^{ i k_2}}+\frac{\chi( 2^{n} (k - p^-))}{\fa + \fb e^{i k_1} + \fc e^{ i k_2}} \Big) \one_{\{k = k'\}}\,,
\]
where $k, k' \in \frac{1}{N} \Z^2 \cap [-\pi, \pi]^2$ are Fourier modes on the torus, $p^-,p^+ \in [-\pi, \pi]^2$ are two conjugate points of Fourier space to be chosen appropriately later and $\chi$ is a smooth (formally in the Gevrey class of order $2$) approximation of the indicator of a ball\footnote{See \cite[Eq.~(6.18)]{GMT20}.}. The points $p^+$ and $p^-$ will be chosen close enough to the two zeros of $\fa + \fb e^{i k_1} + \fc e^{ i k_2}$ so that each $g^{(n)}$ is smooth in Fourier and therefore has finite range in real space (with $e^{-\sqrt{r}}$ decay at distance $r$ given the Gevrey regularity).

One then uses the truncated expectation to do the above sequence of integrals, writing every intermediate term as an exponential. The first one, e.g., gives, for $X_\gamma := \sw_0(\gamma)(\psi(0)+\psi^{(\geq 1)})_\gamma $,
\begin{align*}
\int \exp \Big( \sum_\gamma X_\gamma \Big)\d P_{g^{(0)}} &= \exp \sum_m \frac{1}{m!} \cE_{g^{(0)}}^T\Big(  \underbrace{\sum_\gamma X_\gamma, \ldots, \sum_\gamma X_\gamma}_{m\text{ times}} \Big) \\ &
:= \exp N^2 \sw_0^{(1)}(\emptyset) + \sum_S \sw_0^{(1)}(S) \psi^{(\geq 1)}_S\,,
\end{align*}
where in the last equality $S$ is an arbitrary set (with as many $w$ vertices as $b$ vertices, otherwise $\cE_{g^{(0)}}^T$ would vanish, but unlike $\gamma$ it need not be connected), and $\psi^{(\geq 1)}_S = \prod_{v \in S} \psi^{(\geq 1)}_v$ and $\sw_0^{(1)}(S)$ are defined by the equality; for example,
\[
\sw_0^{(1)}(\emptyset) = \frac{1}{N^2}\sum_{m} \frac{1}{m! } \sum_{\gamma_1, \ldots, \gamma_m}  \prod_j \sw_0( \gamma_j) \cE_{g^{(0)}}^T( \psi^{(0)}_{\gamma_1}, \ldots, \psi^{(0)}_{\gamma_m} )\,, 
\]
and similarly for nonempty sets. Proceeding inductively and assuming for now that all these series  converge, we obtain weights $\sw_0^{(n)}$ on sets for all $n$. Expanding the recursive definition of these weights, we obtain an expression of $\sw_0^{(n)}$ as a sum over labeled rooted trees with the following constraints. They have height $n$ with all leaves at height exactly $n$, each leaf is indexed by a set $\gamma$ and comes with a weight $\sw_0(\gamma)$ and each interior vertex, say at height $n-m$ and with degree $d$ is labeled with a set $S$ and counts the $\psi^{(\geq m)}_S$ coefficient of $\frac{1}{d!} \cE_{g^{(0)}}^T( \psi^{(\geq m-1)}_{S_1}, \ldots, \psi^{(\geq m-1)}_{S_d})$.

To understand the asymptotic behavior of this expression when the number of generations grows, we first observe that an interior vertex of degree $1$ with the same associated set $S$ as its child comes with a weight $1$ and we call such vertices trivial. Also, by \cref{prop:grassmann_formulation}, the leaf weights $\sw_0(\gamma)$ are small and decay exponentially with $|\gamma|$. Finally, it is not hard to see that, since $g^{(n)}(\psi_b, \psi_w) \simeq 2^{-n} \one_{\{ |b-w| \leq 2^{n}\}} $ with $n \geq 0$, each nontrivial vertex $n$ generation above the leaves comes with a weight $2^{-n}$. Overall, it is believable that, as the number of generations increases, the maximal distance from a nontrivial vertex to the leaves stays tight, yielding
\[
\lim_{N \to \infty} \frac{1}{N^2}\log \cZ_\fg = \lim_{n \to \infty} \sw_0^{(n)}(\emptyset) = \text{sum over trees with unrestricted height}.
\]

The above paragraph is however too optimistic: when taking into account properly the entropy terms coming from the positions of the sets in previous scales, we see that already the very simple terms are enough to make the sum over trees diverge. Indeed, consider the contribution coming from just three edges $\sum_{e_1, e_2, e_3} \cE_{g^{(n)}}^T( \psi_{e_1}, \psi_{e_2}, \psi_{e_3})$ at scale $n < 0$. With only $6$ fields, the truncated expectation is explicit and gives, when denoting $e_i = (b_iw_i)$,
\[g^{(n)}( \psi_{b_1}, \psi_{w_2}) g^{(n)}( \psi_{b_2}, \psi_{w_3})g^{(n)}( \psi_{b_3}, \psi_{w_1}) +g^{(n)}( \psi_{b_1}, \psi_{w_3}) g^{(n)}( \psi_{b_3}, \psi_{w_2})g^{(n)}( \psi_{b_2}, \psi_{w_1})\,.\]  A naive bound is of order $(2^{n})^3$ because it contains three long-range propagators\footnote{For this example, one could use the regularity of $g$ to obtain a better bound but for general terms with two fields arising from previous integration steps this becomes difficult.} and, given $e_1$, there are of order $(2^{2n})^2$ possible edges within the range $2^{n}$ of $g$. Overall,  rephrasing in terms of the trees mentioned above, trees containing only a single branching point with three children and only marked with either single edges or the empty set already give an exponentially growing contribution! This issue is well known (see \cite[Sec.~5.2.2]{GMT17} for a more detailed description), and leads to a classification of interaction terms as relevant ones potentially growing with $n$, marginal ones which stay of order $1$ and irrelevant ones which decay exponentially with $n$. 

Luckily, the number of relevant and marginal terms is bounded (there are $3$ of them in this model) so the strategy is to keep track much more closely of them and to treat their effect exactly. Formally, this means that after each integration, the weights $\sw_0^{(n)}$ are carefully separated between the irrelevant terms (which can be bounded naively) and a finite number of relevant and marginal ones whose weights are called ``running coupling constants.''\footnote{The separation is done by the localization operator defined in \cite[Eq.~(6.37)]{GMT20}, the running coupling constants are defined in Sec.~6.2.1.} The integration operation then induces a dynamical system on the running coupling constant which can be made to converge to a fixed point\footnote{See \cite[Eq.~(6.69)--(6.71)]{GMT20}.} by using the freedom in the choice of the sequence $g^{(n)}$.\footnote{See \cite[Sec.~6.4]{GMT20}. The choice of propagators appears both in the initial setup with the free parameters in Eq.~(6.3) chosen later in Eq.~(6.101), and inductively at each scale in Eq.~(6.34).} 

With these modifications, the initial heuristic becomes correct and one obtains an expression for $\lim\log \cZ_\fg$ as a series indexed by so-called ``Gallavotti--Nicol\`{o}'' trees describing trajectories along the successive integrations.\footnote{See \cite[Sec~6.3]{GMT20} for their full definition which notably also involves labels for the localization operator.} Another important consequence is that for each finite $n$, $\sw_0^{(n)}$ has a limit as the size $N$ of the torus goes to infinity which is simply obtained by replacing the torus by the full plane in the sum over positions (which has a very small effect since the propagators have a fast decay in space) and the torus propagators $g^{(n)}_N(\psi_b, \psi_w) := \frac{1}{N^2} \sum_{k} \frac{\chi^{(n)}(k) e^{ i \langle k, b-w\rangle }}{\fa + \fb e^{i k_1} + \fc e^{i k_2}} $ by full plane ones $g^{(n)}_\infty = \iint \frac{\chi^{(n)}(k) e^{ i \langle k, b-w\rangle }}{\fa + \fb e^{i k_1} + \fc e^{i k_2}} \d k_1 \d k_2 $ and let us note for later that the localization is defined only using the full plane limit of the model.

\medskip

We can now go back to the full generating function $\cZ_\fg(\vartheta_1, \vartheta_2, A, \tilde \psi)$. Derivatives with respect to $A$ easily fit in the same picture because the control needed to exchange the limit over $n$ and the derivative is identical to the one needed to ensure the summability in the first place. The only real difference comes from the fact that the set of base weights $\sw_0$ is no longer exactly translation invariant and that some leaf must be labeled with a derivative in addition to a contour $\gamma$. Let us highlight here a small difference between our setting and \cite{GMT20}, with their weight $\partial_{A(e)} \sw_A(\gamma) =0$  when $e \not\subset \gamma$ while for us this requires $e \not\subset \gamma\cup \Int(\gamma)$. This increases slightly the enumeration over contours compatible with a given derivative but it is easily absorbed by the exponential decay with $|\gamma|$ since the total number of derivatives we consider is always fixed. The expansion in terms of the extra variables $\tilde \psi$ also fits into the above framework, it only adds single site patterns and does not affect the summability properties because the expansion in the variables $\tilde \psi$ is only formal by definition and these patterns can never be repeated.

Finally, a convenient way to view $K(\vartheta_1, \vartheta_2)$ is as follows. We start from the full plane Kasteleyn operator $K_{\infty}$, mapping any function defined on white triangles of the infinite triangular lattice to a function on the black triangles. Using the obvious identification between functions on the torus and periodic functions on the full plane, we see that $K$ is the matrix of the action of $K_\infty$ over $N$-periodic functions. Similarly, it is not hard to check that $K(\vartheta_1, \vartheta_2)$ is the matrix of the action of $K_\infty$ over pseudo-periodic functions such that for all $v$ and $j$, $f(v + N e_j) = e^{i N \vartheta_j} f(v)$. With this point of view, it becomes clear that $K(\vartheta_1, \vartheta_2)$ is still diagonal in Fourier with the same eigenvalues $\mu(k) = \fa + \fb e^{ik_1} + \fc e^{i k_2}$, and that the only difference is a shift by $\vartheta_1, \vartheta_2$ of the modes. In the full plane limit of the model, the set of modes becomes continuous and this effect disappears so the full plane model becomes independent of $\vartheta$ and recall that the localization operator and the choice of the propagators $g^{(n)}$ only depend on that model so they are done consistently for all $\vartheta$. The phases $e^{i \vartheta_1 n_\fb(\gamma)+ i \vartheta_2 n_\fc(\gamma)}$ attached to the contour weights in \cref{prop:grassmann_formulation} fit into the same picture: the change of variables $\psi_v\mapsto e^{\pm i\phi_\vartheta(v)}\psi_v$ (with the sign according to the color of $v$), where $\phi_\vartheta$ increases by $\vartheta_1$, $\vartheta_2$ or $0$ along every edge of type $\fb$, $\fc$ or $\fa$ oriented from its black to its white endpoint (this is consistent around every hexagon, so $\phi_\vartheta$ is well defined on the universal cover, and pseudo-periodic), removes these phases and turns $K(\vartheta_1,\vartheta_2)$ into the Kasteleyn matrix with twisted boundary conditions of \cite[\S3.1]{GMT20}. In this picture the interaction is the same for all $\vartheta$ and only the propagator depends on the twist, exactly as in \cite{GMT20}.

A last detail is that, in the initial setup, there is no reason why $K$ has to be invertible and in fact $K(\vartheta_1, \vartheta_2)$ will always be non-invertible for some $\vartheta$. In order to still have a good representation with a propagator and to smooth the behavior of $\det(K(\vartheta_1, \vartheta_2))$ close to its zero, the Fourier modes closest to $p^+$ and $p^-$ (call them $\hat\Psi^\circ_\pm$ and $\hat \Psi^\bullet_\pm$ for the one associated to white and black vertices respectively) are split off from the initial Grassmann integral and the integral over these two will be done last\footnote{See \cite[Eq.~(6.8)]{GMT20} for the resulting first decomposition.}. Ultimately, this gives the following output for the full renormalization scheme\footnote{See \cite[Eq.~(6.112)]{GMT20}.}:
\begin{proposition}\label{prop:renormalization_output}
The following formula holds:
   \begin{multline*}
  \cZ( \vartheta_1, \vartheta_2, A) = \fa^{- N^2 p_\fa} \fb_1^{-N^2 p_\fb} \fc_2^{- N^2 p_\fc} \Big[ \prod_{k \in \mathcal{P}'(\vartheta)} \mu^{(0)}(k) \Big] e^{ N^2( E_{\vartheta}^{(h)} -E_{\vartheta}^{(0)}) + S_\vartheta^{(h)}(A)} \\
  \cdot \int  \exp\Big(  \frac{Z^{(h)}}{N^2} \sum_{\omega \in \{ + , -\}} \mu^{(h)}_{\vartheta, \omega} \hat \Psi^\bullet_\omega \hat \Psi^\circ_\omega + V_{\vartheta}^{(h)}( \sqrt{Z^{(h)}} \Psi, A) \Big)\d \Psi.
   \end{multline*}
   In the above,
     \begin{enumerate}[1.]
        \item $\mathcal{P}'(\vartheta)$ is the set of Fourier modes of $K(\vartheta_1, \vartheta_2)$, namely $\cP(\vartheta):=\frac{2\pi}{N}\{0,\ldots,N-1\}^2+\vartheta$, with the mode closest to $p^+$ and the mode closest to $p^-$ removed. Here $p^\pm=p^\pm(\fa,\fb,\fc,\fg)$ are the \emph{dressed Fermi points} of \cite[\S6.1 and \S6.4.5]{GMT20}: they satisfy $p^-=-p^+$ and are within $O(\sum_r\|\fg_r\|_\infty)$ of the two zeros of $\fa+\fb e^{ik_1}+\fc e^{ik_2}$, but in general differ from them (this is the shift of the slope by the interaction, see \cref{rem:why-tuning}). The function $\mu^{(0)}$ is the reference dispersion of \cite[Eq.~(6.3)]{GMT20}: a smooth bounded function which coincides with $\fa+\fb e^{ik_1}+\fc e^{ik_2}$ outside fixed neighborhoods of $p^\pm$, is within $O(\sum_r\|\fg_r\|_\infty)$ of it in $C^3$, and has simple zeros exactly at $p^+$ and $p^-$, its derivatives there being the dressed Fermi velocities.
        \item $h$ is the index of the last step of the induction, $h = \log_2(N) + O(1)$.
        \item The terms $E_{\vartheta}^{(h)}-E_{\vartheta}^{(0)}$, $S_\vartheta^{(h)}(A)$, $\mu^{(h)}_{\vartheta, \pm}$ and $Z^{(h)}$ are complex numbers and $V_\vartheta^{(h)}(\cdot, A)$ is a polynomial in 4 Grassmann variables. They are analytic functions of the weights of patterns with a representation as absolutely convergent series indexed by Gallavotti--Nicol\`{o} trees, where only $S_\vartheta$ and $V_\vartheta$ depend on the weight perturbation by $A$.\footnote{See \cite[Prop.~3]{GMT20} and the discussion immediately after it, as well as Eqs.~(6.69),(6.70) and Sec.~6.4.5. More precisely, Prop.~3 shows analyticity as a function of the weights and extra variables called the running coupling constant assuming the latter stay small. Eqs.~(6.69),(6.70) formulate precisely the conditions the constants will actually satisfy and Sec.~6.4.5 shows that they depend analytically on the weights and that the dependence can be inverted.}
         \item The set of labeled trees in the previous item is independent of $\vartheta$ and $A$.\footnote{This is immediate from their description in \cite[Sec.~6.3]{GMT20}.}
    \end{enumerate}
\end{proposition}

\begin{remark}
    In an earlier version of this paper, instead of changing the initial Grassmannian formulation of $\cZ$, we used the original one from \cite{GMT17,GMT20} but gave an improvement on the first step in the inductive scheme (the definition of the $\sw_0^{(1)}$ weights). This was done up to a length $O(1/\epsilon)$ and gave a bound $|\sw_0^{(1)}(S) |\leq e^{-\beta}C^{|S|} \exp(- \sqrt{\epsilon\operatorname{Animal}(S)}) $ with $\operatorname{Animal}(S)$ the size of the minimal animal connecting all points of $S$. We then remarked that, thereafter, the proof proceeds according to the original framework of~\cite{GMT20}, where they argued why the rest of the renormalization process carries through in the later scales (see the inductive bound on the weights in~\cite[Eq.(6.61)]{GMT20}).
    (For us, this would have meant that the growth $C^{|S|}$ could be compensated for by the decay of $g$ in the later scales if $\epsilon$ was chosen small enough.) This is unfortunately wrong (in our setting as well as in the setting of~\cite{GMT20}). Indeed, the $ \exp(- \sqrt{\epsilon\operatorname{Animal}(S)})$ decay is not sufficient to beat the entropy of sets covering a large domain with a density $\epsilon$, even if $C$ was replaced by a small constant. In~\cite{GMT20}, the issue traces back to the fact that their bounds are only stated using $\sqrt{\operatorname{Animal}(S)}$ (see \cite[Eq.~(6.31) and~(6.61)]{GMT20}). In their context, this can be fixed without much difficulty because their bound actually comes from the stronger $\min_{G \text{ connected}} \sum_{(v, v') \in G} \sqrt{|v - v'|} $ which is actually enough to beat the entropy. In our case, however, this required significant changes (and new ideas), including changing the Grassmannian formulation as above.
\end{remark}

\subsection{Back to the microcanonical model}\label{sec:micro2}

We are interested in the asymptotics of the correlations of the microcanonical measure $\mu_N$ of \cref{thm:gmt-refinement}, which, up to the error $c^{-1}e^{-N^c}$ allowed there, is the probability measure on the target sector proportional to $\exp[\sum_{x,r}\fg_r(\varphi\restriction_{B(x,r)})]$. All of its joint moments are encoded in the ratios $\cZ_\fg(A)/\cZ_\fg(0)$, with $\cZ_\fg(A)$ from \cref{eq:Z(A)-def}: for a finite set $F$ of edges and $A$ supported on $F$,
\begin{equation}\label{eq:cumulant-def}
\log \frac{\cZ_\fg(A)}{\cZ_\fg(0)}=\log \E\Big[\exp\Big(\sum_{e\in F}A(e)\one_{\{e\in\varphi\}}\Big)\Big]
\end{equation}
is the joint cumulant generating function of the occupation variables $(\one_{\{e\in\varphi\}})_{e\in F}$ under the exactly tilted microcanonical measure, which differs from the one under $\mu_N$ by $O(e^{-N^c})$. Our goal is to prove that
\begin{equation}\label{eq:micro-goal}
\frac{\cZ_\fg(A)}{\cZ_\fg(0)}=e^{S_N(A)}\big(1+O(N^{-\kappa})\big)
\end{equation}
uniformly over complex $A$ supported on $F$ with $\|A\|_\infty\leq\rho_F$, for some $\kappa>0$ and a function $S_N(A)$ supplied by the renormalization group analysis, which does not depend on the twist and whose Taylor coefficients at $A=0$ converge to the infinite-volume cumulants computed in \cite{GMT20}.

By \cref{eq:inversion-Z}, proving \cref{eq:micro-goal} amounts to comparing the two averages $\frac1{N^2}\sum_{\vartheta\in\Theta_N}\cZ_\fg(\vartheta_1,\vartheta_2,A)$ and $\frac1{N^2}\sum_{\vartheta\in\Theta_N}\cZ_\fg(\vartheta_1,\vartheta_2,0)$. The output of \cref{prop:renormalization_output} controls each term of these sums, uniformly in $\vartheta$, up to an additive error which is small relative to the typical size of the terms. This is not enough by itself: the right-hand side of \cref{eq:inversion-Z} is a sum of $N^2$ oscillating terms, and a small error on the terms is a small relative error on the sum only if the sum does not cancel, i.e., only if $\cZ_\fg(0)$ is comparable to the typical size of $|\cZ_\fg(\vartheta_1,\vartheta_2,0)|$. In the variables of \cref{eq:detK}, this is the statement that the target sector $(k,\ell)=(0,0)$ carries a positive fraction of the total mass. In the non-interacting model this holds for the weights associated to $p^{(N)}$, by the discrete Gaussian behavior of the instanton component recalled in \cref{sec:micro1}. In the interacting model it holds only after the weights $(\fa,\fb,\fc)$ have been adjusted so as to absorb the shift of the slope induced by the interaction; see \cref{rem:why-tuning}. The section is organized accordingly. We first record, in \cref{prop:RG-output}, the precise form of the output of the renormalization group analysis that we use, for all twists $\vartheta\in\Theta_N$ and with the dependence on the weights made explicit. We then choose the weights (\cref{lem:tuning}), compare the dressed non-interacting determinant produced by the renormalization with the honest non-interacting determinant at suitable weights (\cref{lem:riemann-sum,cor:riemann-sum}), obtain the lower bound on $\cZ_\fg(0)$ (\cref{prop:sector-lower-bound}), and conclude.

\subsubsection*{The output of the renormalization for all twists}
Recall from \cref{sec:sketch-renormalize} that the Fourier modes of $K(\vartheta)$ are
\[ \cP(\vartheta):=\tfrac{2\pi}{N}\{0,\ldots,N-1\}^2+\vartheta\,,\]
viewed as a subset of the torus $(\R/2\pi\Z)^2$, and that the renormalization scheme of \cite{GMT20} singles out, for each $\vartheta$, the mode $k^+_\vartheta\in\cP(\vartheta)$ closest to the dressed Fermi point $p^+$ and the mode $k^-_\vartheta$ closest to $p^-$ (making an arbitrary choice if there are several), and writes $\cP'(\vartheta):=\cP(\vartheta)\setminus\{k^+_\vartheta,k^-_\vartheta\}$. We write
\begin{equation}\label{eq:tildeZ0-def}
\tilde Z^0_\vartheta:=\prod_{k\in\cP'(\vartheta)}\mu^{(0)}(k)\qquad\mbox{and}\qquad D_0(\vartheta):=\prod_{k\in\cP(\vartheta)}\mu^{(0)}(k)=\mu^{(0)}(k^+_\vartheta)\,\mu^{(0)}(k^-_\vartheta)\,\tilde Z^0_\vartheta\,,
\end{equation}
and denote by $d(k):=\min(|k-p^+|,|k-p^-|)$ the distance of a mode to the dressed Fermi points. Since $\cP(\vartheta)$ is a square grid of spacing $2\pi/N$, one has $|k^\pm_\vartheta-p^\pm|\leq \sqrt2\pi/N$, while every $k\in\cP'(\vartheta)$ satisfies $d(k)\geq\pi/N$ (two distinct modes are at distance at least $2\pi/N$ from each other). We also write $\delta:=\sum_r\|\fg_r\|_\infty$ and recall that, in the setting of \cref{thm:gmt-refinement}, $\delta\leq 2e^{-C}$ can be made as small as desired by taking $C$ large. Finally, we let $\cG(\vartheta,A)$ denote the Grassmann integral in \cref{prop:grassmann_formulation}, so that
\begin{equation}\label{eq:Z-vs-G}
\cZ_\fg(\vartheta_1,\vartheta_2,A)=\fa^{-n_{\fa,N}}\fb_1^{-n_{\fb,N}}\fc_2^{-n_{\fc,N}}\,\cG(\vartheta,A)=w^{-n_N}e^{-i\langle\vartheta,n_N\rangle}\,\cG(\vartheta,A)\,.
\end{equation}

We need the following version of the estimates of \cite[\S6.5 and App.~C]{GMT20}, for all twists $\vartheta\in\Theta_N$ and with the dependence on the weights made explicit; its proof follows by essentially the same arguments as in \cite{GMT20}, and we explain the differences after the statement.

\begin{proposition}[{following from \cite[\S6]{GMT20}}]\label{prop:RG-output}
Fix a compact set $\cK$ of admissible weights $w=(\fa,\fb,\fc)$ with $\fa=1$. There exist $\delta_0>0$, $\kappa\in(0,\frac12]$, a constant $C$ depending only on $\cK$, and, for every finite set $F$ of edges, constants $\rho_F\in(0,1]$ and $C_F<\infty$, such that the following holds whenever $\delta\leq\delta_0$, for every $w\in\cK$, every $N$, every $\vartheta\in\Theta_N$ and every $A\in\C^F$ with $\|A\|_\infty\leq\rho_F$.
\begin{enumerate}[(i)]
\item\label{it:RG-weights} The dressed Fermi points $p^\pm=p^\pm(w,\fg)$ and the reference dispersion $\mu^{(0)}=\mu^{(0)}_{w,\fg}$ of \cref{prop:renormalization_output} satisfy $p^-=-p^+$, $\mu^{(0)}(-k)=\overline{\mu^{(0)}(k)}$, and $\|\mu^{(0)}-\mu_w\|_{C^3(\T^2)}\leq C\delta$, where $\mu_w(k):=\fa+\fb e^{ik_1}+\fc e^{ik_2}$; moreover $\mu^{(0)}$ vanishes exactly at $p^+$ and $p^-$, where its differential (as a map $\R^2\to\C\simeq\R^2$) is invertible. Finally, $w\mapsto p^+(w,\fg)$ is continuously differentiable on $\cK$ and $\|p^+(\cdot,\fg)-p^+_0(\cdot)\|_{C^1(\cK)}\leq C\delta$, where $p^+_0(w)$ denotes the zero of $\mu_w$ with first coordinate in $(0,\pi)$.
\item\label{it:RG-structure} There exist $\Delta=\Delta(w,\fg)\in\C$ and functions $S_N(\cdot)=S_N(\cdot\,;w,\fg)$ and $R_\vartheta(\cdot)$, analytic on $\{\|A\|_\infty\leq\rho_F\}$, such that
\begin{equation}\label{eq:RG-output}
\cG(\vartheta,A)=e^{N^2\Delta+S_N(A)}\big[D_0(\vartheta)+R_\vartheta(A)\big]\,.
\end{equation}
Here $\Delta$ and $S_N$ do not depend on $\vartheta$, $S_N(0)=0$, and $\sup_{\|A\|_\infty\leq\rho_F}|S_N(A)|\leq C_F$.
\item\label{it:RG-bounds} Writing $T_\vartheta:=N^{-2}|\tilde Z^0_\vartheta|$, one has
\begin{equation}\label{eq:RG-remainder}
|R_\vartheta(0)|\leq C(\delta+N^{-1/2})T_\vartheta\,,\qquad
\sup_{\|A\|_\infty\leq\rho_F}|R_\vartheta(A)-R_\vartheta(0)|\leq C_FN^{-\kappa}T_\vartheta\,.
\end{equation}
\item\label{it:RG-limit} If $w_N\to w_\infty$ in $\cK$, then for every multi-index $m\in\N^F$ the Taylor coefficients $\partial^m_AS_N(0;w_N,\fg)$ converge to a limit $\partial^m_AS_\infty(0;w_\infty,\fg)$. The numbers $ \partial_{A(e_1)}\cdots\partial_{A(e_j)}S_\infty(0;w_\infty,\fg) $,
for distinct edges $e_1,\ldots,e_j$, are the joint cumulants of $\one_{\{e_1\in\varphi\}},\ldots,\one_{\{e_j\in\varphi\}}$ under the infinite-volume limit of the canonical measures proportional to $w_\infty^{n(\varphi)}\exp[\sum_{x,r}\fg_r(\varphi\restriction_{B(x,r)})]$, which exists by \cite[Thm.~1]{GMT20}, and they satisfy \cref{eq:cumulants}.
\end{enumerate}
\end{proposition}
\begin{proof}
The estimates follow from \cite[\S6]{GMT20}, applied to the Grassmann representation of \cref{prop:grassmann_formulation} as described in \cref{sec:sketch-renormalize}: \cref{it:RG-structure,it:RG-bounds} follow from \cite[Eqs.~(6.111), (6.113), (6.127), (6.136), (6.144) and (6.145)]{GMT20} (see also \cite[App.~C]{GMT20}), \cref{it:RG-limit} is \cite[Eqs.~(6.129)--(6.134) and (6.149)]{GMT20} together with \cite[Eqs.~(7.4) and (7.6)]{GMT17}, and \cref{it:RG-weights} is \cite[Eq.~(6.3), Prop.~2 and \S6.4.5]{GMT20}.

To obtain the remainder formulation, extract $e^{N^2\Delta+S_N(A)}$ from the cited last-scale formulas. At $A=0$, the remaining expression is
\[ (1+O(\delta))\big[\mu_{\vartheta,+}\mu_{\vartheta,-}+O(\delta N^{-2})\big]\tilde Z^0_\vartheta\,,\]
and its difference from its value at $A=0$ is bounded by $C_FN^{-\kappa}T_\vartheta$ on the source polydisc. Indeed, the finite-size source factor is $1+O_F(N^{-\kappa})$ and the additional source term is $O_F(N^{-\kappa-2})\tilde Z^0_\vartheta$, both with vanishing corrections at $A=0$. Since $|\mu_{\vartheta,\pm}-\mu^{(0)}(k^\pm_\vartheta)|\leq CN^{-3/2}$ (justified below) and $|\mu^{(0)}(k^\pm_\vartheta)|\leq C/N$, the displayed expression differs from $D_0(\vartheta)$ by at most $C(\delta+N^{-1/2})T_\vartheta$. This gives \cref{eq:RG-remainder}; analyticity and the normalization $S_N(0)=0$ follow from the source expansion.

These estimates are proved in \cite{GMT20} for a finite-range interaction, for the four twists $\vartheta\in\{0,\pi/N\}^2$, and for fixed weights, whereas we use them for the contour gas of \cref{prop:grassmann_formulation}, for all $N^2$ twists in $\Theta_N$, and uniformly over $w\in\cK$. We explain why the arguments of \cite{GMT20} apply in each of these three respects.

\emph{The interaction.} The finite range of the interaction is used in \cite{GMT20} only to establish the Grassmann representation \cite[Prop.~1]{GMT20}; from that point on, the multiscale integration of \cite[\S6]{GMT20} only uses the smallness and exponential decay of the contour weights, which \cref{prop:grassmann_formulation} provides in the form $0\leq\sw_A(\gamma)\leq e^{-C^\star|\gamma|}$, together with the lattice Ward identities, which are \cref{lem:lattice_ward}; see \cref{sec:sketch-renormalize}, and in particular the discussion of \cref{prop:renormalization_output}, of which \cref{eq:RG-output} is the final form after the integration of the two special modes, \cite[Eqs.~(6.111)--(6.118) and (6.137)--(6.144)]{GMT20}.

\emph{The twists.} In the gauge described at the end of \cref{sec:sketch-renormalize}, the interaction does not depend on $\vartheta$, and the twist enters the analysis only through the set of modes $\cP(\vartheta)$ of the propagators, in two ways. First, the finite-volume corrections to the propagators at every scale are controlled via the Poisson summation formula \cite[Eq.~(C.7)]{GMT20}, in which the phases $(-1)^{\theta\cdot m}$ become $e^{iN\langle\vartheta,m\rangle}$; since all the bounds there, as well as those of \cite[App.~B]{GMT20} on the finite-size corrections to the effective potentials, are bounds on absolute values, they hold uniformly in $\vartheta$. Second, the two modes $k^\pm_\vartheta$ closest to $p^\pm$ are split off at the beginning, \cite[Eq.~(6.8)]{GMT20}, and integrated last, \cite[\S6.5]{GMT20}; this is done for each $\vartheta$ separately and only uses $|k^\pm_\vartheta-p^\pm|\leq\sqrt2\pi/N$. The bound $|\mu_{\vartheta,\pm}-\mu^{(0)}(k^\pm_\vartheta)|\leq CN^{-3/2}$ used above follows from the definition of $\mu_{\vartheta,\pm}$ in \cite[Eq.~(6.111)]{GMT20} as the dressed dispersion relation at the last scale evaluated at $k^\pm_\vartheta-p^\pm$: its linear part has the dressed velocities of \cite[Eqs.~(6.4)--(6.5)]{GMT20}, which are the derivatives of $\mu^{(0)}$ at $p^\pm$, its remainder is quadratic in $|k^\pm_\vartheta-p^\pm|\leq\sqrt2\pi/N$, and the wave function renormalization dividing that remainder is $N^{O(\delta)}$ by \cite[Eq.~(6.103)]{GMT20}. The only places where \cite{GMT20} uses the specific values of the four twists are \cite[Eqs.~(6.119)--(6.121) and App.~D]{GMT20}: the choice of a twist $\theta_0$ for which $|\mu_{\theta_0,\pm}|\asymp1/N$, the comparison of $\tilde Z^0_\theta$ across twists, and the signs of $Z^0_\theta$. The first two are replaced below by the choice of $\vartheta^*$ in the proof of \cref{prop:sector-lower-bound} and by \cref{cor:riemann-sum}, whose proof follows \cite[App.~D.1]{GMT20}, and the third is not needed, since the signs cancel in \cref{lem:fourier-inversion}.

\emph{The weights.} In \cite{GMT20} the weights are fixed and only the interaction strength varies. The running coupling constants are the unique fixed point, in a ball of radius $\epsilon$, of the contraction $T$ of \cite[\S6.4.2--6.4.4]{GMT20}, which depends on the weights only through the single-scale propagators and the contour weights, analytically and with bounds uniform over the compact set $\cK$; hence the fixed point, and in particular the counterterms $\nu_{0,\omega},a_{0,\omega},b_{0,\omega}$ of \cite[Eq.~(6.101)]{GMT20}, are continuously differentiable (indeed analytic) functions of $w$ as well, by the contraction mapping theorem with parameters, and they are $O(\delta)$ together with their derivatives in $w$, as every term of their expansion carries at least one interaction vertex. The dressed Fermi points are then defined by $\mu_w(p^\omega)=-\nu_{0,\omega}$, \cite[\S6.4.5]{GMT20}, and the implicit function theorem gives the $C^1$ dependence on $w$ together with $\|p^+(\cdot,\fg)-p^+_0\|_{C^1(\cK)}\leq C\delta$ in \cref{it:RG-weights}; the stated properties of $\mu^{(0)}$ follow from its definition \cite[Eq.~(6.3)]{GMT20}, in which the correction to $\mu_w$ is supported near $p^\pm$, is $O(\delta)$ in $C^3$, and vanishes at $p^\pm$ by construction, and from the symmetry \cite[Eq.~(6.7)]{GMT20}. Finally, $S_N$ is the bulk source term \cite[Eq.~(6.129)]{GMT20}, a sum over scales of periodized infinite-volume kernels $W^{(h),\infty}$, each of which is an absolutely convergent sum over trees of terms depending continuously on $w$ (through the full-plane propagators and the contour weights), with the bounds \cite[Eqs.~(6.130) and (6.134)]{GMT20} uniform over $\cK$; this gives the convergence of the Taylor coefficients of $S_N(\cdot\,;w_N,\fg)$ along $w_N\to w_\infty$ in \cref{it:RG-limit}, and \cite[Eq.~(6.149)]{GMT20} identifies the limits with the infinite-volume cumulants, whose decay \cref{eq:cumulants} is \cite[Eqs.~(7.4) and (7.6)]{GMT17}.
\end{proof}
The additive bounds in \cref{eq:RG-remainder} remain valid when $D_0(\vartheta)=0$; in particular, we shall never divide by $\cG(\vartheta,0)$.

\subsubsection*{Choice of the weights}
For $p'$ in the interior of the simplex, let $w_0(p')=(1,\fb,\fc)$ be the weights associated to $p'$ by the triangle construction of \cref{sec:Kasteleyn}, normalized so that $\fa=1$. Recall that the zero of $\mu_{w_0(p')}$ with first coordinate in $(0,\pi)$ is $p^+_0(w_0(p'))=(\pi(p'_\fa+p'_\fb),-\pi(p'_\fa+p'_\fc))$. Consequently, $w\mapsto p^+_0(w)$ is a diffeomorphism from the set of admissible weights (normalized by $\fa=1$) onto an open subset of $(0,\pi)\times(-\pi,0)$, with inverse $p^+\mapsto w_0(p')$ where $p'$ is the unique point of the simplex with $\pi(p'_\fa+p'_\fb)=p^+_1$ and $\pi(p'_\fa+p'_\fc)=-p^+_2$.

\begin{lemma}[tuned weights]\label{lem:tuning}
Let $\cK$, $\delta_0$ and $C$ be as in \cref{prop:RG-output}, and suppose that $\cK$ contains a neighborhood of $w_0(p)$. There exist $\delta_1\in(0,\delta_0]$ and $C_1$ such that, if $\delta\leq\delta_1$, then for every $N$ large enough there is a unique $w_N\in\cK$ with $|w_N-w_0(p^{(N)})|\leq C_1\delta$ such that
\begin{equation}\label{eq:tuning}
p^+(w_N,\fg)=p^+_0\big(\bar w_N\big)\,,\qquad\mbox{where}\qquad \bar w_N:=w_0(p^{(N)})\,.
\end{equation}
Moreover $w_N\to w_\infty$ as $N\to\infty$, where $w_\infty$ is the unique solution of $p^+(w_\infty,\fg)=p^+_0(w_0(p))$ with $|w_\infty-w_0(p)|\leq C_1\delta$.
\end{lemma}
(\cref{eq:tuning} says that the dressed Fermi points of the interacting model with bare weights $w_N$ coincide with the zeros of $\mu_{\bar w_N}$, the dispersion relation of the non-interacting model with slope~$p^{(N)}$.)
\begin{proof}
Let $B$ be a closed ball around $w_0(p)$ contained in $\cK$ on which $p^+_0$ is a diffeomorphism onto its image with $\|Dp^+_0\|$ and $\|(Dp^+_0)^{-1}\|$ bounded by some $L$. Write $G(w):=p^+(w,\fg)$; by \cref{it:RG-weights} of \cref{prop:RG-output}, $\|G-p^+_0\|_{C^1(B)}\leq C\delta$. For $\delta\leq\delta_1$ small enough (depending on $L$ and $C$), $G$ is therefore injective on $B$, with $\|(DG)^{-1}\|\leq 2L$, and $G(B)$ contains the $(\mathrm{radius}(B)/(4L))$-neighborhood of $p^+_0(w_0(p))$; this is the inverse function theorem in the quantitative form provided by the contraction mapping theorem, applied to $v\mapsto v-(Dp^+_0(w'))^{-1}[G(w'+v)-y]$ for $w'$ in the half ball and $y$ close to $p^+_0(w')$. For $N$ large, $\bar w_N=w_0(p^{(N)})$ lies in the half ball, since $p^{(N)}\to p$, and $|G(\bar w_N)-p^+_0(\bar w_N)|\leq C\delta$, so that $y:=p^+_0(\bar w_N)$ is in $G(B)$; the unique solution $w_N\in B$ of $G(w_N)=y$ satisfies $|w_N-\bar w_N|\leq 2L\,|G(\bar w_N)-y|\leq 2LC\delta=:C_1\delta$. Uniqueness in the ball of radius $C_1\delta$ around $\bar w_N$ follows from the injectivity of $G$ on $B$. Finally, $w_N=G^{-1}(p^+_0(\bar w_N))\to G^{-1}(p^+_0(w_0(p)))=:w_\infty$ by continuity of $G^{-1}$ and of $p^+_0$.
\end{proof}

From now on, the weights in \cref{eq:Z-twisted-def} are $w=w_N$, and $p^\pm$, $\mu^{(0)}$, $\Delta$, $\tilde Z^0_\vartheta$, etc., are the corresponding objects. We stress that this choice of weights does not affect the microcanonical generating function $\cZ_\fg(A)$ in \cref{eq:Z(A)-def}, which is the object of interest.

\begin{remark}\label{rem:why-tuning}
The tuning of \cref{lem:tuning} is imperative. Consider $\sum_{x,r}\fg_r(\varphi\restriction_{B(x,r)})=t\,n_\fb(\varphi)$ for a small $t>0$ (take $\fg_1$ to be $t/2$ times the indicator that the lozenge covering the origin triangle is of type $\fb$), together with the weights $w_0(p)$ of the non-interacting model, say for $p=(\frac13,\frac13,\frac13)$, $w_0(p)=(1,1,1)$ and $n_N=(N^2/3,N^2/3)$ for $N\in 3\Z$. Then $\cZ_\fg(\vartheta_1,\vartheta_2,0)$ is, up to the normalization in \cref{eq:Z-vs-G}, the non-interacting twisted determinant with weights $(1,e^t,1)$, so that $N^{-2}\log|\cZ_\fg(\vartheta_1,\vartheta_2,0)|$ converges, for most twists, to the canonical free energy $\max_{p'}[\sigma(p')+tp'_\fb]$ of the weights $(1,e^t,1)$, where $\sigma$ is the entropy function of lozenge tilings; whereas $N^{-2}\log\cZ_\fg(0)=N^{-2}\log Z(n_N)+t/3\to\sigma(p)+t/3$. By strict concavity of $\sigma$, the two limits differ by a positive constant $c(t)\asymp t^2$, so the target sector carries a fraction $e^{-c(t)N^2+o(N^2)}$ of the mass: the lower bound of \cref{prop:sector-lower-bound} fails, and so does \cref{eq:micro-goal} as derived from it. Of course, in this example the interaction is constant on the target sector and the conclusion of \cref{thm:gmt-refinement} is trivially true; the point is that the method requires the weights to be Legendre dual, for the interacting model, to the target slope. The tuned weights of \cref{lem:tuning} are, in this example, $(1,e^{-t},1)$.
\end{remark}

\subsubsection*{Riemann sums with two conical singularities}
The next elementary estimate replaces \cite[App.~D.1]{GMT20}.
\begin{lemma}\label{lem:riemann-sum}
Let $q^+,q^-$ be two points of the torus $(\R/2\pi\Z)^2$, let $d(k):=\min(|k-q^+|,|k-q^-|)$, and let $u$ be a $C^3$ function on $(\R/2\pi\Z)^2\setminus\{q^+,q^-\}$ such that, for some $M>0$,
\begin{equation}\label{eq:u-derivative-bounds}
|\nabla^j u(k)|\leq M\, d(k)^{-j}\qquad\mbox{for all $k\notin\{q^+,q^-\}$ and $j\in\{1,2,3\}$}\,.
\end{equation}
For $N\geq 1$ and $\vartheta\in\R^2$, let $\cP(\vartheta)=\frac{2\pi}N\{0,\ldots,N-1\}^2+\vartheta$, let $Q_k$ be the closed square of side $2\pi/N$ centered at $k\in\cP(\vartheta)$ (so that the squares $Q_k$ tile the torus), let $\cF_\vartheta:=\{k\in\cP(\vartheta): d(k)\geq 10/N\}$ and $\Omega_\vartheta:=\bigcup_{k\in\cF_\vartheta}Q_k$. There exists an absolute constant $C_2$ such that, for all $N\geq 1$ and all $\vartheta$,
\[
\Big|\sum_{k\in\cF_\vartheta}u(k)-\frac{N^2}{4\pi^2}\int_{\Omega_\vartheta}u(q)\,\d q\Big|\leq C_2M\,.
\]
\end{lemma}
\begin{proof}
Fix $k\in\cF_\vartheta$. For $q\in Q_k$ we have $|q-k|\leq\sqrt2\pi/N<5/N\leq d(k)/2$, so $d(q)\geq d(k)/2$ on $Q_k$. By Taylor's formula and \cref{eq:u-derivative-bounds},
\[ u(q)=u(k)+\nabla u(k)\cdot(q-k)+\tfrac12 (q-k)^{\mathsf T}\nabla^2u(k)(q-k)+R_k(q)\,,\qquad |R_k(q)|\leq \tfrac{8}{6} M d(k)^{-3}|q-k|^3\,,\]
and averaging over the square $Q_k$, whose first moments vanish and whose second moments are $\frac{N^2}{4\pi^2}\int_{Q_k}(q-k)_i(q-k)_j\d q=\delta_{ij}\frac{\pi^2}{3N^2}$, we obtain
\begin{equation}\label{eq:cell-average}
\frac{N^2}{4\pi^2}\int_{Q_k}u(q)\,\d q=u(k)+\frac{\pi^2}{6N^2}\,\Delta u(k)+O\big(MN^{-3}d(k)^{-3}\big)\,.
\end{equation}
For $\rho\geq 10/N$, the number of modes $k\in\cP(\vartheta)$ with $\rho\leq d(k)<2\rho$ is at most $C N^2\rho^2$, as these modes are centers of disjoint squares of area $4\pi^2/N^2$ contained in a region of area at most $C\rho^2$. Summing over dyadic shells gives $\sum_{k\in\cF_\vartheta}d(k)^{-3}\leq CN^3$, so the sum over $\cF_\vartheta$ of the error terms in \cref{eq:cell-average} is $O(M)$. For the Laplacian terms, we first compare $\Delta u(k)$ to its average over $Q_k$: as $|\nabla\Delta u|\leq 3M d^{-3}$, the same count yields
\[\Big|\sum_{k\in\cF_\vartheta}\Delta u(k)-\frac{N^2}{4\pi^2}\int_{\Omega_\vartheta}\Delta u\Big|\leq \sum_{k\in\cF_\vartheta} \frac{\sqrt2\pi}{N}\cdot 24Md(k)^{-3}\leq CMN^2\,,\]
which after multiplication by $\frac{\pi^2}{6N^2}$ is $O(M)$. Finally, $\Omega_\vartheta$ is a finite union of closed squares, $u$ is $C^2$ in a neighborhood of $\Omega_\vartheta$, and the torus has no boundary, so by the divergence theorem $\int_{\Omega_\vartheta}\Delta u=\int_{\partial\Omega_\vartheta}\partial_n u\,\d s$, where the boundary $\partial\Omega_\vartheta$ consists of sides of squares $Q_k$ with $k\in\cF_\vartheta$ that are adjacent to a square $Q_{k'}$ with $d(k')<10/N$. Such $k$ satisfy $d(k)<10/N+2\pi/N<17/N$, hence there are at most $C$ of them, and on each of the at most four sides of $Q_k$ in question, of length $2\pi/N$, we have $|\nabla u|\leq M (d(k)/2)^{-1}\leq MN/5$. Thus $|\int_{\Omega_\vartheta}\Delta u|\leq CM$, which after multiplication by $\frac{\pi^2}{6N^2}\cdot\frac{N^2}{4\pi^2}$ is $O(M)$, concluding the proof.
\end{proof}

\begin{corollary}\label{cor:riemann-sum}
Let $\cK$ and $\delta_0$ be as in \cref{prop:RG-output}. There exist $\delta_3\in(0,\delta_0]$ and a constant $C_3$ depending only on $\cK$ such that, if $\delta\leq\delta_3$, then for all $w\in\cK$, all $N$ large enough and all $\vartheta,\vartheta'\in\Theta_N$:
\begin{enumerate}[(a)]
\item\label{it:tildeZ-comparable} $C_3^{-1}\leq|\tilde Z^0_\vartheta/\tilde Z^0_{\vartheta'}|\leq C_3$;
\item\label{it:D0-vs-free} if $\bar w\in\cK$ is such that $p^+_0(\bar w)=p^+(w,\fg)$ and $|\bar w-w|\leq C_1\delta$, then there exist $c_\fg\in\C$, independent of $\vartheta$, and numbers $\xi_\vartheta\in\C$ with $|\xi_\vartheta|\leq C_3\delta$, such that
\begin{equation}\label{eq:D0-vs-free}
D_0(\vartheta)=e^{N^2c_\fg+\xi_\vartheta}\,D_{\bar w}(\vartheta)\,,\qquad D_{\bar w}(\vartheta):=\prod_{k\in\cP(\vartheta)}\mu_{\bar w}(k)=\det K_{\bar w}(\vartheta)\,,
\end{equation}
and moreover $\big|\prod_{k\in\cP'(\vartheta)}\mu^{(0)}(k)/\mu_{\bar w}(k)\big|=e^{N^2\Re c_\fg+O(\delta)}$, with an $O(\delta)$ uniform in $\vartheta$.
\end{enumerate}
\end{corollary}
\begin{proof}
By \cref{it:RG-weights} of \cref{prop:RG-output}, if $\delta$ is small then for all $w\in\cK$ the function $\mu^{(0)}$ satisfies $c\,d(k)\leq|\mu^{(0)}(k)|\leq C d(k)$ for $d(k)\leq r_0$ and $c\leq|\mu^{(0)}(k)|\leq C$ otherwise, for some $r_0,c,C$ depending only on $\cK$ (indeed $\mu_w$ has this property, with zeros at $p^\pm_0(w)$, which are within $C\delta$ of $p^\pm$, and $\mu^{(0)}$ is a $C\delta$-perturbation of $\mu_w$ in $C^3$ vanishing exactly at $p^\pm$ with invertible differential there). Consequently $u:=\log|\mu^{(0)}|$ satisfies \cref{eq:u-derivative-bounds} with $q^\pm=p^\pm$ and $M=C$. By \cref{lem:riemann-sum}, $\sum_{k\in\cF_\vartheta}u(k)=\frac{N^2}{4\pi^2}\int_{\Omega_\vartheta}u+O(1)$. The modes $k\in\cP(\vartheta)\setminus\cF_\vartheta$ number at most $C$, and for each of them $\frac{N^2}{4\pi^2}\int_{Q_k}u=-\log N+O(1)$ (substitute $q=p^\pm+y/N$ and use that $|\log|y||$ is integrable on bounded sets), while $u(k)=\log d(k)+O(1)=-\log N+O(1)$ for those $k\neq k^\pm_\vartheta$, as $\pi/N\leq d(k)<10/N$. Since $\cP'(\vartheta)$ omits precisely $k^+_\vartheta$ and $k^-_\vartheta$, we arrive at
\[ \log|\tilde Z^0_\vartheta|=\sum_{k\in\cP'(\vartheta)}u(k)=\frac{N^2}{4\pi^2}\int_{(\R/2\pi\Z)^2}\log|\mu^{(0)}|+2\log N+O(1)\,,\]
whose right-hand side does not depend on $\vartheta$ except through the $O(1)$ term. This proves \cref{it:tildeZ-comparable}.

For \cref{it:D0-vs-free}, let $h:=\mu^{(0)}-\mu_{\bar w}$. Then $\|h\|_{C^3}\leq\|\mu^{(0)}-\mu_w\|_{C^3}+\|\mu_w-\mu_{\bar w}\|_{C^3}\leq C\delta$, and $h(p^\pm)=0$ since both $\mu^{(0)}$ and $\mu_{\bar w}$ vanish there (the zeros of $\mu_{\bar w}$ are $\pm p^+_0(\bar w)=\pm p^+(w,\fg)$). Hence $|h(k)|\leq C\delta d(k)$, and as $|\mu_{\bar w}(k)|\geq c\,d(k)$ for all $k$ (uniformly over $\bar w\in\cK$), the function $G:=h/\mu_{\bar w}$ satisfies $|G|\leq C\delta$ and, by the quotient rule, $|\nabla^jG(k)|\leq C\delta d(k)^{-j}$ for $j\leq3$. For $\delta\leq\delta_3$ we have $|G|\leq\frac12$, and $u:=\log(1+G)$ (principal branch) is a single-valued $C^3$ function on the torus minus $\{p^+,p^-\}$ with $|u|\leq C\delta$ and $|\nabla^ju|\leq C\delta d^{-j}$, i.e., \cref{eq:u-derivative-bounds} holds with $M=C\delta$. By \cref{lem:riemann-sum}, $\sum_{k\in\cF_\vartheta}u(k)=\frac{N^2}{4\pi^2}\int_{\Omega_\vartheta}u+O(\delta)$, while the at most $C$ remaining modes contribute $O(\delta)$ to both $\sum_{k\in\cP(\vartheta)}u(k)$ and $\frac{N^2}{4\pi^2}\int u$. Therefore
\[\sum_{k\in\cP(\vartheta)}u(k)=N^2c_\fg+\xi_\vartheta\,,\qquad c_\fg:=\frac{1}{4\pi^2}\int_{(\R/2\pi\Z)^2}u\,,\quad|\xi_\vartheta|\leq C\delta\,,\]
and \cref{eq:D0-vs-free} follows from $\mu^{(0)}(k)=\mu_{\bar w}(k)\,e^{u(k)}$ for every $k$ (both sides vanish if $k\in\{p^+,p^-\}$). The last assertion follows in the same way, since $\cP'(\vartheta)$ is obtained from $\cP(\vartheta)$ by removing two of the modes outside $\cF_\vartheta$, each of which contributes at most $C\delta$ to the Riemann sum.
\end{proof}

\subsubsection*{The non-interacting model at the tuned weights}
For positive weights $\bar w$ and $n\in(N\Z)^2$, let $Z(n)$ denote the number of tilings $\varphi$ of $\T_N$ with $n(\varphi)=n$, and let $Z^{\mathrm{can}}_{\bar w}:=\sum_n\bar w^nZ(n)$ be the partition function of the non-interacting dimer model with weights $\bar w$. By \cref{eq:detK-twist} and \cref{lem:fourier-inversion} with $\Phi\equiv1$, for every sector $n$ with $0<n_\fb,n_\fc<N^2$,
\begin{equation}\label{eq:free-sector-inversion}
\epsilon(n)\,\bar w^nZ(n)=\frac1{N^2}\sum_{\vartheta\in\Theta_N}e^{-i\langle\vartheta,n\rangle}D_{\bar w}(\vartheta)\qquad\mbox{and}\qquad \max_{\vartheta\in\Theta_N}|D_{\bar w}(\vartheta)|\leq Z^{\mathrm{can}}_{\bar w}\,,
\end{equation}
with $D_{\bar w}(\vartheta)$ as in \cref{eq:D0-vs-free}. We use the discrete Gaussian behavior of the instanton component of the non-interacting model in the following form.
\begin{lemma}[{follows from the proof of \cite[Thm.~4]{KSW16}}]\label{lem:free-local-clt}
Let $p$ be in the interior of the simplex, let $p^{(N)}\to p$ with $Np^{(N)}\in\Z^3$ be as in \cref{eq:pN-def}, and let $\bar w_N:=w_0(p^{(N)})$. There exist $c_0>0$ and $N_0$ such that for all $N\geq N_0$,
\[ \bar w_N^{\,n_N}Z(n_N)\geq c_0\, Z^{\mathrm{can}}_{\bar w_N}\,.\]
\end{lemma}
\begin{proof}
Consider first a fixed admissible weight vector $w$, and let $p(w)$ denote its slope, i.e., the point of the simplex with $w_0(p(w))=w$. In \cite{KSW16}, the generating function of the instanton component, that is, $D_w(\vartheta)$ for twists $\vartheta$ of order $1/N$ (which is exactly the range of $\Theta_N$), is computed up to a factor $1+o(1)$ uniformly in the twist, and the resulting law of the sector, in the variables $(k,\ell)$ of \cref{eq:detK}, is a discrete Gaussian whose covariance is bounded above and below and whose center is within bounded distance of the sector $(N^2p_\fb(w),N^2p_\fc(w))$ prescribed by the slope; see \cite[Thm.~3]{KSW16}, of which \cite[Thm.~4]{KSW16} is the explicit version, as well as \cite[Thm.~4]{BdT09} for $\fa=\fb=\fc$ and \cite[Thm.~7]{Dubedat15}. A lattice point within bounded distance of the center of such a discrete Gaussian carries a positive fraction of its mass, which gives the claimed bound along any sequence of $N$ for which $Np(w)\in\Z^3$, provided $w$ does not depend on $N$. Here $\bar w_N$ depends on $N$, so what is needed is the same conclusion uniformly over $w$ in a neighborhood of $w_0(p)$, applied to $w=\bar w_N$, whose slope is $p^{(N)}$. This uniformity is not stated in \cite{KSW16}; it follows from the proof there, which only uses the local structure of the characteristic polynomial at its two zeros, and we do not reproduce the argument here.
\end{proof}

\subsubsection*{The lower bound on the target sector} 
With the notation of \cref{prop:RG-output}, we have the following.
\begin{proposition}\label{prop:sector-lower-bound}
There exist $\delta_2>0$, $c>0$ and, for every finite set $F$ of edges, a constant $C'_F$, such that if $\delta\leq\delta_2$ and $N$ is large enough then, with the weights $w=w_N$ of \cref{lem:tuning},
\begin{equation}\label{eq:sector-lower-bound}
\cZ_\fg(0)=\epsilon_N\,\frac1{N^2}\sum_{\vartheta\in\Theta_N}\cZ_\fg(\vartheta_1,\vartheta_2,0)\geq c\,\cM_N\,,\qquad \cM_N:=w_N^{-n_N}\,e^{N^2\Re\Delta}\,N^{-2}\max_{\vartheta\in\Theta_N}|\tilde Z^0_\vartheta|\,,
\end{equation}
and moreover $|\cZ_\fg(\vartheta_1,\vartheta_2,A)|\leq C'_F\,\cM_N$ for all $\vartheta\in\Theta_N$ and all $A\in\C^F$ with $\|A\|_\infty\leq\rho_F$.
\end{proposition}
\begin{proof}
Throughout the proof, $C$ denotes a constant depending only on $\cK$, and we take $\delta\leq\min(\delta_1,\delta_3)$ and $N$ large enough for \cref{lem:tuning,cor:riemann-sum,lem:free-local-clt} to apply. Since $|\mu^{(0)}(k^\pm_\vartheta)|\leq C/N$, we have $|D_0(\vartheta)|\leq CT_\vartheta$. Together with \cref{eq:Z-vs-G,eq:RG-output,eq:RG-remainder} and $|S_N(A)|\leq C_F$, this immediately gives the upper bound.

For the lower bound, combine \cref{eq:RG-output} at $A=0$ with \cref{eq:D0-vs-free}, which applies by \cref{lem:tuning}, to obtain
\[ \cZ_\fg(\vartheta_1,\vartheta_2,0)=w_N^{-n_N}e^{N^2(\Delta+c_\fg)}\,e^{-i\langle\vartheta,n_N\rangle}D_{\bar w_N}(\vartheta)+E_\vartheta\,.\]
Here the error inside the factor $w_N^{-n_N}e^{N^2\Delta-i\langle\vartheta,n_N\rangle}$ is $R_\vartheta(0)+(1-e^{-\xi_\vartheta})D_0(\vartheta)$, so \cref{eq:RG-remainder}, $|\xi_\vartheta|\leq C_3\delta$ and $|D_0(\vartheta)|\leq CT_\vartheta$ give
\[ |E_\vartheta|\leq C(\delta+N^{-1/2})\cM_N\,.\]
Averaging over $\vartheta\in\Theta_N$ and using \cref{eq:free-sector-inversion} for the sector $n_N$ gives
\[ \frac1{N^2}\sum_{\vartheta\in\Theta_N}\cZ_\fg(\vartheta_1,\vartheta_2,0)=\epsilon_N\,w_N^{-n_N}e^{N^2(\Delta+c_\fg)}\,\bar w_N^{\,n_N}Z(n_N)+\frac1{N^2}\sum_{\vartheta\in\Theta_N}E_\vartheta\,,\]
and hence, by \cref{eq:inversion-Z}, $\cZ_\fg(0)\geq w_N^{-n_N}e^{N^2\Re(\Delta+c_\fg)}\,\bar w_N^{\,n_N}Z(n_N)-C(\delta+N^{-1/2})\cM_N$. It remains to bound the first term from below. By \cref{lem:free-local-clt} and \cref{eq:free-sector-inversion}, $\bar w_N^{\,n_N}Z(n_N)\geq c_0\max_\vartheta|D_{\bar w_N}(\vartheta)|$. To estimate the maximum, note that $|k^+_\vartheta-p^+|<1/N$ if and only if $\vartheta$ lies within distance $1/N$ of the lattice $p^+-\frac{2\pi}N\Z^2$; since $\Theta_N$ is a grid of spacing $2\pi/N^2$ in the fundamental domain $[0,2\pi/N)^2$ of this lattice, the number of such $\vartheta\in\Theta_N$ is $(\frac1{4\pi}+o(1))N^2$, and the same holds for $p^-$. We may therefore pick $\vartheta^*\in\Theta_N$ with $|k^\pm_{\vartheta^*}-p^\pm|\geq1/N$ for both signs. Then $|\mu_{\bar w_N}(k^\pm_{\vartheta^*})|\geq c/N$, and by the last assertion of \cref{cor:riemann-sum} together with \cref{it:tildeZ-comparable},
\begin{align*} \max_\vartheta|D_{\bar w_N}(\vartheta)| &\geq|D_{\bar w_N}(\vartheta^*)|\geq \frac{c^2}{N^2}\Big|\prod_{k\in\cP'(\vartheta^*)}\mu_{\bar w_N}(k)\Big|=\frac{c^2}{N^2}\,|\tilde Z^0_{\vartheta^*}|\,e^{-N^2\Re c_\fg+O(\delta)}\\ &\geq \frac{c^2}{2C_3N^2}\,e^{-N^2\Re c_\fg}\max_\vartheta|\tilde Z^0_\vartheta|\,.\end{align*}
Altogether, $\cZ_\fg(0)\geq\big[\frac{c_0c^2}{2C_3}-C(\delta+N^{-1/2})\big]\cM_N$, which gives \cref{eq:sector-lower-bound} with $c:=\frac{c_0c^2}{4C_3}$ once $\delta\leq\delta_2$ and $N$ is large.
\end{proof}

We can now conclude the proof of \cref{thm:gmt-refinement}.
\begin{proof}[Proof of \cref{thm:gmt-refinement}]
Fix $p$ in the interior of the simplex, let $\cK$ be a compact neighborhood of $w_0(p)$ in the set of admissible weights normalized by $\fa=1$, and let $C$ in \cref{eq:many-g} be large enough that $\delta=\sum_r\|\fg_r\|_\infty\leq\min(\delta_0,\delta_1,\delta_2,\delta_3)$ and that \cref{prop:grassmann_formulation} applies. Fix a finite set $F$ of edges, let $N$ be large, and choose the weights $w_N$ of \cref{lem:tuning}.

\emph{Factorization.} For $\|A\|_\infty\leq\rho_F$, \cref{eq:Z-vs-G,eq:RG-output} and $S_N(0)=0$ give
\begin{align*}
&\cZ_\fg(\vartheta_1,\vartheta_2,A)-e^{S_N(A)}\cZ_\fg(\vartheta_1,\vartheta_2,0)\\
&\qquad=w_N^{-n_N}e^{-i\langle\vartheta,n_N\rangle}e^{N^2\Delta+S_N(A)}\big[R_\vartheta(A)-R_\vartheta(0)\big]\,.
\end{align*}
Its modulus is at most $C'_FN^{-\kappa}\cM_N$ by \cref{eq:RG-remainder} and $|S_N(A)|\leq C_F$. Averaging over $\vartheta\in\Theta_N$ and using \cref{eq:inversion-Z,eq:sector-lower-bound},
\[ \big|\cZ_\fg(A)-e^{S_N(A)}\cZ_\fg(0)\big|\leq C'_FN^{-\kappa}\cM_N\leq \frac{C'_F}{c}N^{-\kappa}\,\cZ_\fg(0)\,,\]
and since $|e^{-S_N(A)}|\leq e^{C_F}$, this is \cref{eq:micro-goal}, in the form $\cZ_\fg(A)/\cZ_\fg(0)=e^{S_N(A)}(1+r_N(A))$ with $\sup_{\|A\|_\infty\leq\rho_F}|r_N(A)|\leq C''_FN^{-\kappa}$.

\emph{Cumulants.} For $N$ large enough that $C''_FN^{-\kappa}\leq\frac12$, the function $A\mapsto\log(\cZ_\fg(A)/\cZ_\fg(0))=S_N(A)+\log(1+r_N(A))$ is analytic on the polydisc $\{\|A\|_\infty\leq\rho_F\}$, and by Cauchy's estimates, for distinct edges $e_1,\ldots,e_j\in F$,
\[ \Big|\partial_{A(e_1)}\cdots\partial_{A(e_j)}\log\big(1+r_N(A)\big)\big|_{A=0}\Big|\leq \rho_F^{-j}\sup_{\|A\|_\infty\leq\rho_F}|\log(1+r_N(A))|\leq 2\rho_F^{-j}C''_FN^{-\kappa}\,.\]
By \cref{eq:cumulant-def}, the left-hand side of this display is the difference between the joint cumulant of $\one_{\{e_1\in\varphi\}},\ldots,\one_{\{e_j\in\varphi\}}$ under the exactly tilted microcanonical measure and $\partial_{A(e_1)}\cdots\partial_{A(e_j)}S_N(0;w_N,\fg)$. The joint cumulants under $\mu_N$ differ from the former by $O_F(e^{-N^c})$, by the hypothesis of the theorem, since all the variables involved are bounded by $1$. Since $w_N\to w_\infty$ by \cref{lem:tuning}, \cref{it:RG-limit} of \cref{prop:RG-output} shows that the joint cumulants of $(\one_{\{e\in\varphi\}})_{e\in F}$ under $\mu_N$ converge to the corresponding joint cumulants under the infinite-volume canonical measure $\mu_\infty$ of \cite[Thm.~1]{GMT20} with weights $w_\infty$ and interaction $\fg$.

\emph{Conclusion.} The law of $(\one_{\{e\in\varphi\}})_{e\in F}\in\{0,1\}^F$ is determined by its joint moments, which are polynomials in the joint cumulants; hence, for every finite $F$, the law of $(\one_{\{e\in\varphi\}})_{e\in F}$ under $\mu_N$ converges to its law under $\mu_\infty$. This is \cref{it:mu_N_limit} of \cref{thm:gmt}, with the limit identified as the canonical infinite-volume measure at the weights $w_\infty=w_\infty(p,\fg)$. The limiting cumulants satisfy the decay \cref{eq:cumulants}, and the \GFF scaling limit \cref{it:phi_GFF_limit} then follows exactly as in the canonical case, see \cite[\S7]{GMT17} and \cite[\S6.6]{GMT20}, which concludes the proof.
\end{proof}

\begin{remark}
In the application to the \SOS model, $n_N=(N\lfloor\theta_1N\rfloor,N\lfloor\theta_2N\rfloor)$, so that $p^{(N)}$ differs from $\theta$ by $O(1/N)$ while $N\theta$ is not an integer vector. The theorem allows more generally $p^{(N)}-p=o(1)$. In both cases, the point of working with the $N$-dependent weights $w_N$, which are tuned to $p^{(N)}$ and not to $p$, is that the target sector would otherwise be at distance $N|p^{(N)}-p|$ from the center of the discrete Gaussian in the variables $(k,\ell)$, and would then carry a vanishing fraction of the mass; the uniformity over the weights in \cref{prop:RG-output} and in the proof of \cref{lem:free-local-clt} is what makes this harmless. The limiting objects, namely $\mu_\infty$ and the parameters $\sigma$ and $\fL$ of the \GFF scaling limit, depend on $(p,\fg)$ only, through $w_\infty(p,\fg)$, and \cref{rem:sigma-L-analytic} applies to them since $w_\infty$ is an analytic function of $\fg$ by the implicit function theorem.
\end{remark}

Let us finally note that since the coefficients of $S_\infty$ are sums over trees, they come with exponential decay in the scale variable, or in other words, polynomial decay in the distance.
 As a consequence, all cumulants of the edge occupation variables have a polynomial decay which, following \cite{GMT17,GMT20}, we can express in terms of the size of the minimal animal connecting them; namely, for all $e_1, \ldots, e_k$ we have (see  \cite[Eqs.~(7.4) and (7.6)]{GMT17})
\begin{equation}\label{eq:cumulants}
    \left| \operatorname{Cumulant}(\one_{e_1}, \ldots, \one_{e_k} )  \right| \leq C_k \operatorname{Animal}(e_1, \ldots, e_k)^{-c}
\end{equation}
for some $C, c > 0$.

\subsection*{Acknowledgments}
We are grateful to Loren Coquille for useful discussions at an early stage of this project.
The research of BL was supported by the grant  \textsc{anr-18-ce40-0033 dimers}. The research of EL was supported by the NSF grant \textsc{dms-2054833}.
We are grateful to two anonymous referees for their two years of effort. In the last round of minor revisions (following the arXiv version of Aug 2026), one of the referees provided a detailed \textsc{ChatGPT-6 Pro} report that identified---and suggested fixes for---various gaps in the proof. The biggest one was in the move from canonical to microcanonical ensembles, which our paper had only sketched out (pp.~97--99 of said arXiv version). The current version includes the full details of that part, as outlined by the \textsc{ChatGPT} report.

\bibliographystyle{abbrv}
\bibliography{tilted_sos}

\end{document}